\documentclass[
a4paper 
,11pt 
,leqno 
,english 
]{scrartcl} 
\usepackage{my-preamble-article} 
\usepackage{my-macros} 
\usepackage{my-specific-macros} 
\usepackage{my-counters-related-preprints} 
\usepackage{enumitem} 
\title{Global behaviour of bistable solutions for gradient systems in one unbounded spatial dimension}
\author{Emmanuel \textsc{Risler}}
\begin{document}
\maketitle
\begin{abstract}
This paper is concerned with parabolic gradient systems of the form 
\[
u_t=-\nabla V (u) + u_{xx}\,,
\]
where the spatial domain is the whole real line, the state variable $u$ is multidimensional, and the potential $V$ is coercive 
at infinity. For such systems, under generic assumptions on the potential, the asymptotic behaviour of every 
\emph{bistable solution} --- that is, every solution close at both ends of space to stable homogeneous
equilibria --- is described. Every such solution approaches, far to the left in space a stacked family of bistable 
fronts travelling to the left, far to the right in space a stacked family of bistable fronts travelling to the 
right, and in between a pattern of profiles of stationary solutions homoclinic or heteroclinic to stable homogeneous equilibria, going slowly away from one another. 
This result pushes one step further the program initiated in the late seventies by Fife and McLeod about the global
asymptotic behaviour of bistable solutions, by extending their results to the case of systems. In the absence of 
maximum principle, the arguments are purely variational, and call upon previous results obtained in companion papers.
\end{abstract}
\nnfootnote{%
}
\thispagestyle{empty} 
\pagestyle{empty}
\hypersetup{pageanchor=false} 
\newpage
\tableofcontents
\newpage
\hypersetup{pageanchor=true} 
\pagestyle{plain}
\setcounter{page}{1}
\section{Introduction}
\label{sec:intro}
This paper deals with the global dynamics of nonlinear parabolic systems of the form
\begin{equation}
\label{init_syst}
u_t=-\nabla V (u) + u_{xx} \,,
\end{equation}
where the time variable $t$ and the space variable $x$ are real, the spatial domain is the whole real line, the function $(x,t)\mapsto u(x,t)$ takes its values in $\rr^d$ with $d$ a positive integer, and the nonlinearity is the gradient of a \emph{potential} function $V:\rr^d\to\rr$, which is assumed to be regular (of class $\ccc^2$) and coercive at infinity (see hypothesis \cref{hyp_coerc} in \vref{subsec:coerc_glob_exist}). 

The main feature of system~\cref{init_syst} is that it can be recast, at least formally, as the gradient flow of an energy 
functional. If $(w,w')$ is a pair of vectors of $\rr^d$, let $w\cdot w'$ and $\abs{w} =\sqrt{w\cdot w}$ denote the usual 
Euclidean scalar product and the usual Euclidean norm, respectively, and let us write simply $w^2$ for $\abs{w}^2$. 
For every function $v:x\mapsto v(x)$ defined on $\rr$ with values in $\rr^d$, its \emph{energy} (or \emph{Lagrangian} or \emph{action}) with respect to system \cref{init_syst} is defined (at least formally) by
\begin{equation}
\label{form_en}
\eee[v] = \int_\rr \Bigl(\frac{1}{2}v_x(x)^2+V\bigl(v(x)\bigr)\Bigr)\, dx
\,.
\end{equation}
Formally, the differential of this functional reads (skipping border terms in the integration by parts)
\[
\begin{aligned}
d\eee[v]\cdot \delta v &= \int_\rr \bigl( v_x \cdot (\delta v)_x + \nabla V(v)\cdot \delta v \bigr) \, dx \\
&= \int_\rr \bigl( - v_{xx} + \nabla V(v) \bigr) \cdot \delta v \, dx
\,.
\end{aligned}
\]
In other words, the (formal) gradient of this functional with respect to the $L^2(\rr,\rr^d)$-scalar product reads
\[
\nabla\eee[v] = \nabla V(v) - v_{xx} 
\,,
\]
and system \cref{init_syst} can formally be rewritten under the form
\[
u_t = - \nabla \eee[u(\cdot,t)]
\,.
\]
Accordingly, if $(x,t)\mapsto u(x,t)$ is a solution of this system, then (formally)
\[
\begin{aligned}
\frac{d}{d t}\eee[u(\cdot,t)] &= d\eee[u(\cdot,t)]\cdot u_t(\cdot,t) \\
&= \bigl\langle \nabla\eee[u(\cdot,t)], u_t(\cdot,t) \bigr\rangle_{L^2(\rr,\rr^d)} \\
&= - \int_\rr u_t(x,t)^2 \, dx \le 0
\,.
\end{aligned}
\]
If system \cref{init_syst} is considered on a \emph{bounded} spatial domain with boundary conditions that preserve this gradient structure, then the integrals above (on this spatial domain) converge, thus the system really --- and not only formally --- is of gradient type. In this case the dynamics is (at least from a  qualitative point of view) fairly well understood, up to a fine description of the global attractor that is compact and made of the unstable manifolds of stationary solutions \cite{Hale_asymptBehavDissipSyst_1988,Temam_infiniteDimDynSyst_1997}. According to LaSalle's principle, every solution approaches the set of stationary solutions, and even a single stationary solution under rather general additional hypotheses \cite{Simon_asymptoticsEvolEqu_1983}. 

If space is the whole real line and the solutions under consideration are only assumed to be bounded, then the gradient structure above is only formal and allows a much richer phenomenology (the full attractor is by the way far from being fully understood in this case, see the introduction of \cite{GallaySlijepcevic_energyFlowFormallyGradient_2001} and references therein). A salient feature is the occurrence of travelling fronts, that is travelling waves connecting homogeneous equilibria at both ends of space. These solutions are known to play a major role in the asymptotic behaviour for ``many'' initial conditions. 

This crucial role of travelling fronts can be viewed, abstractly, as a consequence of another fundamental feature of system~\cref{init_syst}: the fact that a formal gradient structure exists not only in the laboratory frame, but also in every frame travelling at a constant speed. Indeed, for every real quantity $c$, if a function $(x,t)\mapsto u(x,t)$ is a solution of system \cref{init_syst}, then the function $(\xi,t)\mapsto v(\xi,t)$ defined as 
\[
v(\xi,t)=u(x,t)
\quad\text{for}\quad
x = ct + \xi
\]
is a solution of
\begin{equation}
\label{syst_mf}
v_t-c v_\xi=-\nabla V(v)+v_{\xi\xi}
\,. 
\end{equation}
Now, for every function $w:\xi\mapsto w(\xi)$ defined on $\rr$ with values in $\rr^d$, let us define (at least formally) the energy of $w$ with respect to system \cref{syst_mf} by
\begin{equation}
\label{form_en_mf}
\eee_c[w]=\int_{\rr} e^{c\xi} \Bigl( \frac{1}{2}w_\xi(\xi)^2 + V\bigl(w(\xi)\bigr) \Bigr)\, d\xi
\,.
\end{equation}
Formally, the differential of $\eee_c[\dot]$ reads (skipping border terms in the integration by parts)
\[
\begin{aligned}
d\eee_c[w]\cdot \delta w &= \int_\rr e^{c\xi}\bigl( w_\xi \cdot (\delta w)_\xi + \nabla V(w)\cdot \delta w \bigr) \, d\xi \\
&= \int_\rr e^{c\xi} \bigl( - w_{\xi\xi} - c w_\xi + \nabla V(w) \bigr) \cdot \delta w \, d\xi
\,.
\end{aligned}
\]
In other words, the (formal) gradient of this functional with respect to the $L^2$-scalar product with weight $e^{c\xi}$ on functions: $\rr\to\rr^d$ reads 
\[
\nabla_c\eee_c[w] = - w_{\xi} - c w_\xi + \nabla V(w)
\,,
\]
and system \cref{syst_mf} can formally be rewritten as
\begin{equation}
\label{form_grad_mf}
v_t = - \nabla_c \eee_c[v(\cdot,t)]
\,.
\end{equation}
Accordingly, if $(\xi,t)\mapsto v(\xi,t)$ is a solution of system \cref{syst_mf}, then (formally)
\begin{equation}
\label{dt_form_en_mf}
\frac{d}{d t}\eee_c[v(\cdot,t)] = - \int_\rr e^{c\xi} v_t(\xi,t)^2 \, d\xi 
\,.
\end{equation}

This gradient structure has been known for a long time \cite{FifeMcLeod_approachTravFront_1977}, but it is only more recently that it received a more detailed attention from several authors (among which S. Heinze, C. B. Muratov, Th. Gallay, R. Joly, and the author \cite{Heinze_variationalApproachTW_2001,Muratov_globVarStructPropagation_2004,GallayRisler_globStabBistableTW_2007,Risler_globCVTravFronts_2008,GallayJoly_globStabDampedWaveBistable_2009}), and that is was shown that this structure is sufficient (in itself, that is without the use of the maximum principle) to prove results of global convergence towards travelling fronts. These ideas have been applied since in different contexts, to prove either global convergence or just existence results, see for instance \cite{Chapuisat_existenceCurvedFront_2007,ChapuisatJoly_asymptProfilesTravFrontBiolEqu_2010,MuratovNovaga_frontPropIVariational_2008,MuratovNovaga_frontPropIISharpReaction_2008,MuratovNovaga_globExpConvTW_2012,AlikakosKatzourakis_heteroclinicTW_2011,AlikakosFusco_ellipticSystemsPhaseTransType_2018,Luo_globStabDampedWaveEqu_2013,BouhoursNadin_variationalApproachRDForcedSpeedDim1_2015,BouhoursGiletti_extinctSpreadClimateAllee_2016,BouhoursGiletti_spreadVanishMonStabRDEqu_2018,OliverBonafoux_heteroclinicTW1dParabSystDegenerate_2021,OliverBonafoux_TWparabAllenCahn_2021,ChenChienHuang_varApproach3PhaseTWgradSyst_2021,ChenCotiZelati_TWSolAllenCahnEqu_2022,OliverBonafouxRisler_globCVPushedTravFronts_2023}. 

A reasonably wide class of solutions of system \cref{init_syst}, large enough to capture the convergence to travelling fronts while limiting the complexity of the dynamics encountered is made of solutions that are close to homogeneous equilibria at both ends of space, at least for large positive times. Among such solutions the simplest case is that of \emph{bistable} solutions, when these equilibria at both ends of space are stable. In the late seventies, substantial breakthroughs have been achieved by P. C. Fife and J. B. McLeod about the global behaviour of such \emph{bistable} solutions in the \emph{scalar} case ($d$ equals $1$). Their results comprise global convergence towards a bistable travelling front \cite{FifeMcLeod_approachTravFront_1977}, global convergence towards a ``stacked family of bistable travelling fronts'' \cite{FifeMcLeod_phasePlaneDisc_1981}, and finally, in the case of a bistable potential, a rather complete description of the global asymptotic behaviour of all solutions that are close enough, at infinity in space,to the local (non global) minimum point \cite{Fife_longTimeBistable_1979}. Many extensions and generalizations of these results have been achieved since, but mostly for monotone systems or in the scalar case $d$ equals $1$ (using maximum principles and order-preserving properties of the solutions), see \cite{RoquejoffreTermanVolpert_globStabTFStackedFamMon_1996,DucrotGiletti_existenceConvergencePropagatingTerrace_2014,Polacik_propagatingTerracesDynFrontLikeSolRDEquationsR_2020} and references therein. 

The aim of this paper, completing the companion papers \cite{Risler_globCVTravFronts_2008,Risler_globalRelaxation_2016}, is to make a step further in this program, by extending these results to the case of \emph{systems} of the form \cref{init_syst}, and by providing for such systems a complete description of the asymptotic behaviour of every bistable solution (\cref{thm:main} below). 
\section{Assumptions, notation, and statement of the results}
\label{sec:assumpt}
This \namecref{sec:assumpt} presents strong similarities with \cite[\GlobalRelaxationSecAssumptionsNotationStatement]{Risler_globalRelaxation_2016}, where more details and comments can be found. 
\subsection{Semi-flow in uniformly local Sobolev space and coercivity hypothesis}
\label{subsec:coerc_glob_exist}
Let us denote by $X$ the uniformly local Sobolev space $\HoneulofR$. System~\cref{init_syst} defines a local semi-flow in $X$ (see for instance D. B. Henry's book \cite{Henry_geomSemilinParab_1981}). 

As in \cite{Risler_globalRelaxation_2016}, let us assume that the potential function $V:\rr^d\to\rr$ is of class $\ccc^2$ and strictly coercive at infinity in the following sense: 
\begin{gather}
\tag{$\text{H}_\text{coerc}$}
\lim_{R\to+\infty}\quad  \inf_{\abs{u}\ge R}\ \frac{u\cdot \nabla V(u)}{\abs{u}^2} >0
\label{hyp_coerc}
\end{gather}
(or in other words there exists a positive quantity $\varepsilon$ such that the quantity $u\cdot \nabla V(u)$ is greater than or equal to $\varepsilon\abs{u}^2$ as soon as $\abs{u}$ is large enough). 

According to this hypothesis \cref{hyp_coerc}, the semi-flow of system \cref{init_syst} is actually global (see \vref{prop:attr_ball}). Let us denote by $(S_t)_{t\ge0}$ this semi-flow. 

In the following, a \emph{solution of system \cref{init_syst}} will refer to a function 
\[
\rr\times[0,+\infty)\to\rr^d\,, \quad (x,t)\mapsto u(x,t)
\,,
\]
such that the function $u_0:x\mapsto u(x,t=0)$ (initial condition) is in $X$ and $u(\cdot,t)$ equals $(S_t u_0)(\cdot)$ for every nonnegative time $t$. 
\subsection{Minimum points and bistable solutions}
\subsubsection{Minimum points}
Everywhere in this paper, the expression ``minimum point'' denotes a point where a function --- namely the potential $V$ --- reaches a local \emph{or} global minimum value. 
\begin{notation}
Let $\mmm$ denote the set of nondegenerate (local or global) minimum points of $V$:
\[
\mmm=\{u\in\rr^d: \nabla V(u)=0 
\quad\text{and}\quad 
D^2V(u)\text{ is positive definite}\}
\,.
\]
\end{notation}
\subsubsection{Bistable solutions}
Let us recall the following definition, already stated in \cite{Risler_globalRelaxation_2016}.
\begin{definition}[bistable solution]
\label{def_bist}
A solution $(x,t)\mapsto u(x,t)$ of system~\cref{init_syst} is called a \emph{bistable solution} if there are two (possibly equal) points $m_-$ and $m_+$ in $\mmm$ such that both quantities:
\[
\limsup_{x\to-\infty} \abs{u(x,t)-m_-}
\quad\text{and}\quad
\limsup_{x\to+\infty} \abs{u(x,t)-m_+}
\]
go to $0$ as time goes to $+\infty$. More precisely, such a solution is called a \emph{bistable solution connecting $m_-$ to $m_+$} (see \cref{fig:bist_sol}). 
\begin{figure}[!htbp]
\centering
\includegraphics[width=0.6\textwidth]{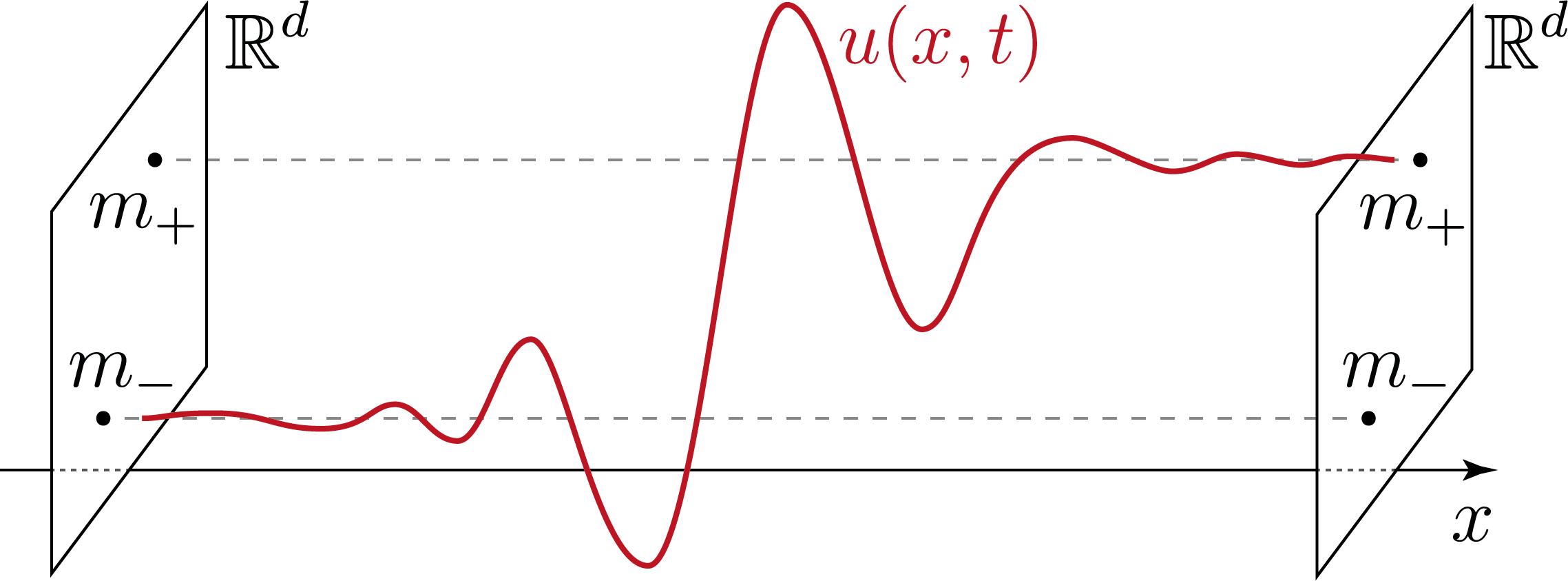}
\caption{A bistable solution connecting $m_-$ to $m_+$.}
\label{fig:bist_sol}
\end{figure}
A function $u_0$ in $X$ is called a \emph{bistable initial condition (connecting $m_-$ to $m_+$)} if the solution of system~\cref{init_syst} corresponding to this initial condition is a bistable solution (connecting $m_-$ to $m_+$).
\end{definition}
Let $m_-$ and $m_+$ denote two (possibly equal) points in $\mmm$.
\begin{notation}
Let
\[
\Xbist(m_-,m_+)
\]
denote the subset of $X$ made of bistable initial conditions connecting $m_-$ to $m_+$.
\end{notation}
By construction, this set is positively invariant under the semi-flow of system \cref{init_syst}; it is in addition nonempty and open in $X$ (for the usual norm on this function space), and contains all functions close enough to $m_-$ and $m_+$ at the ends of space (for proofs see for instance \cite{Risler_globCVTravFronts_2008,Risler_globalRelaxation_2016}).

The aim of this paper is to study the asymptotic behaviour of solutions belonging to the sets $\Xbist(m_-,m_+)$. The description of this asymptotic behaviour involves two kinds of particular solutions: stationary solutions connecting (stable) equilibria and (bistable) fronts travelling at a constant speed. 
\subsection{Stationary solutions, travelling fronts, terraces, asymptotic pattern}
\subsubsection{Stationary solutions and travelling fronts}
Let $c$ be a real quantity. A function
\[
\phi:\rr\to\rr^d,
\quad \xi\mapsto\phi(\xi)
\]
is the profile of a wave travelling at the speed $c$ (or of a standing wave if $c$ vanishes) for system \cref{init_syst} if the function  $(x,t)\mapsto \phi(x-ct)$ is a solution of this system, that is if $\phi$ is a solution of the differential system
\begin{equation}
\label{syst_trav_front}
\phi''=-c\phi'+\nabla V(\phi) 
\,.
\end{equation}
This system can be viewed as a damped oscillator (or a conservative oscillator if $c$ vanishes) in the potential $-V$, the speed $c$ playing the role of the damping coefficient. Its solutions may blow up in finite time, but only global solutions will be considered or encountered along the paper. 
\begin{notation}
If $m_-$ and $m_+$ are two points of $\mmm$ and $c$ is a real quantity, let $\Phi_c(m_-,m_+)$ denote the set of \emph{nonconstant} global solutions of system \cref{syst_trav_front} connecting $m_-$ to $m_+$. With symbols, 
\[
\begin{aligned}
\Phi_c(m_-,m_+) = \bigl\{ &
\phi:\rr\to\rr^d : 
\phi \text{ is a \emph{nonconstant} global solution of system \cref{syst_trav_front}}
\\
& \text{and}\quad\phi(\xi)\xrightarrow[\xi\to -\infty]{} m_-
\quad\text{and}\quad
\phi(\xi)\xrightarrow[\xi\to +\infty]{} m_+
\bigr\}
\,.
\end{aligned}
\]
And, if the quantity $c$ is positive, let $\Phi_c(m_+)$ denote the set of \emph{nonconstant} global and bounded solutions of system \cref{syst_trav_front} converging to $m_+$ at the right end of space (in other words, the set of profiles of bounded waves travelling at the speed $c$ and ``invading'' $m_+$). With symbols, 
\[
\begin{aligned}
\Phi_c(m_+) = \bigl\{ &
\phi:\rr\to\rr^d : 
\phi \text{ is a \emph{nonconstant} global solution of system \cref{syst_trav_front}}
\\
& \text{and}\quad
\sup_{\xi\in\rr}\abs{\phi(\xi)}<+\infty
\quad\text{and}\quad
\phi(\xi)\xrightarrow[\xi\to +\infty]{} m_+
\bigr\}
\,.
\end{aligned}
\]
\end{notation}
Let us make some comments about these sets and this notation. 
\begin{itemize}
\item The notation ``$\phi$'' and ``$\Phi$'' refers to the concept of ``front''. 
\item If $c$ is positive, then according to LaSalle's principle every function $\xi\mapsto \phi(\xi)$ belonging to $\Phi_c(m_+)$ must approach the set of critical points (but not necessarily a single critical point) of $V$ as $\xi$ goes to $-\infty$ (see assertion \cref{item:minus_infty_asymptotics_tw} of \cref{lem:asympt_behav_tw_2}). 
\item If $\phi$ is an element of some set $\Phi_c(m_-,m_+)$, then it follows from system \cref{syst_trav_front} that
\begin{equation}
\label{V_of_u_plus_minus_V_of_u_minus}
V(m_+)-V(m_-) = c \int_{\rr}\phi'(\xi)^2 \, d\xi
\,,
\end{equation}
so that if $c$ is nonzero then $m_-$ and $m_+$ differ, and in this case $\phi$ is indeed the profile of a travelling front. Since its asymptotic values $m_-$ and $m_+$ belong to $\mmm$, this front is qualified as \emph{bistable}. 
\item If conversely $\phi$ is an element of some set $\Phi_0(m_-,m_+)$ (for a null speed), then $V(m_-)$ equals $V(m_+)$ and $m_-$ and $m_+$ may be equal; in such a case, $\phi$ should rather be called a ``pulse'' than a ``front'', but for convenience and homogeneity purposes, the notation $\Phi_0(m_-,m_+)$ and $\phi$ will be maintained also in this case. 
\end{itemize}
\subsubsection{Propagating terrace of bistable travelling fronts}
\label{subsubsec:def_prop_terrace}
This \namecref{subsubsec:def_prop_terrace} and the two next ones are devoted to several definitions. Their purpose is to enable a compact formulation of the main result of this paper (\cref{thm:main} below). Some comments on the terminology and related references are given at the end of \cref{subsubsec:def_stand_terrace}.
\begin{figure}[!htbp]
\centering
\includegraphics[width=0.8\textwidth]{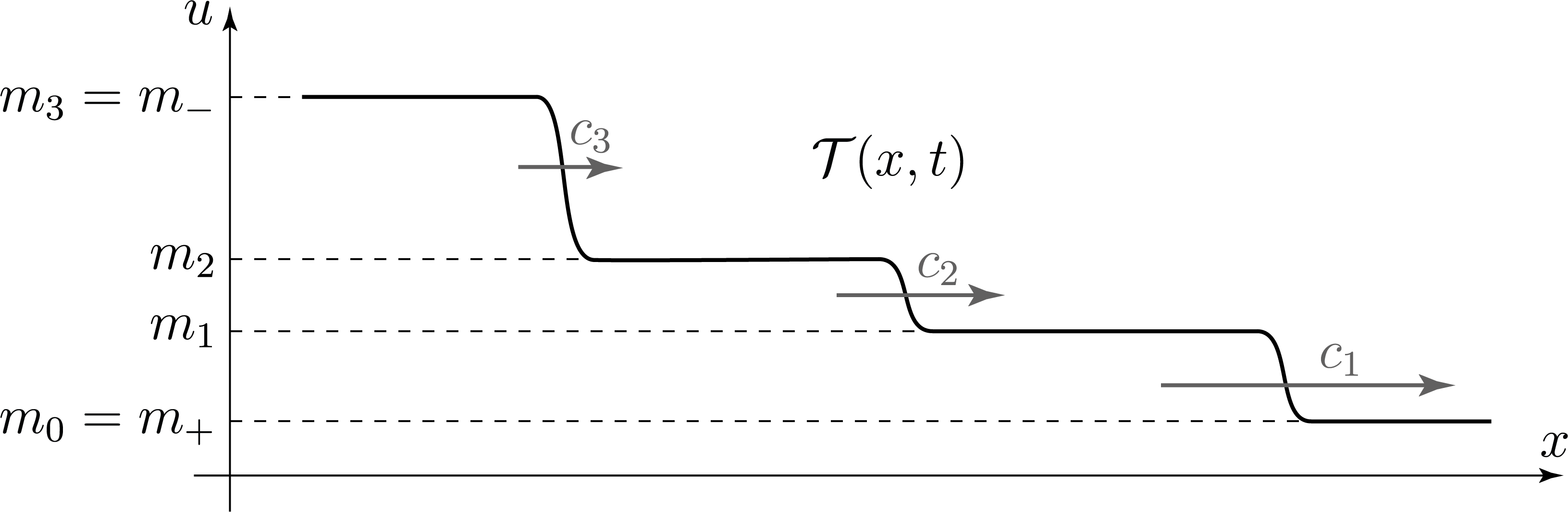}
\caption{Propagating terrace of (bistable) fronts travelling to the right.}
\label{fig:prop_terrace}
\end{figure}
\begin{definition}[propagating terrace of bistable travelling fronts, \cref{fig:prop_terrace}]
Let $m_-$ and $m_+$ be two points of $\mmm$ (satisfying $V(m_-)\le V(m_+)$). A function 
\[
\ttt : \rr\times[0,+\infty)\to\rr^d,\quad (x,t)\mapsto \ttt(x,t)
\]
is called a \emph{propagating terrace of bistable fronts travelling to the right, connecting $m_-$ to $m_+$,} if there exists a nonnegative integer $q$ such that:
\begin{enumerate}
\item if $q$ equals $0$, then $m_-=m_+$ and, for every real quantity $x$ and every nonnegative time $t$, 
\[
\ttt(x,t)=m_-=m_+
\,;
\]
\item if $q$ equals $1$, then there exist:
\begin{itemize}
\item a positive quantity $c_1$,
\item and a function $\phi_1$ in $\Phi_c(m_-,m_+)$ (that is, the profile of a bistable front travelling at the speed $c_1$ and connecting $m_-$ to $m_+$), 
\item and a $\ccc^1$-function $t\mapsto x_1(t)$, defined on $[0,+\infty)$, and satisfying $x_1'(t)\to c_1$ as time goes to $+\infty$, 
\end{itemize}
such that, for every real quantity $x$ and every nonnegative time $t$, 
\[
\ttt(x,t)=\phi_1\bigl(x-x_1(t)\bigr)
\,;
\]
\label{item:def_propagating_terrace_q_equals_one}
\item if $q$ is not smaller than $2$, then there exists $q-1$ points $m_1,\dots,m_{q-1}$ in $\mmm$, satisfying (if $m_+$ is denoted by $m_0$ and $m_-$ by $m_q$)
\[
V(m_0)>V(m_1)>\dots>V(m_q)
\,,
\]
and there exist $q$ positive quantities $c_1$, …, $c_q$ satisfying
\[
c_1\ge\dots\ge c_q
\,,
\]
and for each integer $i$ in $\{1,\dots,q\}$, there exist:
\begin{itemize}
\item a function $\phi_i$ in $\Phi_{c_i}(m_i,m_{i-1})$ (that is, the profile of a bistable front travelling at the speed $c_i$ and connecting $m_i$ to $m_{i-1}$),
\item and a $\ccc^1$-function $t\mapsto x_i(t)$, defined on $[0,+\infty)$, and satisfying $x_i'(t)\to c_i$ as time goes to $+\infty$,
\end{itemize}
such that, for every integer $i$ in $\{1,\dots,q-1\}$, 
\[
x_{i+1}(t)-x_i(t)\to +\infty 
\quad\text{as}\quad
t\to +\infty
\,,
\]
and such that, for every real quantity $x$ and every nonnegative time $t$, 
\[
\ttt(x,t) = m_0 + \sum_{i=1}^q \Bigl[\phi_i\bigl(x-x_i(t)\bigr)-m_{i-1}\Bigr]
\,.
\]
\label{item:def_propagating_terrace_q_larger_than_one}
\end{enumerate}
\end{definition}
\begin{remark}
Item \cref{item:def_propagating_terrace_q_equals_one} may have been omitted in this definition, since it boils down to item \cref{item:def_propagating_terrace_q_larger_than_one} with $q$ equals $1$. 
\end{remark} 
A \emph{propagating terrace of bistable fronts travelling to the left} may be defined similarly. 
\subsubsection{Standing terrace of bistable stationary solutions}
\label{subsubsec:def_stand_terrace}
The next three definitions deal with stationary solutions. They are identical to those of \cite{Risler_globalRelaxation_2016}.
\begin{figure}[!htbp]
\centering
\includegraphics[width=\textwidth]{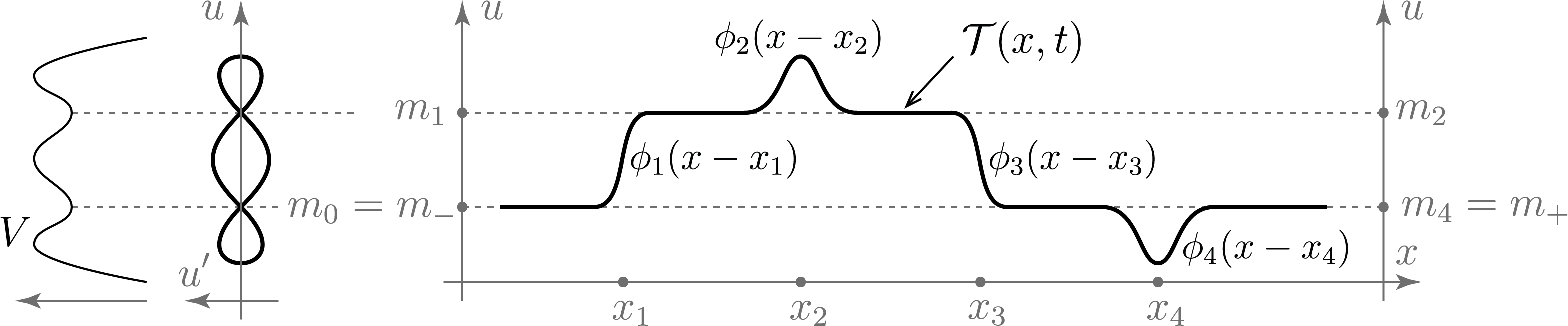}
\caption{Standing terrace (with four items, $q=4$).}
\label{fig:standing_terrace}
\end{figure}
\begin{definition}[standing terrace of bistable stationary solutions, \cref{fig:standing_terrace}]
Let $\valueOfV$ be a real quantity and let $m_-$ and $m_+$ be two points of $\mmm$ such that $V(m_-)=V(m_+)=\valueOfV$. A function
\[
\ttt : \rr\times[0,+\infty)\to\rr^d,\quad (x,t)\mapsto \ttt(x,t)
\]
is called a \emph{standing terrace of bistable stationary solutions, connecting $m_-$ to $m_+$,} if there exists a nonnegative integer $q$ such that:
\begin{enumerate}
\item if $q$ equals $0$, then $m_-=m_+$ and, for every real quantity $x$ and every nonnegative time $t$, 
\[
\ttt(x,t)=m_-=m_+
\,;
\]
\item if $q=1$, then there exist:
\begin{itemize}
\item a bistable stationary solution $\phi_1$ connecting $m_-$ to $m_+$,
\item and a $\ccc^1$-function $t\mapsto x_1(t)$ defined on $[0,+\infty)$ and satisfying $x_1'(t)\to0$ as time goes to $+\infty$,
\end{itemize}
such that, for every real quantity $x$ and every nonnegative time $t$, 
\[
\ttt(x,t) = \phi_1\bigl( x-x_1(t)\bigr)
\,;
\]
\label{item:def_standing_terrace_q_equals_one}
\item if $q$ is not smaller than $2$, then there exist $q-1$ (not necessarily distinct) points $m_1,\dots,m_{q-1}$ in $\mmm$, all in the level set $V^{-1}(\{\valueOfV\})$, and if $m_-$ is denoted by $m_0$ and $m_+$ by $m_q$, then for each integer $i$ in $\{1,\dots,q\}$, there exist:
\begin{itemize}
\item a bistable stationary solution $\phi_i$ connecting $m_{i-1}$ to $m_i$,
\item and a $\ccc^1$-function $t\mapsto x_i(t)$ defined on $[0,+\infty)$ and satisfying $x_i'(t)\to0$ as time goes to $+\infty$,
\end{itemize}
such that, for every integer $i$ in $\{1,\dots,q-1\}$,
\[
x_{i+1}(t)-x_i(t)\to +\infty 
\quad\text{as}\quad
t\to +\infty
\,,
\]
and such that, for every real quantity $x$ and every nonnegative time $t$, 
\[
\ttt(x,t) = m_0 + \sum_{i=1}^q \Bigl[\phi_i\bigl(x-x_i(t)\bigr)-m_{i-1}\Bigr]
\,.
\]
\label{item:def_standing_terrace_q_larger_than_one}
\end{enumerate}
\end{definition}
\begin{remark}
Once again item \cref{item:def_standing_terrace_q_equals_one} may have been omitted in this definition, since it boils down to item \cref{item:def_standing_terrace_q_larger_than_one} with $q$ equals $1$. 
\end{remark} 
The terminology ``propagating terrace'' was introduced by A. Ducrot, T. Giletti, and H. Matano in \cite{DucrotGiletti_existenceConvergencePropagatingTerrace_2014} (and subsequently used by several other authors \cite{Polacik_propagatingTerracesAsymptOneDimSym_2017,Polacik_propTerracesProofGibbonsConj_2016,GilettiRossi_pulsatingSolMultBistMultiStab_2019,MatanoPolacik_dynNonnegSolOneDimRDII_2020,Polacik_propagatingTerracesDynFrontLikeSolRDEquationsR_2020,GilettiMatano_existenceUniquenessPropTerr_2020,PauthierRademacherU_WeakStrongInteractKinks_2021}) to denote a stacked family of travelling fronts in a (scalar) reaction-diffusion equation. This led the author to keep the same terminology in the present context, and to introduce the term ``standing terrace'' for sake of homogeneity. Those terminologies are convenient to denote objects that would otherwise require a long description. They are also used in the companion papers \cite{Risler_globalBehaviourHyperbolicGradient_2017,Risler_globalBehaviourRadiallySymmetric_2017}.

The author hopes that these advantages balance some drawbacks of this terminological choice. Like the fact that the word ``terrace'' is probably more relevant in the scalar case $d$ equals $1$ (see the pictures in \cite{DucrotGiletti_existenceConvergencePropagatingTerrace_2014,Polacik_propagatingTerracesDynFrontLikeSolRDEquationsR_2020}) than in the more general case of systems considered here. Or the fact that the definitions above and in \cite{Risler_globalRelaxation_2016} are different from the original definition of \cite{DucrotGiletti_existenceConvergencePropagatingTerrace_2014} in that they involve not only the profiles of particular (standing or travelling) solutions, but also their positions (denoted above by $x_i(t)$). 

To finish, observe that in the present context terraces are only made of bistable solutions, by contrast with the propagating terraces introduced and used by the authors cited above; that (still in the present context) terraces are approached by solutions but are (in general) not solutions themselves; and that a (standing or propagating) terrace may be nothing but a single stable homogeneous equilibrium (if $q$ equals $0$) or may involve a single travelling front or a single inhomogeneous stationary solution (if $q$ equals $1$). 
\subsubsection{Energy of a bistable stationary solution and of a standing terrace}
\label{subsubsec:def_energy_stat_sol_stand_terrace}
\begin{definition}[energy of a bistable stationary solution]
If $x\mapsto u(x)$ is a bistable stationary solution connecting two points $m_-$ and $m_+$ of $\mmm$ (in this case $V(m_-)$ must equal $V(m_+)$), let us call \emph{energy of $u$}, and let us denote by $\eee[u]$, the quantity:
\[
\eee[u] = \int_{\rr}\Bigl(\frac{1}{2}\abs{u'(x)} ^2+V\bigl(u(x)\bigr)-V(m_\pm)\Bigr)\, dx
\,.
\]
Observe that this integral converges, since $u(x)$ approaches its limits $m_-$ and $m_+$ at both ends of space at an exponential rate. 
\end{definition}
\begin{definition}[energy of a standing terrace]
For every standing terrace $\ttt$ of bistable stationary solutions, let us call \emph{energy of $\ttt$}, and let us denote by $\eee[\ttt]$, the quantity defined (with the notation of the two definitions above) as follows:
\begin{enumerate}
\item if $q$ equals $0$, then $\eee[\ttt]=0$,
\item if $q$ equals $1$, then $\eee[\ttt]=\eee[\phi_1]$,
\item if $q$ is not smaller than $2$, then $\eee[\ttt]=\sum_{i=1}^q\eee[\phi_i]$.
\end{enumerate}
\end{definition}
\subsubsection{Bistable asymptotic pattern}
\label{subsubsec:def_asympt_patt}
\begin{figure}[!htbp]
\centering
\includegraphics[width=\textwidth]{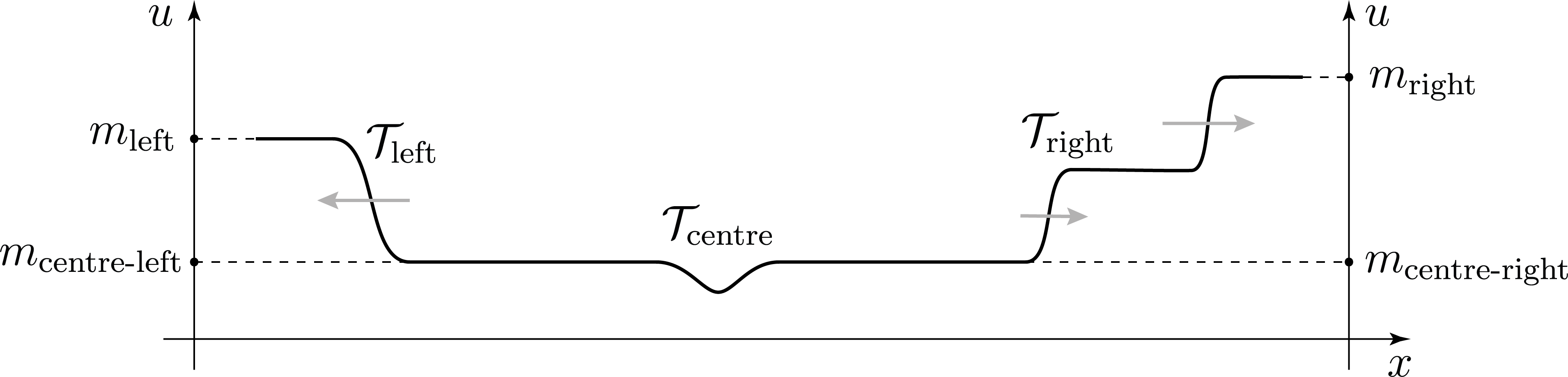}
\caption{Bistable asymptotic pattern.}
\label{fig:bist_asympt_pattern}
\end{figure}
\begin{definition}[bistable asymptotic pattern, \cref{fig:bist_asympt_pattern}]
\label{def:bistable_asympt_pattern}
Let $\mLeft$ and $\mRight$ be two points of $\mmm$. A function
\[
\ppp : \rr\times[0,+\infty)\to\rr^d,\quad (x,t)\mapsto \ppp(x,t)
\]
is called a \emph{bistable asymptotic pattern connecting $\mLeft$ to $\mRight$} if there exist:
\begin{itemize}
\item two points $\mLeftBehind$ and $\mRightBehind$ in $\mmm$, belonging to the same level set of $V$,
\item and a propagating terrace $\tttLeft$ of bistable fronts travelling to the left, connecting $\mLeft$ to $\mLeftBehind$, 
\item and a standing terrace $\tttCentre$ of bistable stationary solutions, connecting $\mLeftBehind$ to $\mRightBehind$, 
\item and a propagating terrace $\tttRight$ of bistable fronts travelling to the right, connecting $\mRightBehind$ to $\mRight$, 
\end{itemize} 
such that, for every real quantity $x$ and for every nonnegative time $t$, 
\[
\ppp(x,t) = \bigl[ \tttLeft(x,t) - \mLeftBehind \bigr] + \tttCentre(x,t) + \bigl[ \tttRight(x,t) - \mRightBehind \bigr]
\,.
\]
\end{definition}
The main result of this paper (\cref{thm:main} below) states that, under some generic assumptions on the potential $V$, every bistable solution approaches such a bistable asymptotic pattern. Results of the the same flavour have recently been obtained by H. Matano and P. Poláčik in the scalar case $d$ equals $1$ (under weaker assumptions otherwise, and by completely different methods specific to the scalar case),  \cite{MatanoPolacik_dynNonnegSolOneDimRDII_2020,Polacik_propagatingTerracesDynFrontLikeSolRDEquationsR_2020} (compare \cite[Figure~1]{MatanoPolacik_dynNonnegSolOneDimRDII_2020} and \cite[Figure~1.1]{Polacik_propagatingTerracesDynFrontLikeSolRDEquationsR_2020} with \cref{fig:bist_asympt_pattern} above). 
\subsection{Generic hypotheses on the potential}
\label{subsec:generic_assupmt_pot}
\subsubsection{Escape distance of a minimum point}
\label{subsubsec:Escape_dist}
\begin{notation}
For every $u$ in $\rr^d$, let $\sigma\bigl(D^2V(u)\bigr)$ denote the spectrum (the set of eigenvalues) of the Hessian matrix of $V$ at $u$, and let $\eigVmin(u)$ denote the minimum of this spectrum:
\begin{equation}
\label{def_eigVmin_of_u}
\eigVmin(u) = \min \Bigl(\sigma\bigl(D^2V(u)\bigr)\Bigr)
\,.
\end{equation}
\end{notation}
\begin{definition}[Escape distance of a nondegenerate minimum point]
For every $m$ in $\mmm$, let us call \emph{Escape distance of $m$}, and let us denote by $\dEsc(m)$, the supremum of the set
\begin{equation}
\label{set_for_definition_Escape_distance}
\Bigl\{\delta \in[0,1]: \text{ for all } u \text{ in } \rr^d \text{ satisfying } \abs{u-m}\le \delta, \quad\eigVmin(u) \ge\frac{1}{2} \eigVmin(m) \Bigr\}
\,.
\end{equation}
\end{definition}
Since the quantity $\eigVmin(u)$ varies continuously with $u$, this Escape distance $\dEsc(m)$ is positive (thus in $(0,1]$). In addition, for all $u$ in $\rr^d$ such that $\abs{u-m}$ is not larger than $\dEsc(m)$, the following inequality holds:
\begin{equation}
\label{property_dEsc}
\eigVmin(u) \ge\frac{1}{2} \eigVmin(m)
\,.
\end{equation}
This ``Escape'' distance will be used in two different ways. 
\begin{enumerate}
\item To ``track'' the position in space where a solution ``escapes'' a neighbourhood of $m$ (this position is called ``leading edge'' by Muratov in a framework including monostable invasion \cite{Muratov_globVarStructPropagation_2004,MuratovNovaga_globExpConvTW_2012,MuratovZhong_thresholdSymSol_2013}). The reason for the capital letter ``E'' in ``Esc'' is to make a difference with another escape distance ``$\desc(m)$'' that will be required later (see definition \vref{def_desc}).
\item To normalize the bistable stationary solutions with respect to translation invariance (in the next \namecref{subsubsec:breakup}).
\end{enumerate}
\subsubsection{Breakup of space translation invariance for stationary solutions and travelling fronts}
\label{subsubsec:breakup}
For every real quantity $c$, for every ordered pair $(m_-,m_+)$ of points of $\mmm$, and for every function $\phi$ in $\Phi_c(m_-,m_+)$, 
\[
\sup_{\xi \in\rr}\abs{\phi(\xi)-m_-}>\dEsc(m_-)
\quad\text{and}\quad
\sup_{\xi \in\rr}\abs{\phi(\xi)-m_+}>\dEsc(m_+)
\]
(assertion \cref{item:escape_spatial_asymptotics_tw} of \vref{lem:asympt_behav_tw_2}). See \cref{fig:esc_distance}. 
\begin{figure}[!htbp]
\centering
\includegraphics[width=0.6\textwidth]{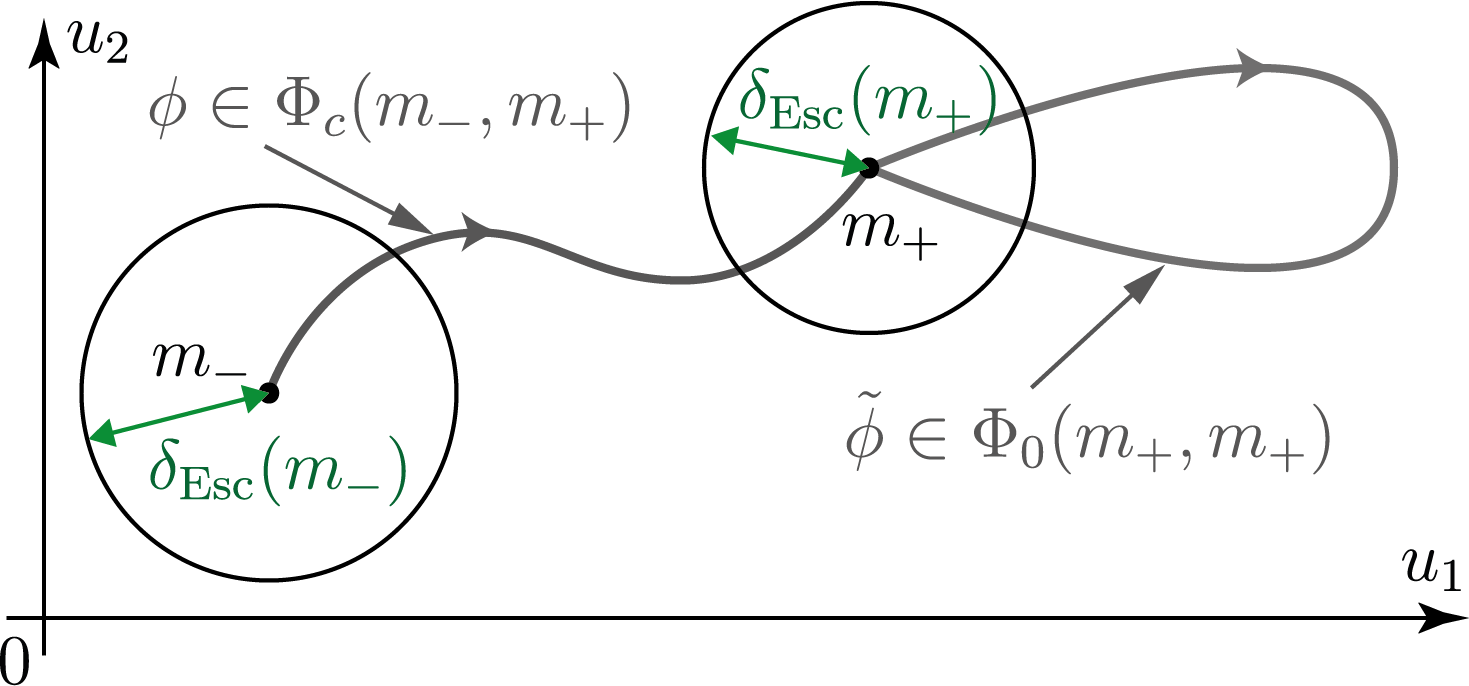}
\caption{Every function in $\Phi_{c}(m_-,m_+)$ escapes at least at distance $\dEsc(m_-)$ of $m_-$ and at distance $\dEsc(m_+)$ of $m_+$; every function in $\Phi_0(m_+,m_+)$ escapes at least at distance $\dEsc(m_+)$ of $m_+$.}
\label{fig:esc_distance}
\end{figure}
This provides a way to pick a representative among the family of all translates of $\phi$, by demanding that, say, the translate be exactly at distance $\dEsc(m_+)$ of his right-end limit $m_+$ at $\xi=0$, and closer for every positive $\xi$ (see \cref{fig:norm_stat}). 
\begin{figure}[!htbp]
\centering
\includegraphics[width=0.75\textwidth]{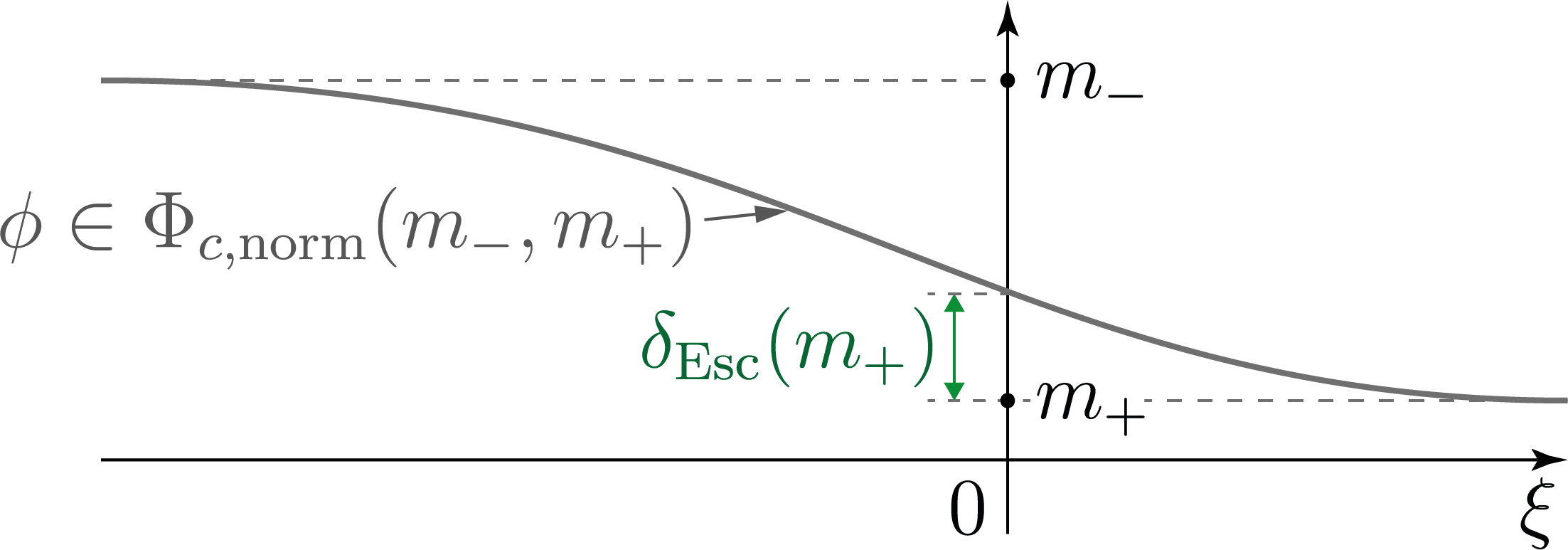}
\caption{Normalized (standing or travelling) bistable front.}
\label{fig:norm_stat}
\end{figure}
Here is a more formal definition. For $c$ in $\rr$ and $(m_-, m_+)$ in $\mmm^2$, let us introduce the set of \emph{normalized profiles of bistable fronts/stationary solutions connecting $m_-$ to $m_+$}:
\begin{equation}
\label{def_norm}
\begin{aligned}
\PhicNorm{c}(m_-,m_+) = \bigl\{ & \phi\in\Phi_{c}(m_-,m_+): \abs{\phi(0)-m_+} =\dEsc(m_+)\quad \text{and}\\
& \abs{\phi(\xi)-m_+} <\dEsc(m_+)\quad\text{for all}\quad \xi \text{ in }(0,+\infty)\bigr\} 
\,.
\end{aligned}
\end{equation}
And if $c$ is positive, let us introduce the set of \emph{normalized profiles of bounded waves travelling at the speed $c$ and ``invading'' $m_+$}:
\[
\begin{aligned}
\PhicNorm{c}(m_+) = \bigl\{ & \phi\in\Phi_{c}(m_+): \abs{\phi(0)-m_+} =\dEsc(m_+)\quad \text{and}\\
& 
\abs{\phi(\xi)-m_+} <\dEsc(m_+)\quad\text{for all}\quad \xi \text{ in } (0,+\infty)\bigr\} 
\,. 
\end{aligned}
\]
\subsubsection{Statement of the generic hypotheses}
\label{subsubsec:gen_hyp_V}
The main result stated in the next \cref{subsec:main_result} requires a number of generic hypotheses on the potential $V$, that will now be stated. A formal proof of the genericity of these hypotheses is provided in \cite{JolyRisler_genericTransversalityTravStandFrontsPulses_2023}. 
\begin{description}
\item[\hypOnlyBistLabel]\hypertarget{hypOnlyBist} Every nonconstant bounded wave travelling at a nonzero speed and invading a stable equilibrium (a point of $\mmm$) is a bistable travelling front. With symbols, for every $m_+$ in $\mmm$ and every positive quantity $c$, 
\[
\begin{aligned}
\Phi_c(m_+) &= \bigcup_{m_-\in\mmm} \Phi_c(m_-,m_+) \,,\\
\text{or equivalently}\quad
\PhicNorm{c}(m_+) &= \bigcup_{m_-\in\mmm} \PhicNorm{c}(m_-,m_+)
\,.
\end{aligned}
\]
\end{description}
In the next two hypotheses, the subscript ``disc'' refers to the concept of ``discontinuity'' or ``discreteness''. 
\begin{description}
\item[\hypDiscVelLabel]\hypertarget{hypDiscVel} For every $m_+$ in $\mmm$, the set:
\[
\bigl\{ c\text{ in }(0,+\infty) : \Phi_c(m_+)\not=\emptyset \bigr\} 
\]
has an empty interior. 
\item[\hypDiscFrontLabel]\hypertarget{hypDiscFront} For every point $m_+$ in $\mmm$ and every real quantity $c$, the set
\[
\bigl\{ \bigl(\phi(0),\phi'(0)\bigr) : \phi\in\PhicNorm{c}(m_+) \bigr\}
\]
is totally discontinuous --- if not empty --- in $\rr^{2d}$. That is, its connected components are singletons. Equivalently, the set $\PhicNorm{c}(m_+)$ is totally disconnected for the topology of compact convergence (uniform convergence on compact subsets of $\rr$).
\end{description}
The next hypothesis will be required to ensure that the number of travelling fronts involved in the asymptotic behaviour of a bistable solution is finite: 
\begin{description}
\item[\hypCriticalValuesLabel]\hypertarget{hypCriticalValues} The set of \emph{critical values} of $V$, that is the set
\[
\bigl\{V(u) : u\in\rr^d\text{ and }\nabla V(u)=0 \bigr\} 
\,,
\]
is finite. 
\end{description}
The next hypothesis will be required in order to apply the relaxation results of \cite{Risler_globalRelaxation_2016} (for the relaxation of the solution behind the travelling fronts).
\begin{description}
\item[\hypOnlyMinLabel]\hypertarget{hypOnlyMin} Every critical point of $V$ that belongs to the same level set as a point of $\mmm$ is itself in $\mmm$. 
\end{description}
In other words, for all points $u_1$ and $u_2$ in $\rr^d$,
\[
\Bigl[
\nabla V (u_1) = \nabla V (u_2) = 0 
\ \text{and}\ 
V(u_1) = V(u_2)
\ \text{and}\ 
D^2V(u_1)>0
\Bigr]
\implies
D^2V(u_2)>0
\,.
\]
Finally, let us us call \cref{hyp_gen} the union of these five generic hypotheses:
\begin{gather}
\tag{G}
\text{\textup{(\hyperlink{hypOnlyBist}{\hypOnlyBistRef})} and \textup{(\hyperlink{hypDiscVel}{\hypDiscVelRef})} and \textup{(\hyperlink{hypDiscFront}{\hypDiscFrontRef})} and \textup{(\hyperlink{hypCriticalValues}{\hypCriticalValuesRef})} and \textup{(\hyperlink{hypOnlyMin}{\hypOnlyMinRef})}}.
\label{hyp_gen}
\end{gather}
\subsection{Main result: global asymptotic behaviour}
\label{subsec:main_result}
\begin{theorem}[global asymptotic behaviour]
\label{thm:main}
Let $V$ denote a function in $\ccc^2(\rr^d,\rr)$ satisfying the coercivity hypothesis \cref{hyp_coerc} and the generic hypotheses \cref{hyp_gen}. Then, for every bistable solution $(x,t)\mapsto u(x,t)$ of system \cref{init_syst}, there exists a bistable asymptotic pattern $\ppp$ such that
\[
\sup_{x\in\rr}\abs{u(x,t)-\ppp(x,t)}\to 0
\quad\text{as}\quad
t\to + \infty
\,.
\]
\end{theorem}
\subsection{Additional results}
%
\subsubsection{Residual asymptotic energy}
%
Here is an additional conclusion to \cref{thm:main}. 
\begin{proposition}[residual asymptotic energy]
\label{prop:residual_asympt_en}
Assume that the assumptions of \cref{thm:main} hold. With the notation of this theorem, if $\tttCentre$ denotes the standing terrace involved in $\ppp$ and if $\valueOfVcentre$ denotes the value taken by $V$ at each of the two points of $\mmm$ connected by $\tttCentre$, then, for every small enough positive quantity $\varepsilon$, 
\[
\int_{-\varepsilon t}^{\varepsilon t} \Bigl( \frac{1}{2}u_x(x,t)^2 + V\bigl(u(x,t)\bigr) - \valueOfVcentre\Bigr) \, dx \to \eee[\tttCentre]
\quad\text{as}\quad
t\to+\infty
\,.
\]
\end{proposition}
The quantity $\eee[\tttCentre]$ may be called the \emph{residual asymptotic energy} of the solution. 
\subsubsection{Regularity of the correspondence between a solution and its asymptotic pattern}
\begin{notation}
Let
\[
\Xbist(\mmm) = \bigsqcup_{(m_-,m_+)\in\mmm^2} \Xbist(m_-,m_+)
\,.
\]
For $u_0$ in $\Xbist(\mmm)$, if $(x,t)\mapsto u(x,t)$ denotes the corresponding solution, using the notation of \cref{thm:main} and of \cref{def:bistable_asympt_pattern} (definition of a bistable asymptotic pattern), let: 
\begin{itemize}
\item $\qLeft$ denote the number of items involved in the left-propagating terrace $\tttLeft$, 
\item $\qRight$ denote the number of items involved in the right-propagating terrace $\tttRight$, 
\item $\cOneLeft$ denote the real quantity defined as 
\begin{itemize}
\item if $\qLeft$ equals $0$ then $\cOneLeft=0$,
\item if $\qLeft$ is not smaller than $1$ then $\cOneLeft$ is the speed of the ``first'' travelling front involved in the left-propagating terrace $\tttLeft$ (the one invading $m_-$), with the convention that $\cOneLeft$ is positive (the velocity of the front is $-\cOneLeft$), 
\end{itemize}
\item $\cOneRight$ denote the real quantity defined as 
\begin{itemize}
\item if $\qRight$ equals $0$ then $\cOneRight=0$,
\item if $\qRight$ is not smaller than $1$ then $\cOneRight$ is the (positive) speed of the ``first'' travelling front involved in the right-propagating terrace $\tttRight$ (the one invading $m_+$), 
\end{itemize}
\item $\valueOfVcentre$ denote the quantity $V(\mLeftBehind)=V(\mRightBehind)$,
\item $\eee$ denote the energy of the centre standing terrace $\tttCentre$.
\end{itemize}
This defines maps:
\[
\begin{aligned}
\qRight : \ & \Xbist(\mmm)\to \nn \,, \\
\qLeft : \ & \Xbist(\mmm)\to \nn \,, \\
\valueOfVcentre : \ & \Xbist(\mmm)\to \rr \,, \\
\eee : \ & \Xbist(\mmm)\to [0,+\infty) \,, \\
\cOneLeft : \ & \Xbist(\mmm)\to (0,+\infty) \,, \\
\cOneRight : \ & \Xbist(\mmm)\to (0,+\infty) \,. \\
\end{aligned}
\]
Finally, let 
\[
\XbistNoInv(\mmm) = \Xbist(\mmm) \cap \qLeft^{-1}(\{0\}) \cap \qRight^{-1}(\{0\}) 
\,.
\]
In this notation the subscript "no-inv" refers to the fact that these solutions are those for which none of the two stable equilibria at both ends of space is ``invaded'' by a travelling front. Note that for every solution in $\XbistNoInv(\mmm)$, the equilibria approached by the solution at both ends of spatial domain must belong to the same level set of $V$ (this follows from \cref{thm:main}). 
\end{notation}
The following proposition states some regularity properties (upper or lower semi-continuity) of the ``correspondences'' between a solution and its asymptotic pattern defined above. The underlying phenomenon is in essence nothing else than the standard upper semi-continuity with respect to initial condition of the asymptotic level set of a (descendent) gradient flow (of say a Morse function on a finite-dimensional manifold). All the assertions stand with respect to the topology induced by $\norm{\cdot}_X$ (for the domain spaces) and the topology induced by the usual distance on $\rr$ (for the arrival spaces).
\begin{proposition}[continuity properties of the asymptotic pattern with respect to initial condition]
The following assertions hold:
\begin{enumerate}
\item the maps $\cOneLeft$ and $\cOneRight$ are lower semi-continuous;
\item the restriction of the function $\eee$ to the set $\XbistNoInv(\mmm)$ is upper semi-continuous. 
\end{enumerate}
\end{proposition}
\begin{proof}
The fist assertion is proved in \cite{Risler_globCVTravFronts_2008}. The second assertion is proved in \cite{Risler_globalRelaxation_2016}. 
\end{proof}
There are many other natural questions concerning the regularity of those correspondences. For instance it seems likely that the asymptotic potential level (the function $\valueOfVcentre[\cdot]$) is upper semi-continuous on $\Xbist(\mmm)$. And it would be nice to equip the ``space of asymptotic patterns'' with a topological structure ensuring similar semi-continuity properties for other (all?) features of the asymptotic pattern (speeds of the travelling fronts, values of the potential at minimum points connected by these travelling fronts, energies of the components of the centre standing terrace). This question goes beyond the scope of this paper; it is discussed in more details (in a more restricted case) in \cite{Risler_globalRelaxation_2016}. 
\subsection{Additional questions and extensions}
\subsubsection{Additional questions}
%
Besides the questions concerning the regularity of the correspondence between a solution and its asymptotic pattern mentioned above, here are some other natural questions (either that I have not been able to solve, or that are beyond the scope of this paper). 
\begin{itemize}
\item Does \cref{thm:main} hold without hypothesis \textup{(\hyperlink{hypDiscVel}{\hypDiscVelRef})}? The question is legitimate since this hypothesis is not required to prove convergence towards the ``first travelling fronts'' (those called $\phi_{1,+}$ and $\phi_{1,-}$ in the statement of \cref{thm:main}); indeed no hypothesis of this kind is made in \cite{Risler_globCVTravFronts_2008}. However, I have not been able to get rid of that hypothesis to prove convergence towards a travelling front ``following'' a previous one. For additional comments and details see \cref{subsec:cv_mean_inv_vel}.
\item Does \cref{thm:main} hold without hypothesis \textup{(\hyperlink{hypOnlyMin}{\hypOnlyMinRef})}? This question is discussed in \cite{Risler_globalRelaxation_2016}.
\item All the convergence results stated in this paper are purely qualitative (there is no quantitative estimate about the rate of convergence of a solution towards its asymptotic pattern). F. Béthuel, G. Orlandi and D. Smets \cite{BethuelOrlandi_slowMotion_2011} have obtained such quantitative estimates for the same gradient systems but in a different setting (they do not consider convergence towards travelling fronts but only relaxation towards stationary solutions connecting global minimum points of $V$). It would be interesting to see if the same approach can yield similar quantitative estimates but for general bistable solutions (and asymptotic behaviour involving travelling fronts). 
\item Concerning the existence of travelling fronts, the results obtained by Alikakos--Katzourakis and the author (\cite{AlikakosKatzourakis_heteroclinicTW_2011,Risler_globCVTravFronts_2008}) have been recently extended (by a calculus of variations approach) to the settings of degenerate stable states \cite{OliverBonafoux_heteroclinicTW1dParabSystDegenerate_2021} and parabolic systems in space dimension two \cite{OliverBonafoux_TWparabAllenCahn_2021,ChenChienHuang_varApproach3PhaseTWgradSyst_2021}. It would be interesting to know if global convergence results as those of the present paper can also be obtained in these settings. 
\end{itemize}
\subsubsection{Extensions}
%
Results similar to \cref{thm:main} hold in the following two cases that are considered in the companion papers \cite{Risler_globalBehaviourHyperbolicGradient_2017,Risler_globalBehaviourRadiallySymmetric_2017}. 
\begin{itemize}
\item Damped hyperbolic systems of the form
\[
\alpha u_{tt}+u_t=-\nabla V(u)+u_{xx}
\]
obtained by adding an inertial term $\alpha u_{tt}$ (where $\alpha$ is a positive non necessarily small quantity) to the parabolic system~\cref{init_syst} considered here. Note that this extension was already achieved by Gallay and Joly in the scalar case $d$ equals $1$ (for a bistable potential) using the gradient structure in every travelling frame \cref{dt_form_en_mf}, see \cite{GallayJoly_globStabDampedWaveBistable_2009}. 
\item Radially symmetric solutions in higher space dimension $\dSpace$, governed by systems of the form 
\[
u_t = -\nabla V(u)+ \frac{\dSpace-1}{r} u_r + u_{rr}
\]
where the nonnegative quantity $r$ denotes the radius (distance to the origin) in $\rr^{\dSpace}$. 
\end{itemize}
\subsection{Organization of the paper}
The next \cref{sec:preliminaries} is devoted to some preliminaries (existence of solutions, preliminary computations on spatially localized functionals, notation). 

Proof of \cref{thm:main} is essentially based on two propositions: \cref{prop:inv_cv,prop:relax}, together with the results of the companion paper \cite{Risler_globalRelaxation_2016}. \Cref{sec:inv_impl_cv,sec:no_inv_implies_relax} are devoted to these two propositions.

\Cref{prop:inv_cv} is the main step and is proved in \vref{sec:inv_impl_cv}. It is an extension of the main result of global convergence towards travelling fronts proved in the previous paper \cite{Risler_globCVTravFronts_2008}. As in \cite{Risler_globCVTravFronts_2008}, the proof is based on estimates about the time derivatives of energy functionals of the form \cref{form_en_mf}, with the exponential weight replaced by an integrable one. By contrast with the situation investigated in \cite{Risler_globCVTravFronts_2008}, the hypotheses of \cref{prop:inv_cv} cope with the case where the solution is not necessarily close to a single point of $\mmm$ in the whole domain ahead of the ``next'' travelling front, but may already behave in this domain as a propagating terrace of bistable travelling fronts. With respect to \cite{Risler_globCVTravFronts_2008}, the main difficulty is that an additional term appears in the time derivatives of the localized energy functionals in this case. For this reason, the relaxation scheme requires more care. 

\Cref{prop:relax} is easier and is proved in \vref{sec:no_inv_implies_relax}. It follows from \cref{prop:nonneg_asympt_en}, which states that, if a solution is close to stable homogeneous equilibria in large spatial domains on the left and on the right, and if these domains are not invaded (at a nonzero mean speed) from the ``centre'' area in between, then the energy of the solution (in a standing frame) between these two areas remains nonnegative. Again, the proof is based on estimates for localized energy functionals in a travelling frame --- this time at zero or small nonzero speed. 

The proof of \cref{thm:main} is completed in \vref{sec:proof_thm1}, by combining \cref{prop:inv_cv}, \cref{prop:relax}, and the results of \cite{Risler_globalRelaxation_2016}.

Finally, elementary properties of the profiles of travelling fronts are recalled in \vref{sec:prop_stand_trav}.
\section{Preliminaries}
\label{sec:preliminaries}
As everywhere else, let us introduce a function $V$ in $\ccc^2(\rr^d,\rr)$ satisfying the coercivity hypothesis \cref{hyp_coerc}. 
\subsection{Global existence of solutions and attracting ball for the semi-flow}
\label{subsec:glob_exist}
The following proposition follows from general existence results and the assumption \cref{hyp_coerc}. For the proof see for instance \cite{Risler_globCVTravFronts_2008,Risler_globalRelaxation_2016}. Its two conclusions are somehow redundant, since an upper bound for the $H^1_\text{ul}$-norm immediately yields an upper bound for the $L^\infty$-norm, and the converse is also true since the flow is regularizing. The reason for this redundant statement is that both quantities $\Rattinfty$ and $\RattX$ used to express these bounds will be explicitly used at some places. 
\begin{proposition}[global existence of solutions and attracting ball]
\label{prop:attr_ball}
For every function $u_0$ in $X$, system \cref{init_syst} has a unique globally defined solution $t\mapsto S_t u_0$ in $\ccc^0([0,+\infty),X)$ with initial condition $u_0$. In addition, there exist positive quantities $\Rattinfty$ and $\RattX$ (radius of attracting ball for the $L^\infty$-norm and the $H^1_\textup{\text{ul}}$-norm, respectively), depending only on the potential $V$, 
such that, for every large enough positive time $t$, 
\[
\norm{x\mapsto (S_t u_0)(x)}_{\Linfty} \le \Rattinfty 
\qquad\text{and}\qquad
\norm{x\mapsto (S_t u_0)(x)}_{X}  \le \RattX 
\,.
\]
\end{proposition}
In addition, system~\cref{init_syst} has smoothing properties (Henry \cite{Henry_geomSemilinParab_1981}). Due to these properties, since $V$ is of class $\ccc^2$ and thus the nonlinearity $\nabla V$ is of class $\ccc^1$, for every quantity $\alpha$ in the interval $(0,1)$, every solution $t\mapsto S_t u_0$ in $\ccc^0([0,+\infty),X)$ actually belongs to
\[
\ccc^0\left((0,+\infty),\cccb{2,\alpha}\right)\cap \ccc^1\left((0,+\infty),\cccb{0,\alpha}\right),
\]
and, for every positive quantity $\varepsilon$, the quantities
\begin{equation}
\label{bound_u_ut_ck}
\sup_{t\ge\varepsilon}\norm{S_t u_0}_{\cccb{2,\alpha}}
\quad\text{and}\quad
\sup_{t\ge\varepsilon}\norm{\frac{d(S_t u_0)}{dt}(t)}_{\cccb{0,\alpha}}
\end{equation}
are finite. 
\subsection{Asymptotic compactness of solutions}
\label{subsec:compactness}
The following compactness property will be called upon several times (at the ends of the proofs of \cref{prop:inv_cv,prop:relax}).
\begin{lemma}[asymptotic compactness]
\label{lem:compactness}
For every solution $(x,t)\mapsto u(x,t)$ of system \cref{init_syst}, and for every sequence $(x_n,t_n)_{n\in\nn}$ in $\rr\times[0,+\infty)$ such that $t_n\to+\infty$ as $n\to+\infty$, there exists a entire solution $\widebar{u}$ of system \cref{init_syst} in 
\[
\ccc^0\left(\rr,\cccb{2}\right)\cap \ccc^1\left(\rr,\cccb{0}\right)
\,,
\]
such that, up to replacing the sequence $(x_n,t_n)_{n\in\nn}$ by a subsequence, 
\begin{equation}
\label{compactness}
D^{2,1}u(x_n+\cdot,t_n+\cdot)\to D^{2,1}\widebar{u}
\quad\text{as}\quad
n\to+\infty
\,,
\end{equation}
uniformly on every compact subset of $\rr^2$, where the symbol $D^{2,1}v$ stands for $(v,v_x,v_{xx},v_t)$ (for $v$ equal to $u$ or $\widebar{u}$). 
\end{lemma}
\begin{proof}
See \cite[1963]{MatanoPolacik_entireSolutionBistableParabEquTwoCollidingPulses_2017} or \cite[proof of \GlobalRelaxationLemAsymptCompactness]{Risler_globalRelaxation_2016}. 
\end{proof}
\subsection{Time derivative of (localized) energy and \texorpdfstring{$L^2$}{L2}-norm of a solution in a travelling frame}
\label{subsec:1rst_ord}
Let $(x,t)\mapsto u(x,t)$ be a solution of system \cref{init_syst}, and let $m$ be a point of $\mmm$. Let $c$ be a real quantity, and let us consider the same solution viewed in a frame travelling at the speed $c$, that is the function $(\xi,t)\mapsto v(\xi,t)$ defined as
\[
v(\xi,t) = u(x,t)
\quad\text{for}\quad
x=ct+\xi
\,.
\]
As mentioned in introduction, this function is a solution of the system
\begin{equation}
\label{syst_mf_bis}
v_t - cv_\xi = -\nabla V(v) + v_{\xi\xi}
\,,
\end{equation}
that can be formally rewritten as the (descendent) gradient of the following energy (Lagrangian, action) functional:
\[
\int_{\rr} e^{c\xi} \Bigl(\frac{1}{2}v_\xi(\xi,t)^2 + V\bigl(v(\xi,t)\bigr)-V(m)\Bigr)\, d\xi
\,.
\]
The key ingredients of the proofs rely on appropriate combinations of this functional with the other most natural functional to consider, namely the $L^2$-norm of the distance to $m$ (with the same exponential weight):
\[
\int_{\rr} e^{c\xi} \frac{1}{2}\bigl(v(\xi,t)-m\bigr)^2 \, d\xi
\,.
\]
To simplify the presentation, let us assume (only in this \namecref{subsec:1rst_ord}) that
\[
m=0_{\rr^d}
\quad\text{and}\quad
V(m) = V(0_{\rr^d}) = 0
\,.
\]
In order to ensure the convergence of those last two integrals, it is necessary to localize the integrands. Let $(\xi,t)\mapsto \psi(\xi,t)$ be a function defined on $\rr\times[0,+\infty)$ and such that, for all $t$ in $(0,+\infty)$, the function $\xi\mapsto \psi(\xi,t)$ belongs to $W^{2,1}(\rr,\rr)$ and its time derivative $\xi\mapsto \psi_t(\xi,t)$ is defined and belongs to $L^1(\rr,\rr)$. 
Then, the time derivatives of the two aforementioned functionals --- localized by the ``weight'' function $\psi (\xi,t)$ --- read:
\begin{equation}
\label{ddt_loc_en}
\frac{d}{dt}\int_{\rr} \psi\Bigl( \frac{1}{2}v_\xi^2 + V(v)\Bigr) \, d\xi = 
\int_{\rr} \biggl[ - \psi\, v_t^2 + \psi_t \Bigl( \frac{1}{2}v_\xi^2 + V(v)\Bigr) + (c\psi - \psi_\xi)\, v_\xi \cdot v_t \biggr] \, d\xi
\,,
\end{equation}
and
\begin{equation}
\label{ddt_loc_L2}
\frac{d}{dt}\int_{\rr} \psi \frac{1}{2}v^2 \, d\xi = 
\int_{\rr} \biggl[ \psi\bigl( - v\cdot \nabla V(v) -  v_\xi^2 \bigr) + \frac{1}{2}(\psi_t + \psi_{\xi\xi} - c\psi_\xi)v^2 \biggr] \, d\xi
\,.
\end{equation}
Here are some basic observations about these expressions. 
\begin{itemize}
\item The time derivative of the (localized) energy is the sum of a (nonpositive) ``dissipation'' term and two additional ``flux'' terms.
\item The time derivative of the (localized) $L^2$-norm is similarly made of two ``main'' terms and an additional ``flux'' term. Among the two main terms, the second is nonpositive, and so is the first if $\abs{v}$ is small. 
\item The flux terms involving $\psi_t$ are small if $\psi$ varies slowly with respect to time. 
\item The other flux term in the time derivative of energy is small if the quantity $c\psi - \psi_\xi$ is small, that is if $\psi(\xi,t)$ behaves nearly as $\exp(c\xi)$ (up to a positive multiplicative constant). But this cannot hold for all $\xi$ in $\rr$ since $\xi\mapsto\psi(\xi,t)$ must be in $L^1(\rr,\rr)$. As a consequence, it will not be possible to avoid that this second flux term be weighted by a ``non small'' weight of the order of $c\psi$, at least somewhere in space.
\item Hopefully, the remaining flux term in the time derivative of the $L^2$-norm is ``nicer'', in the sense that it is small under any of two conditions instead of a single one:
\begin{itemize}
\item either if $\psi(\xi,t)$ is close to behave like $\exp(c\xi)$ (up to a positive multiplicative constant, like the previous flux term),
\item or if $\psi(\xi,t)$ varies slowly with $\xi$. 
\end{itemize}
\item Finally, the ``tricky'' flux term in the time derivative of energy can be balanced by the other terms of those expressions by considering a linear combination (with positive coefficients) of energy and $L^2$-norm with a large enough coefficient of the $L^2$-norm with respect to the coefficient of energy. 
\end{itemize}
These observations will be put in practice several times along the following pages. 
\subsection{Miscellanea}
\label{subsec:misc}
\subsubsection{Second order estimates for the potential around a minimum point}
\begin{lemma}[second order estimates for the potential around a minimum point]
\label{lem:estim_from_def_escape}
For every $m$ in $\mmm$ and every $u$ in $\rr^d$ satisfying $\abs{u-m} \le\dEsc(m)$, the following estimates hold:
\begin{align}
\label{posit_pot_around_loc_min}
V(u)-V(m) &\ge \frac{\eigVmin(m)}{4} (u-m)^2 \,, \\
\text{and}\qquad
\label{v_nablaV_controls_square_around_loc_min}
(u-m)\cdot \nabla V(u) &\ge \frac{\eigVmin(m)}{2} (u-m)^2 \,, \\
\text{and}\qquad
\label{v_nablaV_controls_pot_around_loc_min}
(u-m)\cdot \nabla V(u) &\ge V(u) -V(m)
\,.
\end{align}
\end{lemma}
\begin{proof}
The three inequalities follow from inequality \vref{property_dEsc} ensured by the definition of $\dEsc(m)$ and from three variants of Taylor's Theorem with Lagrange remainder applied to the function $f$ defined on $[0,1]$ by $f(\theta)=V\bigl(m+\theta (u-m)\bigr)$ (see \cite[\GlobalRelaxationLemEstimFromEscDist]{Risler_globalRelaxation_2016}).
\end{proof}
\subsubsection{Lower quadratic hulls of the potential at minimum points}
\label{subsubsec:low_quad_hull}
For the computations carried in \cref{subsec:def_fire_zero} below, it will be convenient to introduce the quantity $\qLowHull$ defined as the minimum of the convexities of the negative quadratic hulls of $V$ at the points of $\mmm$ (see \cref{fig:low_quad_hull}). 
\begin{figure}[!htbp]
\centering
\includegraphics[width=0.5\textwidth]{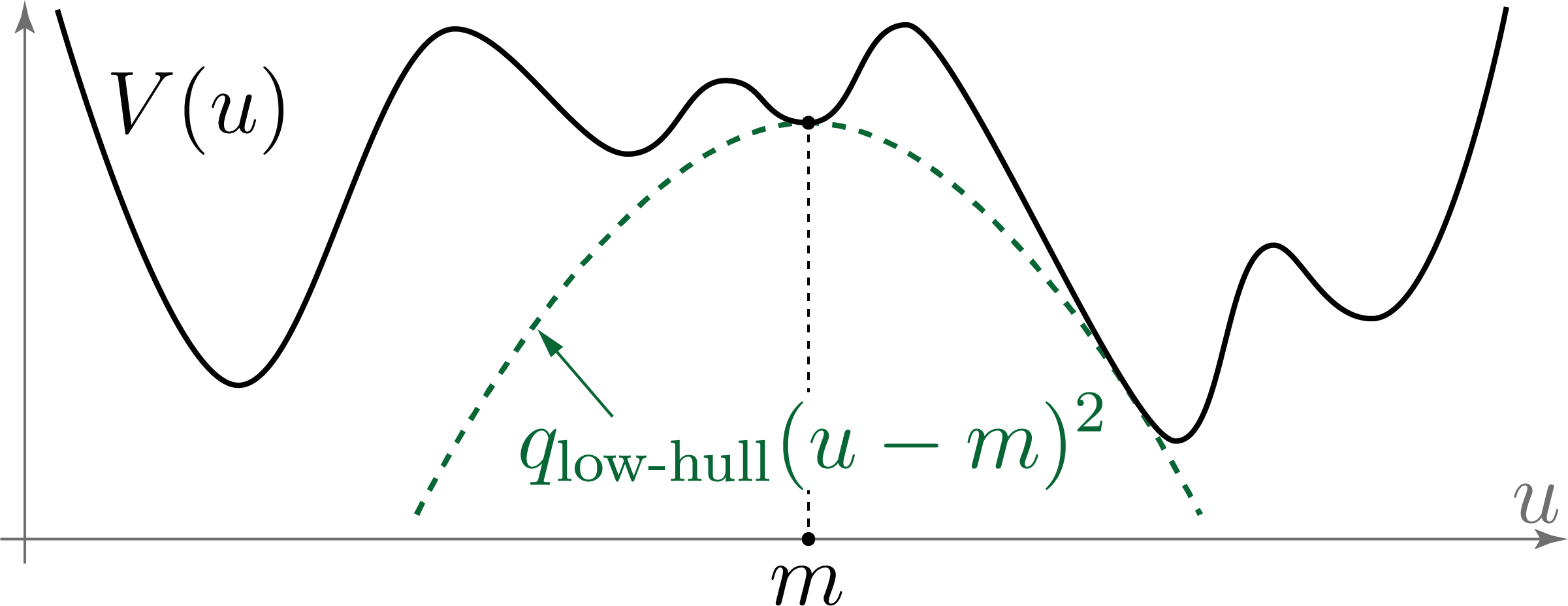}
\caption{Lower quadratic hull of the potential at a minimum point (definition of the quantity $\qLowHull$).}
\label{fig:low_quad_hull}
\end{figure}
With symbols:
\[
\qLowHull = \min_{m\in\mmm}\ \inf_{u\in\rr^d\setminus \{m\}}\ \frac{V(u)-V(m)}{(u-m)^2}
\]
(a similar quantity was defined in \cite{Risler_globalRelaxation_2016}). 
This definition ensures (as displayed by \cref{fig:low_quad_hull}) that, for every point $m$ of $\mmm$ and for every $u$ in $\rr^d$,
\[
V(u) -V(m) \ge \qLowHull (u-m)^2
\,.
\]
Let us introduce the following quantity (it will be used to define the \emph{coefficient of the energy} in the firewall functions defined in \vref{subsec:def_fire_zero,subsubsec:def_firewall}):
\[
\coeffEnZero=\frac{1}{\max(1, - 4\, \qLowHull)}
\,.
\]
It follows from this definition that, for every $m$ in $\mmm$ and for all $u$ in $\rr^d$,
\begin{equation}
\label{def_weight_en}
\coeffEnZero \, \bigl( V(u) - V(m) \bigr) +  \frac{1}{4}(u-m)^2\ge 0
\,.
\end{equation}
\section{Invasion implies convergence}
\label{sec:inv_impl_cv}
\subsection{Definitions and hypotheses}
\label{subsec:inv_cv_def_hyp}
As everywhere else, let us consider a function $V$ in $\ccc^2(\rr^d,\rr)$ satisfying the coercivity hypothesis \cref{hyp_coerc}. Let us consider a point $m$ in $\mmm$, a function (initial condition) $u_0$ in $X$, and the solution $(x,t)\mapsto u(x,t)$ of system \cref{init_syst} for this initial condition. 

It will not be assumed that this solution is bistable, but instead, as stated by the next hypothesis \textup{(\hyperlink{hypHomRight}{\hypHomRightRef})}, that there exists a growing interval, travelling at a positive speed, where the solution is close to $m$ (the subscript ``hom'' in the definitions below refers to this ``homogeneous'' area), see \cref{fig:inv_cv}.
\begin{figure}[!htbp]
\centering
\includegraphics[width=.8\textwidth]{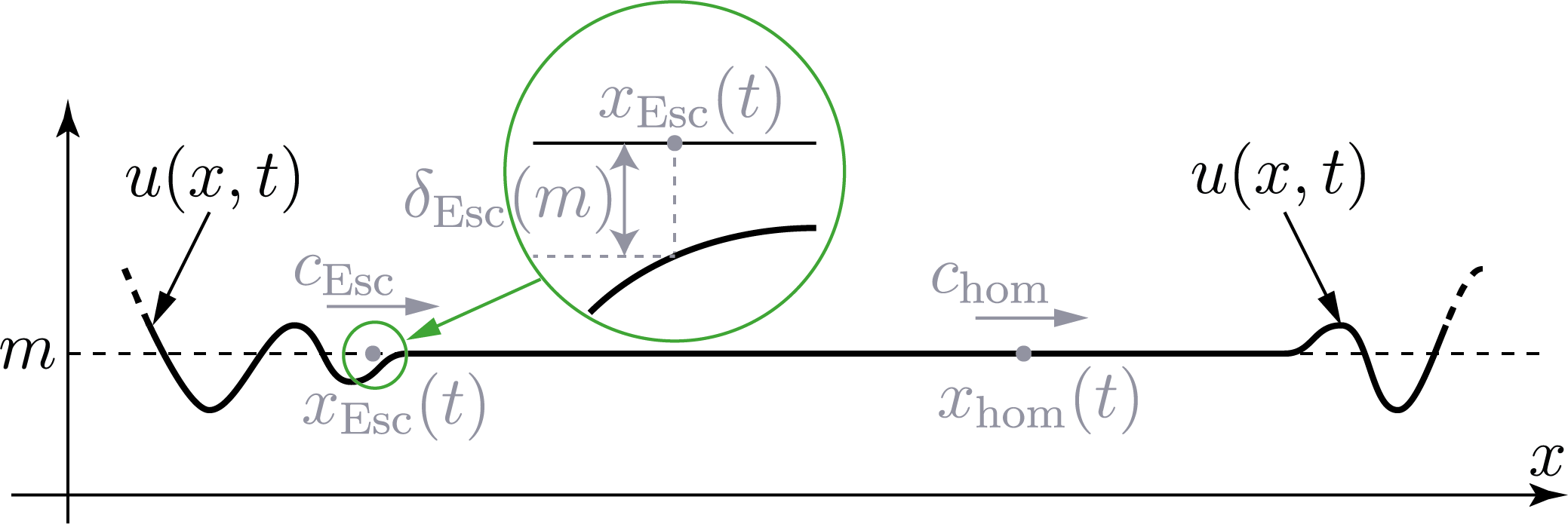}
\caption{Illustration of hypotheses \textup{(\hypHomRightRef)} and \textup{(\hypInvRef)}.}
\label{fig:inv_cv}
\end{figure}
\begin{description}
\item[\hypHomRightLabel]\hypertarget{hypHomRight} There exists a positive quantity $\cHom$ and a $\ccc^1$-function 
\[
\xHom : [0,+\infty)\to \rr\,,
\quad\text{satisfying}\quad
\xHom'(t) \to \cHom
\quad\text{as}\quad
t\to + \infty
\,,
\]
such that, for every positive quantity $L$, 
\[
\sup_{\xi\in[-L,L]} \abs{u\bigl( \xHom(t) + \xi, t\bigr) - m }\to 0
\quad\text{as}\quad
t\to + \infty
\,.
\]
\end{description}
For every $t$ in $[0+\infty)$, let us denote by $\xEsc(t)$ the supremum of the set:
\[
\Bigl\{ x\in \bigl(-\infty,\xHom(t)\bigr] : \abs{u(x,t) - m} = \dEsc(m) \Bigr\}
\,,
\]
with the convention that $\xEsc(t)$ equals $-\infty$ if this set is empty. In other words, $\xEsc(t)$ is the first point at the left of $\xHom(t)$ where the solution ``escapes'' at the distance $\dEsc(m)$ from the stable homogeneous equilibrium $m$. This point will be referred to as the ``Escape point'' (with an upper-case ``E'', by contrast with another ``escape point'' that will be introduced later, with a lower-case ``e'' and a slightly different definition). Let us consider the upper limit of the mean speeds between $0$ and $t$ of this Escape point:
\[
\cEsc = \limsup_{t\to +\infty} \frac{\xEsc(t)}{t}
\,,
\]
and let us make the following hypothesis, stating that the area around $\xHom(t)$ where the solution is close to $m$ is ``invaded'' from the left at a nonzero (mean) speed.
\begin{description}
\item[\hypInvLabel]\hypertarget{hypInv} The quantity $\cEsc$ is positive. 
\end{description}
\subsection{Statement}
\label{subsec:inv_cv_stat}
The aim of \cref{sec:inv_impl_cv} is to prove the following proposition, which is the main step in the proof of \cref{thm:main}. The proposition is illustrated by \cref{fig:inv_cv_bis}.
\begin{figure}[!htbp]
\centering
\includegraphics[width=\textwidth]{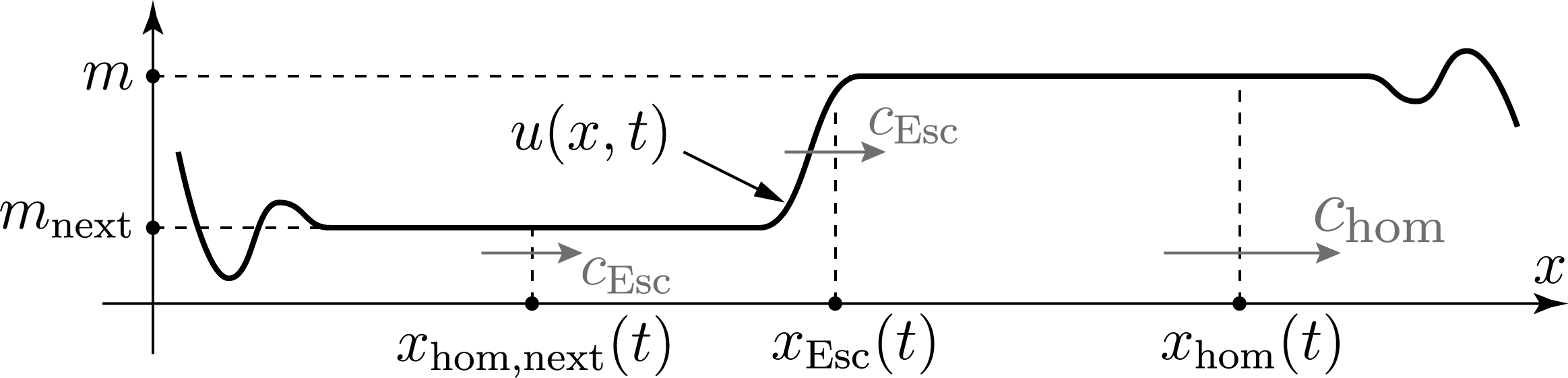}
\caption{Illustration of \cref{prop:inv_cv}.}
\label{fig:inv_cv_bis}
\end{figure}
\begin{proposition}[invasion implies convergence]
\label{prop:inv_cv}
Assume that $V$ satisfies the coercivity hypothesis \cref{hyp_coerc} and the generic hypotheses \textup{(\hyperlink{hypOnlyBist}{\hypOnlyBistRef})} and \textup{(\hyperlink{hypDiscVel}{\hypDiscVelRef})} and \textup{(\hyperlink{hypDiscFront}{\hypDiscFrontRef})}, and, keeping the definitions and notation above, let us assume that for the solution under consideration hypotheses \textup{(\hyperlink{hypHomRight}{\hypHomRightRef})} and \textup{(\hyperlink{hypInv}{\hypInvRef})} hold. Then the following conclusions hold. 
\begin{enumerate}
\item The function $t\mapsto \xEsc(t)$ is of class $\ccc^1$ as soon as $t$ is large enough and 
\[
\xEsc'(t)\to \cEsc
\quad\text{as}\quad
t\to +\infty
\,.
\]
\item There exist:
\begin{itemize}
\item a point $\mNext$ in $\mmm$ satisfying $V(\mNext)<V(m)$,
\item a profile of travelling front $\phi$ in $\PhicNorm{\cEsc}(\mNext,m)$,
\item a $\ccc^1$-function $[0,+\infty)\to\rr$, $t\mapsto \xHomNext(t)$,
\end{itemize}
such that, as time goes to $+\infty$, the following limits hold:
\[
\xEsc(t)-\xHomNext(t)\to +\infty
\quad\text{and}\quad
\xHomNext'(t)\to \cEsc
\]
and
\[
\sup_{x\in[\xHomNext(t) \, ,\ \xHom(t)]} \abs{u(x,t) - \phi\bigl( x-\xEsc(t) \bigr)} \to 0 
\]
and, for every positive quantity $L$, 
\[
\sup_{\xi\in[-L,L]}\abs{u\bigl( \xHomNext(t)+\xi,t\bigr) - \mNext}\to 0 
\,.
\]
\end{enumerate}
\end{proposition}
In this statement, the very last conclusion is actually redundant with the previous one. The reason why this last conclusion is stated this way is that it emphasizes the fact that a property similar to \textup{(\hyperlink{hypHomRight}{\hypHomRightRef})} is recovered ``behind'' the travelling front. As can be expected this will be used to prove \cref{thm:main} by re-applying \cref{prop:inv_cv} as many times as required (to the left and to the right), as long as ``invasion of the equilibria behind the last front'' occurs. 
\subsection{Set-up for the proof, 1}
\label{subsec:inv_cv_set_pf}
Let us keep the notation and assumptions of \cref{subsec:inv_cv_def_hyp}, and let us assume that the hypotheses \cref{hyp_coerc} and \textup{(\hyperlink{hypOnlyBist}{\hypOnlyBistRef})} and \textup{(\hyperlink{hypDiscVel}{\hypDiscVelRef})} and \textup{(\hyperlink{hypDiscFront}{\hypDiscFrontRef})} and \textup{(\hyperlink{hypHomRight}{\hypHomRightRef})} and \textup{(\hyperlink{hypInv}{\hypInvRef})} of \cref{prop:inv_cv} hold. 
\subsubsection{Assumptions holding up to changing the origin of time}
Without loss of generality, up to changing the origin of time, it may be assumed that the following properties hold. 
\begin{itemize}
\item According to \cref{prop:attr_ball}, it may be assumed that, for every nonnegative time $t$, 
\end{itemize}
\begin{align}
\label{attr_ball_infty_inv_implies_cv}
&\norm{x\mapsto u(x,t)}_{\Linfty} \le \Rattinfty \\
\text{and}\qquad
\label{attr_ball_Hone_inv_implies_cv}
&\norm{x\mapsto u(x,t)}_X \le \RattX
\,.
\end{align}
\begin{itemize}
\item According to the bounds \vref{bound_u_ut_ck}, it may be assumed that
\end{itemize}
\begin{equation}
\label{bound_u_ut_ck_bis}
\sup_{t\ge0}\norm{x\mapsto u(x,t)}_{\cccb{2}} < +\infty
\qquad\text{and}\qquad
\sup_{t\ge0}\norm{x\mapsto u_t(x,t)}_{\cccb{0}} < +\infty
\,.
\end{equation}
\begin{itemize}
\item According to \textup{(\hyperlink{hypHomRight}{\hypHomRightRef})}, it may be assumed that, for all $t$ in $[0,+\infty)$, 
\end{itemize}
\begin{equation}
\label{hyp_xHom_prime_pos}
\xHom'(t)\ge 0
\,.
\end{equation}
\begin{itemize}
\item According to \textup{(\hyperlink{hypInv}{\hypInvRef})}, it may be assumed that, for all $t$ in $[0,+\infty)$, 
\end{itemize}
\begin{equation}
\label{hyp_xEsc_finite}
-\infty<\xEsc(t)
\,.
\end{equation}
\subsubsection{Normalized potential and corresponding solution}
\label{subsubsec:normalized_potential_solution}
For notational convenience, let us introduce:
\begin{itemize}
\item a new ``normalized'' potential $V^\dag:\rr^d\to\rr$, $v\mapsto V^\dag(v)$,
\item and the corresponding solution $u^\dag:\rr\times[0,+\infty)\to\rr$, $(x,t)\mapsto u^\dag(x,t)$, 
\end{itemize}
defined as
\[
V^\dag(v)=V(m +v)-V(m)
\quad\text{and}\quad
u^\dag(x,t) = u(x,t)-m
\,.
\]
Thus the origin $0_{\rr^d}$ of $\rr^d$ is to $V^\dag$ what $m$ is to $V$, it is a nondegenerate minimum point for $V^\dag$ (with $V^\dag(0_{\rr^d})=0$), and $u^\dag$ is a solution of system \cref{init_syst} with potential $V^\dag$ instead of $V$; and, for all $(x,t)$ in $\rr\times[0,+\infty)$, 
\[
V^\dag\bigl(u^\dag(x,t)\bigr) = V\bigl(u(x,t)\bigr)-V(m)
\,.
\]
It follows from inequality \cref{def_weight_en} satisfied by $\coeffEnZero$ that, for all $v$ in $\rr^d$, 
\begin{equation}
\label{def_weight_en_V_dag}
\coeffEnZero \, V^\dag(v) +  \frac{1}{4}v^2\ge 0
\,,
\end{equation}
and it follows from inequalities \cref{v_nablaV_controls_square_around_loc_min,v_nablaV_controls_pot_around_loc_min} that, for all $v$ in $\rr^d$ satisfying $\abs{v}\le\dEsc(m)$, 
\begin{align}
\label{v_nablaV_controls_square_around_loc_min_dag}
v\cdot \nabla V^\dag(v) &\ge \frac{\eigVmin(m)}{2} v^2 \,, \\
\text{and}\qquad
\label{v_nablaV_controls_pot_around_loc_min_dag}
v\cdot \nabla V^\dag(v) &\ge V^\dag(v)
\,.
\end{align}
This notation $V^\dag$ and $u^\dag$ will be used up to \cref{subsec:relax_sch_tr_fr}. From \cref{subsec:cv_mean_inv_vel} on, the notation $V$ and $u$ denoting the initial potential and solution will be used again, instead. 
\subsubsection{Looking for another definition of the escape point}
Unfortunately, the Escape point $\xEsc(t)$ presents a significant drawback: there is no reason why it should display any form of continuity (it may jump back and forth while time increases). This lack of control is problematic with respect to the purpose of writing down a dissipation argument precisely around the position in space where the solution escapes from $m$. 

The answer to this difficulty will be to define another ``escape point'' (this one will be denoted by ``$\xesc(t)$'' --- with a lower-case ``e'' --- instead of $\xEsc(t)$). This second definition is a bit more involved than that of  $\xEsc(t)$, but the resulting escape point will have the significant advantage of growing at a finite (and even bounded) rate (\cref{lem:inv} below). The material required to define this escape point is introduced in the next \namecref{subsec:def_fire_zero}. 
\subsection{Firewall functions in the laboratory frame}
\label{subsec:def_fire_zero}
\subsubsection{Definition} 
\label{subsubsec:def_firewall_lab_frame}
The content of this \namecref{subsec:def_fire_zero} and of the next one is almost identical to that of \cite[\GlobalRelaxationSecSpatialAsymptotics]{Risler_globalRelaxation_2016}, where details, proofs, and comments can be found (the sole difference 
in \cite{Risler_globalRelaxation_2016} is the existence of a positive definite ``diffusion matrix'' --- whereas in the present paper this diffusion matrix equals identity). Only the minimum required definitions and statements are provided below. 

The notation is the same as in \cite[\GlobalRelaxationSecSpatialAsymptotics]{Risler_globalRelaxation_2016} with an additional ``$0$'' subscript to point out that all these objects are defined in the standing frame. Similar objects will be defined in the next \cref{subsec:relax_sch_tr_fr} but this time in a travelling referential (this time without the ``$0$'' subscript). 

Let 
\begin{equation}
\label{def_kappaZero}
\kappa_0 = \min\biggl(\frac{1}{\sqrt{\coeffEnZero}},\frac{\sqrt{\eigVmin(m)}}{4}\biggr)
\,
\end{equation}
and let us introduce the weight function $\psi_0$ defined as
\[
\psi_0(x) = \exp(-\kappa_0\abs{x})
\,.
\]
For $\bar{x}$ in $\rr$, let $T_{\bar{x}}\psi_0$ denote the translate of $\psi_0$ by $\bar{x}$, that is the function defined as
\[
T_{\bar{x}}\psi_0(x) = \psi_0 (x-\bar{x})
\,,
\]
see \cref{fig_weight_fire}. 
\begin{figure}[!htbp]
\centering
\includegraphics[width=\textwidth]{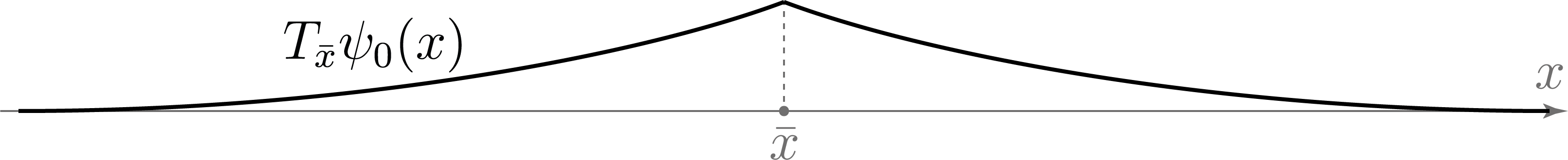}
\caption{Graph of the weight function $x\mapsto T_{\bar{x}}\psi(x)$ used to define the firewall function $\fff_0(\bar{x},t)$. The slope is small, according to the definition of $\kappa$.}
\label{fig_weight_fire}
\end{figure}
For all $t$ in $[0,+\infty)$ and $\bar{x}$ in $\rr$, let us introduce the ``firewall'' $\fff_0(\bar{x},t)$ defined as
\begin{equation}
\label{def_firewall_lab_frame}
\fff_0(\bar{x},t) = \int_{\rr} T_{\bar{x}}\psi_0(x)\biggl(\coeffEnZero\Bigl(\frac{1}{2}u^\dag_x(x,t)^2 + V^\dag \bigl(u^\dag(x,t)\bigr)\Bigr) + \frac{1}{2}u^\dag(x,t)^2\biggr)\, dx
\,,
\end{equation}
(the quantity $\coeffEnZero$ was defined in \cref{subsubsec:low_quad_hull}). 
\subsubsection{Coercivity} 
%
\begin{lemma}[coercivity of firewall function in the laboratory frame]
\label{lem:fireZero_coercivity}
For all $t$ in $[0,+\infty)$ and $\bar{x}$ in $\rr$, 
\begin{equation}
\label{fireZero_coercivity}
\fff_0(\bar x,t)\ge \min\Bigl(\frac{\coeffEnZero}{2},\frac{1}{4}\Bigr)\int_{\rr} T_{\bar{x}}\psi_0(x)\bigl(u^\dag_x(x,t)^2 +u^\dag(x,t)^2\bigr)\, dx
\,.
\end{equation}
\end{lemma}
\begin{proof}
Inequality \cref{fireZero_coercivity} follows from inequality \cref{def_weight_en_V_dag}.
\end{proof}
\subsubsection{Linear decrease up to pollution} 
%
Let 
\begin{equation}
\label{def_Sigma_Esc_0}
\SigmaEscZero(t)=\bigl\{x\in\rr: \abs{u^\dag(x,t)} >\dEsc(m)\bigr\} 
\,.
\end{equation}
\begin{lemma}[firewall linear decrease up to pollution]
\label{lem:dt_fire}
There exist positive quantities $\nuFZero$ and $\KFZero$, both depending only on $V$ and $m$, such that, for all $\bar{x}$ in $\rr$ and $t$ in $[0,+\infty)$,
\begin{equation}
\label{dt_fire}
\partial_t \fff_0(\bar{x},t)\le-\nuFZero\, \fff_0(\bar{x},t) + \KFZero\, \int_{\SigmaEscZero(t)} T_{\bar{x}}\psi_0(x)\, dx
\,.
\end{equation}
\end{lemma}
\begin{proof}
See \cite[\GlobalRelaxationLemFirewallDecreaseUpToPollution]{Risler_globalRelaxation_2016}. The quantities $\nuFZero$ and $\KFZero$ may be chosen as follows:
\begin{equation}
\label{def_nu_fire_zero}
\nuFZero = \min\Bigl(\frac{1}{4\coeffEnZero}, \frac{\eigVmin(m)}{4}\Bigr) \,, 
\end{equation}
and, according to the uniform bound \vref{attr_ball_infty_inv_implies_cv} for the solution,
\begin{equation}
\label{def_K_fire_zero}
\KFZero = \max_{v\in\rr^d, \ \abs{v}\le \Rattinfty}\biggl[ - (v-m)\cdot \nabla V(v) +\frac{1}{2}\abs{V(v)-V(m)}+ \frac{\eigVmin(m)}{4}\, (v-m)^2 \biggr] 
\,.
\end{equation}
\end{proof}
\begin{remark}
In order the proof of \cite[\GlobalRelaxationLemFirewallDecreaseUpToPollution]{Risler_globalRelaxation_2016} to apply, the quantities $\kappa_0$ and $\nuFZero$ should fulfill the following inequalities:
\begin{align}
\label{conditions_kappaZero_lab}
\frac{\coeffEnZero\, \kappa_0^2}{4}\le \frac{1}{2}
\quad&\text{and}\quad
\frac{\kappa_0^2}{2}\le \frac{\eigVmin(m)}{8}
\\
\text{and}\quad
\label{conditions_nuFireZero_lab}
\nuFZero\coeffEnZero \le \frac{1}{2}
\quad&\text{and}\quad
\frac{\nuFZero}{2} \le \frac{\eigVmin(m)}{8}
\end{align}
(compare with, respectively, the inequalities \cite[\GlobalRelaxationConditionsKappa{} and \GlobalRelaxationConditionsNuFire]{Risler_globalRelaxation_2016}), out of which the ``natural'' values for, respectively, $\kappa_0$ and $\nuFZero$ would be:
\[
\min\biggl(\sqrt{\frac{2}{\coeffEnZero}},\frac{\sqrt{\eigVmin(m)}}{2}\biggr)
\qquad\text{and}\qquad
\min\Bigl(\frac{1}{2\coeffEnZero}, \frac{\eigVmin(m)}{4}\Bigr)
\,.
\]
The reason for the (slightly more stringent) expressions \cref{def_kappaZero,def_nu_fire_zero} is that they are convenient to be used again in \cref{sec:no_inv_implies_relax}. 
\end{remark}
\subsection{Upper bound on the invasion speed}
\label{subsec:up_bd_inv_vel}
Let
\begin{equation}
\label{def_desc}
\desc(m) = \dEsc(m) \sqrt{\frac{2\min\Bigl(\frac{\coeffEnZero}{2},\frac{1}{4}\Bigr)}{1+\kappa_0}}
\,.
\end{equation}
As the quantity $\dEsc(m)$ defined in \cref{subsubsec:breakup}, this quantity $\desc(m)$ will provide a way to measure the vicinity of the solution $u^\dag$ to the minimum point $0_{\rr^d}$ of $V^\dag$, this time in terms of the firewall function $\fff_0$. The value chosen for $\desc(m)$ ensures the validity of the following lemma. 
\begin{lemma}[escape/Escape]
\label{lem:esc_Esc}
For all $x$ in $\rr$ and $t$ in $[0,+\infty)$, the following assertion holds:
\begin{equation}
\label{ineq_esc_Esc}
\fff_0(x,t)\le \desc(m)^2
\implies
\abs{u^\dag(x,t)} \le\dEsc(m)
\,.
\end{equation}
\end{lemma}
\begin{proof}
See \cite[\GlobalRelaxationCorescapeEscape]{Risler_globalRelaxation_2016}.
\end{proof}
Let $L$ be a positive quantity, large enough so that
\[
2 \KFZero \frac{\exp(-\kappa_0 L)}{\kappa_0} \le \nuFZero \frac{\desc(m)^2}{8}
\,,
\quad\text{namely}\quad
L = \frac{1}{\kappa_0}\log\Bigl(\frac{16 \,\KFZero}{\nuFZero\, \desc(m)^2\, \kappa_0}\Bigr)
\,,
\]
let $\hullnoesc:\rr\to\rr\cup\{+\infty\}$ (``no-escape hull'') be the function defined as
\begin{equation}
\label{def_h_noesc}
\hullnoesc(\xi) = 
\left\{
\begin{aligned}
& +\infty & \quad\text{for}\quad & \xi<0 \,, \\
& \frac{\desc(m)^2}{2}\Bigl(1-\frac{\xi}{2\,L}\Bigr) & \quad\text{for}\quad & 0\le \xi\le L \,, \\
& \frac{\desc(m)^2}{4} & \quad\text{for}\quad & \xi\ge L \,,
\end{aligned}
\right.
\end{equation}
\begin{figure}[!htbp]
\centering
\includegraphics[width=0.6\textwidth]{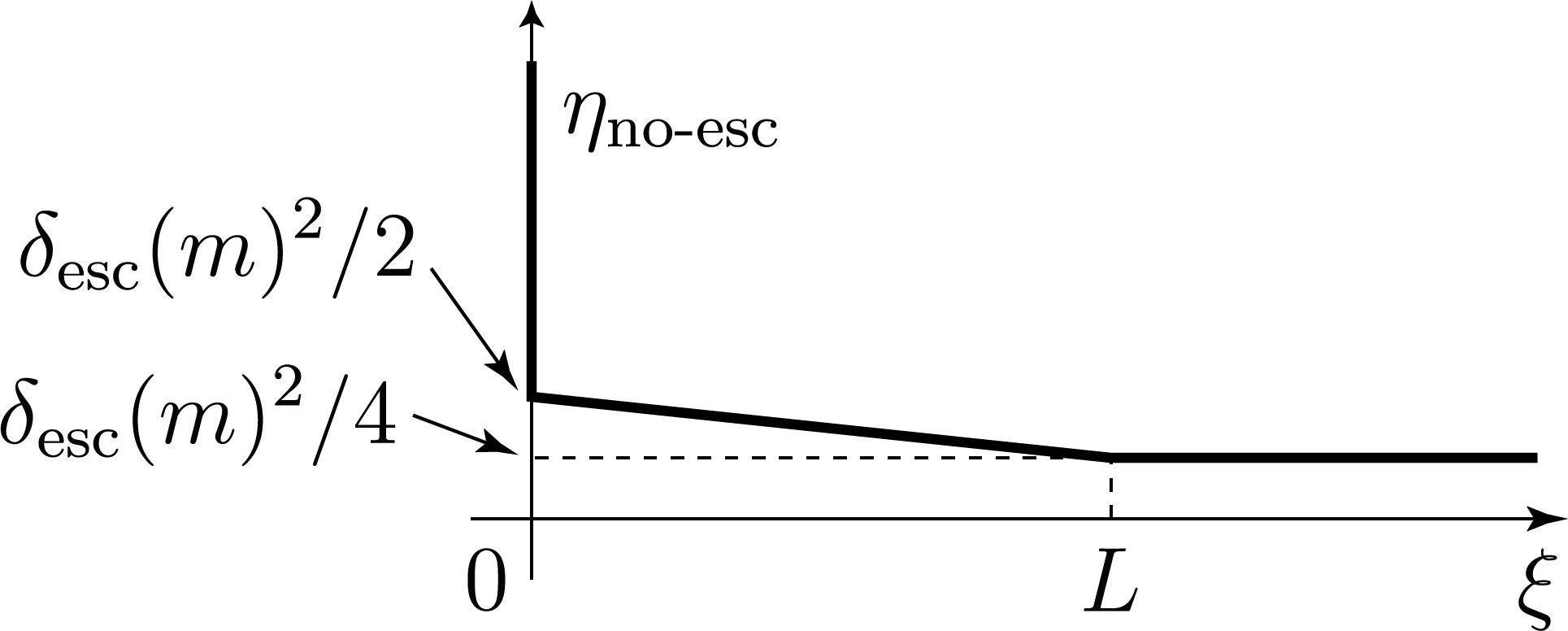}
\caption{Graph of the hull function $\hullnoesc$.}
\label{fig:graph_hull}
\end{figure}
see \cref{fig:graph_hull}, and let $\cnoesc$ (``no-escape speed'') be a positive quantity, large enough so that
\[
\cnoesc \frac{\desc(m)^2}{4L} \ge 2 \frac{\KFZero}{\kappa_0}
\,,
\quad\text{namely}\quad
\cnoesc = \frac{8\,\KFZero \, L}{\kappa_0 \, \desc(m)^2}
\,.
\]
As for the quantities $\kappa_0$ and $\nuFZero$ and $\KFZero$, the quantities $L$ and $\cnoesc$ and the function $\hullnoesc$ depend on $V$ \emph{and} $m$. 
The following lemma, illustrated by \cref{fig:hull_invasion}, is a variant of \cite[\GlobalRelaxationLemBoundInvasionSpeed]{Risler_globalRelaxation_2016}. 
\begin{figure}[!htbp]
\centering
\includegraphics[width=\textwidth]{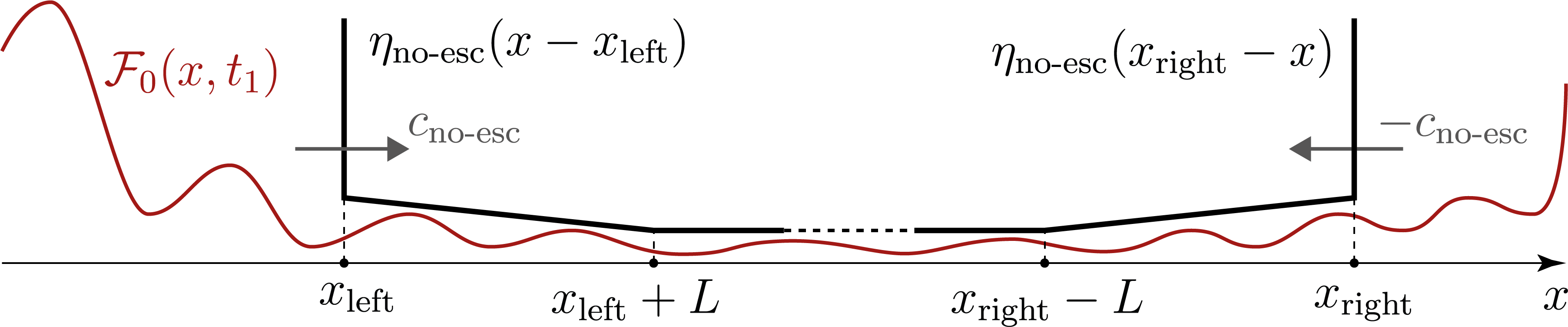}
\caption{Illustration of \cref{lem:inv}; if the firewall function is below the maximum of two mirror hulls at a certain time $t_1$ and if these two hulls travel at opposite speeds $\pm\cnoesc$, then the firewall will remain below the maximum of those travelling hulls in the future (note that after they cross this maximum equals $+\infty$ thus the assertion of being ``below'' is empty).}
\label{fig:hull_invasion}
\end{figure}
\begin{lemma}[bound on invasion speed]
\label{lem:inv}
For all real quantities $\xLeft$ and $\xRight$ and nonnegative time $t_1$, if
\[
\fff_0(x,t_1) \le \max\bigl( \hullnoesc(x-\xLeft) , \hullnoesc(\xRight - x) \bigr)
\quad\text{for all } x \text{ in }\rr
\,,
\]
then, for every time $t$ greater than or equal to $t_1$ and every real quantity $x$, 
\[
\fff_0(x,t) \le \max\Bigl( \hullnoesc\bigl(\xLeft-\cnoesc\, (t-t_1)\bigr), \hullnoesc\bigl(\xRight + \cnoesc\, (t-t_1)- x\bigr) \Bigr)
\,.
\]
\end{lemma}
\begin{proof}
See \cite[\GlobalRelaxationLemBoundInvasionSpeed]{Risler_globalRelaxation_2016}.
\end{proof}
\subsection{Set-up for the proof, 2: escape point and associated speeds}
\label{subsec:inv_cv_set_pf_cont}
With the notation and results of the previous \cref{subsec:def_fire_zero,subsec:up_bd_inv_vel} at hand, let us pursue the set-up for the proof of \cref{prop:inv_cv} ``invasion implies convergence''. According to hypothesis \textup{(\hyperlink{hypHomRight}{\hypHomRightRef})} and to the bounds \cref{bound_u_ut_ck_bis} on the solution, it may be assumed, up to changing the origin of time, that, for all $t$ in $[0,+\infty)$ and for all $x$ in $\rr$,
\begin{equation}
\label{hyp_for_def_xesc}
\fff_0(x,t) \le \max\biggl( \hullnoesc\Bigl( x-\bigl( \xHom(t)-1\bigr) \Bigr) , \hullnoesc \bigl( \xHom(t) - x \bigr) \biggr)
\,.
\end{equation}
As a consequence, for all $t$ in $[0,+\infty)$, the set
\[
\begin{aligned}
\IHom(t) = \Bigl\{ & x_{\ell} \le \xHom(t) : \text{ for all } x \text{ in } \rr\,, 
\\
& \fff_0(x,t) \le \max\Bigl( \hullnoesc( x-x_{\ell}) , \hullnoesc \bigl( \xHom(t) - x \bigr) \Bigr)
\Bigr\}
\end{aligned}
\]
is a nonempty interval (containing $[\xHom(t)-1,\xHom(t)]$) that must be bounded from below (see \cref{fig:def_xesc}). 
\begin{figure}[!htbp]
\centering
\includegraphics[width=\textwidth]{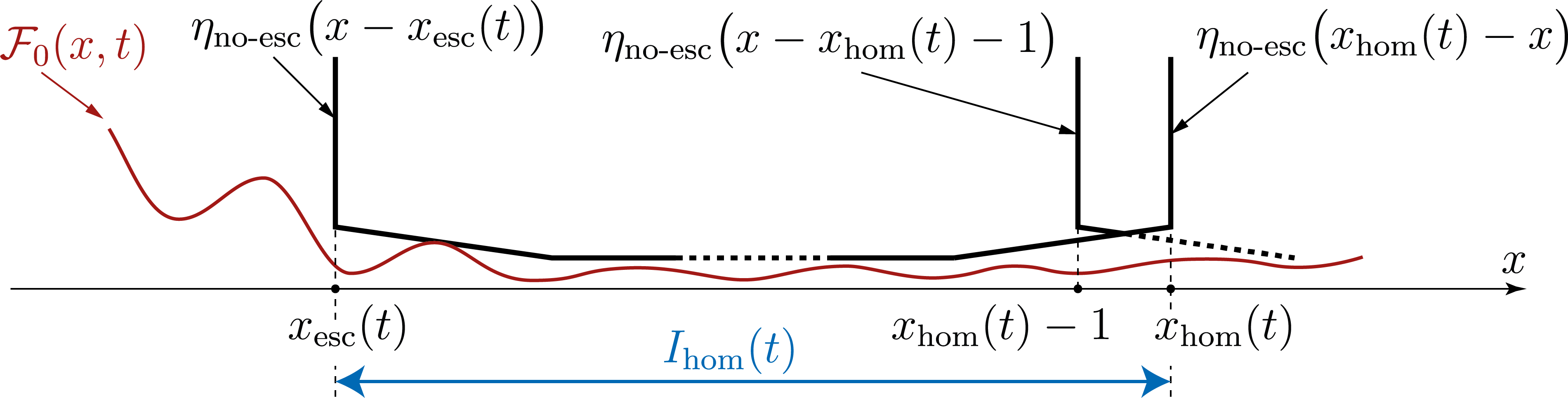}
\caption{Interval $\IHom(t)$ and definition of $\xesc(t)$.}
\label{fig:def_xesc}
\end{figure}
Indeed, if at a certain time it was not bounded from below --- in other words if it was equal to $(-\infty,\xHom(t)]$ --- then according to \cref{lem:inv} this would remain unchanged in the future, thus according to \cref{lem:esc_Esc} the point $\xEsc(t)$ would remain equal to $-\infty$ in the future, a contradiction with hypothesis \textup{(\hyperlink{hypInv}{\hypInvRef})}.

For every nonnegative time $t$, let 
\begin{equation}
\label{def_xesc}
\xesc(t) = \inf \bigl( \IHom(t) \bigr) 
\quad
\text{(thus } \xesc(t)>-\infty \text{).}
\end{equation}
Somehow like $\xEsc(t)$, this point represents the first point at the left of $\xHom(t)$ where the solution $u(x,t)$ ``escapes'' (in a sense defined by the firewall function $\fff_0$ and the no-escape hull $\hullnoesc$) at a certain distance from $m$. In the following, this point $\xesc(t)$ will be called the ``escape point'' (by contrast with the ``Escape point'' $\xEsc(t)$ defined before). According to assumption \cref{hyp_xEsc_finite} and to the ``hull inequality'' \cref{hyp_for_def_xesc} and \cref{lem:esc_Esc} (``escape/Escape''), for every nonnegative time $t$,
\begin{equation}
\label{xEsc_xesc_xHom}
-\infty<\xEsc(t) \le \xesc(t) \le \xHom(t)-1 
\quad\text{and}\quad
\SigmaEscZero(t) \cap [\xEsc(t),\xHom(t)] = \emptyset
\,,
\end{equation}
and, according to hypothesis \textup{(\hyperlink{hypHomRight}{\hypHomRightRef})} and to the bounds \cref{bound_u_ut_ck_bis} on the solution, 
\begin{equation}
\label{xHom_minus_xesc}
\xHom(t) - \xesc(t) \to +\infty
\quad\text{as}\quad
t\to +\infty
\,.
\end{equation}
The big advantage of $\xesc(\cdot)$ with respect to $\xEsc(\cdot)$ is that, according to \cref{lem:inv}, the growth of $\xesc(\cdot)$ is more under control. More precisely, according to this lemma, for all nonnegative quantities $t$ and $s$, 
\begin{equation}
\label{control_escape}
\xesc(t+s)\le \xesc(t) + \cnoesc \, s
\,.
\end{equation}
For every $s$ in $[0,+\infty)$, let us consider the ``upper and lower bounds of the variations of $\xesc(\cdot)$ over all time intervals of length $s$'':
\begin{figure}[!htbp]
\centering
\includegraphics[width=\textwidth]{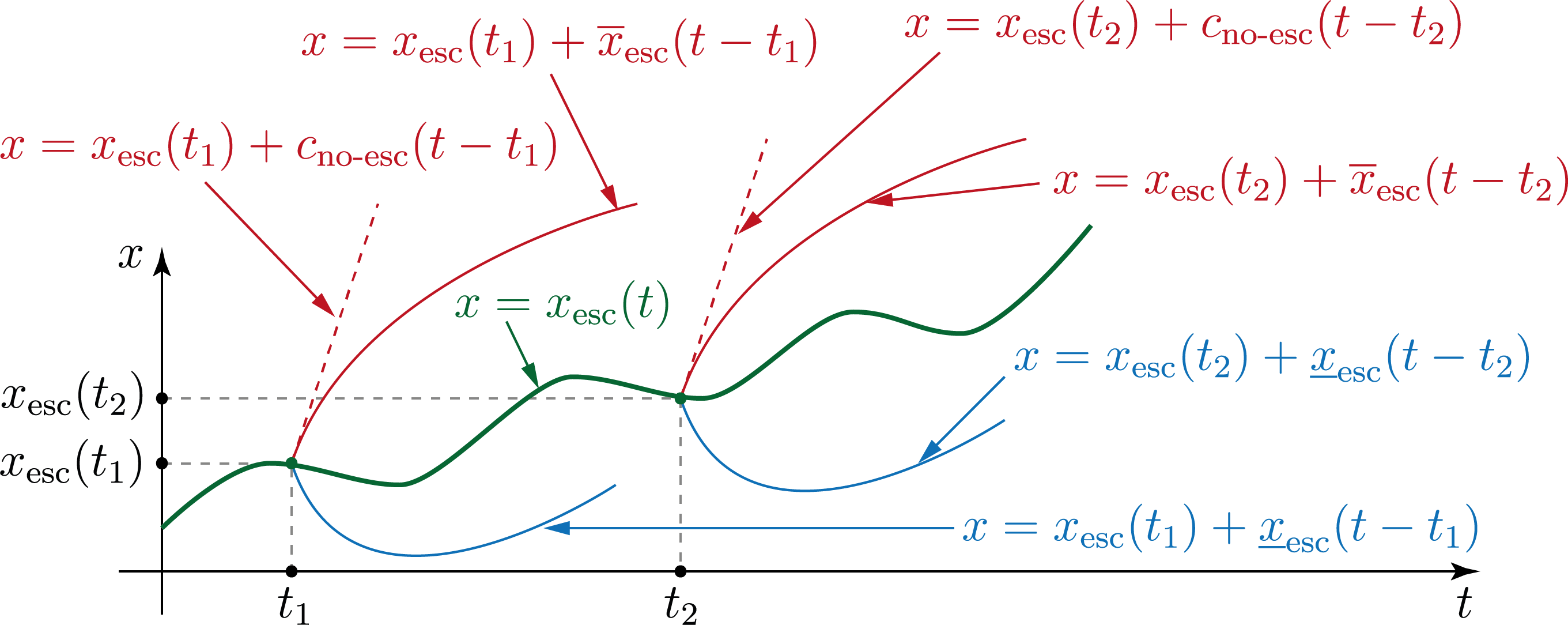}
\caption{Illustration of the bounds \cref{various_bounds_x_esc}.}
\label{fig:def_bar_underbar_xesc}
\end{figure}
\begin{equation}
\label{def_bar_underbar_xesc}
\barxesc(s) = \sup_{t\in[0,+\infty)} \xesc(t+s) - \xesc(t)
\quad\text{and}\quad
\underbarxesc(s) = \inf_{t\in[0,+\infty)} \xesc(t+s) - \xesc(t)
\,,
\end{equation}
see \cref{fig:def_bar_underbar_xesc}. According to these definitions and to inequality \cref{control_escape} above, for all $t$ and $s$ in $[0,+\infty)$,
\begin{equation}
\label{various_bounds_x_esc}
-\infty\le\underbarxesc(s)\le \xesc(t+s)-\xesc(t)\le \barxesc(s) \le \cnoesc\, s
\,.
\end{equation}
Let us consider the four limit mean speeds:
\[
\cescinf = \liminf_{t\to+\infty}\frac{\xesc(t)}{t}
\quad\text{and}\quad
\cescsup = \limsup_{t\to+\infty}\frac{\xesc(t)}{t}
\]
and
\[
\underbarcescinf = \liminf_{s\to+\infty}\frac{\underbarxesc(s)}{s}
\quad\text{and}\quad
\barcescsup = \limsup_{s\to+\infty}\frac{\barxesc(s)}{s}
\,.
\]
The following inequalities follow from these definitions and from hypothesis \textup{(\hyperlink{hypInv}{\hypInvRef})}:
\[
-\infty \le \underbarcescinf \le \cescinf \le \cescsup \le \barcescsup \le \cnoesc
\quad\text{and}\quad
0 < \cEsc \le \cescsup
\,.
\]
The four limit mean speeds defined above will turn out to be equal. The proof of this equality is based of the ``relaxation scheme'' set up in the next \namecref{subsec:relax_sch_tr_fr}. 
\begin{remark}
In the previous paper \cite{Risler_globCVTravFronts_2008} where convergence towards a single travelling front was proved, no object similar to the lower bound $\underbarxesc(\cdot)$ or the quantity $\underbarcescinf$ was defined. Here those objects will be specifically required to prove convergence towards the travelling fronts ``following'' the ``first'' ones (in the statement of \cref{thm:main}, the ``first'' ones are $\phi_{1,+}$ and $\phi_{1,-}$, and the ``following'' ones are the $\phi_{i,+}$ and $\phi_{i,-}$ with $i$ larger than $1$). Indeed, for those ``following'' travelling fronts, a tighter control over the escape point will be required.  
\end{remark}
\subsection{Relaxation scheme in a travelling frame}
\label{subsec:relax_sch_tr_fr}
The aim of this \namecref{subsec:relax_sch_tr_fr} is to set up an appropriate relaxation scheme in a travelling frame. This means defining an appropriate localized energy and controlling the ``flux'' terms occurring in the time derivative of this localized energy. The considerations made in \cref{subsec:1rst_ord} will be put in practice. 
\subsubsection{Notation for the travelling frame}
\label{subsubsec:def_trav_f}
Let us keep the notation and hypotheses introduced above (since the beginning of \cref{subsec:inv_cv_set_pf}), and let us introduce the following real quantities that will play the role of ``parameters'' for the relaxation scheme below (see \cref{fig:trav_fr}):
\begin{figure}[!htbp]
\centering
\includegraphics[width=.8\textwidth]{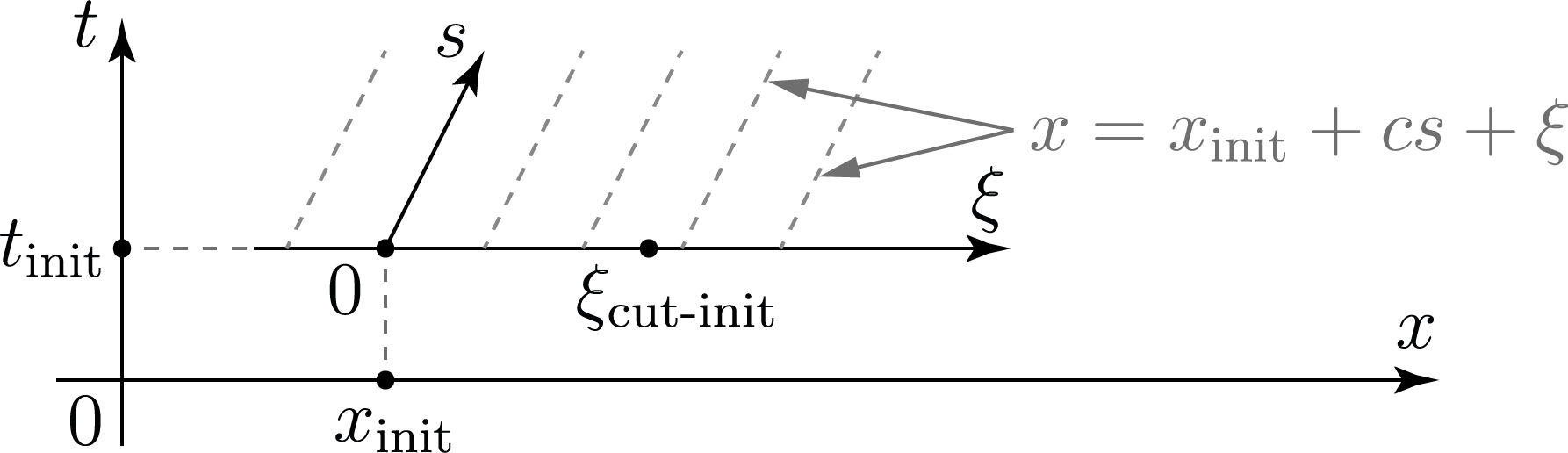}
\caption{Space coordinate $\xi$ and time coordinate $s$ in the travelling frame, and parameters $\tInit$ and $\xInit$ and $c$ and $\initCut$.}
\label{fig:trav_fr}
\end{figure}
\begin{itemize}
\item the ``initial time'' $\tInit$ of the time interval of the relaxation;
\item the position $\xInit$ of the origin of the travelling frame at initial time $t=\tInit$ (in practice it will be chosen equal to $\xesc(\tInit)$);
\item the speed $c$ of the travelling frame;
\item a quantity $\initCut$ that will be the the position of the maximum point of the weight function $\xi\mapsto\chi(\xi,\tInit)$ localizing energy at initial time $t=\tInit$ (this weight function is defined below); the subscript ``cut'' refers to the fact that this weight function displays a kind of ``cut-off'' on the interval between this maximum point and $+\infty$. Thus the maximum point is in some sense the point ``where the cut-off begins''. 
\end{itemize}
Let us make on these parameters the following hypotheses:
\begin{equation}
\label{hyp_param_relax_sch}
0\le \tInit
\quad\text{and}\quad
0 < c \le \cnoesc
\quad\text{and}\quad
0 \le \initCut
\,.
\end{equation}
The relaxation scheme will be applied several time in the next pages, for various choices of these three parameters. 

For all $\xi$ in $\rr$ and $s$ in $[0,+\infty)$, let
\[
v(\xi,s) = u^\dag(x,t)
\quad\text{where}\quad
x = \xInit + cs + \xi
\quad\text{and}\quad
t = \tInit + s
\,.
\]
This function $(\xi,s)\mapsto v(\xi,s)$ is thus defined from $(x,t)\mapsto u^\dag(x,t)$ by considering this solution in a frame travelling at the speed $c$, with $\tInit$ as the origin of times and $\xInit$ as the origin of space. It satisfies the differential system already written in \vref{syst_mf} and \vref{syst_mf_bis} (with $V^\dag$ instead of $V$), that is:
\begin{equation}
\label{syst_mf_V_dag}
v_s - c v_\xi = - \nabla V^\dag (v) + v_{\xi\xi}
\,.
\end{equation}
\subsubsection{Principle of the relaxation scheme}
Two functions will now be defined, each with its own weight function:
\begin{itemize}
\item a localized energy $s\mapsto\eee(s)$;
\item a localized ``firewall'' function $s\mapsto\fff(s)$, that will be a linear combination of the energy and the $L^2$-norm with appropriate coefficients.
\end{itemize}
Here are the constraints that have to be faced in the choice of these definitions. Some of them have already been mentioned in \cref{subsec:1rst_ord}. Note that these constraints are slightly more involved than in a standing frame (this latter case is easier and treated in details in \cite{Risler_globalRelaxation_2016}).
\begin{enumerate}
\item Both weight functions should vary slowly with time.
\item The weight function for the energy should be equal --- or at least close --- to $\exp (c\xi)$ (up to a positive multiplicative constant) in a subset of the space real line ``as large as possible'' since such a subset does not ``contribute'' to the flux terms in the time derivative of localized energy. 
\item The weight function for the firewall functional should either be equal --- or at least close --- to $\exp (c\xi)$ (up to a positive multiplicative constant), or vary slowly with respect to space and time, since each of these conditions ensures the smallness of the flux term of the $L^2$-norm. 
\item The (positive) coefficients of the energy and of the $L^2$-norm in the definition of the firewall functional should face two independent constraints, both in favour of a larger coefficient for the $L^2$-norm:
\begin{itemize}
\item the firewall should be coercive;
\item in the time derivative of the firewall, the ``non small'' flux terms of the derivative of energy (where the weight function of energy is \emph{not} close to $\exp (c\xi)$) should be balanced by the ``main'' (nonpositive) terms in the derivative of the $L^2$-norm.
\end{itemize}
\item Finally, in order the positive part of the total contribution of the flux terms of the energy to be \emph{small} (and not only bounded), as will be required to prove convergence towards stationary solutions in a travelling referential, the initial weight function for the energy will have to be chosen equal to $\exp (c\xi)$ up to far to the right, and it is the parameter $\initCut$ that will enable this ``tuning''. 
\end{enumerate}
\subsubsection{Choice of the parameters \texorpdfstring{$\kappa$}{kappa} and \texorpdfstring{$\cCut$}{cCut} and \texorpdfstring{$\coeffEn$}{wEn}}
Let $\kappa$ (rate of decrease of the weight functions), $\cCut$ (speed of the cut-off point in the travelling frame), and $\coeffEn$ (coefficient of energy in the firewall function) be three positive quantities, small enough so that
\begin{equation}
\label{conditions_kappa_cCut_coeffEn}
\begin{aligned}
&\coeffEn(c+\kappa)\Bigl(\frac{\cCut}{2}+\frac{c+\kappa}{4}\Bigr) \le \frac{1}{2} 
\qquad\text{and}\qquad
\coeffEn \cCut (c+\kappa) \le \frac{1}{4} \\
\text{and}\qquad
&\frac{(\cCut+\kappa)(c+\kappa)}{2} \le \frac{\eigVmin(m)}{8}
\end{aligned}
\end{equation}
(compare with conditions \vref{conditions_kappaZero_lab}) and
\begin{equation}
\label{coeffEn_smaller_than_coeffEnZero}
\coeffEn \le \coeffEnZero
\end{equation}
(the quantity $\coeffEnZero$ was defined in \vref{subsubsec:low_quad_hull}). 
Conditions \cref{conditions_kappa_cCut_coeffEn} will be used to prove inequality \vref{ds_fire_before_nufire}. 
Condition \cref{coeffEn_smaller_than_coeffEnZero} will be used in the proofs of \cref{lem:ds_en_trav_f} (energy decrease up to firewall) and \cref{lem:nonnegativity_F} (nonnegativity of firewall).
These quantities may be chosen as follows ($\kappa$ and $\cCut$ are chosen so that the third inequality of \cref{conditions_kappa_cCut_coeffEn} be fulfilled, and then $\coeffEn$ is chosen according to the first two inequalities of \cref{conditions_kappa_cCut_coeffEn} and to \cref{coeffEn_smaller_than_coeffEnZero}):
\[
\begin{aligned}
\kappa &= \min\Bigl( \frac{\sqrt{\eigVmin(m)}}{4} , \frac{\eigVmin(m)}{16\cnoesc} \Bigr) 
\qquad\text{and}\qquad
\cCut = \frac{\eigVmin(m)}{8(\cnoesc+\kappa)} \,, \\
\text{and}\qquad\coeffEn &= \min\Bigl(\frac{2}{(\cnoesc + \kappa)(\cnoesc + \kappa+2\cCut)} , \frac{1}{4\cCut(\cnoesc + \kappa)},  \coeffEnZero\Bigr) \,.
\end{aligned}
\]
These three quantities depend on $V$ \emph{and} $m$. 
\subsubsection{Localized energy}
\label{subsubsec:def_loc_en}
For every real quantity $s$, let us introduce the two intervals
\[
\begin{aligned}
\iMain(s) &= ( - \infty , \initCut + \cCut s] \,, \\
\text{and}\qquad
\iRight(s) &= [ \initCut + \cCut s , +\infty) \,,
\end{aligned}
\]
and let us introduce the function $\chi(\xi,s)$ (weight function for the localized energy) defined as
\[
\chi(\xi,s) =
\left\{
\begin{aligned}
&\exp(c\xi) 
& &\text{if}\quad
\xi \in \iMain(s) \,, \\
&\exp\Bigl[ c (\initCut + \cCut s) - \kappa \bigl(\xi - (\initCut +\cCut s) \bigr)\Bigr]
& &\text{if}\quad
\xi \in \iRight(s) 
\,,
\end{aligned}
\right.
\]
\begin{figure}[!htbp]
\centering
\includegraphics[width=0.9\textwidth]{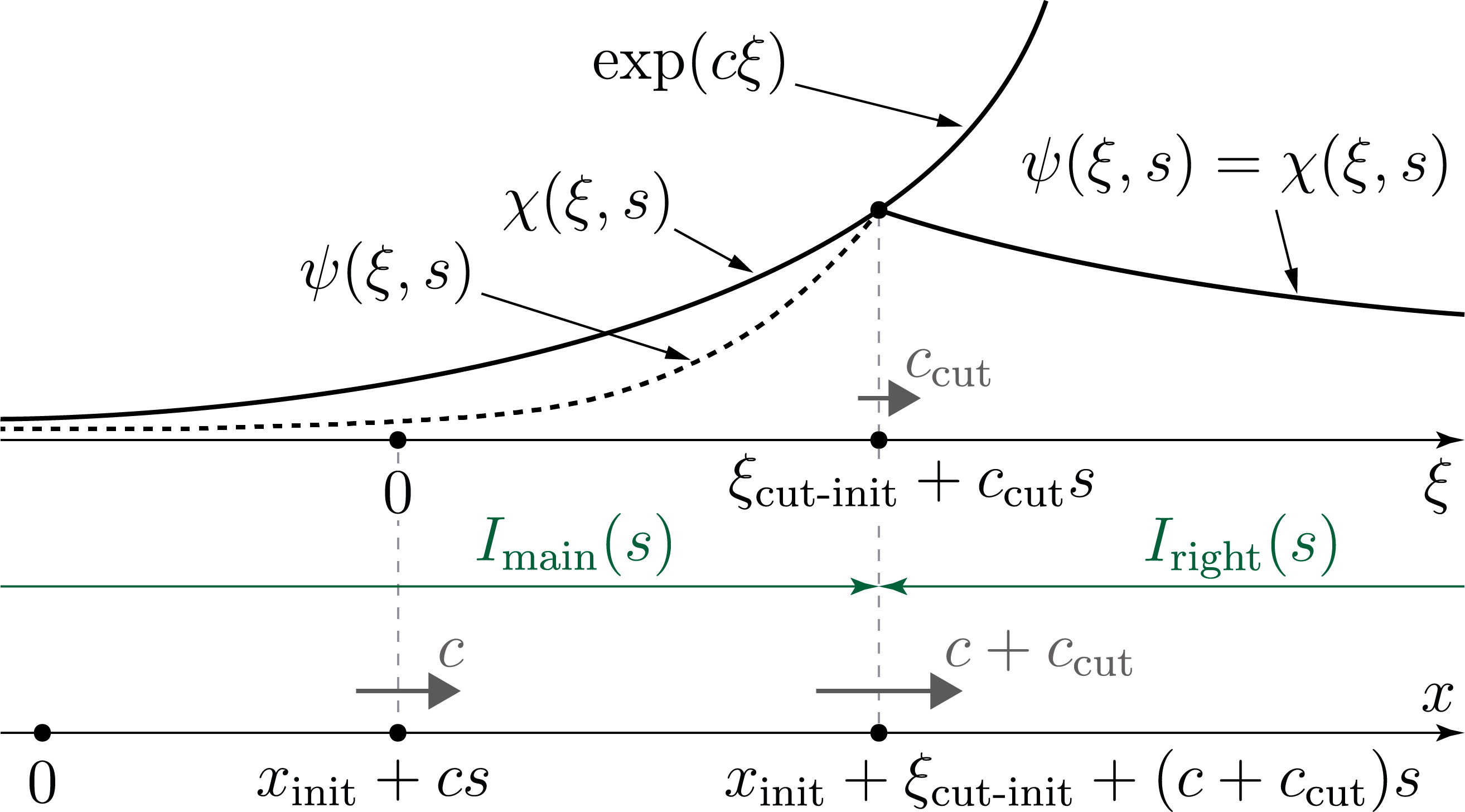}
\caption{Graphs of the functions $\xi\mapsto\chi(\xi,s)$ and $\xi\mapsto\psi(\xi,s)$.}
\label{fig:chi_psi}
\end{figure}
see \cref{fig:chi_psi}; and, for all $s$ in $[0,+\infty)$, let us define the ``energy'' $\eee(s)$ by
\[
\eee(s) = \int_{\rr} \chi(\xi,s) \left(\frac{1}{2}v_\xi(\xi,s)^2 + V^\dag\bigl(v(\xi,s)\bigr)\right) \, d\xi
\,.
\]
Note that the choice of the decrease rate of $\chi(\cdot,s)$ on the interval $[\initCut + \cCut s, +\infty)$ is not crucial (it might have been chosen equal to $1$, for instance). The sole advantage of this choice is that the weight functions $\chi$ and $\psi$ (the latter one defined below) are identical on this interval, which is convenient for the estimates below. 
\subsubsection{Time derivative of localized energy}
\label{subsubsec:der_loc_en}
For all $s$ in $[0,+\infty)$, let us define the ``dissipation'' function by
\begin{equation}
\label{def_dissip_tr_fr}
\ddd(s) = \int_{\rr} \chi(\xi,s)\, v_s(\xi,s)^2 \, d\xi
\,.
\end{equation}
\begin{lemma}[time derivative of localized energy]
\label{lem:ds_en_trav_f_prel}
For all $s$ in $[0,+\infty)$, 
\begin{equation}
\label{ds_en_trav_f_prel}
\eee'(s) \le -\frac{1}{2} \ddd(s) + 
(c+\kappa)\int_{\iRight(s)} \chi \Bigl(\frac{c+\cCut+\kappa}{2} v_\xi^2 + \cCut V^\dag(v) \Bigr) \, d\xi 
\,.
\end{equation}
\end{lemma}
\begin{proof}
For all $s$ in $[0,+\infty)$, it follows from expression \vref{ddt_loc_en} for the derivative of a localized energy that
\[
\eee'(s) = - \ddd(s) + \int_{\rr}\biggl(\chi_s\Bigl( \frac{1}{2}v_\xi^2 + V^\dag(v)\Bigr) + (c\chi-\chi_\xi) v_\xi \cdot v_s\biggl) \, d\xi 
\,.
\]
It follows from the definition of $\chi$ that:
\[
\chi_s(\xi,s) = 
\left\{
\begin{aligned}
&0  
& &\text{if}\quad \xi \in \iMain(s) \,, \\ 
&\cCut(c+\kappa)\, \chi(\xi,s) 
& &\text{if}\quad \xi \in \iRight(s) \,,
\end{aligned}
\right.
\]
and
\[
(c\chi-\chi_\xi)(\xi,s) = 
\left\{
\begin{aligned}
&0 
& &\text{if}\quad \xi \in \iMain(s) \,, \\ 
&(c+\kappa)\, \chi(\xi,s) 
& &\text{if}\quad \xi \in\iRight(s) \,.
\end{aligned}
\right.
\]
Thus it follows from these expressions that
\[
\eee'(s) = - \ddd(s) + (c+\kappa)\int_{\iRight(s)} \chi \biggl(\cCut \Bigl( \frac{1}{2}v_\xi^2 + V^\dag(v)\Bigr) + v_\xi \cdot v_s\biggl) \, d\xi 
\,.
\]
Using the inequality
\[
(c+\kappa)\,  v_\xi \cdot v_s \le \frac{1}{2} v_s^2 + \frac{(c+\kappa)^2}{2} v_\xi^2
\,,
\]
inequality \cref{ds_en_trav_f_prel} follows. \Cref{lem:ds_en_trav_f_prel} is proved. 
\end{proof}
\subsubsection{Firewall function} 
\label{subsubsec:def_firewall}
A second function (the ``firewall'') will now be defined, to get some control over the second term of the right-hand side of inequality \cref{ds_en_trav_f_prel}.
Let us introduce the function $\psi(\xi,s)$ defined as
\[
\psi(\xi,s) = 
\left\{
\begin{aligned}
&\exp\Bigl[\kappa\bigl( \xi -(\initCut + \cCut s)\bigr) \Bigr] \chi(\xi,s)
& &\text{if}\quad \xi \in \iMain(s) \,, \\ 
&\chi(\xi,s) 
& &\text{if}\quad \xi \in \iRight(s) \,,
\end{aligned}
\right.
\]
see \cref{fig:chi_psi}; and, for all $s$ in $[0,+\infty)$, let us define the ``firewall'' $\fff(s)$ by
\begin{equation}
\label{def_fff}
\fff(s) = \int_{\rr} \psi(\xi,s) \Biggl(\coeffEn \left(\frac{1}{2}v_\xi(\xi,s)^2 + V^\dag\bigl(v(\xi,s)\bigr)\right) + \frac{1}{2}v(\xi,s)^2\Biggr) \, d\xi
\,.
\end{equation}
\subsubsection{Energy decrease up to firewall}
\label{subsubsec:loc_energy_decrease}
\begin{lemma}[energy decrease up to firewall]
\label{lem:ds_en_trav_f}
There exists a positive quantiy $\KEF$, depending (only) on $V$ and $m$, such that for every nonnegative quantity $s$, 
\begin{equation}
\label{ds_en_trav_f}
\eee'(s) \le -\frac{1}{2} \ddd(s) + \KEF \fff(s)
\,.
\end{equation}
\end{lemma}
\begin{proof}
Inequality \cref{ds_en_trav_f_prel} of \cref{lem:ds_en_trav_f_prel} can be rewritten (without changing the value of the right-hand side) as follows (note the substitution of $\chi$ by $\psi$, which is allowed since these two functions are equal on $\iRight(s)$):
\[
\begin{aligned}
\eee'(s) 
&\le -\frac{1}{2} \ddd(s) + (c+\kappa)\int_{\iRight(s)} \psi \biggl(\frac{c+\cCut+\kappa}{2} v_\xi^2 + \cCut V^\dag(v) \biggr) \, d\xi \\
&\le -\frac{1}{2} \ddd(s) + (c+\kappa)\int_{\iRight(s)}  \psi \biggl(\frac{c+\cCut+\kappa}{2} v_\xi^2 + \cCut\Bigl(V^\dag(v) + \frac{1}{2\coeffEn} v^2\Bigr) \biggr) \, d\xi
\,.
\end{aligned}
\]
Since according to inequality \vref{coeffEn_smaller_than_coeffEnZero} the quantity $\coeffEn$ is not larger than $\coeffEnZero$, it follows from inequality \vref{def_weight_en_V_dag} that the quantity
\[
V^\dag(v) + \frac{1}{2\coeffEn} v^2 = \frac{1}{\coeffEn}\Bigl(\coeffEn V^\dag(v) + \frac{1}{2} v^2\Bigr)
\]
is nonnegative, and as a consequence the previous inequality still holds if the factor $\cCut$ of this quantity is replaced by the larger factor $\cCut+c+\kappa$. After this replacement, the inequality reads
\begin{equation}
\label{ds_en_trav_f_proof}
\eee'(s) \le -\frac{1}{2} \ddd(s) +  \frac{(c+\kappa)(c+\cCut+\kappa)}{\coeffEn} \int_{\iRight(s)} \psi\biggl(\coeffEn\Bigl(\frac{1}{2} v_\xi^2 +V^\dag(v) \Bigr) + \frac{1}{2} v^2\biggr) \, d\xi
\,.
\end{equation}
Again according to \cref{def_weight_en_V_dag,coeffEn_smaller_than_coeffEnZero}, this inequality \cref{ds_en_trav_f_proof} still holds if the domain of integration is extended to the whole real line; after this extension, this inequality \cref{ds_en_trav_f_proof} reads: 
\[
\eee'(s) \le -\frac{1}{2} \ddd(s) +  \frac{(c+\kappa)(c+\cCut+\kappa)}{\coeffEn} \fff(s)
\,.
\]
Finally, introducing the (positive) quantity $\KEF$ (depending only on $V$) defined as
\[
\KEF = \frac{(\cnoesc+\kappa)(\cnoesc+\cCut+\kappa)}{\coeffEn}
\,,
\]
inequality \cref{ds_en_trav_f} follows from the last inequality above. \Cref{lem:ds_en_trav_f} is proved. 
\end{proof}
\subsubsection{Relaxation scheme inequality, 1}
Let $\sFin$ be a nonnegative quantity (denoting the length of the time interval on which the relaxation scheme will be applied). It follows from the previous inequality that
\begin{equation}
\label{diff_s_en_trav_f}
\frac{1}{2} \int_0^{\sFin} \ddd(s) \, ds \le \eee(0) - \eee(\sFin) + \KEF \int_0^{\sFin} \fff(s) \, ds
\,.
\end{equation}
The approximate decrease of of the localized energy (up to a remaining term controlled by the firewall function) above, and more specifically the integrated form \cref{diff_s_en_trav_f} are the core of the relaxation scheme that is set up. Indeed, if some control can be obtained over the right-hand side (upper bound on the firewall function, upper bound on the initial localized energy, lower bound on the final localized energy), then it will provide some control on the integral of the dissipation, an information related to the vicinity of the solution to travelling fronts. The next goal (in the next \namecref{subsubsec:der_fire}) is to gain some control over the firewall function. 
\subsubsection{Firewall linear decrease up to pollution}
\label{subsubsec:der_fire}
For every $s$ in $[0,+\infty)$, let us introduce the set --- the domain of space where the solution $v$ (resp. $u$) ``Escapes'' at distance $\dEsc(m)$ from $0_{\rr^d}$ (resp. from $m$):
\[
\SigmaEsc(s)=\{\xi\in\rr: \abs{v(\xi,s)} >\dEsc(m)\} 
\,.
\]
To make the connection with the definition \vref{def_Sigma_Esc_0} of the related set $\SigmaEscZero(t)$, observe that, for all $s$ in $[0,+\infty)$ and $\xi$ in $\rr$, 
\[
\xi\in \SigmaEsc(s) \iff 
\xInit + cs + \xi \in \SigmaEscZero(\tInit+s)
\,.
\]
The next step is the following lemma (observe the strong similarity with \vref{lem:dt_fire}).
\begin{lemma}[firewall linear decrease up to pollution]
\label{lem:decrease_firewall}
There exist positive quantities $\nuF$ and $\KF$, both depending (only) on $V$ and $m$, such that, for all $s$ in $[0,+\infty)$,
\begin{equation}
\label{ds_fire}
\fff'(s) \le - \nuF \fff(s) + \KF \int_{\SigmaEsc(s)} \psi(\xi,s) \, d\xi 
\,.
\end{equation}
\end{lemma}
\begin{proof}
According to expressions \vref{ddt_loc_en,ddt_loc_L2} for the time derivatives of a localized energy and a localized $L^2$ functional, for all $s$ in $[0,+\infty)$,
\[
\begin{aligned}
\fff'(s) =  \int_{\rr} \Biggl[ & \psi \Bigl( - \coeffEn v_s^2 - v\cdot \nabla V^\dag(v) - v_\xi^2 \Bigr) 
+ \coeffEn \psi_s   \Bigl( \frac{1}{2}v_\xi^2 + V^\dag(v) \Bigr)   \\
& + \coeffEn (c\psi - \psi_\xi)  v_\xi \cdot v_s 
+ \frac{\psi_s + \psi_{\xi\xi} - c\psi_\xi}{2}v^2 \Biggr] \, d\xi
\,.
\end{aligned}
\]
According to the definition of $\psi$, 
\[
\psi_s(\xi,s) = 
\left\{
\begin{aligned}
&- \kappa\cCut\psi(\xi,s)
& &\text{if}\quad \xi \in \iMain(s) \,, \\ 
&(c+\kappa)\cCut\psi(\xi,s) 
& &\text{if}\quad \xi \in \iRight(s) \,,
\end{aligned}
\right.
\]
and 
\begin{equation}
\label{expression_c_psi_minus_psi_xi}
c\psi(\xi,s)-\psi_\xi(\xi,s) = 
\left\{
\begin{aligned}
&-\kappa\psi(\xi,s)
& &\text{if}\quad \xi \in \iMain(s) \,, \\ 
&(c+\kappa)\psi(\xi,s) 
& &\text{if}\quad \xi \in \iRight(s) \,,
\end{aligned}
\right.
\end{equation}
and, for all $\xi$ in $\rr$, if $\delta_{\initCut + \cCut s}(\cdot)$ denotes the Dirac mass at $\xi=\initCut + \cCut s$,
\begin{equation}
\label{expression_psi_xixi_minus_c_psi_xi}
\psi_{\xi\xi}(\xi,s)- c\psi_\xi(\xi,s)= \kappa(c+\kappa) \psi(\xi,s)-(c+2\kappa)\exp\bigl[c(\initCut + \cCut s)\bigr]\delta_{\initCut + \cCut s}(\xi)
\,.
\end{equation}
As a consequence, the following inequalities hold for all values of the arguments:
\begin{equation}
\label{bounds_psi_psi_s_psi_xi_etc}
\begin{aligned}
\abs{\psi_s}&\le \cCut(c+\kappa) \, \psi \,, \\
\text{and}\quad
\abs{c\psi - \psi_\xi} &\le (c + \kappa) \,  \psi \,, \\
\text{and}\quad
\psi_{\xi\xi}-c\psi_\xi &\le \kappa(c+\kappa) \, \psi
\,.
\end{aligned}
\end{equation}
Thus, for all $s$ in $[0,+\infty)$,
\[
\begin{aligned}
\fff'(s) \le  \int_{\rr} \psi \Biggl[ & - \coeffEn v_s^2 - v\cdot \nabla V^\dag(v) - v_\xi^2 + \coeffEn\, \cCut(c+\kappa) \Bigl( \frac{1}{2}v_\xi^2 + \abs{V^\dag(v)} \Bigr)  \\
& + \coeffEn (c+\kappa) \abs{v_\xi \cdot v_s} +  \frac{(\cCut + \kappa)(c+\kappa)}{2}v^2 \Biggr] \, d\xi
\,.
\end{aligned}
\]
Using the inequality
\[
\coeffEn (c+\kappa) \abs{v_\xi \cdot v_s} \le \coeffEn v_s^2 + \coeffEn \frac{(c+\kappa)^2}{4} v_\xi^2
\,,
\]
it follows that
\[
\begin{aligned}
\fff'(s)\le \int_{\rr} \psi \Biggl[ &
\biggl( \coeffEn (c+\kappa)\Bigl( \frac{\cCut}{2} + \frac{c+\kappa}{4} \Bigr) - 1 \biggr) v_\xi^2 - v\cdot \nabla V^\dag(v) \\
& + \coeffEn \, \cCut (c+\kappa) \, \abs{V^\dag(v)} + \frac{(\cCut + \kappa)(c+\kappa)}{2}v^2  \Biggr] \, d\xi
\,,
\end{aligned}
\]
and according to the conditions \vref{conditions_kappa_cCut_coeffEn} satisfied by $\kappa$ and $\cCut$ and $\coeffEn$, it follows that
\begin{equation}
\label{ds_fire_before_nufire}
\fff'(s) \le \int_{\rr} \psi\Bigl[ -\frac{1}{2}v_\xi^2 - v\cdot \nabla V^\dag(v) + \frac{1}{4} \abs{V^\dag(v)} + \frac{\eigVmin(m)}{8} v^2 \Bigr] \, d\xi
\end{equation}
(compare with \cite[inequality \GlobalRelaxationTimeDerFireBeforeNuFire]{Risler_globalRelaxation_2016} in a standing frame; the additional term $\abs{V^\dag(v)}/4$ comes from the time dependence of the weight $\psi$).

Let $\nuF$ be a positive quantity to be chosen below. It follows from the previous inequality and from the definition \cref{def_fff} of $\fff(s)$ that
\begin{equation}
\label{ds_fire_prel}
\begin{aligned}
\fff'(s) + \nuF\fff(s) \le \int_{\rr} \psi \biggl[&-\frac{1}{2}(1-\nuF\,\coeffEn)v_\xi^2 - v\cdot \nabla V^\dag(v)  \\
&  + \Bigl(\frac{1}{4}+\nuF\coeffEn\Bigr) \abs{V^\dag(v)} +  \Bigl(\frac{\eigVmin(m)}{8} + \frac{\nuF}{2}\Bigr)v^2  \biggr]\, d\xi 
\,.
\end{aligned}
\end{equation}
In view of this expression and of inequalities \vref{v_nablaV_controls_square_around_loc_min_dag,v_nablaV_controls_pot_around_loc_min_dag}, let us assume that $\nuF$ is small enough so that
\begin{equation}
\label{conditions_on_nuFire}
\nuF\, \coeffEn \le 1 
\quad\text{and}\quad
\nuF\, \coeffEn \le \frac{1}{4}
\quad\text{and}\quad
\frac{\nuF}{2} \le \frac{\eigVmin(m)}{8} 
\end{equation}
(compare with conditions \vref{conditions_nuFireZero_lab} satisfied by the quantity $\nuFZero$); this quantity $\nuF$ may be chosen as
\[
\nuF = \min\Bigl(\frac{1}{4\coeffEn}, \frac{\eigVmin(m)}{4} \Bigr)
\,.
\]
Then, it follows from \cref{ds_fire_prel,conditions_on_nuFire} that
\begin{equation}
\label{ds_fire_prel_bis}
\fff'(s) + \nuF\fff(s) \le \int_{\rr} \psi \Bigl[- v\cdot \nabla V^\dag(v) + \frac{1}{2}\abs{V^\dag(v)} + \frac{\eigVmin(m)}{4}v^2 \Bigr]\, d\xi
\,.
\end{equation}
According to \cref{v_nablaV_controls_square_around_loc_min_dag,v_nablaV_controls_pot_around_loc_min_dag}, the integrand of the integral at the right-hand side of this inequality is nonpositive as long as $\xi$ is \emph{not} in $\SigmaEsc(s)$.
Therefore this inequality still holds if the domain of integration of this integral is changed from $\rr$ to $\SigmaEsc(s)$. Besides, observe that, in terms of the ``initial'' potential $V$ and solution $u(x,t)$, the factor of $\psi$ under the integral of the right-hand side of this last inequality reads
\[
- (u-m)\cdot \nabla V(u) + \frac{1}{2}\abs{V(u)-V(m)} + \frac{\eigVmin(m)}{4}(u-m)^2
\,,
\]
where $u$ denotes $u(x,t)$ with $t=\tInit +s$ and $x=\xInit + cs + \xi$. 
Thus, according to the $L^\infty$ bound \vref{attr_ball_infty_inv_implies_cv} for the solution, if $\KF$ denotes the quantity $\KFZero$ already defined in \vref{def_K_fire_zero}, then inequality \cref{ds_fire} follows from \cref{ds_fire_prel_bis} (with the domain of integration of the integral on the right-hand side restricted to $\SigmaEsc(s)$). This finishes the proof of \cref{lem:decrease_firewall}.
\end{proof}
\begin{remark}
Because of the term $c\psi$ in the expression \cref{expression_c_psi_minus_psi_xi} of $c\psi-\psi_\xi$ when $\xi$ is in $\iRight(s)$, there is no hope to obtain an inequality like \cref{ds_fire_before_nufire} without a condition involving $\cnoesc$ on $\coeffEn$; similarly, because of the term $\kappa c\psi$ in the expression \cref{expression_psi_xixi_minus_c_psi_xi} of $\psi_{\xi\xi}-c\psi_\xi$, there is no more hope to obtain this kind of inequality without a condition involving $\cnoesc$ on $\kappa$. This is the reason why the quantities $\coeffEnZero$ and $\kappa_0$ of \cref{subsec:relax_sc_stand} could not be reused in this \cref{subsec:relax_sch_tr_fr}, and why the quantities $\coeffEn$ and $\kappa$ had to be introduced instead. 
\end{remark}
\subsubsection{Nonnegativity of firewall} 
%
\begin{lemma}[nonnegativity of firewall]
\label{lem:nonnegativity_F}
For all $s$ in $[0,+\infty)$, 
\begin{equation}
\label{coercivity_fff}
\fff(s) \ge 0
\,.
\end{equation}
\end{lemma}
\begin{proof}
Since according to inequality \vref{coeffEn_smaller_than_coeffEnZero} the quantity $\coeffEn$ is not larger than $\coeffEnZero$, it follows from inequality \vref{def_weight_en_V_dag} that the quantity 
\[
\coeffEn V^\dag\bigl(v(\xi,s)\bigr) + \frac{1}{2}v(\xi,s)^2 \ge \frac{1}{4}v(\xi,s)^2
\]
is nonnegative, and this proves \cref{lem:nonnegativity_F}.
\end{proof}
\begin{remark}
As for the firewall $\fff_0(t)$ in the laboratory frame (\cref{lem:fireZero_coercivity}), inequalities \cref{def_weight_en_V_dag,coeffEn_smaller_than_coeffEnZero} actually yield the stronger coercivity property:
\[
\fff(s) \ge \min\Bigl(\frac{\coeffEn}{2},\frac{1}{4}\Bigr) \int_{\rr} \psi(\xi,s)\bigl(v_\xi(\xi,s)^2 + v(\xi,s)^2\bigr)\, d\xi
\,;
\]
however, by contrast with $\fff_0(t)$, only the fact that $\fff(s)$ is nonnegative will be used in the following (see the next \namecref{subsubsec:relax_scheme_ineq_2}). 
\end{remark}
\subsubsection{Relaxation scheme inequality, 2}
\label{subsubsec:relax_scheme_ineq_2}
For all $s$ in $[0,+\infty)$, let 
\[
\mathcal{G} (s) = \int_{\SigmaEsc(s)} \psi(\xi,s) \, d\xi 
\,.
\]
According to \vref{lem:nonnegativity_F}, the quantity $\fff(\sFin)$ is nonnegative. As a consequence, integrating inequality \cref{ds_fire} between $0$ and a nonnegative quantity $\sFin$ yields:
\[
\int_0^{\sFin} \fff(s) \, ds \le \frac{1}{\nuF}\Bigl(\fff(0) + \KF \int_0^{\sFin} \mathcal{G} (s) \, ds\Bigr)
\,,
\]
and the ``relaxation scheme'' inequality \cref{diff_s_en_trav_f} becomes:
\begin{equation}
\label{relax_scheme_ggg}
\frac{1}{2} \int_0^{\sFin} \ddd(s) \, ds \le \eee(0) - \eee(\sFin) + \frac{\KEF}{\nuF}  \Bigl(\fff(0) + \KF \int_0^{\sFin} \mathcal{G} (s) \, ds\Bigr)
\,.
\end{equation}
The next step is to gain some control over the quantity $\mathcal{G}(s)$.
\subsubsection{Control over the pollution in the time derivative of the firewall function}
\label{subsubsec:flux_der_fire}
Let us assume that the parameter $\xInit$ is equal to $\xesc(\tInit)$, and, for every nonnegative quantity $s$, let
\[
\begin{aligned}
\xiHom(s) &= \xHom (\tInit + s) - \xInit - cs \,, \\
\text{and}\quad\xiesc(s) &= \xesc (\tInit + s) - \xInit - cs  
\,,
\end{aligned}
\]
see \cref{fig:def_yesc_yhom,fig:pos_trav_fr}.
\begin{figure}[!htbp]
\centering
\includegraphics[width=0.8\textwidth]{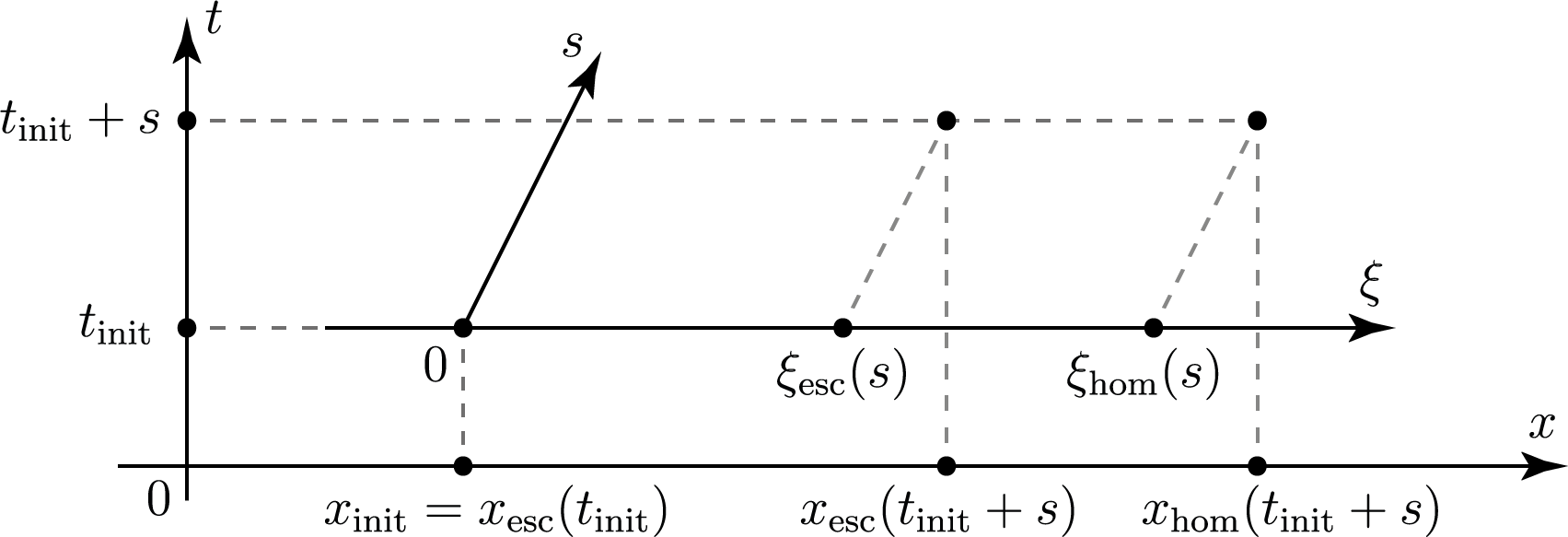}
\caption{Definitions of the ``escape'' point $\xiesc(s)$ and the point $\xiHom(s)$ marking the ``homogeneous'' area in the travelling frame.}
\label{fig:def_yesc_yhom}
\end{figure}
\begin{figure}[!htbp]
\centering
\includegraphics[width=\textwidth]{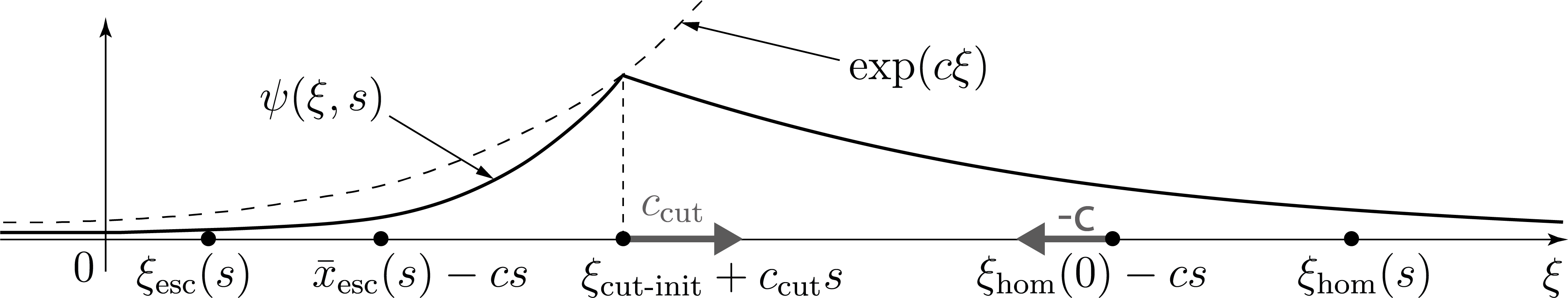}
\caption{Typical relative positions of the points $\xiesc(s)$, $\initCut + \cCut s$, $\xiHom(0)-cs$, and $\xiHom(s)$ in the travelling referential, and graph of $\xi\mapsto\psi(\xi,s)$.}
\label{fig:pos_trav_fr}
\end{figure}
According to properties \vref{xEsc_xesc_xHom} for the set $\SigmaEscZero(t)$, 
\[
\SigmaEsc(s) \subset \bigl(-\infty, \xiesc(s)\bigr] \cup \bigl[\xiHom(s),+\infty\bigr)
\,,
\]
thus, introducing the quantities 
\[
\Gback(s) = \int_{-\infty}^{\xiesc(s)} \psi (\xi,s)\, d\xi 
\quad\text{and}\quad
\Gfront(s) = \int_{\xiHom(s)}^{+\infty} \psi (\xi,s)\, d\xi 
\,,
\]
the following inequality holds:
\[
\mathcal{G}(s) \le \Gback(s)+\Gfront(s) 
\,.
\]
The aim of this \namecref{subsubsec:flux_der_fire} is to prove the bounds on $\Gback(s)$ and $\Gfront(s)$ provided by the next lemma. The following additional technical hypothesis will be required for the bound on $\Gback(s)$:
\begin{equation}
\label{hyp_c_close_to_barcescsup}
\barcescsup - \frac{\kappa \cCut}{4(\cnoesc+\kappa)} \le c
\,.
\end{equation}
This hypothesis is satisfied as soon as the speed is close enough to $\barcescsup$. It ensures that the escape point $\xiesc(s)$ remains ``more and more far away to the left'' with respect to the position $\initCut + \cCut\, s$ of the cut-off, as $s$ increases. 
\begin{lemma}[upper bounds on pollution terms in the derivative of the firewall]
\label{lem:bound_Gback_Gfront}
There exists a positive quantity $K[u_0]$, depending on $V$ and $m$ and the initial condition $u_0$ under consideration, such that for every nonnegative quantity $s$ the following estimates hold: 
\begin{align}
\label{up_bd_G_back}
\Gback(s) & \le K[u_0] \exp(-\kappa \, \initCut) \exp\Big( -\frac{\kappa\, \cCut}{2}s\Bigl)\,, \\
\label{up_bd_G_front}
\Gfront(s) & \le \frac{1}{\kappa}\exp\bigl[ (\cnoesc+\kappa)\,  \initCut + (\cnoesc + \kappa) (\cCut + \kappa)  s -\kappa \, \xiHom(0)\bigr]
\,.
\end{align}
\end{lemma}
\begin{proof}
First here are some considerations about the way the various parameters will be chosen for the relaxation scheme; according to these considerations, \cref{fig:pos_trav_fr} displays the expected positions of the various relevant points in the travelling frame. 
\begin{itemize}
\item The position $\xInit$ of the origin of the travelling frame at time $\tInit$ will always be chosen equal to the position $\xesc(\tInit)$ of the escape point at this time. 
\item The time $\tInit$ will be chosen large, so that $\xiHom(0)$ is large positive, and $\xiHom(s)\ge \xiHom(0) - cs$ remains large positive for the values of $s$ under consideration.
\item The initial position $\initCut$ of the cut-off will be chosen either equal to $0$, or large positive (thus in both cases nonnegative). 
\item The speed $c$ will be chosen close in the interval $(0,\barcescsup)$ and very close to the upper bound of this interval. As a consequence, the point $\xiesc(s)$ (the ``escape point'' viewed in the travelling frame) is expected to remain, for most of values of $s$, and for sure for $s$ large positive, at the left of the ``cut-off'' point $\initCut + \cCut s$, since this cut-off point travels to the right at a definite nonzero speed $\cCut$ in the travelling frame. 
\end{itemize}
These considerations lead us to bound from above the integrand $\psi(\xi,s)$ in the expression of $\Gback(s)$ and $\Gfront(s)$ by the following quantity:
\[
\begin{aligned}
\exp\bigl( (c+\kappa) \, \xi - \kappa(\initCut + \cCut\,  s) \bigr)
& \quad\text{for}\quad
\Gback(s) \,, \\
\text{and}\quad
\exp\bigl( (c+\kappa) (\initCut + \cCut \, s) - \kappa \, \xi\bigr)
& \quad\text{for}\quad
\Gfront(s)
\,.
\end{aligned}
\]

First let us consider the quantity $\Gback(s)$. By explicit calculation,
\[
\begin{aligned}
\Gback(s) &\le \frac{1}{c+\kappa}\exp\bigl( (c+\kappa) \xiesc(s) - \kappa \initCut - \kappa\, \cCut s\bigr) \\
&\le \frac{1}{\kappa} \exp\bigl( (c+\kappa) \, \xiesc(s) - \kappa \, \initCut - \kappa\, \cCut s\bigr)
\,.
\end{aligned}
\]
According to the definition of $\xiesc(\cdot)$ and $\barxesc(\cdot)$ and provided that $\xInit = \xesc(\tInit)$, for all $s$ in $[0,+\infty)$, 
\[
\xiesc(s) \le  \barxesc(s) - cs
\,.
\]
It follows that
\begin{align}
\nonumber
\Gback(s) &\le \frac{1}{\kappa} \exp(-\kappa \, \initCut) \exp\Bigl( (c+\kappa)\bigl(\barxesc(s) - cs\bigr) - \kappa\, \cCut s \Bigr) \\
\label{upp_bd_Gback_prel}
&\le \frac{1}{\kappa} \exp(-\kappa \, \initCut)  \exp\Bigl( (c+\kappa)\bigl(\barxesc(s) - c s\bigr)-\frac{\kappa\, \cCut}{2}s\Bigr) \exp\Big( -\frac{\kappa\, \cCut}{2}s\Bigl)
\,.
\end{align}
Let us us denote by $\beta(s)$ the argument of the second exponential of the right-hand side of this last inequality:
\[
\begin{aligned}
\beta(s) &= (c+\kappa)\bigl(\barxesc(s) - c s\bigr)-\frac{\kappa\, \cCut}{2}s \\
&= (c+\kappa)\bigl(\barxesc(s) - \barcescsup \, s \bigr)+ \Bigl( (c+\kappa)(\barcescsup -c) -\frac{\kappa\, \cCut}{2} \Bigr) s 
\,.
\end{aligned}
\]
According to hypothesis \cref{hyp_c_close_to_barcescsup} above, the following inequality holds:
\[
(c+\kappa)(\barcescsup -c) \le \frac{\kappa\, \cCut}{4}
\,,
\]
thus, for all $s$ in $[0,+\infty)$, 
\[
\beta(s) \le (c+\kappa)\bigl(\barxesc(s)- \barcescsup \, s \bigr) -\frac{\kappa\, \cCut}{4} s 
\,,
\]
and thus, according to the definition of $\barcescsup$ this quantity $\beta(s)$ goes to $-\infty$ as $s$ goes to $+\infty$. It follows from the definition of $\barxesc(\cdot)$ that $\beta(0)$ equals $0$, and, for all $s$ in $(0,+\infty)$, it follows from the last inequality that
\[
\beta(s)>0 \implies \barxesc(s)- \barcescsup \, s >0
\,.
\]
Thus, the following (nonnegative) quantity:
\[
\bar{\beta}[u_0] = \sup_{s\ge0} \ (\cnoesc+\kappa)\bigl(\barxesc(s)- \barcescsup \, s \bigr) -\frac{\kappa\, \cCut}{4} s 
\,,
\]
is an upper bound for all the values of $\beta(s)$, for all $s$ in $[0,+\infty)$. This quantity depends on $V$ and on the function $x\mapsto \barxesc(s)$, in other words on the initial condition $u_0$, but not on the parameters $\tInit$ and $c$ and $\initCut$ of the relaxation scheme. Let
\[
K[u_0] = \frac{1}{\kappa} \exp\bigl( \bar{\beta}[u_0] \bigr)
\,;
\]
with this notation, the upper bound \cref{up_bd_G_back} on $\Gback(s)$ follows from inequality \cref{upp_bd_Gback_prel}.

Now let us consider the second quantity $\Gfront(s)$. The control that will be required on this quantity is simpler, since it only relies on the value of $\xiHom(0)$, which can be assumed to be arbitrarily large positive if $\tInit$ is large enough positive. Since $\xHom'(\cdot)$ is nonnegative (see \vref{hyp_xHom_prime_pos}), for all $s$ in $[0,+\infty)$, 
\[
\xiHom'(s) \ge -c 
\quad\text{thus}\quad
\xiHom(s) \ge \xiHom(0) - cs
\,.
\]
By explicit calculation, it follows from the upper bound on the integrand of $\Gfront(s)$ that
\[
\Gfront(s) \le \frac{1}{\kappa} \exp\bigl( (c+\kappa)\,  \initCut \bigr) \exp\Bigl( \bigl((c+\kappa) \, \cCut + \kappa \,  c\bigr) s\Bigr) \exp \bigl( -\kappa \, \xiHom(0)\bigr)
\,,
\]
and inequality \cref{up_bd_G_front} on $\Gfront(s)$ follows. \Cref{lem:bound_Gback_Gfront} is proved.
\end{proof}
\subsubsection{Relaxation scheme inequality, final}
\label{subsubsec:fin_relax}
Let us introduce the quantity
\[
\KGback[u_0] = \frac{2 \, \KEF \, \KF \, K[u_0]}{ \nuF \, \kappa \cCut }
\,,
\]
and, for every nonnegative quantity $s$, the quantity
\[
\KGfront (s) = \frac{\KEF \, \KF}{\nuF \, \kappa \, (\cnoesc + \kappa) (\cCut + \kappa)}
\exp \bigl( (\cnoesc + \kappa) (\cCut + \kappa)  s \bigr)
\,.
\]
Then, according to inequalities \cref{up_bd_G_back,up_bd_G_front} of \cref{lem:bound_Gback_Gfront}, the ``relaxation scheme'' inequality \vref{relax_scheme_ggg} can be rewritten under the following (final) form:
\begin{equation}
\label{relax_scheme_final}
\begin{aligned}
\frac{1}{2} \int_0^{\sFin} \ddd(s) \, ds \le & \ \eee(0) - \eee(\sFin) + \frac{\KEF}{\nuF} \fff(0) 
+ \KGback[u_0] \exp(-\kappa \, \initCut) \\
& + \KGfront (\sFin)\exp\bigl( (\cnoesc+\kappa)\,  \initCut \bigr) \exp \bigl( -\kappa \, \xiHom(0)\bigr)
\,.
\end{aligned}
\end{equation}
In order to derive from this estimate useful information (a nice upper bound on the dissipation integral, stating that this dissipation integral is bounded or even small), the following conditions should be fulfilled:
\begin{itemize}
\item the ``initial'' value $\eee(0)$ of the localized energy should be bounded from above;
\item the ``final'' value $\eee(\sFin)$ of the localized energy should be bounded from below (or, better, close to $\eee(0)$);
\item the ``initial'' value $\fff(0)$ of the firewall function should be small (or at least\\ bounded);
\item the ``initial'' position $\initCut$ of the ``cut-off point'' should be large positive;
\item the ``initial'' position $\xiHom(0)$ of the ``homogeneous point'' should be large positive, the condition on its side depending on $\initCut$ and $\sFin$. 
\end{itemize}
\subsection{Convergence of the mean invasion speed}
\label{subsec:cv_mean_inv_vel}
\begin{remark}
For the remaining of this \namecref{sec:inv_impl_cv}, the notation $V^\dag$ and $u^\dag$ introduced in \cref{subsubsec:normalized_potential_solution} does not provide any clear benefit with respect to the initial potential $V$ and initial solution $u$. As a consequence, only the objects $V$ and $u$ will be considered. 
\end{remark}
\subsubsection{Statement and set-up of the proof}
%
The aim of this \namecref{subsec:cv_mean_inv_vel} is to prove the following proposition. 
\begin{proposition}[mean invasion speed]
\label{prop:cv_mean_inv_vel}
The following equalities hold:
\[
\cescinf = \cescsup = \barcescsup
\,.
\]
\end{proposition}
\begin{proof}[Set-up for the proof of \cref{prop:cv_mean_inv_vel}]
\renewcommand{\qedsymbol}{}
Let us proceed by contradiction and assume that
\[
\cescinf < \barcescsup
\,.
\]
Then, let us take and fix a positive quantity $c$ satisfying the following conditions:
\begin{equation}
\label{hyp_c_cv_mean_vel}
\cescinf < c < \barcescsup \le c + \frac{\kappa \cCut}{4(\cnoesc+\kappa)}
\quad\text{and}\quad
\Phi_c(m) = \emptyset
\,.
\end{equation}
The first condition is satisfied as soon as $c$ is smaller than and close enough to $\barcescsup$, thus existence of a quantity $c$ satisfying the two conditions follows from hypothesis \textup{(\hyperlink{hypDiscVel}{\hypDiscVelRef})}.

The contradiction will follow from the relaxation scheme set up in \cref{subsec:relax_sch_tr_fr}. The main ingredient is: since the set $\Phi_c(m)$ is empty, some dissipation must occur permanently around the escape point in a referential travelling at the speed $c$. This is stated by the following lemma. 
\end{proof} 
\subsubsection{Nonzero dissipation around escape point}
\begin{lemma}[nonzero dissipation around escape point due to the absence of travelling front]
\label{lem:dissip_no_tf_vel}
There exist positive quantities $L$ and $\epsDissip$ such that, for every large enough positive time $t$, 
\begin{equation}
\label{dissip_no_tf_vel}
\norm{\xi\mapsto u_t\bigl( \xesc (t) + \xi, t\bigr) + c u_x \bigl( \xesc (t) + \xi, t\bigr) }_{L^2([-L,L],\rr^d)} \ge \epsDissip
\,.
\end{equation}
\end{lemma}
\begin{remark}
This statement is actually still true without the condition ``$t$ large enough positive'', which is assumed only to fit with \cref{lem:compactness} (compactness).
\end{remark}
\begin{proof}[Proof of \cref{lem:dissip_no_tf_vel}]
Let us proceed by contradiction and assume that the converse is true. Then, there exists a sequence $(t_n)_{n\in\nn}$ of positive quantities going to $+\infty$ as $n\to+\infty$ such that, 
for every positive integer $n$, 
\begin{equation}
\label{proof_lem_dissip}
\norm{\xi\mapsto u_t\bigl( \xesc (t_n) + \xi, t_n \bigr) + c u_x \bigl( \xesc (t_n) + \xi, t_n \bigr) }_{L^2([-n,n],\rr^d)} \le \frac{1}{n}
\,.
\end{equation}
By compactness (\cref{lem:compactness}), up to replacing the sequence $(t_n)_{n\in\nn}$ by a subsequence, there exists an entire solution $\widebar{u}$ of system \cref{init_syst} such that, with the notation of \cref{compactness},
\[
D^{2,1}u\bigl( \xesc (t_n) + \cdot, t_n + \cdot \bigr)\to D^{2,1}\widebar{u}
\quad\text{as}\quad 
n\to+\infty
\,,
\]
uniformly on every compact subset of $\rr^2$. According to hypothesis \cref{proof_lem_dissip}, the function $\xi\mapsto\widebar{u}_t(\xi,0) + c \widebar{u}_x(\xi,0)$ vanishes identically, so that the function $\xi\mapsto \widebar{u}(\xi,0)$ is a solution of system \vref{syst_trav_front} governing the profiles of waves travelling at the speed $c$ for system \cref{init_syst}. According to the properties of the escape point \vref{xEsc_xesc_xHom,xHom_minus_xesc}, 
\[
\sup_{\xi\in[0,+\infty)} \abs{\widebar{u}(\xi,0)-m} \le \dEsc(m)
\,,
\]
thus it follows from assertion \cref{item:cv_spatial_asymptotics_tw} of \vref{lem:asympt_behav_tw_2} that $\widebar{u}(\xi,0)$ goes to $m$ as $\xi$ goes to $+\infty$. On the other hand, according to the bound \cref{attr_ball_infty_inv_implies_cv} on the solution, $\abs{\widebar{u}(\cdot,0)}$ is bounded (by $\Rattinfty$), and since $\Phi_c(m)$ is empty, it follows from hypothesis \textup{(\hyperlink{hypOnlyBist}{\hypOnlyBistRef})} that $\widebar{u}(\cdot,0)$ must be identically equal to $m$, a contradiction with the definition of $\xesc(\cdot)$. 
\end{proof}
\subsubsection{Time interval and origin of space for the relaxation scheme}
%
The next step is the choice of the time interval and the travelling frame (at the speed $c$) where the relaxation scheme will be applied. Those should display the following features:
\begin{enumerate}
\item the escape point should not go far to the left (or else, due to the exponential weight, no lower bound on the dissipation would hold);
\item at the end of the time interval the escape point should not be far to the right, so that the final value of the localized energy be bounded from below;
\item the length of the time interval should be large enough so that a sufficient dissipation to occur;
\item for a given length of that time interval, its lower bound should be chosen arbitrarily large positive, in order to control the ``front'' flux term involving the quantity $\KGfront$ in the relaxation scheme final inequality \vref{relax_scheme_final}. 
\end{enumerate}
\paragraph{Large excursions to the right and returns for the escape point.}
The following lemma is a first attempt to find such a time interval (see \cref{fig:first_attempt_time_int}).
\begin{figure}[!htbp]
\centering
\includegraphics[width=0.75\textwidth]{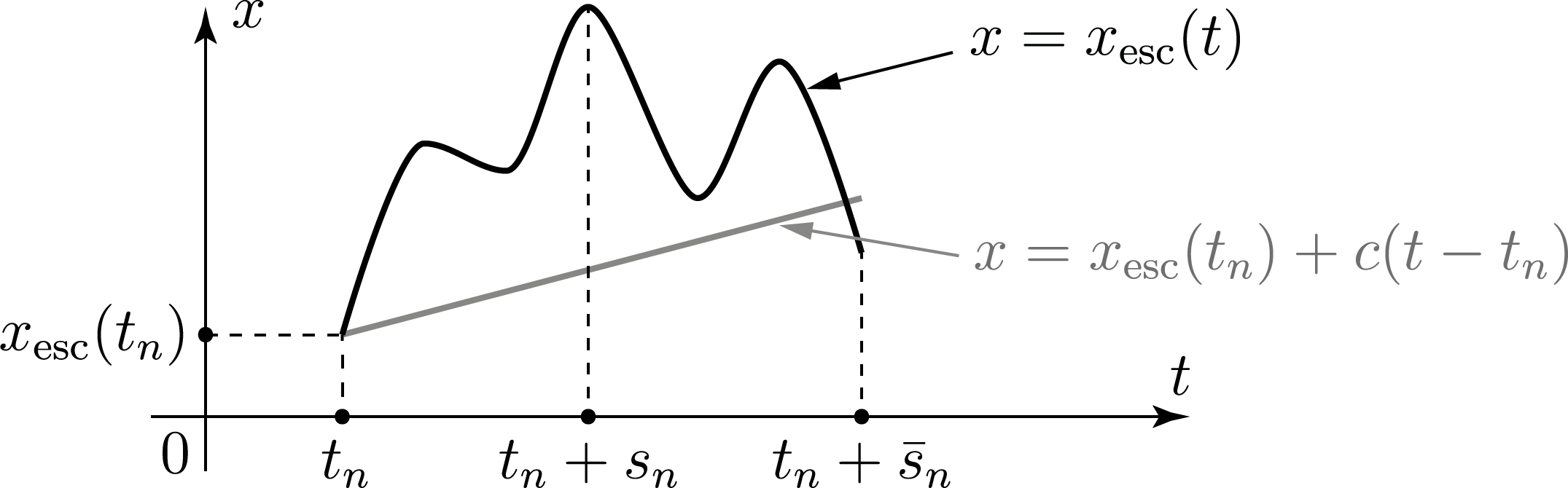}
\caption{Illustration of \cref{lem:first_attempt_time_int}.}
\label{fig:first_attempt_time_int}
\end{figure}
\begin{lemma}[large excursions to the right and returns for the escape point in travelling frame]
\label{lem:first_attempt_time_int}
There exist sequences $(t_n)_{n\in\nn}$ and $(s_n)_{n\in\nn}$ and $(\bar{s}_n)_{n\in\nn}$ of real quantities such that the following properties hold.
\begin{enumerate}
\item For every $n$ in $\nn$, the following inequalities hold: $0\le t_n$ and $0 \le s_n \le \bar{s}_n$\,;
\label{item:item_one_lem_first_attempt_time_int}
\item $s_n\to+\infty$ and $\xesc(t_n+s_n) - \xesc(t_n) - c s_n \to +\infty$ as $n$ goes to $+\infty$\,;
\label{item:item_two_lem_first_attempt_time_int}
\item For every $n$ in $\nn$, the following inequality holds: $\xesc(t_n+\bar{s}_n) - \xesc(t_n) - c \bar{s}_n \le 0$\,.
\label{item:item_three_lem_first_attempt_time_int}
\end{enumerate}
\end{lemma}
\begin{proof}[Proof of \cref{lem:first_attempt_time_int}]
According to the definition of $\barcescsup$, there exists a sequence \\
$(s_n)_{n\in\nn}$ of positive real quantities, satisfying
\[
s_n\to +\infty
\quad\text{and}\quad
\frac{\barxesc(s_n)}{s_n} \to \barcescsup
\quad\text{as}\quad
n\to +\infty
\,.
\]
Then, according to the definition \vref{def_bar_underbar_xesc} of $\barxesc(\cdot)$, for every $n$ in $\nn$ there exists a nonnegative quantity $t_n$ such that
\[
\xesc(t_n + s_n) - \xesc(t_n) \ge \barxesc(s_n) - 1
\,.
\]
Then,
\[
\frac{\xesc(t_n + s_n) - \xesc(t_n) - c s_n}{s_n} \ge \frac{\barxesc(s_n) - 1 - c s_n}{s_n} \xrightarrow[n\to +\infty]{}  \barcescsup - c > 0
\]
thus
\[
\xesc(t_n + s_n) - \xesc(t_n) - c s_n\to +\infty
\quad\text{as}\quad
n\to +\infty
\,,
\]
which proves property \cref{item:item_two_lem_first_attempt_time_int}.
On the other hand, for every $n$ in $\nn$, 
\[
\liminf_{s\to+\infty} \frac{\xesc(t_n+s) - \xesc(t_n) - c s }{s} = \cescinf -c <0 
\,,
\]
thus
\[
\liminf_{s\to+\infty} \xesc(t_n+s) - \xesc(t_n) - c s = -\infty
\,,
\]
and thus there exists $\bar{s}_n$ greater than $s_n$ such that 
\[
\xesc(t_n+\bar{s}_n) - \xesc(t_n) - c \bar{s}_n \le 0
\,,
\]
which proves property \cref{item:item_three_lem_first_attempt_time_int}. According to how $t_n$ and $s_n$ and $\bar{s}_n$ where chosen, property \cref{item:item_one_lem_first_attempt_time_int} also holds. \Cref{lem:first_attempt_time_int} is proved.
\end{proof}
\paragraph{Time intervals of controlled length.}
Intervals $[t_n,t_n+\bar{s}_n]$ defined by the previous lemma are not convenient to apply the relaxation scheme because their length is arbitrary large (thus the ``front'' flux term involving the quantity $\KGfront$ in the relaxation scheme final inequality \vref{relax_scheme_final} cannot be controlled on such time intervals). The following lemma provides another sequence of intervals (derived from the sequences defined in \cref{lem:first_attempt_time_int}) without this drawback (see \cref{fig:second_attempt_time_int}). 
\begin{figure}[!htbp]
\centering
\includegraphics[width=0.9\textwidth]{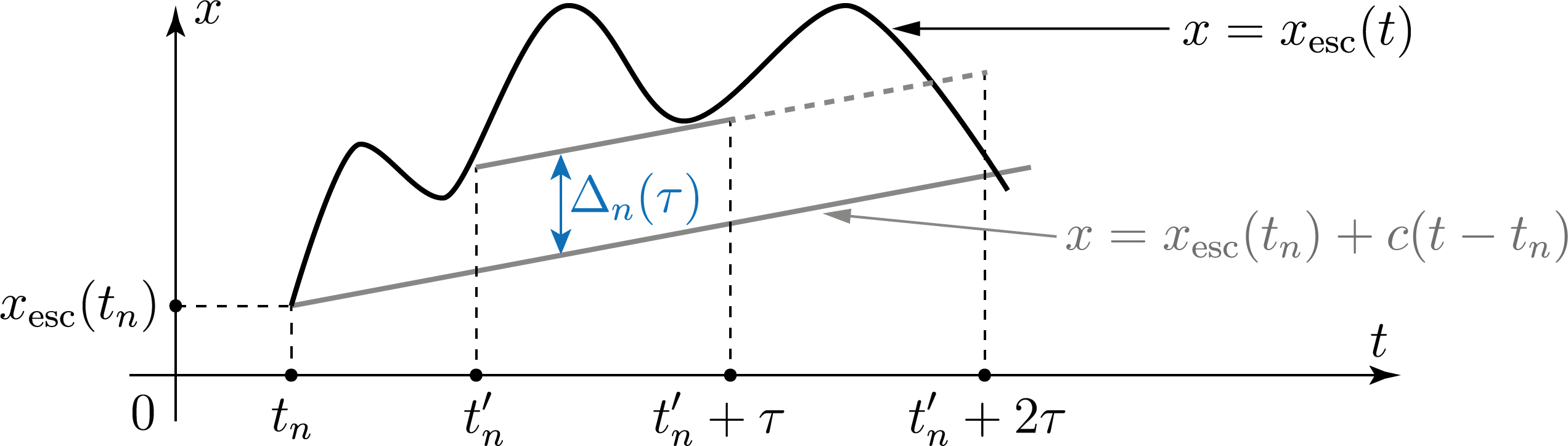}
\caption{Illustration of \cref{lem:sec_attempt_time_int}.}
\label{fig:second_attempt_time_int}
\end{figure}

Let $\tau$ denote a (large) positive quantity to be chosen below. This quantity will determine the length of the time intervals where the relaxation scheme will be applied (more precisely that length will be between $\tau$ and $2\tau$). The value of $\tau$ will be chosen large enough so that a sufficient amount of dissipation occurs during the relaxation scheme. 
\begin{lemma}[escape point remains to the right and ends up to the left in travelling frame, controlled duration]
\label{lem:sec_attempt_time_int}
There exist sequences $(t'_n)_{n\in\nn}$ and $(s'_n)_{n\in\nn}$ such that, for every $n$ in $\nn$ the following properties hold:
\begin{enumerate}
\item $0\le t'_n$ and $\tau \le s'_n \le 2\tau$\,,
\label{item:item_one_lem_sec_attempt_time_int}
\item for all $s$ in $[0,\tau]$, the following inequality holds: $\xesc(t'_n + s) - \xesc(t'_n) - cs \ge 0$\,,
\label{item:item_two_lem_sec_attempt_time_int}
\item $\xesc(t'_n + s'_n) - \xesc(t'_n) - c s'_n \le 1$\,,
\label{item:item_three_lem_sec_attempt_time_int}
\end{enumerate}
and such that
\[
t'_n\to+\infty
\quad\text{as}\quad
n\to +\infty
\,.
\]
\end{lemma}
\begin{proof}[Proof of \cref{lem:sec_attempt_time_int}]
For every $n$ in $\nn$ let us introduce the set:
\[
\begin{aligned}
\bigl\{ \Delta \in [0,+\infty) : \ & \text{there exists } t \text{ in } [t_n,t_n + \bar{s}_n] \text{ such that, for all } s \text{ in } [0,\tau], \\
& \xesc(t+s) - \xesc(t_n) - c (t+s-t_n) \ge \Delta \bigr\}
\,,
\end{aligned}
\]
and let us denote by $\Delta_n(\tau)$ the supremum of this set (with the convention that $\Delta_n(\tau)$ equals $-\infty$ if this set is empty). 

First let us prove that $\Delta_n(\tau)$ goes to $+\infty$ as $n$ goes to $+\infty$. For this purpose, observe that, according to the control on the growth of $\xesc(\cdot)$ \vref{control_escape}, for every nonnegative integer $n$ and for all $s$ in $[0,\tau]$, 
\[
\xesc(t_n + s_n -\tau + s ) \ge \xesc ( t_n + s_n) - \cnoesc (\tau - s)
\,,
\]
and as a consequence,
\[
\begin{aligned}
\xesc(t_n + s_n -\tau + s ) - \xesc (t_n) & - c (s_n -  \tau + s ) \\
&\ge \xesc(t_n + s_n) - \xesc (t_n)  - c s_n - ( \cnoesc -c) (\tau - s) \\
&\ge \xesc(t_n + s_n) - \xesc (t_n)  - c s_n - ( \cnoesc -c) \tau
\,,
\end{aligned}
\]
and according to \cref{lem:first_attempt_time_int} this last quantity goes to $+\infty$ as $n$ goes to $+\infty$. This shows that $\Delta_n(\tau)$ goes to $+\infty$ as $n$ goes to $+\infty$. Up to replacing the sequence $(t_n,s_n,\bar{s}_n)_{n\in\nn}$ by a subsequence, let us assume that, for every nonnegative integer $n$ the quantity $\Delta_n(\tau)$ is not smaller than $1$.

Then, for every nonnegative integer $n$, according to the definition of $\Delta_n(\tau)$, the set
\[
\Bigl\{ t\in[t_n,t_n+\bar{s}_n] : \text{ for all } s \text{ in } [0,\tau], 
\xesc(t+s)-\xesc(t_n) - c (t+s-t_n) \ge \Delta_n(\tau) -\frac{1}{2} \Bigr\}
\]
is nonempty. Let $t'_n$ denote the infimum of this set. Then $t'_n$ is greater than $t_n$, and, according to the control on the growth of $\xesc(\cdot)$ \vref{control_escape}, 
\begin{equation}
\label{t_prime_starts_at_center}
\xesc(t'_n) - \xesc(t_n) - c (t'_n-t_n) = \Delta_n(\tau) - \frac{1}{2}
\,,
\end{equation}
and, for all $s$ in $[0,\tau]$, 
\begin{equation}
\label{t_prime_n_remains_to_the_right}
\xesc(t'_n+s) - \xesc(t_n) - c (t'_n+s-t_n) \ge \Delta_n(\tau) - \frac{1}{2}
\,.
\end{equation}
Since $t'_n$ is less than or equal to $t_n+\bar{s}_n$ and according to the last assertion of \cref{lem:first_attempt_time_int}, and since $\Delta_n(\tau) -1/2$ is positive, this shows that
\[
t'_n+ \tau \le t_n + \bar{s}_n
\,.
\]
As a consequence, according to the definition of $\Delta_n(\tau)$, there exists $s'_n$ in $[\tau,2\tau]$ such that
\begin{equation}
\label{t_prime_n_finishes_to_the_left}
\xesc(t'_n + s'_n) - \xesc(t_n) - c (t'_n + s'_n - t_n) \le \Delta_n(\tau) + \frac{1}{2}
\,.
\end{equation}
Finally, it follows from equality \cref{t_prime_starts_at_center} and inequality \cref{t_prime_n_remains_to_the_right} that, for all all $s$ in $[0,\tau]$, 
\[
\xesc(t'_n+s) - \xesc(t'_n) - cs \ge 0
\,,
\]
which proves property \cref{item:item_two_lem_sec_attempt_time_int}, and it follows from inequalities \cref{t_prime_n_remains_to_the_right,t_prime_n_finishes_to_the_left} that
\[
\xesc(t'_n + s'_n) - \xesc(t'_n) - c s'_n \le 1
\,,
\]
which proves property \cref{item:item_three_lem_sec_attempt_time_int}. According to how $t'_n$ and $s'_n$ where chosen, property \cref{item:item_one_lem_sec_attempt_time_int} also holds. 
In addition, according to equality \cref{t_prime_starts_at_center} and to the control on the growth of $\xesc(\cdot)$ \vref{control_escape}, since $\Delta_n(\tau)$ goes to $+\infty$ as $n$ goes to $+\infty$, 
\[
t'_n-t_n\to+\infty
\quad\text{as}\quad
n\to +\infty
\quad\text{thus}\quad
t'_n\to+\infty
\quad\text{as}\quad
n\to +\infty
\,.
\]
\Cref{lem:sec_attempt_time_int} is proved. 
\end{proof} 
Since $t'_n\to+\infty$ as $t\to+\infty$, we may assume, up to dropping the first terms of the sequence $(t'_n,s'_n)_{n\in\nn}$, that, for every $n$ in $\nn$ and for every time $t$ greater than or equal to $t'_n$, inequality \cref{dissip_no_tf_vel} of \cref{lem:dissip_no_tf_vel} holds.  
\subsubsection{Relaxation scheme}
%
\Cref{lem:sec_attempt_time_int} provides intervals $[t'_n,t'_n+s'_n]$ suitable to apply the relaxation scheme set up in \cref{subsec:relax_sch_tr_fr}. For every nonnegative integer $n$ this relaxation scheme is going to be applied for the following parameters: 
\[
\tInit = t'_n
\quad\text{and}\quad
\xInit = \xesc(\tInit)
\quad\text{and}\quad
c \quad\text{as chosen above, and}\quad
\initCut = 0 
\]
(the relaxation scheme thus depends on $n$). Observe that, according to hypothesis \vref{hyp_c_cv_mean_vel} for the speed $c$, both hypotheses \vref{hyp_param_relax_sch} and \vref{hyp_c_close_to_barcescsup} (required to apply this relaxation scheme) hold. 
Let us denote by 
\[
v^{(n)}(\cdot,\cdot)
\quad\text{and}\quad
\eee^{(n)}(\cdot)
\quad\text{and}\quad
\ddd^{(n)}(\cdot)
\quad\text{and}\quad
\fff^{(n)}(\cdot)
\quad\text{and}\quad
{\xiesc}^{(n)}(\cdot)
\quad\text{and}\quad
{\xiHom}^{(n)}(\cdot)
\]
the objects defined in \cref{subsec:relax_sch_tr_fr} (with the same notation except the ``$(n)$'' superscript that is here to remind that all these objects depend on the integer $n$). 
By definition the quantity ${\xiesc}^{(n)}(0)$ equals $0$, and according to the conclusions of \cref{lem:sec_attempt_time_int}, 
\[
{\xiesc}^{(n)}(s)\ge 0 \text{ for all }s\text{ in }[0,\tau]
\quad\text{and}\quad
{\xiesc}^{(n)}(s'_n) \le 1
\,.
\]
The next two lemmas will be shown to be in contradiction with the relaxation scheme final inequality \vref{relax_scheme_final} (see inequality \cref{relax_scheme_final_control_invasion} below), and this will complete the proof of \cref{prop:cv_mean_inv_vel}.
\paragraph{Bounds on energy and firewall at the ends of the relaxation scheme.}
\begin{lemma}[bounds on energy and firewall at the ends of the relaxation scheme]
\label{lem:claim_energy}
The quantities $\eee^{(n)}(0)$ and $\fff^{(n)}(0)$ are bounded from above and the quantity $\eee^{(n)}(s'_n)$ is bounded from below, and these bounds are uniform with respect to $\tau$ and $n$. 
\end{lemma}
\begin{proof}
The fact that $\eee^{(n)}(0)$ and $\fff^{(n)}(0)$ are bounded from above follows from the fact that $\initCut$ equals $0$ and from the bounds \vref{attr_ball_infty_inv_implies_cv,attr_ball_Hone_inv_implies_cv} for the solution. The fact that $\eee^{(n)}(s'_n)$ is bounded from below follows from the fact that ${\xiesc}^{(n)}(s'_n) \le 1$ and from the fact that, according to hypothesis \cref{hyp_coerc}, $V$ is bounded from below.
\end{proof}
\paragraph{Large dissipation integral.}
\begin{lemma}[large dissipation integral]
\label{lem:claim_dissip}
The quantity 
\[
\int_0^{s'_n} \ddd^{(n)}(s) \, ds
\]
goes to $+\infty$ as $\tau$ goes to $+\infty$, uniformly with respect to $n$. 
\end{lemma}

\begin{proof}
By definition of $\ddd^{(n)}(\cdot)$, for all $s$ in $[0,s'_n]$, 
\[
\ddd^{(n)}(s) \ge \int_{-\infty}^{\cCut s} e^{c\xi} \ v_s^{(n)}(\xi,s) ^2 \, d\xi
\,. 
\]
thus, for all $s$ in $[0,\tau]$, since ${\xiesc}^{(n)}(s)\ge 0$ (performing the change of variables $\xi = {\xiesc}^{(n)}(s) + \tilde{\xi}$), 
\[
\ddd^{(n)}(s) \ge \int_{-\infty}^{ \cCut s - {\xiesc}^{(n)}(s) } e^{c\tilde{\xi}} \ v_s^{(n)} \bigl( {\xiesc}^{(n)}(s) + \tilde{\xi} , s \bigr) ^2 \, d\tilde{\xi}
\,.
\]
For all $s$ in $[0,+\infty)$ and $\tilde{\xi}$ in $\rr$, 
\[
v_s^{(n)} \bigl( {\xiesc}^{(n)}(s) + \tilde{\xi} , s \bigr) = u_t \bigl(\xesc(t'_n +s) + \tilde{\xi}, t'_n + s \bigr) + c \, u_x \bigl(\xesc(t'_n +s) + \tilde{\xi}, t'_n + s \bigr)
\,,
\]
thus, according to inequality \cref{dissip_no_tf_vel} of \vref{lem:dissip_no_tf_vel}, there exist positive quantities $L$ and $\epsDissip$ such that, uniformly with respect to $n$ in $\nn$ and $s$ in $[0,+\infty)$, 
\begin{equation}
\label{non_zero_dissip_tr_fr}
\int_{-L}^L v_s^{(n)} \bigl( {\xiesc}^{(n)}(s) + \tilde{\xi} , s \bigr) ^2 \, d\tilde{\xi} \ge \epsDissip
\,.
\end{equation}
Observe that the condition \vref{hyp_c_cv_mean_vel} satisfied by $c$ implies that $\barcescsup - c$ is less than or equal to $\cCut/4$. Thus, since 
\[
{\xiesc}^{(n)}(s) \le \barxesc(s) -cs \quad\text{for all } s \text{ in } [0,+\infty)
\,,
\]
there exists a positive quantity $s_0$, depending only on $L$ and on the function $\barxesc(\cdot)$ such that, for every $s$ in $[s_0,+\infty)$, 
\[
\cCut s - {\xiesc}^{(n)}(s) \ge L
\,.
\]
It follows from inequality \cref{non_zero_dissip_tr_fr} that, for all $s$ in $[s_0,\tau]$, 
\[
\ddd^{(n)}(s) \ge e^{-cL} \epsDissip
\,,
\]
and finally,
\[
\int_0^{s'_n} \ddd^{(n)}(s) \, ds \ge (\tau-s_0) e^{-cL} \epsDissip
\,,
\]
and this finishes the proof of \cref{lem:claim_dissip}.
\end{proof}
\subsubsection{Contradiction and end of the proof of \texorpdfstring{\cref{prop:cv_mean_inv_vel}}{Proposition 4}}
%
\begin{proof}[End of the proof of \cref{prop:cv_mean_inv_vel}]
For every nonnegative integer $n$, the relaxation scheme final inequality \vref{relax_scheme_final} yields (since $\initCut$ equals $0$):
\begin{equation}
\label{relax_scheme_final_control_invasion}
\begin{aligned}
\frac{1}{2} \int_0^{s'_n} \ddd^{(n)}(s) \, ds \le & \ \eee^{(n)}(0) - \eee^{(n)}(s'_n) + \frac{\KEF}{\nuF} \fff^{(n)}(0) 
+ \KGback[u_0]  \\
& + \KGfront (2\tau) \exp \bigl( -\kappa \, {\xiHom}^{(n)}(0)\bigr)
\,.
\end{aligned}
\end{equation}
Since $t'_n$ goes to $+\infty$ as $n$ goes to $+\infty$, the quantity ${\xiHom}^{(n)}(0)$ also goes to $+\infty$ as $n$ goes to $+\infty$. Thus, according to \cref{lem:claim_energy}, the right-hand side of inequality \cref{relax_scheme_final_control_invasion} is bounded, uniformly with respect to $\tau$, provided that $n$ (depending on $\tau$) is large enough. This is contradictory to \cref{lem:claim_dissip}, and completes the proof of \vref{prop:cv_mean_inv_vel}. 
\end{proof}
\subsubsection{Definition of escape speed \texorpdfstring{$\cesc$}{cesc}}
%
According to \cref{prop:cv_mean_inv_vel}, the three quantities $\cescinf$ and $\cescsup$ and $\barcescsup$ are equal; let
\[
\cesc
\]
denote their common value.
\subsection{Further control on the escape point}
\label{subsec:further_control}
At this stage two cases can be distinguished. 
\begin{enumerate}
\item Either $\cesc$ is less than $\cHom$, and in this case the end of the proof of \cref{prop:inv_cv} ``invasion implies convergence'' is very similar to the end of the proof of the main result of \cite{Risler_globCVTravFronts_2008}. The sole difference is the presence of the additional ``front'' flux term coming from $\Gfront(\cdot)$. But if $\cesc < \cHom$ this term can be easily handled provided that the parameter $c$ involved in the relaxation scheme and the quantity $\cCut$ are chosen in such a way that
\[
c+\cCut < \cHom 
\]
(this is possible since $\cesc$ is less than $\cHom$). Then the flux terms due to $\Gfront(\cdot)$ decrease at an exponential rate, and can be made arbitrarily small provided that $\tInit$ is large enough. The details are skipped since this approach is anyway not sufficient in the case where $\cesc$ equals $\cHom$.
\item The other case $\cesc$ equals $\cHom$ is slightly more delicate, and taking this case into account will require the more precise control on the invasion point provided by \cref{prop:further_control} below. 
\end{enumerate}
The remaining of the proof of \cref{prop:inv_cv} covers both cases above, but is specifically designed to take into account the (more difficult) second case (again, in the first case $\cesc$ less than $\cHom$, adapting the proof of \cite{Risler_globCVTravFronts_2008} as sketched above would lead to a simpler proof). 

The aim of this \namecref{subsec:further_control} is to prove the following proposition that enforces the control on the behaviour of the ``escape'' point $\xesc(\cdot)$ (this will be used for the additional relaxation arguments carried on in the next \namecrefs{subsec:further_control}).
\begin{proposition}[mean invasion speed, further control]
\label{prop:further_control}
The following equality holds:
\[
\underbarcescinf = \cesc
\,.
\]
\end{proposition}
\begin{proof}
\renewcommand{\qedsymbol}{}
The proof is rather similar to that of \cref{prop:cv_mean_inv_vel}. Recall that the quantity $\underbarcescinf$ is less than or equal to $\cesc$. Let us proceed by contradiction and assume that
\[
\underbarcescinf < \cesc
\,.
\]
Then, let us take and fix a positive quantity $c$ satisfying the following conditions:
\[
\underbarcescinf < c < \cesc < c + \cCut
\quad\text{and}\quad
c \ge \cesc - \frac{\kappa \cCut}{4(\cnoesc+\kappa)}
\quad\text{and}\quad
\Phi_c(m) = \emptyset
\,.
\]
The two first conditions are satisfied as soon as $c$ is smaller than and close enough to $\cesc$, thus existence of a quantity $c$ satisfying the three conditions follows from hypothesis \textup{(\hyperlink{hypDiscVel}{\hypDiscVelRef})}. The following lemma is identical to \vref{lem:first_attempt_time_int} (but the proof will be slightly different, see \cref{fig:fist_attempt_time_int_bis}).
 \begin{figure}[!htbp]
\centering
\includegraphics[width=0.9\textwidth]{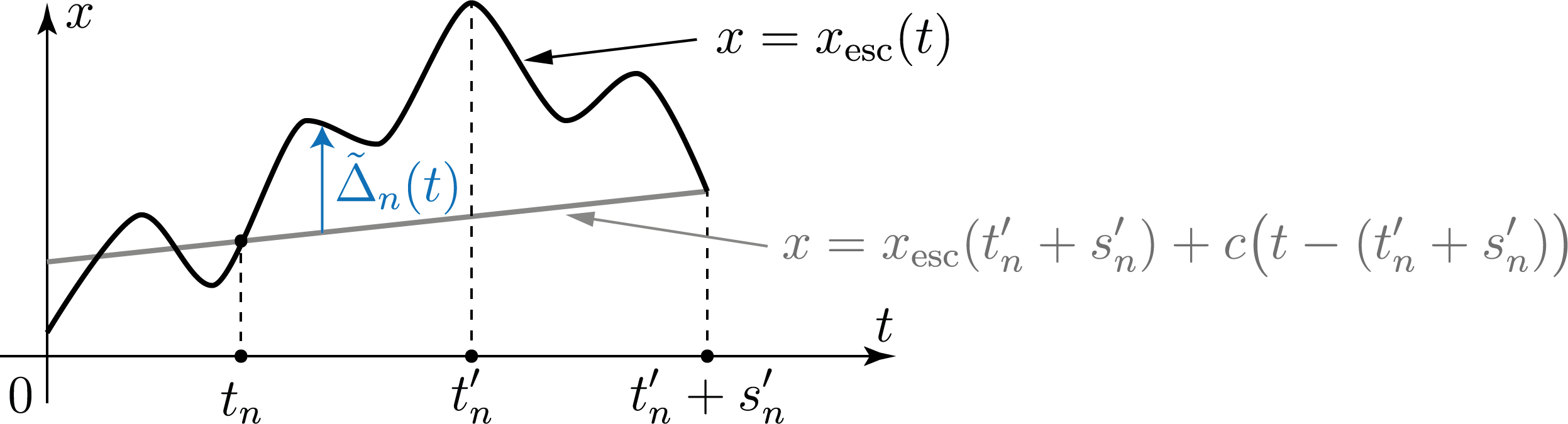}
\caption{Illustration of the proof of \cref{lem:first_attempt_time_int_bis}.}
\label{fig:fist_attempt_time_int_bis}
\end{figure}
\begin{lemma}[large excursions to the right and returns for escape point in travelling frame]
\label{lem:first_attempt_time_int_bis}
There exist sequences $(t_n)_{n\in\nn}$ and $(s_n)_{n\in\nn}$ and $(\bar{s}_n)_{n\in\nn}$ of real quantities such that the following properties hold.
\begin{enumerate}
\item For every nonnegative integer $n$ the following inequalities hold: $0\le t_n$ and $0 \le s_n \le \bar{s}_n$\,;
\label{item:item_one_lem_first_attempt_time_int_bis}
\item $\xesc(t_n+s_n) - \xesc(t_n) - c s_n \to +\infty$ as $n$ goes to $+\infty$\,;
\label{item:item_two_lem_first_attempt_time_int_bis}
\item For every nonnegative integer $n$ the following inequality holds: $\xesc(t_n+\bar{s}_n) - \xesc(t_n) - c \bar{s}_n \le 0$\,.
\label{item:item_three_lem_first_attempt_time_int_bis}
\end{enumerate}
\end{lemma}
\end{proof} 
\begin{proof}[Proof of \cref{lem:first_attempt_time_int_bis}]
According to the definition of $\underbarcescinf$, there exists a sequence $(s'_n)_{n\in\nn}$ of positive real quantities satisfying
\[
s'_n\to+\infty
\quad\text{and}\quad
\frac{\underbarxesc(s'_n)}{s'_n}\to \underbarcescinf
\quad\text{as}\quad
n\to+\infty
\,.
\]
In addition, up to replacing the sequence $(s'_n)_{n\in\nn}$ by a subsequence, it may be assumed that
\[
\underbarxesc(s'_n) > -\infty
\quad\text{for all}\quad p\text{ in }\nn
\,.
\]
Then, by definition of $\underbarxesc(\cdot)$, for every nonnegative integer $n$ there exists a nonnegative real quantity $t'_n$ such that
\begin{equation}
\label{underbar_x_esc_of_s_prime_n_larger_than}
\xesc(t'_n + s'_n) \le \xesc(t'_n) + \underbarxesc(s'_n) - 1
\,.
\end{equation}
Observe that $t'_n$ must go to $+\infty$ as $n$ goes to $+\infty$, or else this last inequality would yield
\[
\liminf_{n\to+\infty} \frac{\xesc(t'_n + s'_n) }{t'_n + s'_n} \le \underbarcescinf < \cesc
\]
and this would be contradictory to the definition of $\cescinf$ and \cref{prop:cv_mean_inv_vel}. For every nonnegative integer $n$ and every time $t$ in $[0,+\infty)$, let
\[
\tilde{\Delta}_n(t) = \xesc(t) - \Bigl( \xesc(t'_n + s'_n) + c\bigl( t - (t'_n+s'_n)\bigr)\Bigr)
\,.
\]
Then $\tilde{\Delta}_n(t'_n+s'_n)=0$ and it follows from \cref{underbar_x_esc_of_s_prime_n_larger_than} that
\[
\frac{\tilde{\Delta}_n(t'_n)}{s'_n} = \frac{\xesc(t'_n)-\xesc(t'_n+s'_n)}{s'_n} + c \ge \frac{1-\underbarxesc(s'_n)}{s'_n} + c
\,,
\]
thus
\[
\liminf_{n\to+\infty} \frac{\tilde{\Delta}_n(t'_n)}{s'_n} \ge c-\underbarcescinf > 0
\,,
\]
and finally 
\begin{equation}
\label{limit_of_tilde_Delta_n_of_t_prime_n}
\tilde{\Delta}_n(t'_n) \to +\infty 
\quad\text{as}\quad
n\to +\infty \,.
\end{equation}
On the other hand, 
\[
\frac{\tilde{\Delta}_n(0)}{t'_n+s'_n} = \frac{\xesc(0)-\xesc(t'_n+s'_n)}{t'_n+s'_n} + c 
\]
and this quantity goes to the negative quantity $c-\cesc$ as $n$ goes to $+\infty$. As a consequence,
\[
\tilde{\Delta}_n(0)\to -\infty 
\quad\text{as}\quad
n\to +\infty \,.
\]
Thus, up to replacing the sequence $(t'_n,s'_n)_{n\in\nn}$ by a subsequence, it may be assumed that, for every nonnegative integer $n$, $\tilde{\Delta}_n(0)$ is negative. Then, for every nonnegative integer $n$, let  
\[
t_n = \sup \, \bigl\{t\in[0,t'_n] : \tilde{\Delta}_n(t)\le 0\bigr\}
\quad\text{and}\quad
s_n = t'_n - t_n 
\quad\text{and}\quad
\bar{s}_n = t'_n + s'_n - t_n
\,.
\]
Property \cref{item:item_one_lem_first_attempt_time_int_bis} follows from these definitions, and, according to the control on the growth of $\xesc(\cdot)$ \vref{control_escape}, the quantity $\tilde{\Delta}_n(t_n)$ must be equal to $0$; or in other words,
\begin{equation}
\label{consequence_of_tilde_Delta_n_of_tn_vanishes}
\xesc(t_n+\bar{s}_n) = \xesc(t_n) + c \bar{s}_n
\,.
\end{equation}
Property \cref{item:item_three_lem_first_attempt_time_int_bis} follows from equality \cref{consequence_of_tilde_Delta_n_of_tn_vanishes}. It also follows from equality \cref{consequence_of_tilde_Delta_n_of_tn_vanishes} that
\[
\tilde{\Delta}_n (t'_n) = \xesc(t_n+s_n) - \xesc(t_n) - c s_n
\,,
\]
so that property \cref{item:item_two_lem_first_attempt_time_int_bis} follows from the limit \cref{limit_of_tilde_Delta_n_of_t_prime_n}. \Cref{lem:first_attempt_time_int_bis} is proved.
\end{proof}
\begin{proof}[End of the proof of \cref{prop:further_control}]
Since the conclusions of \cref{lem:first_attempt_time_int_bis} are identical to those of \vref{lem:first_attempt_time_int}, the end of the proof of \cref{prop:further_control} (``mean invasion speed, further control'') can be achieved exactly as for \cref{prop:cv_mean_inv_vel} (``mean invasion speed''). 
\end{proof}
\subsection{Dissipation approaches zero at regularly spaced times}
\label{subsec:dissip_zero_some_times}
The key argument behind \vref{prop:cv_mean_inv_vel,prop:further_control} (``mean invasion speed'' and ``… further control'') is that a dissipation uniformly bounded from below around the escape point cannot occur during large time intervals, since it is forbidden by the relaxation scheme set up in \cref{subsec:relax_sch_tr_fr} (and in particular by the relaxation scheme final inequality \vref{relax_scheme_final}). The aim of this \namecref{subsec:dissip_zero_some_times} is to state another result (\cref{prop:dissp_zero_some_times} below) that just formalizes this argument, this time considering the dissipation in a frame travelling precisely at the ``sole relevant'' speed $\cesc$ given by \cref{prop:cv_mean_inv_vel,prop:further_control}.

For all $t$ in $[0,+\infty)$, the following set:
\[
\left\{\varepsilon \text{ in } (0,+\infty) : \int_{-1/\varepsilon}^{1/\varepsilon} \Bigl( u_t\bigl(\xesc(t)+\xi,t\bigr)+\cesc u_x\bigl(\xesc(t)+\xi,t\bigr) \Bigr)^2\, d\xi \le \varepsilon \right\}
\]
is (according to the bounds \vref{bound_u_ut_ck_bis} for the solution) a nonempty (and unbounded from above) interval. Let 
\begin{equation}
\label{def_deltaDissip}
\deltaDissip(t)
\end{equation}
denote the infimum of this interval. This quantity measures to what extent the solution is, at time $t$ and around the escape point $\xesc(t)$, close to be stationary in a frame travelling at the speed $\cesc$. The aim of the next \namecref{subsec:dissip_zero_some_times} will be to prove that
\[
\deltaDissip(t) \to 0
\quad\text{as}\quad
t\to +\infty
\,.
\]
\cref{prop:dissp_zero_some_times} below can be viewed as a first step towards this goal. 
\begin{proposition}[regular occurrence of small dissipation]
\label{prop:dissp_zero_some_times}
For every positive quantity $\varepsilon$, there exists a positive quantity $T(\varepsilon)$ such that, for every $t$ in $[0,+\infty)$, 
\[
\inf_{t'\in[t,t+T(\varepsilon)]} \deltaDissip(t') \le \varepsilon
\,.
\]
\end{proposition}
\begin{proof}
\renewcommand{\qedsymbol}{}
Let us proceed by contradiction and assume that the converse holds. Then there exist a positive quantity $\varepsilon_0$ and a sequence $(t'_n)_{n\in\nn}$ of nonnegative quantities such that, for every $t$ in $[t'_n,t'_n+n]$,
\begin{equation}
\label{lower_bound_dissip}
\int_{-1/\varepsilon_0}^{1/\varepsilon_0} \Bigl( u_t\bigl(\xesc(t)+\xi,t\bigr)+\cesc u_x\bigl(\xesc(t)+\xi,t\bigr) \Bigr)^2\, d\xi\ge \varepsilon_0
\,.
\end{equation}
Up to replacing $t'_n$ by $t'_{2n}+n$, it may be assumed that $t'_n$ goes to $+\infty$ as $n$ goes to $+\infty$. 

As for the proof of \cref{prop:cv_mean_inv_vel,prop:further_control}, the strategy is to apply the relaxation scheme in travelling frame where the escape point $\xesc(\cdot)$:
\begin{itemize}
\item remains around the origin or to the right during a significant time interval (to recover enough dissipation);
\item finishes around the origin or to the left (so that the final energy be bounded from below).
\end{itemize}
In order these two conditions to hold simultaneously, the relaxation scheme will not be applied on the whole intervals $[t'_n,t'_n+n]$, but rather on smaller convenient intervals that will be introduced. Because of the second of these conditions, the speed of the travelling frame will be chosen slightly greater than $\cesc$, more precisely equal to
\[
\cesc + \frac{1}{\tau}
\,,
\]
where $\tau$ is a positive quantity to be chosen later (this quantity will play the same kind of role as in the proof of \cref{prop:cv_mean_inv_vel}). According to \cref{prop:cv_mean_inv_vel}, there exists a positive quantity $\bar{\tau}$ larger than $\tau$ (thus depending on $\tau$) such that
\[
\frac{\barxesc(\bar{\tau})}{\bar{\tau}} \le \cesc + \frac{1}{\tau}
\]
(this will ensure that the second of the conditions above is satisfied on every interval of length $\bar{\tau}$). 
Besides, according to hypothesis \cref{lower_bound_dissip} and to the bounds \cref{bound_u_ut_ck_bis} on the solution, it may be assumed that $\tau$ is large enough so that, for every $t$ in $[t'_n,t'_n+n]$, 
\[
\int_{-1/\varepsilon_0}^{1/\varepsilon_0} \biggl( u_t\bigl(\xesc(t)+\xi,t\bigr)+\Bigl(\cesc + \frac{1}{\tau}\Bigr) u_x\bigl(\xesc(t)+\xi,t\bigr) \biggr)^2\, d\xi \ge \frac{\varepsilon_0}{2}
\,.
\]
The following lemma provides the initial times of the time intervals where the relaxation scheme will be applied, ensuring that the first of the two conditions above is fulfilled despite the fact that the travelling speed is slightly above $\cesc$.
\end{proof} 
\begin{lemma}[escape point remains to the right in travelling frame]
\label{lem:time_int_ter}
For every large enough positive integer $n$, there exists a time $t_n$ in the interval $[t'_n,t'_n + n - \tau]$ such that, for every $s$ in $[0,\tau]$, 
\[
\xesc(t_n + s) - \xesc(t_n) - \cesc  s \ge -1
\,.
\]
\end{lemma}
\begin{proof}[Proof of \cref{lem:time_int_ter}]
Let us proceed by contradiction and assume that the converse holds. Then, there exists an arbitrarily large positive integer $n$ such that, for every $t$ in the interval $[t'_n,t'_n + n - \tau]$, there exists $s$ in $(0,\tau]$ such that 
\[
\xesc(t + s) - \xesc(t) - \cesc  s < -1
\,,
\]
ensuring that the mean speed of $\xesc(\cdot)$ on the interval $[t,t+s]$ is less than $\cesc - 1/\tau$. This shows that there exists $t'$ in the interval $[t'_n + n - \tau, t'_n +n ]$ such that the interval $[t'_n,t']$ can be cut into a finite number of subintervals (defined one after another, starting from $t'_n$) so that, on each of these subintervals, the mean speed of $\xesc(\cdot)$ is less than $\cesc - 1/\tau$. As a consequence, the mean speed of $\xesc(\cdot)$ on the whole interval $[t'_n,t']$ is less than $\cesc - 1/\tau$. 

But on the other hand, according to \cref{prop:further_control}, for every large enough positive quantity $s$, the following inequality holds:
\[
\underbarxesc(s) \ge \Bigl(\cesc-\frac{1}{\tau}\Bigr) s
\,,
\]
in other words the mean speed of $\xesc(\cdot)$ cannot not be less than $\cesc - 1/\tau$ on a large enough time interval, a contradiction with the previous assertion if $n$ is large enough. \Cref{lem:time_int_ter} is proved. 
\end{proof}
\begin{proof}[End of the proof of \cref{prop:dissp_zero_some_times}]
Thus, according to this lemma, for every nonnegative integer $n$, large enough so that $t_n$ is defined, the following assertion holds: for every $s$ in $[0,\tau]$,
\[
\xesc(t_n+s) - \xesc(t_n) - \Big(\cesc + \frac{1}{\tau}\Bigr) s \ge -2
\,,
\]
and
\[
\xesc(t_n+\bar{\tau}) - \xesc(t_n) - \Big(\cesc + \frac{1}{\tau}\Bigr) s \le 0
\,.
\]
These two last assertions are of the same nature as those of \cref{lem:first_attempt_time_int}, and as a consequence the end of the proof of \cref{prop:dissp_zero_some_times} can be carried out along the same lines as the proof of \cref{prop:cv_mean_inv_vel} (or \cref{prop:further_control}). 
\end{proof}
\subsection{Relaxation}
\label{subsec:relax}
The aim of this \namecref{subsec:relax} is to prove the following proposition. 
\begin{proposition}[relaxation]
\label{prop:dissip_app_zero}   
The following assertion holds:
\[
\deltaDissip(t)\to0\quad\text{as}\quad
 t\to+\infty \,. 
\]
\end{proposition}
\begin{proof}
Let us proceed by contradiction and assume that the converse assertion holds. Then there exists a positive quantity $\varepsilon_0$ and a sequence $(t_n)_{n\in\nn}$ of (positive) times such that $t_n$ goes to $+\infty$ as $n$ goes to $+\infty$ and such that, for every nonnegative integer $n$, 
\begin{equation}
\label{deltaDissip_of_tn_larger_than_eps_zero}
\deltaDissip(t_n)\ge\varepsilon_0
\,.
\end{equation}
In other words, there is a ``bump'' of dissipation at each time $t_n$. On the other hand, according to \vref{prop:dissp_zero_some_times}, on every large enough time interval, there exist times where the dissipation around the escape point is low. Roughly speaking, the strategy will be to apply the relaxation scheme set up in \cref{subsec:relax_sch_tr_fr} on a time interval containing a dissipation bump (at a certain time $t_n$) and bounded by two times where the dissipation is low. At both ends of the intervals, it will follow from the smallness of the dissipation that the solution is close to a front travelling at the speed $\cesc$, and therefore that its energy $\eee_{\cesc}$ (properly localized) is close to $0$. Provided that the energy fluxes can be sufficiently controlled along the relaxation scheme on this time interval, this will be in contradiction with the dissipation bump occurring at time $t_n$. 

In order to reach this contradiction, a number of conditions need to be fulfilled. Here are three of them. 
\begin{itemize}
\item The relaxation scheme will actually be applied twice, on each side of the dissipation bump occurring at time $t_n$. Indeed, as illustrated on \cref{fig:two_intervals_required}, applying the relaxation scheme only once on the whole interval may lead to a dissipation bump occurring ``far to the left'' in the appropriate travelling frame, ending up with a negligible influence on the (localized) energy $\eee_{\cesc}$ of the solution. 
\begin{figure}[!htbp]
\centering
\includegraphics[width=\textwidth]{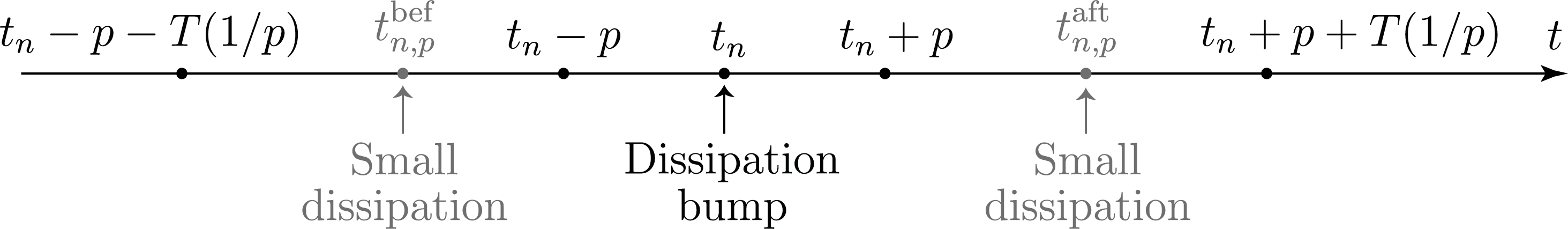}
\caption{Time of ``dissipation bump'' $t_n$, framed by two times where dissipation is almost zero.}
\label{fig:dissipation_bump}
\end{figure}
\item The lengths of the two time intervals where the relaxation scheme will be applied need to be large, in order the mean speed of the escape point on each of these interval to be close to its asymptotic value $\cesc$ (indeed, it is for this speed $\cesc$ that the dissipation will be close to $0$ at the ends and that a dissipation bump occurs at time $t_n$). 
\item Finally, depending on the length of the time intervals where the relaxation scheme is applied, the integer $n$ will have to be chosen large enough, in order the ``front'' flux term (the last term in the right hand side of inequality \vref{relax_scheme_final}) to be controlled. 
\end{itemize}
The third of these conditions leads us to introduce a second integer parameter $p$ that will be related altogether to the length of these two time intervals and to the smallness of the dissipation at both ends. Let us call upon the notation $T(\cdot)$ introduced in \cref{prop:dissp_zero_some_times}. Up to replacing the sequence $(t_n)_{n\in\nn}$ by a subsequence, it may be assumed that this sequence is increasing. Then, since $t_n$ goes to $+\infty$ as $n$ goes to $+\infty$, for every positive integer $p$ there exists a nonnegative integer $n_{\min}(p)$  such that, for every integer $n$, 
\[
n\ge n_{\min}(p) \iff t_n \ge p + T(1/p)
\,.
\]
Note that $n_{\min}(p)$ goes to $+\infty$ as $p$ goes to $+\infty$. 
According to \cref{prop:dissp_zero_some_times}, for every positive integer $p$ and every integer $n$ greater than or equal to $n_{\min}(p)$, there exist 
(see \cref{fig:dissipation_bump}):
\[
\tnpBef \text{ in } \bigl[t_n-p-T(1/p),t_n-p\bigr]
\quad\text{and}\quad
\tnpAft \text{ in } \bigl[t_n+p,t_n+p+T(1/p)\bigr]
\]
such that
\begin{equation}
\label{small_dissip_frame}
\deltaDissip(\tnpBef)\le 1/p
\quad\text{and}\quad
\deltaDissip(\tnpAft)\le 1/p
\,.
\end{equation}
The mentions ``bef'' and ``aft'' are reminders of the fact that these times occur ``before'' or ``after'' the ``dissipation bump time'' $t_n$. In those definitions, the significant features are that both quantities 
\[
\deltaDissip(\tnpBef)
\quad\text{and}\quad
\deltaDissip(\tnpAft)
\]
go to $0$ as $p$ goes to $+\infty$ (uniformly with respect to $n$ greater than $n_{\min}(p)$), and both quantities
\[
t_n - \tnpBef 
\quad\text{and}\quad
\tnpAft - t_n
\]
go to $+\infty$ as $p$ goes to $+\infty$, while remaining bounded with respect to $n$ for every fixed positive integer $p$. 
\begin{figure}[!htbp]
\centering
\includegraphics[width=0.5\textwidth]{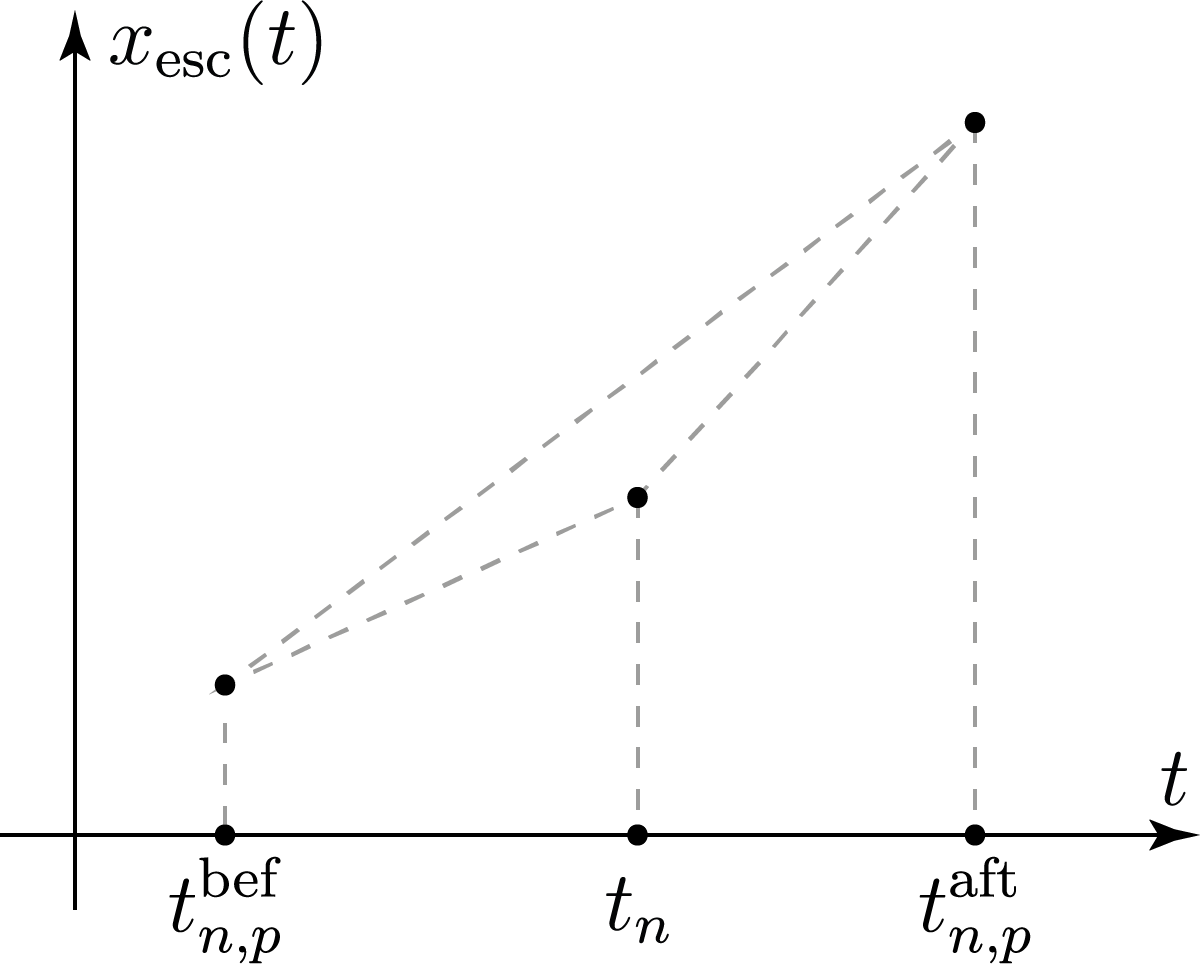}
\caption{Illustration of the reason why the relaxation scheme is applied separately on two intervals: although the slopes on these two (sub)intervals are close to $\cesc$, applying the relaxation scheme on a single interval may lead to a dissipation bump occurring ``far to the left'' of the origin in the moving frame.}
\label{fig:two_intervals_required}
\end{figure}
According to inequality \cref{control_escape} (controlling the growth of the escape point), both quantities 
\[
\frac{ \xesc(t_n) - \xesc(\tnpBef)}{t_n - \tnpBef}
\quad\text{and}\quad
\frac{ \xesc(\tnpAft) - \xesc(t_n) }{\tnpAft - t_n}
\]
(that is the mean speeds of the escape point on the two intervals surrounding the dissipation bump time $t_n$) are bounded from above by $\cnoesc$, and according to \cref{prop:cv_mean_inv_vel,prop:further_control}, both quantities go to $\cesc$ as $p$ goes to $+\infty$, uniformly with respect to $n$ greater than or equal to $n_{\min}(p)$. 

By compactness (\cref{lem:compactness}), up to replacing the sequence $(t_n)_{n\in\nn}$ by a subsequence, there exists an entire solution $\widebar{u}$ of system \cref{init_syst} such that, with the notation of \cref{compactness},
\begin{equation}
\label{asymt_compactness_relax_trav_frame}
D^{2,1}u\bigl( \xesc (t_n) + \cdot, t_n +\cdot \bigr)\to D^{2,1}\widebar{u}
\quad\text{as}\quad
n\to+\infty
\,,
\end{equation}
uniformly on every compact subset of $\rr^2$. Let us denote by $u_\infty$ the function $\xi\mapsto \widebar{u}(\xi,0)$. 

In the two next \namecrefs{subsubsec:relax_sch_left}, the relaxation scheme set up in \cref{subsec:relax_sch_tr_fr} will be applied to the two intervals $[\tnpBef,t_n]$ and $[t_n,\tnpAft]$. The following notation will be extensively used. 
\begin{notation}
For every nonnegative time $t$ and every real quantity $\xi$, let 
\[
\begin{aligned}
E(\xi,t) =& \frac{1}{2}u_x\bigl(\xesc(t) + \xi, t\bigr)^2 + V\Bigr( u\bigl( \xesc(t) + \xi, t\bigr)\Bigr)-V(m)  \,, \\
F(\xi,t) =&\,  \coeffEn\biggl( \frac{1}{2}u_x\bigl(\xesc(t) + \xi, t\bigr)^2 + V\Bigr( u\bigl( \xesc(t) + \xi, t\bigr)\Bigr)-V(m) \biggr) \\
&+ \frac{1}{2}\Bigl(u\bigl(\xesc(t) + \xi, t\bigr)-m\Bigr)^2 \,,
\end{aligned}
\]
and, for every function $\phi$ in $\ccc^1(\rr,\rr^d)$, let 
\[
\begin{aligned}
E_{[\phi]}(\xi) &= \frac{1}{2}\phi'(\xi)^2 + V\bigl(\phi(\xi)\bigr)-V(m) \,, \\
F_{[\phi]}(\xi) &= \coeffEn\Bigl(\frac{1}{2}\phi'(\xi)^2 + V\bigl(\phi(\xi)\bigr)-V(m)\Bigr) + \frac{1}{2}\bigl(\phi(\xi)-m\bigr)^2 \,,
\end{aligned}
\]
\end{notation}
\subsubsection{Relaxation scheme to the left of the dissipation bump}
\label{subsubsec:relax_sch_left}
The aim of this \namecref{subsubsec:relax_sch_left} is to prove the following lemma, stating that the energy at the right-end of the ``left-hand'' interval $[\tnpBef,t_n]$ is negative. To this end, the relaxation scheme set up in \cref{subsec:relax_sch_tr_fr} will be applied to this interval. 
\begin{lemma}[negative energy at right-end of left-hand interval]
\label{lem:neg_energy}
The following inequality holds (and the integral on the left hand side converges):
\begin{equation}
\label{neg_en_pf_dissip_app_0}
\int_{-\infty}^{+\infty} \exp(\cesc \xi) \, E_{[u_{\infty}]}(\xi) \, d\xi <0
\,.
\end{equation}
\end{lemma}
\begin{proof}[Proof of \cref{lem:neg_energy}]
Let us still consider two integers $n$ and $p$ with $p$ positive and $n$ greater than or equal to $n_{\min}(p)$, and let
\[
s_{n,p} = t_n - \tnpBef
\quad\text{and}\quad
c_{n,p} = \frac{ \xesc(t_n) - \xesc(\tnpBef)}{s_{n,p}}
\,.
\]
Let us assume that $p$ is large enough so that 
\[
0< c_{n,p}
\quad\text{and}\quad
\cesc - \frac{\kappa \cCut}{4(\cnoesc+\kappa)} \le c_{n,p}
\,,
\]
and let $\ell$ denote a nonnegative quantity to be chosen below. 
The relaxation scheme set up in \cref{subsec:relax_sch_tr_fr} will be applied with the following parameters:
\[
\tInit = \tnpBef
\quad\text{and}\quad
\xInit = \xesc(\tInit)
\quad\text{and}\quad
c = c_{n,p}
\quad\text{and}\quad
\initCut = \ell
\,.
\]
Thus the relaxation scheme will depend on the three parameters $(n,p,\ell)$. Observe that both hypotheses \vref{hyp_param_relax_sch} and \vref{hyp_c_close_to_barcescsup} (required to apply the relaxation scheme) hold. 
Let us denote by 
\[
\begin{aligned}
& 
v^{(n,p)}(\cdot,\cdot)
\quad\text{and}\quad
\chi^{(n,p,\ell)}(\cdot,\cdot)
\quad\text{and}\quad
\eee^{(n,p,\ell)}(\cdot)
\quad\text{and}\quad
\ddd^{(n,p,\ell)}(\cdot)
\\
&
\text{and}\quad
\psi^{(n,p,\ell)}(\cdot,\cdot)
\quad\text{and}\quad
\fff^{(n,p,\ell)}(\cdot)
\end{aligned}
\]
the objects defined in \cref{subsec:relax_sch_tr_fr} (with the same notation except the ``$(n,p)$'' or ``$(n,p,\ell)$'' superscripts to emphasize the dependency with respect to the parameters). 

The proof is based on the relaxation scheme final inequality \vref{relax_scheme_final} on the $s$-time interval $[0,s_{n,p}]$, which will provide an upper bound on the quantity $\eee^{(n,p,\ell)}(s_{n,p})$ (the localized energy at the right end of this time interval). This will require a careful choice of the three parameters $n$, $p$, and $\ell$ to control the various other quantities in this inequality and the difference between $\eee^{(n,p,\ell)}(s_{n,p})$ and the integral \cref{neg_en_pf_dissip_app_0}. Here are the quantities that have to be controlled:
\begin{enumerate}
\item the dissipation term $\int_0^{s_{n,p}}\ddd^{(n,p,\ell)}(s)\, ds$\,;
\item the initial value $\eee^{(n,p,\ell)}(0)$ of the localized energy;
\item the initial value $\fff^{(n,p,\ell)}(0)$ of the firewall function;
\item the ``back flux'' term (involving the factor $\KGback[u_0]$);
\item the ``front flux'' term (involving the factor $\KGfront$);
\item the difference between the final energy $\eee^{(n,p,\ell)}(s_{n,p})$ and the integral \cref{neg_en_pf_dissip_app_0}. 
\end{enumerate}
Those controls will be stated by a series of lemmas (\cref{lem:dissip_bump,lem:small_init_en_bef,lem:non_small_fin_en_bef,lem:small_init_fire_bef,lem:small_back_flux_fire_bef,lem:small_front_flux_fire_bef}). 
\Cref{fig:dependency_param_relax} summarizes the requirements on the three parameters $n$, $p$, and $\ell$, and the dependencies between those parameters as well, in order all these controls to hold. 
\begin{figure}[!htbp]
\centering
\includegraphics[width=0.75\textwidth]{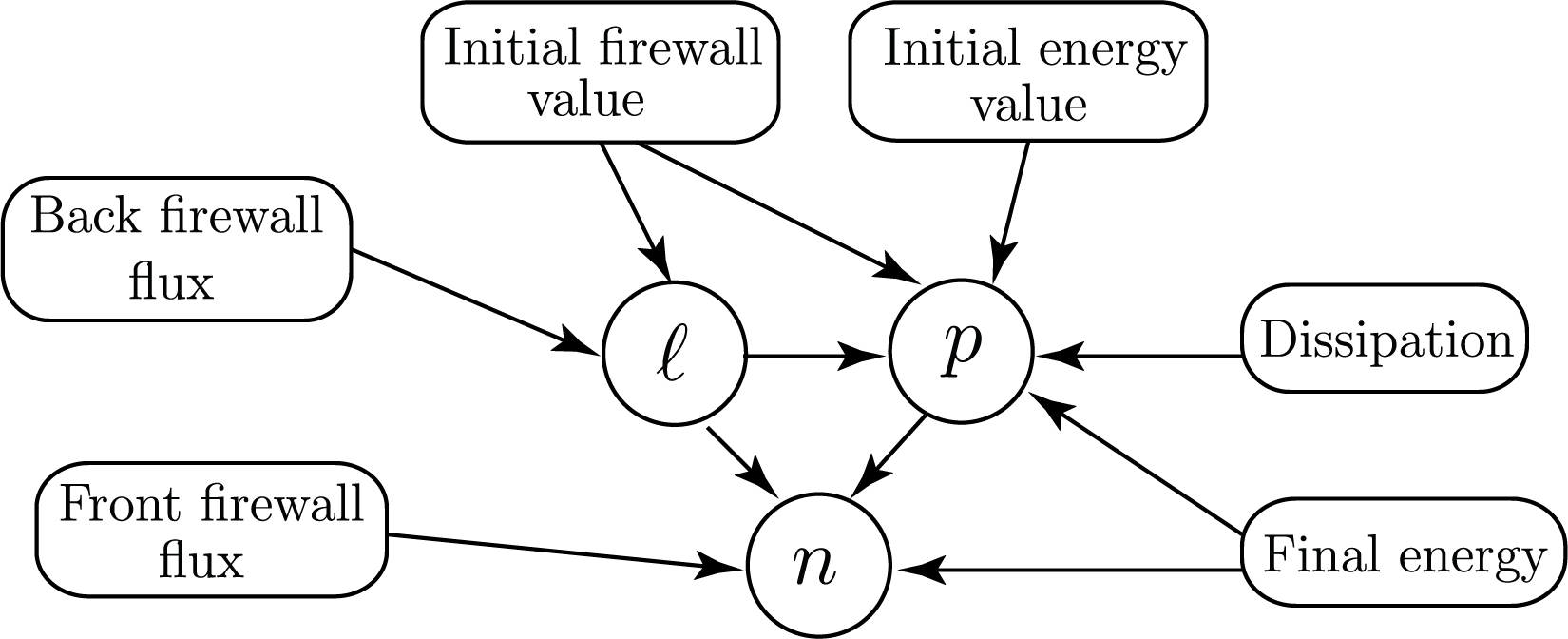}
\caption{Illustration of the constraints on parameters $n$, $p$ and $\ell$ to ensure the estimates required to use the relaxation scheme final inequality \vref{relax_scheme_final}.}
\label{fig:dependency_param_relax}
\end{figure}
As illustrated by this figure, the final choice of those parameters, satisfying all these requirements, will be, in short: 
\begin{itemize}
\item a quantity $\ell$ large enough positive, 
\item an integer $p$ large enough positive (depending on $\ell$),
\item an integer $n$ large enough positive (depending on $\ell$ and $p$). 
\end{itemize}
The proof of \cref{lem:neg_energy} will be presented as a sequence of lemmas. Here is the first one.
\begin{lemma}[lower bound on dissipation]
\label{lem:dissip_bump}
There exists a positive quantity $\epsDissip$ and a positive integer $\pMinDissip$ such that, for every integer $p$ greater than or equal to $\pMinDissip$, for every integer $n$ greater than or equal to $n_{\min}(p)$, and for every nonnegative quantity $\ell$, the following inequality holds:
\[
\frac{1}{2} \int_0^{s_{n,p}} \ddd^{(n,p,\ell)} (s) \, ds \ge \epsDissip
\,.
\]
\end{lemma}
\begin{proof}[Proof of \cref{lem:dissip_bump}]
Indeed, if the converse was true, there would exist a sequence $(p_j)_{j\in\nn}$ of integers, going to $+\infty$ as $j$ goes to $+\infty$, and a sequence $(n_j)_{j\in\nn}$ of integers with $n_j$ greater than or equal to $n_{\min}(p_j)$ for every nonnegative integer $j$ (thus $n_j$ also goes to $+\infty$ as $j$ goes to $+\infty$), and a sequence $(\ell_j)_{j\in\nn}$ of nonnegative quantities, such that, if we consider the integral
\[
I_j = \int_0^{s_{n_j,p_j}} \ddd^{(n_j,p_j,\ell_j)} (s) \, ds
\,,
\]
then $I_j\to0$ as $j\to +\infty$.

Observe that
\[
I_j \ge \int_{s_{n_j,p_j}-1}^{s_{n_j,p_j}} \ddd^{(n_j,p_j,\ell_j)} (s) \, ds= \int_{-1}^0 \ddd^{(n_j,p_j,\ell_j)} (s_{n_j,p_j}+s) \, ds
\,,
\]
so that, according to the definition \cref{def_dissip_tr_fr} of $\ddd(\cdot)$,
\[
\begin{aligned}
I_j &\ge \int_{-1}^0 \left(\int_{\rr}\chi^{(n_j,p_j,\ell_j)}(\xi,s_{n_j,p_j}+s)v^{(n_j,p_j)}(\xi,s_{n_j,p_j}+s)^2\, d\xi\right)\, ds \\
&\ge \int_{-1}^0 \left(\int_{-\infty}^{\ell_j+\cCut s_{n_j,p_j}}\exp(c_{n_j,p_j} \xi)v^{(n_j,p_j)}(\xi,s_{n_j,p_j}+s)^2\, d\xi\right)\, ds \\
&\ge \int_{-1}^0 \left(\int_{-\infty}^{\cCut s_{n_j,p_j}}\exp(c_{n_j,p_j} \xi)v^{(n_j,p_j)}(\xi,s_{n_j,p_j}+s)^2\, d\xi\right)\, ds \\
&\ge \int_{-1}^0 \left(\int_{-\infty}^{\cCut s_{n_j,p_j}}\exp(c_{n_j,p_j} \xi)(u_t+c_{n_j,p_j}u_x)(\xesc(t_n)+c_{n_j,p_j} s + \xi,t_n+s)^2\, d\xi \right)\, ds 
\,,
\end{aligned}
\]
and since $s_{n_j,p_j}$ goes to $+\infty$ as $j$ goes to $+\infty$, if $j$ is large enough positive then $I_j$ is greater than or equal to the quantity
\[
\int_{-1}^0 \left(\int_{-2/\varepsilon_0}^{2/\varepsilon_0+c_{n_j,p_j}}\exp(c_{n_j,p_j} \xi)\Bigl((u_t+c_{n_j,p_j}u_x)\bigl(\xesc(t_n)+c_{n_j,p_j} s + \xi,t_n+s\bigr)\Bigr)^2\, d\xi\right)\, ds 
\,,
\]
so that
\[
\begin{aligned}
I_j \exp&\left(\frac{2}{\epsilon_0}c_{n_j,p_j}\right) \\
&\ge \int_{-1}^0 \left(\int_{-2/\varepsilon_0}^{2/\varepsilon_0+c_{n_j,p_j}}\Bigl((u_t+c_{n_j,p_j}u_x)\bigl(\xesc(t_n)+c_{n_j,p_j} s + \xi,t_n+s\bigr)\Bigr)^2\, d\xi\right)\, ds \\
&\ge \int_{-1}^0 \left(\int_{-2/\varepsilon_0 + c_{n_j,p_j} s}^{2/\varepsilon_0+c_{n_j,p_j}(1+s)}\Bigl((u_t+c_{n_j,p_j}u_x)\bigl(\xesc(t_n) + \xi,t_n+s\bigr)\Bigr)^2\, d\xi\right)\, ds \\
&\ge \int_{-1}^0 \left(\int_{-2/\varepsilon_0}^{2/\varepsilon_0}\Bigl((u_t+c_{n_j,p_j}u_x)\bigl(\xesc(t_n) + \xi,t_n+s\bigr)\Bigr)^2\, d\xi\right)\, ds
\,.
\end{aligned}
\]
In view of the limit \cref{asymt_compactness_relax_trav_frame}, is follows that
\[
\liminf_{j\to+\infty} I_j \exp\left(\frac{2}{\epsilon_0}\cesc\right)\ge \int_{-1}^0 \left(\int_{-2/\varepsilon_0}^{2/\varepsilon_0}\bigl(\widebar{u}_t(\xi,s)+\cesc\widebar{u}_x(\xi,s)\bigr)^2\, d\xi\right)\, ds
\,,
\]
and since $I_j\to0$ as $j\to+\infty$, it follows that the function $\xi\mapsto \widebar{u}_t(\xi,0)+\cesc\widebar{u}_x(\xi,0)$ must vanish identically on the interval $[-2/\varepsilon_0,2/\varepsilon_0]$, a contradiction with inequality \cref{deltaDissip_of_tn_larger_than_eps_zero} and the definition \cref{def_deltaDissip} of $\deltaDissip(\cdot)$. \Cref{lem:dissip_bump} is proved. 
\end{proof}
The conclusions of \cref{lem:neg_energy} will follow from the next five (independent) lemmas. 
\begin{lemma}[upper bound on initial energy]
\label{lem:small_init_en_bef}
For every nonnegative quantity $\ell$, there exists a positive integer $\pMinInitEn(\ell)$ such that, for every integer $p$ greater than or equal to $\pMinInitEn(\ell)$ and every integer $n$ greater than or equal to $n_{\min}(p)$, 
\[
\eee^{(n,p,\ell)}(0) \le \frac{\epsDissip}{8}
\,.
\]
\end{lemma}
\begin{lemma}[lower bound on final energy]
\label{lem:non_small_fin_en_bef}
For every nonnegative quantity $L$, there exists a positive integer $\pMinFinEn(L)$ such that, for every integer $p$ greater than or equal to $\pMinFinEn(L)$, and for every nonnegative quantity $\ell$, there exists a nonnegative integer $\nMinFinEn(L,\ell,p)$ such that, for every integer $n$ greater than or equal to $\nMinFinEn(L,\ell,p)$ and to $n_{\min}(p)$, 
\[
\eee^{(n,p,\ell)}(s_{n,p}) \ge \int_{-\infty}^L \exp(\cesc \xi) \, E_{[u_{\infty}]}(\xi) \, d\xi - \frac{\epsDissip}{8}
\,.
\]
\end{lemma}
\begin{lemma}[upper bound on initial firewall]
\label{lem:small_init_fire_bef}
There exists a positive quantity $\ellMinInitFire$ such that, for every quantity $\ell$ greater than or equal to $\ellMinInitFire$, there exists a positive integer $\pMinInitFire(\ell)$ such that, for every integer $p$ greater than or equal to $\pMinInitFire(\ell)$, and for every integer $n$ greater than or equal to $n_{\min}(p)$, the following inequality holds (the constants $\KEF$ and $\nuF$ being those of inequality \vref{relax_scheme_final}):
\[
\frac{\KEF}{\nuF} \fff^{(n,p,\ell)}(0) \le \frac{\epsDissip}{8}
\]
\end{lemma}
\begin{lemma}[upper bound on back flux for the firewall]
\label{lem:small_back_flux_fire_bef}
There exists a nonnegative \\
quantity $\ellMinBackFire$ such that, for every positive integer $p$ and every integer $n$ greater than or equal to $n_{\min}(p)$, the ``back-flux'' term in inequality \vref{relax_scheme_final} is less than or equal to $\epsDissip/8$.
\end{lemma}
\begin{lemma}[upper bound on front flux for the firewall]
\label{lem:small_front_flux_fire_bef}
For every nonnegative quantity $\ell$ and for every positive integer $p$, there exists a nonnegative integer $\nMinFrontFire(\ell,p)$ such that, for every integer $n$ greater than or equal to $\nMinFrontFire(\ell,p)$ and greater than or equal to $n_{\min}(p)$, the ``front flux'' term in inequality \vref{relax_scheme_final} is less than or equal to $\epsDissip/8$.
\end{lemma}
Postponing the proofs of these five lemmas, let us first conclude with the proof of \cref{lem:neg_energy}. For every nonnegative quantity $L$, let:
\[
\begin{aligned}
\ell &= \max( \ellMinInitFire , \ellMinBackFire )
\,, \\
p &= \max\bigl( \pMinDissip , \pMinInitEn(\ell) , \pMinFinEn(L) , \pMinInitFire(\ell) \bigr)
\,, \\ 
n &= \max\bigl(n_{\min}(p), \nMinFinEn(L,\ell,p) , \nMinFrontFire(\ell,p) \bigr) 
\,.
\end{aligned}
\]
Then, according to \cref{lem:dissip_bump,lem:small_init_en_bef,lem:non_small_fin_en_bef,lem:small_init_fire_bef,lem:small_back_flux_fire_bef,lem:small_front_flux_fire_bef}, it follows from inequality \vref{relax_scheme_final} that
\[
\int_{-\infty}^L \exp(\cesc \xi) \, E_{[u_{\infty}]}(\xi) \, d\xi \le -\frac{3}{8}\epsDissip
\,,
\]
and since the nonnegative quantity $L$ is any, this finishes the proof of \cref{lem:neg_energy} (provided that \cref{lem:small_init_en_bef,lem:non_small_fin_en_bef,lem:small_init_fire_bef,lem:small_back_flux_fire_bef,lem:small_front_flux_fire_bef} hold).
\end{proof}
\begin{proof}[Proof of \cref{lem:small_init_en_bef} (upper bound on initial energy)]
Let us proceed by contradiction \\ 
and assume that the converse holds. Then there exists a nonnegative quantity $\ell_0$, a sequence $(p_j)_{j\in\nn}$ of positive integers going to $+\infty$ as $j$ goes to $+\infty$, and a sequence $(n_j)_{j\in\nn}$ of integers such that $n_j$ is greater than or equal to $n_{\min}(p_j)$ for all $j$ in $\nn$ (thus $n_j$ goes to $+\infty$ as $j$ goes to $+\infty$) and such that, for every $j$ in $\nn$, 
\begin{equation}
\label{hyp_proof_small_init_en_bef}
\eee^{(n_j,p_j,\ell_0)}(0) \ge \frac{\epsDissip}{8}
\,.
\end{equation}

By compactness (\cref{lem:compactness}), up to replacing the sequence $\bigl((n_j,p_j)\bigr)_{j\in\nn}$ by a subsequence, there exists an entire solution $\tilde{u}$ of system \cref{init_syst} such that, with the notation of \cref{compactness},
\begin{equation}
\label{compactness_bef}
D^{2,1}u\bigl( \xesc(\tnjpjBef) + \cdot, \tnjpjBef + \cdot \bigr)\to D^{2,1}\tilde{u}
\quad\text{as}\quad
j\to+\infty
\,,
\end{equation}
uniformly on every compact subset of $\rr^2$. According to the definition \vref{small_dissip_frame} of $\tnpBef$, for every nonnegative integer $j$, 
\[
\deltaDissip(\tnjpjBef) \le \frac{1}{p_j}
\,,
\]
therefore the function
\[
\xi\mapsto\tilde u_t(\xi,0) + \cesc \tilde u_x(\xi,0)
\]
must be identically zero on $\rr$. Let us denote by $\phi$ the function $\xi\mapsto\tilde u_t(\xi,0)$. It follows that $\phi$ is a solution of system \vref{syst_trav_front} for the speed $\cesc$, or in other words is the profile of a wave travelling at the speed $\cesc$ for system \cref{init_syst}. Since (as stated in \vref{xHom_minus_xesc}) the difference $\xHom(t)-\xesc(t)$ goes to $+\infty$ as time goes to $+\infty$, the following estimate holds:
\begin{equation}
\label{est_front_bef}
\abs{\phi (\xi)-m} \le \dEsc(m)
\quad\text{for all}\quad \xi \text{ in }[0,+\infty)
\,,
\end{equation}
thus according to assertion \cref{item:cv_spatial_asymptotics_tw} of \vref{lem:asympt_behav_tw_2}, 
\[
\phi(\xi)\to m
\quad\text{as}\quad
\xi \to +\infty
\,.
\]
On the other hand, according to the bound \vref{attr_ball_infty_inv_implies_cv}, the function $\abs{\phi(\cdot)}$ is bounded on $\rr$, and according to the definition \vref{def_xesc} of $\xesc(\cdot)$ the function $\abs{\phi(\cdot)-m}$ cannot vanish identically. In short, the function $\phi$ must belong to the set $\Phi_{\cesc}(m)$ of bounded profiles of waves travelling at the speed $\cesc$ and ``invading'' the equilibrium $m$. As a consequence, according to \vref{lem:zero_en_tw}, 
\begin{equation}
\label{zero_en_tw_small_init_en_bef}
\int_{\rr} \exp(\cesc \xi) E_{[\phi]}(\xi) \, d\xi =0
\,.
\end{equation}
Recall that
\[
\chi^{(n_j,p_j,\ell_0)}(\xi,0) = \left\{
\begin{aligned}
& \exp(c_{n_j,p_j} \xi) 
& & \quad\text{if}\quad \xi\le \ell_0 \,, \\
& \exp \bigl( c_{n_j,p_j} \ell_0 - \kappa(\xi-\ell_0) \bigr)
& & \quad\text{if}\quad \xi\ge \ell_0 \,,
\end{aligned}
\right.
\]
and let
\begin{equation}
\label{def_chi_infty}
\chi^{(\infty,\infty,\ell_0)}(\xi) = \left\{
\begin{aligned}
& \exp(\cesc \xi) 
& & \quad\text{if}\quad \xi\le \ell_0 \,, \\
& \exp \bigl( \cesc\ell_0 - \kappa(\xi-\ell_0) \bigr)
& & \quad\text{if}\quad \xi\ge \ell_0 \,.
\end{aligned}
\right.
\end{equation}
Since the convergence
\[
\chi^{(n_j,p_j,\ell_0)}(\xi,0) \to 0
\quad\text{as}\quad
\xi\to \pm\infty
\]
is uniform with respect to $j$ provided that $j$ is large enough, it follows that
\[
\chi^{(n_j,p_j,\ell_0)}(\cdot,0) \to \chi^{(\infty,\infty,\ell_0)}(\cdot) 
\quad\text{in}\quad 
L^1(\rr)
\quad\text{as}\quad
j\to +\infty
\]
and (according to the definition \cref{compactness_bef} of $\phi$ and the bounds \vref{bound_u_ut_ck_bis} for the solution) that 
\[
\xi\mapsto \chi^{(n_j,p_j,\ell_0)}(\xi,0) E(\xi,t_{n_j,p_j}^{\textup{bef}})
\quad\text{approaches}\quad
\xi\mapsto \chi^{(\infty,\infty,\ell_0)}(\xi) E_{[\phi]}(\xi) 
\quad\text{in}\quad
L^1(\rr,\rr)
\]
as $j$ goes to $+\infty$.
As a consequence, 
\begin{equation}
\label{conv_en_bef}
\eee^{(n_j,p_j,\ell_0)}(0) \to \int_{\rr} \chi^{(\infty,\infty,\ell_0)}(\xi) E_{[\phi]}(\xi) \, d\xi
\quad\text{as}\quad
j\to +\infty
\,.
\end{equation}
According to inequality \cref{est_front_bef} the quantity $V\bigl(\phi(\xi)\bigr)\bigr)-V(m)$ is nonnegative (actually positive) for all $\xi$ in $[0,+\infty)$, therefore according to the identity \cref{zero_en_tw_small_init_en_bef},
\begin{equation}
\label{neg_en_front_trunc_weight_bef}
\int_{\rr} \chi^{(\infty,\infty,\ell_0)}(\xi) E_{[\phi]}(\xi) \, d\xi \le 0
\,.
\end{equation}
The contradiction with hypothesis \cref{hyp_proof_small_init_en_bef} follows from \cref{conv_en_bef} and \cref{neg_en_front_trunc_weight_bef}.
\Cref{lem:small_init_en_bef} is proved.
\end{proof}
\begin{proof}[Proof of \cref{lem:non_small_fin_en_bef} (lower bound on final energy)]
Recall that
\[
\eee^{(n,p,\ell)} (s_{n,p})= \int_{\rr} \chi^{(n,p,\ell)}(\xi,s_{n,p}) \, E(\xi,t_n) \, d\xi
\]
and that
\[
\chi^{(n,p,\ell)}(\xi,s_{n,p}) = \left\{
\begin{aligned}
& \exp(c_{n,p} \xi) 
& &  \text{if}\quad \xi\le \ell + \cCut s_{n,p} \,, \\
& \exp \Bigl( c_{n,p} (\ell + \cCut s_{n,p}) - \kappa\bigl(\xi-(\ell + \cCut s_{n,p}) \bigr) \Bigr)
& &  \text{if}\quad \xi\ge \ell + \cCut s_{n,p} \,.
\end{aligned}
\right.
\]
It follows from the definition of $\xesc(\cdot)$ that
\[
V\Bigl( u\bigl(\xesc(t_n) + \xi, t_n\bigr) -m\Bigr) - V(m)\ge 0
\quad\text{for all}\quad
\xi
\quad\text{in}\quad
[0,\xHom(t_n) - \xesc(t_n)]
\,.
\]
Let $L$ be a positive quantity, and let us assume that $n$ is large enough so that 
\[
L\le \xHom(t_n) - \xesc(t_n)
\]
(this is possible according to assertion \vref{xHom_minus_xesc}). Then, 
\begin{equation}
\label{low_bd_non_small_fin_en_bef}
\begin{aligned}
\eee^{(n,p,\ell)} (s_{n,p}) \ge & \int_{-\infty}^L \chi^{(n,p,\ell)}(\xi,s_{n,p}) \, E(\xi,t_n) \, d\xi  \\
& + \int_{\xHom(t_n) - \xesc(t_n)}^{+\infty} \chi^{(n,p,\ell)}(\xi,s_{n,p}) \, E(\xi,t_n) \, d\xi 
\,.
\end{aligned}
\end{equation}
Recall that the choice of $\tnpBef$ ensures that $s_{n,p}$ is greater than or equal to $p$, therefore as soon as $p$ is large enough (namely greater than or equal to $L/\cCut$), the first integral at the right-hand of this last inequality \cref{low_bd_non_small_fin_en_bef} equals
\begin{equation}
\label{first_int_non_small_fin_en_bef}
\int_{-\infty}^L \exp(c_{n,p} \xi) \, E(\xi,t_n) \, d\xi 
\,.
\end{equation}
Since the convergence 
\[
\exp(c_{n,p} \xi) \to 0 
\quad\text{as}\quad
\xi\to -\infty
\]
is uniform with respect to $n$ and $p$ provided that $p$ is large enough, and according to the bounds \vref{bound_u_ut_ck_bis}, 
\[
\xi\mapsto \exp(c_{n,p} \xi) \, E(\xi,t_n)
\quad\text{approaches}\quad
\xi\mapsto \exp(\cesc \xi) \, E_{[u_{\infty}]}(\xi)
\]
in $L^1\bigl((-\infty,L],\rr\bigr)$, as both $n$ and $p$ go to $+\infty$.

On the other hand, for $\ell$ and $p$ fixed and $n$ large enough (depending on the values of $\ell$ and $p$), the second integral of the right-hand of inequality \cref{low_bd_non_small_fin_en_bef} equals
\begin{equation}
\label{second_int_non_small_fin_en_bef}
\exp \bigl( (c_{n,p}+\kappa) (\ell + \cCut s_{n,p}) \bigr) \int_{\xHom(t_n) - \xesc(t_n)}^{+\infty} e^{- \kappa\xi} \, E(\xi,t_n) \, d\xi 
\,.
\end{equation}
According to the bounds \cref{bound_u_ut_ck_bis} and since according to the choice of $t_{n,p}^{\textup{bef}}$ the quantity $s_{n,p}$ is less than or equal to $p+T(1/p)$ (it is thus bounded from above, uniformly with respect to $n$), this last quantity \cref{second_int_non_small_fin_en_bef} goes to $0$ as $n$ goes to $+\infty$ and $\ell$ and $p$ are fixed. 

The desired lower bound of \cref{lem:non_small_fin_en_bef} thus follows from inequality \cref{low_bd_non_small_fin_en_bef}. 
This finishes the proof of \cref{lem:non_small_fin_en_bef}.
\end{proof}
The next step is to prove \cref{lem:small_init_fire_bef}. Let us first introduce some notation and state an intermediate lemma, before actually proceeding to the proof. Let
\begin{equation}
\label{definition_rho_infty}
\psi^{(\infty,\infty,\ell)} (\xi) = 
\left\{
\begin{aligned}
& \exp\bigl((\cesc + \kappa) \xi - \kappa \ell \bigr) & & \quad\text{if}\quad \xi \le \ell \\
& \exp\bigl((\cesc+\kappa)\ell - \kappa \xi \bigr) & & \quad\text{if}\quad \xi \ge \ell \,.
\end{aligned} 
\right. \\
\end{equation} 
Observe that, for every $\xi$ in $\rr$, 
\[
\psi^{(n,p,\ell)} (\xi,0) \to \psi^{(\infty,\infty,\ell)} (\xi) 
\quad\text{as}\quad
p\to +\infty
\,,
\]
uniformly with respect to $n$ greater than or equal to $n_{\min}(p)$. Let
\[
\epsDissipFire = \frac{\nuF \, \epsDissip}{8 \, \KEF}
\,.
\]
The aim of \cref{lem:small_init_fire_bef} is to prove, under suitable conditions on the integers $n$ and $p$ and the quantity $\ell$, that $\fff^{(n,p,\ell)}(0)$ is less than or equal to $\epsDissipFire$. The following intermediate lemma provides a step towards this purpose.
\begin{lemma}[smallness of firewall with a weight bulk far to the right for a travelling front]
\label{lem:small_fire_tw}
There exists a positive quantity $\ell_1$ such that, for every quantity $\ell$ greater than or equal to $\ell_1$, and for every function $\phi$ in the set $\Phi_{\cesc}(m)$ of bounded profiles of waves travelling at the speed $\cesc$ and invading $m$, such that $\abs{\phi(\xi)-m}$ is not larger than $\dEsc(m)$ for every nonnegative quantity $\xi$, the following inequality holds:
\begin{equation}
\label{ineq_small_fire_tw}
\int_{\rr} \psi^{(\infty,\infty,\ell)} (\xi) \, F_{[\phi]}(\xi) \, d\xi \le \frac{\epsDissipFire}{2}
\,.
\end{equation}
\end{lemma}
\begin{proof}[Proof of \cref{lem:small_fire_tw}]
Observe that, for every real quantity $\xi$ in $(-\infty,\ell]$, 
\[
\psi^{(\infty,\infty,\ell)} (\xi) = \exp\bigl(-\frac{\kappa\ell}{2}\bigr) \exp\biggl( (\cesc + \kappa) \Bigl( \xi - \frac{\kappa \ell}{2(\cesc + \kappa)}\Bigr) \biggr)
\,.
\]
Thus, according to the bounds \cref{bound_u_ut_ck_bis}, there exists a (large) positive quantity $\ell_2$ such that, for every $\ell$ greater than or equal to $\ell_2$, 
\[
\int_{-\infty}^{\kappa \ell/\bigl(2(\cesc + \kappa)\bigr)} \psi^{(\infty,\infty,\ell)} (\xi) \, F_{[\phi]}(\xi) \, d\xi \le \frac{\epsDissipFire}{4}
\,.
\]
It follows that
\[
\int_{\rr} \psi^{(\infty,\infty,\ell)} (\xi) \, F_{[\phi]}(\xi) \, d\xi 
\le 
\frac{\epsDissipFire}{4} + \int_{\kappa \ell/\bigl(2(\cesc + \kappa)\bigr)}^{+\infty} \exp(\cesc \xi) \, F_{[\phi]}(\xi) \, d\xi 
\,.
\]
According to assertion \cref{item:exp_cv_spatial_asymptotics_tw} of \vref{lem:asympt_behav_tw_2}, the second term (the integral) in the right-hand side of this inequality goes to $0$ as $\ell$ goes to $+\infty$. In addition, according to the Local Stable Manifold Theorem, this convergence is uniform with respect to the profile $\phi$ belonging to $\Phi_{\cesc}(m)$ and satisfying $\abs{\phi(\xi)-m}\le\dEsc(m)$ for every nonnegative quantity $\xi$. This proves \cref{lem:small_fire_tw}.
\end{proof}
\begin{proof}[Proof of \cref{lem:small_init_fire_bef} (upper bound on initial firewall)]
Let us proceed by contradiction and assume that the converse holds. Then, there exists a quantity $\ell_3$ greater than or equal to the quantity $\ell_1$ introduced in \cref{lem:small_fire_tw}, a sequence $(p_j)_{j\in\nn}$ of positive integers going to $+\infty$, and a sequence $(n_j)_{j\in\nn}$ of integers with $n_j$ greater than or equal to $n_{\min}(p_j)$ for all $j$ (thus $n_j$ also goes to $+\infty$ as $j$ goes to $+\infty$), such that, for every nonnegative integer $j$, 
\[
\fff^{(n_j,p_j,\ell_3)}(0) \ge \epsDissipFire
\,.
\]
Up to replacing the sequence $\bigl((n_j,p_j)\bigr)_{j\in\nn}$ by a subsequence, and proceeding as in the proof of \cref{lem:small_init_en_bef}, it may be assumed that there exists a function $\phi$ in the set $\Phi_{\cesc}(m)$ of bounded profiles of waves travelling at the speed $\cesc$ and ``invading'' the equilibrium $m$, such that, for every positive quantity $L$, 
\[
\norm{ \xi\mapsto m+v^{(n_j,p_j)}(\xi,0) - \phi(\xi) }_{\ccc^1\bigl([-L,L],\rr^d\bigr)} \to 0
\quad\text{as}\quad
j \to +\infty
\,.
\]
Moreover, according to the definition of $\xesc(\cdot)$ and to \vref{lem:esc_Esc} (escape / Escape), the quantity $\abs{\phi(\xi)-m}$ is not larger than $\dEsc(m)$ for every nonnegative quantity $\xi$. 
According to the bounds \vref{bound_u_ut_ck_bis}, the function
\begin{equation}
\label{integrand_small_init_fire_bef}
\xi\mapsto \psi^{(n_j,p_j,\ell_3)}(\xi,0) \, F(\xi,t_{n_j,p_j}^{\textup{bef}}) 
\end{equation}
goes to $0$ as $\xi$ goes to $\pm\infty$, and this convergence is uniform with respect to $j$ (provided that $j$ is large enough). As a consequence, the function \cref{integrand_small_init_fire_bef} above approaches 
\[
\xi\mapsto \psi^{(\infty,\infty,\ell_3)} (\xi) \, F_{[\phi]}(\xi)
\]
in $L^1(\rr)$, as $j$ goes to $+\infty$. It follows that $\fff^{(n_j,p_j,\ell_3)}(0)$ goes to the quantity
\[
\int_{\rr} \psi^{(\infty,\infty,\ell_3)} (\xi) \, F_{[\phi]}(\xi) \, d\xi
\]
as $j$ goes to $+\infty$, a contradiction with \cref{lem:small_fire_tw}. \Cref{lem:small_init_fire_bef} is proved. 
\end{proof}
\begin{proof}[Proof of \cref{lem:small_back_flux_fire_bef} (upper bound on back flux for the firewall)]
The statement follows from the expression of the back flux term in the relaxation scheme final inequality \vref{relax_scheme_final}. 
\end{proof}
\begin{proof}[Proof of \cref{lem:small_front_flux_fire_bef} (upper bound on front flux for the firewall)]
The statement\\ 
follows from the expression of the front flux term in the relaxation scheme final inequality \vref{relax_scheme_final}.
\end{proof}
The proof of \cref{lem:neg_energy} is complete.
\subsubsection{Relaxation scheme to the right of the dissipation bump}
\label{subsubsec:relax_sch_right}
The purpose of this \namecref{subsubsec:relax_sch_right} is to apply once again the relaxation scheme set up in \cref{subsec:relax_sch_tr_fr} to the ``second'' sub-interval (between $t_n$ and $t_{n,p}^{\textup{aft}}$, see \vref{fig:two_intervals_required}), in order to complete the proof of \cref{prop:dissip_app_zero}. The arguments are very similar to those of the proof of \cref{lem:neg_energy} in the previous \namecref{subsubsec:relax_sch_right}.

Let us introduce the positive quantity $\epsEnergy$ defined as
\[
\epsEnergy = - \int_{\rr} \exp(\cesc \xi) E_{[u_\infty]}(\xi) \, d\xi
\,.
\]
The same notation as in the previous \namecref{subsubsec:relax_sch_right} will be used again to denote objects that are defined similarly but with respect to the ``second'' subinterval $[t_n,t_{n,p}^{\textup{aft}}]$ (by contrast with the ``first'' one $[t_{n,p}^{\textup{bef}},t_n]$). In other words, for notational simplicity, the superscripts ``aft'' (versus ``bef'') will be omitted for those objects. This begins with the following notation. 

For every positive integer $p$ and every integer $n$ greater than or equal to $n_{\min}(p)$, let
\[
s_{n,p} = t_{n,p}^{\textup{aft}} - t_n 
\quad\text{and}\quad
c_{n,p} = \frac{\xesc(t_{n,p}^{\textup{aft}})-\xesc(t_n)}{s_{n,p}}
\,.
\]
Let us assume that $p$ is large enough so that
\[
0 < c_{n,p} 
\quad\text{and}\quad
c_{n,p} \ge \cesc - \frac{\kappa \cCut}{4(\cnoesc+\kappa)}
\,,
\]
and let $\ell$ denote a nonnegative quantity to be chosen below. The relaxation scheme set up in \cref{subsec:relax_sch_tr_fr} will be applied with the following parameters:
\[
\tInit = t_{p} 
\quad\text{and}\quad
\xInit = \xesc(\tInit)
\quad\text{and}\quad
c = c_{n,p}
\quad\text{and}\quad
\initCut = \ell
\,.
\]
As in the previous \namecref{subsubsec:relax_sch_right}, the relaxation scheme thus depends on the three parameters $n$, $p$, and $\ell$. Observe that both hypotheses \vref{hyp_param_relax_sch} and \vref{hyp_c_close_to_barcescsup} (required to apply the relaxation scheme) hold. Let us denote by
\[
\begin{aligned}
& 
v^{(n,p)}(\cdot,\cdot)
\quad\text{and}\quad
\chi^{(n,p,\ell)}(\cdot,\cdot)
\quad\text{and}\quad
\eee^{(n,p,\ell)}(\cdot)
\quad\text{and}\quad
\ddd^{(n,p,\ell)}(\cdot)
\\
&
\text{and}\quad
\psi^{(n,p,\ell)}(\cdot,\cdot)
\quad\text{and}\quad
\fff^{(n,p,\ell)}(\cdot)
\end{aligned}
\]
the objects defined in \cref{subsec:relax_sch_tr_fr} (with the same notation except the ``$(n,p)$'' or ``$(n,p,\ell)$'' superscripts to emphasize the dependency with respect to the parameters). The contradiction completing the proof of \cref{prop:dissip_app_zero} will follow from the next five lemmas.
\begin{lemma}[upper bound on initial energy]
\label{lem:small_init_en_aft}
For every nonnegative quantity $\ell$, there exists a positive integer $\pMinInitEn(\ell)$ such that, for every integer $p$ greater than or equal to $\pMinInitEn(\ell)$ and every integer $n$ greater than or equal to $n_{\min}(p)$, 
\[
\eee^{(n,p,\ell)}(0) \le -\frac{7}{8} \epsEnergy
\,.
\]
\end{lemma}
\begin{lemma}[lower bound on final energy]
\label{lem:non_small_fin_en_aft}
There exists a positive integer $\pMinFinEn$ such that, for every integer $p$ greater than or equal to $\pMinFinEn$ and for every nonnegative quantity $\ell$, there exists an nonnegative integer $\nMinFinEn(\ell,p)$ such that, for every integer $n$ greater than or equal to $\nMinFinEn(\ell,p)$ and to $n_{\min}(p)$, 
\[
\eee^{(n,p,\ell)}(s_{n,p}) \ge -\frac{\epsEnergy}{8}
\,.
\]
\end{lemma}
\begin{lemma}[upper bound on initial firewall]
\label{lem:small_init_fire_aft}
There exists a positive quantity $\ellMinInitFire$ such that, for every quantity $\ell$ greater than or equal to $\ellMinInitFire$, there exists a positive integer $\pMinInitFire(\ell)$ such that, for every integer $p$ greater than or equal to $\pMinInitFire(\ell)$, and for every integer $n$ greater than or equal to $n_{\min}(p)$, the following inequality holds (the constants $\KEF$ and $\nuF$ being those of inequality \vref{relax_scheme_final}):
\[
\frac{\KEF}{\nuF} \fff^{(n,p,\ell)}(0) \le \frac{\epsEnergy}{8}
\]
\end{lemma}
\begin{lemma}[upper bound on back flux for the firewall]
\label{lem:small_back_flux_fire_aft}
There exists a nonnegative \\
quantity $\ellMinBackFire$ such that, for every positive integer $p$ and every integer $n$ greater than or equal to $n_{\min}(p)$, the ``back-flux'' term in inequality \vref{relax_scheme_final} is less than or equal to $\epsEnergy/8$.
\end{lemma}
\begin{lemma}[upper bound on front flux for the firewall]
\label{lem:small_front_flux_fire_aft}
For every nonnegative quantity $\ell$ and for every positive integer $p$, there exists an integer $\nMinFrontFire(\ell,p)$ such that, for every integer $n$ greater than or equal to $\nMinFrontFire(\ell,p)$ and to $n_{\min}(p)$, the ``front flux'' term in inequality \vref{relax_scheme_final} is less than or equal to $\epsEnergy/8$.
\end{lemma}
Postponing the proofs of these five lemmas, let us first conclude with the proof of \cref{prop:dissip_app_zero}. Let
\[
\begin{aligned}
\ell &= \max( \ellMinInitFire , \ellMinBackFire )
\,, \\
p &= \max\bigl( \pMinInitEn(\ell) , \pMinFinEn , \pMinInitFire(\ell) \bigr)
\,, \\ 
n &= \max\bigl(n_{\min}(p), \nMinFinEn(\ell,p) , \nMinFrontFire(\ell,p) \bigr) 
\,.
\end{aligned}
\] 
Then, according to \cref{lem:small_init_en_aft,lem:non_small_fin_en_aft,lem:small_init_fire_aft,lem:small_back_flux_fire_aft,lem:small_front_flux_fire_aft}, inequality \vref{relax_scheme_final} leads to an immediate contradiction. This finishes the proof of \cref{prop:dissip_app_zero} (provided that \cref{lem:small_init_en_aft,lem:non_small_fin_en_aft,lem:small_init_fire_aft,lem:small_back_flux_fire_aft,lem:small_front_flux_fire_aft} hold).
\end{proof}
\begin{proof}[Proof of \cref{lem:small_init_en_aft} (upper bound on initial energy)]
Let $\ell$ denote a nonnegative quantity, and let (as in the definition \vref{def_chi_infty})
\[
\chi^{(\infty,\infty,\ell)}(\xi) = \left\{
\begin{aligned}
& \exp(\cesc \xi) 
& & \quad\text{if}\quad \xi\le \ell \,, \\
& \exp \bigl( \cesc\ell - \kappa(\xi-\ell) \bigr)
& & \quad\text{if}\quad \xi\ge \ell \,.
\end{aligned}
\right.
\]
For every large enough positive integer $p$, and every integer $n$ greater than or equal to $n_{\min}(p)$, the mean speed is close to $\cesc$, thus (say) greater than $\cesc/2$. As a consequence, the convergence
\[
\chi^{(n,p,\ell)}(\xi,0) \to \chi^{(\infty,\infty,\ell)}(\xi)
\quad\text{as}\quad
\xi\to \pm\infty
\] 
is uniform with respect to $n$ and $p$ (provided that $p$ is large enough and that $n$ is greater than or equal to $n_{\min}(p)$ and that $\ell$ is fixed). It follows from the convergence above that
\[
\xi\mapsto\chi^{(n,p,\ell)}(\xi,0) \, E(\xi,t_n) 
\quad\text{approaches}\quad
\xi\mapsto\chi^{(\infty,\infty,\ell)}(\xi) \, E_{[u_{\infty}]}(\xi)
\]
in $L^1(\rr,\rr)$ as $p$ goes to $+\infty$ (uniformly with respect to $n$ greater than or equal to $n_{\min}(p)$). Thus
\[
\eee^{(n,p,\ell)}(0) \to \int_{\rr} \chi^{(\infty,\infty,\ell)}(\xi) \, E_{[u_{\infty}]}(\xi) \, d\xi
\quad\text{as}\quad
p\to +\infty
\,,
\]
uniformly with respect to $n$ greater than or equal to $n_{\min}(p)$. Since the quantity $V\bigl(u_{\infty}(\xi)\bigr)-V(m)$ is nonnegative for every nonnegative quantity $\xi$, the following inequality holds:
\[
\int_{\rr} \chi^{(\infty,\infty,\ell)}(\xi) \, E_{[u_{\infty}]}(\xi) \, d\xi 
\le
\int_{\rr} \exp(\cesc \xi) \, E_{[u_{\infty}]}(\xi) \, d\xi = -\epsEnergy
\,,
\]
and \cref{lem:small_init_en_aft} follows. 
\end{proof}
\begin{proof}[Proof of \cref{lem:non_small_fin_en_aft} (lower bound on final energy)]
According to\\ 
\vref{lem:zero_en_tw}, there exists a positive quantity $L$ such that, for every function $\phi$ in the set $\Phi_{\cesc}(m)$ of bounded profiles of waves travelling at the speed $\cesc$ and ``invading'' the equilibrium $m$, satisfying $\abs{\phi(\xi)-m}\le \dEsc(m)$ for every nonnegative quantity $\xi$, the following estimate holds:
\[
\int_{-\infty}^L \exp(\cesc \xi) \, E_{[\phi]}(\xi) \, d\xi \ge -\frac{\epsEnergy}{24}
\,.
\]
As in the proof of \cref{lem:small_init_en_bef}, let us assume that $n$ is a positive integer, large enough so that
\[
L \le \xHom(t_{n,p}^{\textup{aft}}) - \xesc(t_{n,p}^{\textup{aft}})
\]
(this is possible according to assertion \vref{xHom_minus_xesc}). Then,
\begin{equation}
\label{low_bd_non_small_fin_en_aft}
\begin{aligned}
\eee^{(n,p,\ell)}(s_{n,p}) \ge  & 
\int_{-\infty}^L \chi^{(n,p,\ell)} (\xi,s_{n,p}) \, E (\xi,t_{n,p}^{\textup{aft}}) \, d\xi \\
& + \int_{\xHom(t_{n,p}^{\textup{aft}}) - \xesc(t_{n,p}^{\textup{aft}})}^{+\infty} \chi^{(n,p,\ell)} (\xi,s_{n,p}) \, E (\xi,t_{n,p}^{\textup{aft}}) \, d\xi
\,.
\end{aligned}
\end{equation}
Recall that the choice of $\tnpAft$ ensures that $s_{n,p}$ is greater than or equal to $p$, therefore as soon as $p$ is large enough (namely greater than or equal to $L/\cCut$), the quantity $\ell + \cCut s_{n,p}$ is greater than or equal to $L$, and therefore the first integral at the right-hand of this last inequality \cref{low_bd_non_small_fin_en_aft} equals
\begin{equation}
\label{first_int_non_small_fin_en_aft}
\int_{-\infty}^L \exp(c_{n,p} \xi) \, E(\xi,t_{n,p}^{\textup{aft}}) \, d\xi 
\,.
\end{equation}
The following lemma deals with this integral.
\begin{lemma}[lower bound on final energy, integral between $-\infty$ and $L$]
\label{lem:non_small_fin_en_aft_step}
There exists a positive integer $\pMinFinEn$ such that, for every integer $p$ greater than or equal to $\pMinFinEn$, for every integer $n$ greater than or equal to $n_{\min}(p)$, and for every nonnegative quantity $\ell$, 
\[
\int_{-\infty}^L \exp(c_{n,p} \xi) \, E(\xi,t_{n,p}^{\textup{aft}}) \, d\xi \ge -\frac{\epsEnergy}{12}
\,.
\]
\end{lemma}
\begin{proof}[Proof of \cref{lem:non_small_fin_en_aft_step}]
Let us proceed by contradiction and assume that the converse holds. Then there exist a sequence $(p_j)_{j\in\nn}$ of positive integers going to $+\infty$ as $j$ goes to $+\infty$, a sequence $(n_j)_{j\in\nn}$ of nonnegative integers such that $n_j$ is greater than or equal to than $n_{\min}(p_j)$ for every nonnegative integer $j$, and a sequence $(\ell_j)_{j\in\nn}$ of nonnegative quantities such that, for every nonnegative integer $j$, 
\begin{equation}
\label{hyp_proof_lem_non_small_fin_en_aft_step}
\int_{-\infty}^L \exp(c_{n_j,p_j} \xi) \, E (\xi,t_{n_j,p_j}^{\textup{aft}}) \, d\xi \le -\frac{\epsEnergy}{12}
\,.
\end{equation}
Up to replacing the sequence $\bigl((n_j,p_j,\ell_j)\bigr)_{j\in\nn}$ by a subsequence, it may be assumed (proceeding as in the proof of \cref{lem:small_init_en_bef}) that there exists a function $\phi$ in the set $\Phi_{\cesc}(m)$ of bounded profiles of waves travelling at the speed $\cesc$ and ``invading'' the equilibrium $m$, such that, for every positive quantity $L'$, 
\[
\norm{\xi\mapsto u \bigl(\xesc(\tnjpjAft)+\xi,\tnjpjAft\bigr) - \phi(\xi)}_{\ccc^2([-L',L'],\rr^d)} \to 0
\quad\text{as}\quad
j\to +\infty
\,,
\]
and such that $\abs{\phi(\xi)-m}\le\dEsc(m)$ for all $\xi$ in $[0,+\infty)$. 
Since the convergence
\[
\exp(c_{n_j,p_j} \xi) \to 0
\quad\text{as}\quad
\xi\to -\infty
\]
is uniform with respect to $j$ (provided that $j$ is large enough positive), it follows that
\[
\xi\mapsto \exp(c_{n_j,p_j} \xi) \, E (\xi,t_{n_j,p_j}^{\textup{aft}})
\quad\text{approaches}\quad
\xi\mapsto \exp(\cesc \xi) \, E_{[\phi]}(\xi)
\quad\text{in}\quad
L^1\bigl((-\infty,L],\rr\bigr)
\]
as $j$ goes to $+\infty$. According to \vref{lem:zero_en_tw}, 
\[
\int_{\rr} \exp(\cesc \xi) \, E_{[\phi]}(\xi) \, d\xi = 0
\quad\text{thus}\quad
\int_{-\infty}^L \exp(\cesc \xi) \, E_{[\phi]}(\xi) \, d\xi \le 0
\,,
\]
a contradiction with hypothesis \cref{hyp_proof_lem_non_small_fin_en_aft_step}. \Cref{lem:non_small_fin_en_aft_step} is proved.
\end{proof}
Let us pursue with the end of the proof of \cref{lem:non_small_fin_en_aft}. Let us assume that $p$ is greater than $\pMinFinEn$. For such fixed integer $p$ and for every nonnegative fixed quantity $\ell$, since $s_{n,p}$ is bounded from above uniformly with respect to $n$, the following inequality holds for every large enough positive integer $n$:
\[
\ell + \cCut s_{n,p} \ge \xHom(t_{n,p}^{\textup{aft}}) - \xesc(t_{n,p}^{\textup{aft}})
\,.
\]
Thus, the second integral of the right-hand side of inequality\cref{low_bd_non_small_fin_en_aft} reads
\[
\exp\bigl((c_{n,p}+\kappa)(\ell+\cCut s_{n,p})\bigr) \int_{\xHom(t_{n,p}^{\textup{aft}}) - \xesc(t_{n,p}^{\textup{aft}})}^{+\infty} \exp(-\kappa \xi) \, E (\xi,t_{n,p}^{\textup{aft}}) \, d\xi
\,.
\]
According to the bounds \cref{bound_u_ut_ck_bis}, this last quantity goes to $0$ as $n$ goes to $+\infty$ and $\ell$ and $p$ are fixed. 
This finishes the proof of \cref{lem:non_small_fin_en_aft}.
\end{proof}
The next step is to prove \cref{lem:small_init_fire_aft}. Let us first state an intermediate result. 
For every nonnegative quantity $\ell$, every nonnegative integer $p$, every integer $n$ greater than or equal to $n_{\min}(p)$, let us define the function $\xi\mapsto \psi^{(\infty,\infty,0)} (\xi)$ as in \cref{definition_rho_infty}. Let 
\[
\tidleEpsEnergy = \frac{\nuF}{8\KEF} \epsEnergy
\,.
\]
\begin{lemma}[small upper bound on the firewall with a weight bulk far to the right for $u_{\infty}$]
\label{lem:small_fire_tw_aft}
There exists a positive quantity $\ell_4$ such that, for every quantity $\ell$ greater than or equal to $\ell_4$, the following inequality holds:
\begin{equation}
\label{small_fire_tw_aft}
\int_{\rr} \psi^{(\infty,\infty,\ell)} (\xi) \, F_{[u_{\infty}]}(\xi) \, d\xi \le \frac{\tidleEpsEnergy}{2}
\,.
\end{equation}
\end{lemma}
\begin{proof}[Proof of \cref{lem:small_fire_tw_aft}]
Let us proceed like in the proof of \cref{lem:small_fire_tw}. According to the bounds \cref{bound_u_ut_ck_bis}, there exists a (large) positive quantity $\ell_5$ such that, for every $\ell$ greater than or equal to $\ell_5$, 
\[
\int_{-\infty}^{\kappa\ell / \bigl(2(\cesc + \kappa)\bigr)} \psi^{(\infty,\infty,\ell)} (\xi) \, F_{[u_{\infty}]}(\xi) \, d\xi \le \frac{\tidleEpsEnergy}{4}
\,.
\]
It follows that
\begin{equation}
\label{small_fire_tw_aft_intermediary_step}
\int_{\rr} \psi^{(\infty,\infty,\ell)} (\xi) \, F_{[u_{\infty}]}(\xi) \, d\xi \le \frac{\tidleEpsEnergy}{4} + \int_{\kappa\ell / \bigl(2(\cesc + \kappa)\bigr)}^{+\infty}\exp(\cesc \xi) \, F_{[u_{\infty}]}(\xi) \, d\xi
\,.
\end{equation}
Besides, according to inequality \vref{posit_pot_around_loc_min}, for every nonnegative $\xi$, 
\[
\bigl(u_{\infty}(\xi)-m\bigr)^2 \le \frac{4}{\eigVmin(m)} \Bigl(V\bigl(u_\infty(\xi)\bigr)-V(m)\Bigr)
\,,
\]
thus it follows from \vref{lem:neg_energy} that the function $\xi\mapsto \exp(\cesc \xi) \, F_{[u_{\infty}]}(\xi)$ is in $L^1(\rr,\rr)$. As a consequence the second term (the integral) on the right-hand side of inequality \cref{small_fire_tw_aft_intermediary_step} goes to $0$ as $\ell$ goes to $+\infty$. This proves \cref{lem:small_fire_tw_aft}.
\end{proof}
\begin{proof}[Proof of \cref{lem:small_init_fire_aft} (upper bound on initial firewall)]
Let us assume that $\ell$ is greater than or equal to the quantity $\ell_4$ introduced in \cref{lem:small_fire_tw_aft} above. According to the bounds \cref{bound_u_ut_ck_bis}, the function
\begin{equation}
\label{small_init_fire_aft}
\xi\mapsto \psi^{(n,p,\ell)}(\xi,0) \, F(\xi,t_n)
\end{equation}
goes to $0$ as $\xi$ goes to $\pm\infty$, and this convergence is uniform with respect to $n$ and $p$ (provided that $p$ is large enough and that $n$ is greater than or equal to $n_{\min}(p)$). As a consequence, the function \cref{small_init_fire_aft} above approaches the function
\[
\xi\mapsto \psi^{(\infty,\infty,\ell)} (\xi) \, F_{[u_{\infty}]}(\xi)
\]
in $L^1(\rr,\rr)$ as $p$ goes to $+\infty$, uniformly with respect to $n$ greater than or equal to $n_{\min}(p)$. This proves \cref{lem:small_init_fire_aft}. 
\end{proof}
\begin{proof}[Proof of \cref{lem:small_back_flux_fire_aft} (upper bound on back flux for the firewall)]
As for \cref{lem:small_back_flux_fire_bef} on \\ page \pageref{lem:small_back_flux_fire_bef}, the statement follows from the expression of the back flux term in the relaxation inequality \vref{relax_scheme_final}.
\end{proof}
\begin{proof}[Proof of \cref{lem:small_front_flux_fire_aft} (upper bound on front flux for the firewall)]
As for \cref{lem:small_front_flux_fire_bef} on \\ page \pageref{lem:small_front_flux_fire_bef}, the statement follows from the expression of the back flux term in the relaxation inequality \vref{relax_scheme_final}.
\end{proof}
The proof of \vref{prop:dissip_app_zero} (``relaxation'') is complete. Note that, at this stage, hypothesis \textup{(\hyperlink{hypDiscFront}{\hypDiscFrontRef})} has not been used yet. 
\subsection{Convergence}
\label{subsec:convergence}
The end of the proof of \vref{prop:inv_cv} (``invasion implies convergence'') is a straightforward consequence of \cref{prop:dissip_app_zero}, and is very similar to the end of the proof of the main result of \cite{Risler_globCVTravFronts_2008} or of \cite{Risler_globalRelaxation_2016}, thus this \namecref{subsec:convergence} is rather similar to the corresponding sections of those papers. 

Let us call upon the notation $\xEsc(t)$ and $\xesc(t)$ and $\xHom(t)$ introduced in \cref{subsec:inv_cv_def_hyp,subsec:inv_cv_set_pf_cont}. Recall that, according to inequalities \cref{hyp_xEsc_finite,xEsc_xesc_xHom} and to the hypotheses of \cref{prop:inv_cv}, for every nonnegative time $t$, 
\[
-\infty < \xEsc(t) \le \xesc(t) \le \xHom(t) < +\infty
\,.
\]
\begin{lemma}[existence of Escape point and transversality]
\label{lem:xEsc_finite_transv}
The following inequalities hold:
\begin{align}
\label{xesc_minus_xEsc_bounded_from_above}
&\limsup_{t\to+\infty} \xesc(t) - \xEsc(t) < +\infty \,,\\
\text{and}\quad
\label{transversality_at_xEsc}
&\limsup_{t\to+\infty} \Bigl(u\bigl( \xEsc(t), t \bigr)-m\Bigr) \cdot u_x\bigl( \xEsc(t), t \bigr) < 0
\,.
\end{align}
\end{lemma}
\begin{proof}
To prove inequality \cref{xesc_minus_xEsc_bounded_from_above}, let us proceed by contradiction and assume that the converse holds. Then there exists a sequence $(t_n)_{n\in\nn}$ of nonnegative quantities going to $+\infty$ such that $\xesc(t_n) - \xEsc(t_n)$ goes to $+\infty$ as $n$ goes to $+\infty$. Proceeding as in the proof of \vref{lem:small_init_en_bef}, it may be assumed, up to replacing the sequence $(t_n)_{n\in\nn}$ by a subsequence, that there exists a function $\phi_1$ in the set $\Phi_{\cesc}(m)$ of bounded profiles of waves travelling at the speed $\cesc$ and ``invading'' the local minimum $m$ such that, for every positive quantity $L$, 
\[
\norm{\xi\mapsto u \bigl( \xesc(t_n) +\xi, t_n\bigr) - \phi_1(\xi) }_{\ccc^2([-L,L],\rr^d)} \to 0
\quad\text{as}\quad
n\to +\infty
\,.
\]
In addition, it follows from the definition of $\xEsc(\cdot)$ that
\[
\abs{\phi_1(\xi)-m} \le \dEsc(m)
\quad\text{for all}\quad 
\xi\text{ in } \rr
\,,
\]
a contradiction with assertion \cref{item:escape_spatial_asymptotics_tw} of \vref{lem:asympt_behav_tw_2}. Inequality \cref{xesc_minus_xEsc_bounded_from_above} is proved. 

The proof of inequality \cref{transversality_at_xEsc} is similar. Let us proceed by contradiction and assume that the converse holds. Then there exists a sequence $(t'_n)_{n\in\nn}$ of nonnegative quantities going to $+\infty$ such that, for every positive integer $n$,
\begin{equation}
\label{contradiction_barely_transversal_at_Escape_point}
\Bigl(u\bigl( \xEsc(t'_n), t'_n \bigr)-m\Bigr) \cdot u_x\bigl( \xEsc(t'_n), t'_n \bigr) \ge -\frac{1}{n}
,.
\end{equation}
Proceeding as in the proof of \vref{lem:small_init_en_bef},  may  assume, up to replacing the sequence $(t'_n)_{n\in\nn}$ by a subsequence, that there exists a function $\phi_2$ in the set $\PhicNorm{\cesc}(m)$ of (normalized) bounded profiles of waves travelling at the speed $\cesc$ and ``invading'' the local minimum $m$ such that, for every positive quantity $L$, 
\[
\norm{\xi\mapsto u \bigl( \xEsc(t_n) +\xi, t_n\bigr) - \phi_2(\xi) }_{\ccc^2([-L,L],\rr^d)} \to 0
\quad\text{as}\quad
n\to +\infty
\,.
\]
It follows from inequality \cref{contradiction_barely_transversal_at_Escape_point} that
\[
(\phi_2(0)-m) \cdot \phi'_2(0) \ge 0
,,
\]
a contradiction with assertion \cref{item:transv_spatial_asymptotics_tw} of \vref{lem:asympt_behav_tw_2}. Inequality \cref{transversality_at_xEsc} and \cref{lem:xEsc_finite_transv} are proved.
\end{proof}
\begin{lemma}[regularity of Escape point]
\label{regularity_Esc_point}
The function $t\mapsto\xEsc(t)$ is of class $\ccc^1$ on a neighbourhood of $+\infty$ and
\[
\xEsc'(t)\to\cesc
\quad\text{as}\quad
t\to+\infty
\,.
\]
\end{lemma}
\begin{proof}
Let us introduce the function
\[
f:\rr\times[0,+\infty) \to \rr, \quad (x,t)\mapsto \frac{1}{2}\Bigl(\abs{(u(x,t)-m}^2-\dEsc(m)^2\Bigr)
\,.
\]
According to the regularity of the solution (see \cref{subsec:glob_exist}), this function $f$ is of class $\ccc^1$, and, for every nonnegative time $t$, the quantity $f\bigl(\xEsc(t),t\bigr)$ is equal to $0$ and 
\[
\partial_x f\bigl(\xEsc(t), t \bigr) = \Bigl(u\bigl(\xEsc(t), t\bigr)-m\Bigr)  \cdot u_x \bigl(\xEsc(t), t\bigr)
\,.
\]
Thus, according to inequality \cref{transversality_at_xEsc} of \cref{lem:xEsc_finite_transv}, for every large enough positive time $t$,
\[
\partial_x f\bigl(\xEsc(t), t \bigr) < 0
\,.
\]
Thus it follows from the Implicit Function Theorem that the function $x\mapsto \xEsc(t)$ is of class $\ccc^1$ on a neighbourhood of $+\infty$, and that, for every large enough positive time $t$,
\begin{equation}
\label{xEsc_prime_of_t}
\xEsc'(t) 
= -\frac{\partial_t f \bigl(\xEsc(t),t\bigr)}{\partial_x f \bigl(\xEsc(t),t\bigr)}
= -\frac{\Bigl(u\bigl(\xEsc(t), t\bigr)-m\Bigr) \cdot u_t \bigl(\xEsc(t), t\bigr)}{\Bigl(u\bigl(\xEsc(t), t\bigr)-m\Bigr) \cdot u_x \bigl(\xEsc(t), t\bigr)}
\,.
\end{equation}
According to inequality \cref{transversality_at_xEsc} of \cref{lem:xEsc_finite_transv}, the denominator of this expression remains bounded away from $0$ as time goes to plus infinity. On the other hand, according to inequality \cref{xesc_minus_xEsc_bounded_from_above} of the same lemma and to \vref{prop:dissip_app_zero} and to the bounds \vref{bound_u_ut_ck} for the solution, 
\[
u_t\bigl(\xEsc(t)+\xi,t\bigr)+\cesc u_x\bigl(\xEsc(t)+\xi,t\bigr) \to 0
\quad\text{as}\quad t\to +\infty
\,.
\]
Thus it follows from \cref{xEsc_prime_of_t} that $\xEsc'(t)$ goes to $\cEsc$ as time goes to $+\infty$. \Cref{regularity_Esc_point} is proved.
\end{proof}
The next lemma is the only place throughout the proof of \cref{prop:inv_cv} where hypothesis \textup{(\hyperlink{hypDiscFront}{\hypDiscFrontRef})} is required. 
\begin{lemma}[convergence around Escape point]
\label{lem:cv}
There exists a function $\phi$ in the set $\PhicNorm{\cesc}(m)$ of (normalized) bounded profiles of waves travelling at the speed $\cesc$ and invading the equilibrium $m$ such that, for every positive quantity $L$, 
\[
\sup_{x\in\bigl[\xEsc(t)-L,\xEsc(t)+L\bigr]} \abs{u(x,t)-\phi\bigl(x-\xEsc(t)\bigr)} \to 0 
\quad\text{as}\quad 
t\to+\infty 
\,.
\]
In particular, the set $\PhicNorm{\cesc}(m)$ is nonempty.
\end{lemma}
\begin{proof}
Take a sequence $(t_n)_{n\in\nn}$ of positive times going to $+\infty$ as $n$ goes to $+\infty$. Proceeding as in the proof of \vref{lem:small_init_en_bef}, it may be assumed, up to replacing the sequence $(t_n)_{n\in\nn}$ by a subsequence, that there exists a function $\phi$ in the set $\Phi_{\cesc}(m)$ of bounded profiles of waves travelling at the speed $\cesc$ and ``invading'' the local minimum $m$ such that, for every positive quantity $L$, 
\[
\norm{\xi\mapsto u \bigl( \xEsc(t_n) +\xi, t_n\bigr) - \phi(\xi) }_{\ccc^2([-L,L],\rr^d)} \to 0
\quad\text{as}\quad
n\to +\infty
\,.
\]
According to the definition of $\xEsc(\cdot)$, 
\[
\abs{\phi(0)-m} = \dEsc(m)
\quad\text{and}\quad
\abs{\phi(\xi)-m}\le \dEsc(m)
\quad\text{for all}\quad
\xi \text{ in } [0,+\infty)
\,,
\]
thus according to assertion \cref{item:closer_spatial_asymptotics_tw} of \vref{lem:asympt_behav_tw_2}, it follows that $\phi$ actually belongs to the set $\PhicNorm{\cesc}(m)$ of ``normalized'' profiles. 

Let $\mathcal{L}$ denote the set of all possible limits (in the sense of uniform convergence on compact subsets of $\rr$) of sequences of maps
\[
\xi\mapsto u \bigl( \xEsc(t'_n) +\xi, t'_n\bigr)
\]
for all possible sequences $(t'_n)_{n\in\nn}$ such that $t'_n$ goes to $+\infty$ as $n$ goes to $+\infty$. This set $\mathcal{L}$ is included in the set $\PhicNorm{\cesc}(m)$, and, because the semi-flow of system \cref{init_syst} is continuous on $X$, this set $\mathcal{L}$ is a continuum (a compact connected subset) of $X$. 

Since on the other hand --- according to hypothesis \textup{(\hyperlink{hypDiscFront}{\hypDiscFrontRef})} --- the set $\PhicNorm{\cesc}(m)$ is totally disconnected in $X$, this set $\mathcal{L}$ must actually be reduced to the singleton $\{\phi\}$. \Cref{lem:cv} is proved. 
\end{proof}
\begin{lemma}[convergence up to $\xHom(t)$]
\label{lem:cv_right_of_front}
For every positive quantity $L$, 
\[
\sup_{x\in[\xEsc(t)-L,\xHom(t)]} \abs{u(x,t)-\phi\bigl(x-\xEsc(t)\bigr)} \to 0 
\quad\text{as}\quad 
t\to+\infty 
\,.
\]
\end{lemma}
\begin{proof}
Let us proceed by contradiction and assume that the converse holds. Then, according to \cref{lem:cv} above, there exists a positive quantity $\varepsilon$, and sequences $(t_n)_{n\in\nn}$ and $(\xi_n)_{n\in\nn}$ of real quantities both going to $+\infty$ as $n$ goes to $+\infty$, such that, for every positive integer $n$, 
\[
\abs{u\bigl( \xEsc(t_n)+\xi_n,t_n \bigr)-\phi(\xi_n)} \ge \varepsilon
\,,
\]
and thus, if $n$ is large enough, 
\[
\abs{u\bigl( \xEsc(t_n)+\xi_n,t_n \bigr)-m} \ge \varepsilon
\,.
\]
Using the notation $\fff_0(\cdot,\cdot)$ of \cref{subsec:def_fire_zero}, this yields the existence of a positive quantity $\varepsilon'$ such that, if $n$ is large enough, 
\begin{equation}
\label{fff_zero_larger_than_small_positive_quantity}
\fff_0 \bigl( \xEsc(t_n)+\xi_n,t_n \bigr) \ge \varepsilon'
\,.
\end{equation}
According to hypothesis \textup{(\hyperlink{hypHomRight}{\hypHomRightRef})} and to the bounds \cref{bound_u_ut_ck_bis} on the solution,
\[
\xHom(t_n) - \bigl( \xEsc(t_n) + \xi_n\bigr) \to +\infty
\quad\text{as}\quad
n\to +\infty
\,.
\]
According to inequality \vref{dt_fire} about the time derivative of $\fff_0$ and to the fact that both derivatives $\xEsc'(t)$ and $\xHom'(t)$ have finite limits as time goes to $+\infty$, this shows that the function:
\[
[0,t_n]\to \rr, 
\quad
s\mapsto \fff_0 \bigl( \xEsc(t_n)+\xi_n,t_n - s\bigr)
\]
is increasing at an exponential rate on an arbitrarily large interval starting from $0$, provided that $n$ is large enough. In view of \cref{fff_zero_larger_than_small_positive_quantity} this contradicts the bounds \cref{bound_u_ut_ck_bis} on the solution. 
\end{proof}
\subsection{Homogeneous point behind the travelling front}
\label{subsec:hom_behind_front}
According to hypothesis \textup{(\hyperlink{hypOnlyBist}{\hypOnlyBistRef})}, the function $\phi$ is actually the profile of a bistable travelling front; in other words, there exists $\mNext$ in $\mmm$ such that
\begin{equation}
\label{phi_converges_towards_mnext_at_minus_infinity}
\phi(\xi)\to\mNext
\quad\text{as}\quad
\xi\to-\infty
\,.
\end{equation}
The following lemma completes the proof of \cref{prop:inv_cv} (``invasion implies convergence''). 
\begin{lemma}[``next'' homogeneous point behind the front]
\label{lem:next_hom_point}
There exists a $\rr$-valued function $\xHomNext$, defined and of class $\ccc^1$ on a neighbourhood of $+\infty$, such that the following limits hold as time goes to $+\infty$:
\[
\begin{aligned}
& \xEsc(t) - \xHomNext(t) \to +\infty
\quad\text{and}\quad
\xHomNext'(t) \to \cesc \\
\text{and}\quad
& \sup_{x\in[\xHomNext(t),\ \xHom(t)]} \abs{u(x,t)-\phi\bigl(x-\xEsc(t)\bigr)} \to 0 
\,,
\end{aligned}
\]
and, for every positive quantity $L$, 
\[
\sup_{\xi\in[-L,L]}\abs{u\bigl(\xHomNext(t) + \xi, t\bigr) - \mNext } \to 0
\,.
\] 
\end{lemma}
\begin{proof}
Let us introduce the sequence of times $(t_n)_{n\in\nn}$ defined as follows: $t_0$ is positive and large enough so that the function $t\mapsto \xEsc(t)$ is defined and of class $\ccc^1$ on $[t_0,+\infty)$, and, for every positive integer $n$, 
\[
t_n = \max \left( t_{n-1} + n, \sup \left\{ t \text{ in } [0,+\infty) : \sup_{\xi\in [-2n,0]} \abs{u\bigl(\xesc(t) + \xi, t\bigr)- \phi(\xi)}\ge \frac{1}{n}\right\}\right)
\]
(the key point being that, according to \cref{lem:cv} above, this quantity $t_n$ is finite). 
Let $\chi$ denote a smooth function $\rr\to\rr$ satisfying
\[
\chi\equiv 0 \text{ on } (-\infty,0]
\quad\text{and}\quad
\chi\equiv 1 \text{ on } [1,+\infty)
\quad\text{and}\quad
0 \le \chi \le 1 \text{ and }\chi' \ge 0\text{ on }[0,1]
\,,
\]
\begin{figure}[!htbp]
\centering
\includegraphics[width=0.85\textwidth]{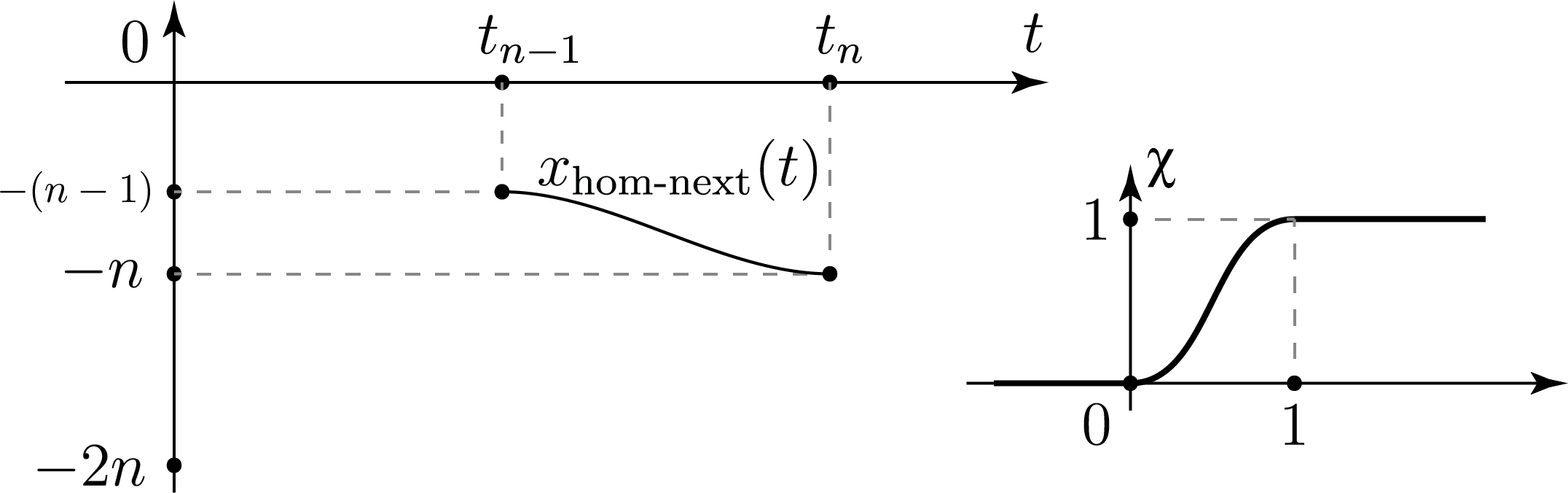}
\caption{Illustration of the definition of the function $\xHomNext(\cdot)$.}
\label{fig:hom_behind_front}
\end{figure}
see \cref{fig:hom_behind_front}; and let us define the function $\xHomNext:[t_0,+\infty)\to\rr$ by
\[
\xHomNext(t) = \xEsc(t) - n - \chi\left(\frac{t-t_{n-1}}{t_{n}-t_{n-1}}\right) 
\quad\text{for $n$ in $\nn^*$ and $t$ in $[t_{n-1},t_{n}]$.}
\]
This function is of class $\ccc^1$ on $[t_0,+\infty)$ and the other conclusions of \cref{lem:next_hom_point} follow from the limit \cref{phi_converges_towards_mnext_at_minus_infinity} and the definition of $\xHomNext(\cdot)$. 
\end{proof}
The proof of \cref{prop:inv_cv} is complete. 
\section{No invasion implies relaxation}
\label{sec:no_inv_implies_relax}
As everywhere else, let us consider a function $V$ in $\ccc^2(\rr^d,\rr)$ satisfying the coercivity hypothesis \cref{hyp_coerc}. 
\subsection{Definitions and hypotheses}
\label{subsec:def_hyp_dichot}
Let us consider two points $m_-$ and $m_+$ in $\mmm$ and a solution $(x,t)\mapsto u(x,t)$ of system \cref{init_syst} defined on $\rr\times[0,+\infty)$. 
Without assuming that this solution is bistable, let us make the following hypothesis \textup{(\hyperlink{hypHom}{\hypHomRef})}, which is similar to hypothesis \textup{(\hyperlink{hypHomRight}{\hypHomRightRef})} made in \cref{sec:inv_impl_cv} (``invasion implies convergence''), but this time both to the right and to the left in space (see \cref{fig:inv_relax_dichot}).
\begin{figure}[!htbp]
\centering
\includegraphics[width=\textwidth]{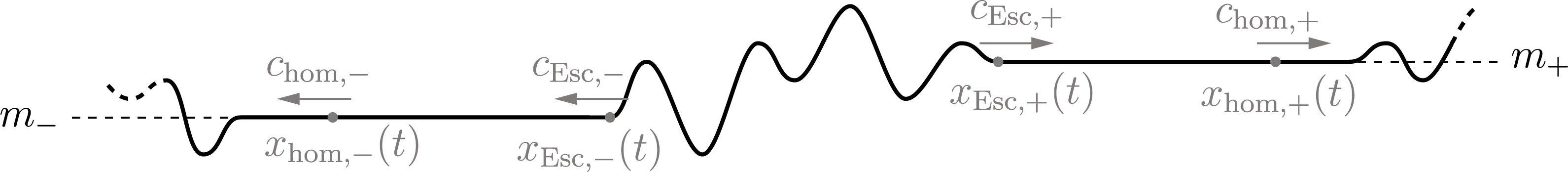}
\caption{Illustration of hypothesis \textup{(\hypHomRef)} and of \cref{prop:relax}.}
\label{fig:inv_relax_dichot}
\end{figure}
\begin{description}
\item[\hypHomLabel]\hypertarget{hypHom} There exist a positive quantity $\cHomPlus$ and a negative quantity $\cHomMinus$ and $\ccc^1$-functions
\[
\begin{aligned}
& \xHomMinus:[0,+\infty)\to\rr
\quad\text{satisfying}\quad
\xHomMinus'(t)\to \cHomMinus 
\quad\text{as}\quad
t\to +\infty \\
\text{and}\quad
& \xHomPlus:[0,+\infty)\to\rr
\quad\text{satisfying}\quad
\xHomPlus'(t)\to \cHomPlus 
\quad\text{as}\quad
t\to +\infty 
\end{aligned}
\]
such that, for every positive quantity $L$, 
\[
\begin{aligned}
&\sup_{\xi\in[-L,L]}\abs{u \bigl( \xHomMinus(t) + \xi, t \bigr) - m_-} \to 0
\quad\text{as}\quad
t\to +\infty \\
\text{and}\quad
&\sup_{\xi\in[-L,L]}\abs{u \bigl( \xHomPlus(t) + \xi, t \bigr) - m_+} \to 0
\quad\text{as}\quad
t\to +\infty 
\,.
\end{aligned}
\]
\end{description}
For every $t$ in $[0,+\infty)$, let us denote by $\xEscMinus(t)$ the infimum of the set
\[
\bigl\{ x\in\rr : \xHomMinus(t) \le x \le \xHomPlus(t) \quad\text{and}\quad \abs{u(x,t)-m_-}=\dEsc(m_-) \bigr\}
\]
(with the convention that $\xEscMinus(t)$ equals $+\infty$ if this set is empty), and let us denote by $\xEscPlus(t)$ the supremum of the set
\[
\bigl\{ x\in\rr : \xHomMinus(t) \le x \le \xHomPlus(t) \quad\text{and}\quad \abs{u(x,t)-m_+}=\dEsc(m_+) \bigr\}
\]
(with the convention that $\xEscPlus(t)$ equals $-\infty$ if this set is empty). Let 
\[
\cEscMinus = \liminf_{t\to+\infty} \frac{\xEscMinus(t)}{t}
\quad\text{and}\quad
\cEscPlus = \limsup_{t\to+\infty} \frac{\xEscPlus(t)}{t}
\,,
\]
see \cref{fig:inv_relax_dichot}. It follows from the definitions of $\xEscMinus(t)$ and $\xEscPlus(t)$ that, for all $t$ in $[0,+\infty)$, 
\[
\xHomMinus(t) \le \xEscMinus(t)
\quad\text{and}\quad
\xEscPlus(t) \le \xHomPlus(t)
\]
thus
\[
\cHomMinus \le \cEscMinus
\quad\text{and}\quad
\cEscPlus \le \cHomPlus
\,.
\]
If the quantity $\cEscPlus$ was positive or if the quantity $\cEscMinus$ was negative, this would mean that the corresponding equilibrium is ``invaded'' at a nonzero mean speed, a situation already studied in \cref{sec:inv_impl_cv}. Let us introduce the following (converse) ``no invasion'' hypothesis:
\begin{description}
\item[\hypNoInvLabel]\hypertarget{hypNoInv} The following inequalities hold:
\[
0\le \cEscMinus
\quad\text{and}\quad
\cEscPlus \le 0 
\,.
\]
\end{description}
\subsection{Statement}
\label{subsec:statement_dichot}
The aim of \cref{sec:no_inv_implies_relax} is to prove the following proposition. 
\begin{proposition}[no invasion implies relaxation]
\label{prop:relax}
Assume that $V$ satisfies hypothesis \\ \cref{hyp_coerc} and that the solution $(x,t)\mapsto u(x,t)$ under consideration satisfies hypotheses \textup{(\hyperlink{hypHom}{\hypHomRef})} and \textup{(\hyperlink{hypNoInv}{\hypNoInvRef})}. Then the following conclusions hold. 
\begin{enumerate}
\item The quantities $V(m_-)$ and $V(m_+)$ are equal.
\label{item:Vm_minus_equals_Vm_plus}
\item There exists a nonnegative quantity $\eeeResAsympt[u]$ (``residual asymptotic energy'') such that, for all quantities $c_-$ in $(\cHomMinus,0)$ and $c_+$ in $(0,\cHomPlus)$,
\begin{equation}
\label{asympt_energy_prop_relax}
\int_{c_-t}^{c_+t}\left(\frac{1}{2}u_x(x,t)^2 + V\bigl(u(x,t)-V(m_\pm)\right) \to \eeeResAsympt[u]
\quad\text{as}\quad
t\to+\infty
\,.
\end{equation}
\label{item:asympt_energy_prop_relax}
\item The following limits hold as time goes to $+\infty$:
\begin{equation}
\label{dissip_goes_to_zero_prop_relax}
\sup_{x\in[\xHomMinus(t),\xHomPlus(t)]} \abs{u_t(x,t)} \to 0
\,,
\end{equation}
and, for every quantity $\varepsilon$ which is positive and smaller than $\min\left(\abs{\cHomMinus},\cHomPlus\right)$, 
\begin{equation}
\label{sol_goes_to_m_plus_minus_prop_relax}
\sup_{x\in[\xHomMinus(t),-\varepsilon t]}\abs{u(x,t)-m_-} \to 0
\quad\text{and}\quad
\sup_{x\in[\varepsilon t,\xHomPlus(t)]}\abs{u(x,t)-m_+} \to 0
\,.
\end{equation}
\label{item:zero_limits_prop_relax}

\end{enumerate}
\end{proposition}
\subsection{Relaxation scheme in a standing or almost standing frame}
\label{subsec:relax_sc_stand}
\subsubsection{Principle}
The aim of this \namecref{subsec:relax_sc_stand} is to set up an appropriate relaxation scheme in a standing or almost standing frame. This means defining an appropriate localized energy and controlling the ``flux'' terms occurring in the time derivative of that localized energy. The argument will be quite similar to that of \cref{subsec:relax_sch_tr_fr} (the relaxation scheme in the travelling frame), the main difference being that $c$ is now either equal or close to $0$, and as a consequence the weight function for the localized energy will be defined with a cut-off on the right and another on the left, instead of a single one; accordingly firewall functions will be introduced to control the fluxes along each of these cutoffs.

Let us keep the notation and assumptions of \cref{subsec:def_hyp_dichot}, and let us assume that hypotheses \cref{hyp_coerc} and \textup{(\hyperlink{hypHom}{\hypHomRef})} and \textup{(\hyperlink{hypNoInv}{\hypNoInvRef})} of \cref{prop:relax} hold. According to \vref{prop:attr_ball}, it may be assumed (without loss of generality, up to changing the origin of time) that, for all $t$ in $[0,+\infty)$, 
\begin{align}
\label{att_ball_dichot}
\norm{x\mapsto u(x,t)}_{\Linfty} &\le \Rattinfty \\
\text{and}\qquad 
\label{att_ball_X_dichot}
\norm{x\mapsto u(x,t)}_{X} &\le \RattX
\,.
\end{align}
\subsubsection{Notation for the travelling frame}
%
Let $c$ denote a real quantity. By contrast with \cref{subsec:relax_sch_tr_fr}, the other parameters --- namely the initial time $\tInit$, the initial position of origin of travelling frame $\xInit$, and the initial position of the cut-off $\initCut$ --- are not required here; or in other words they are chosen equal to $0$). 
The relaxation scheme will be applied in the next \cref{subsec:low_bd_en} for $c$ very close or equal to $0$. 

Let us introduce the function $(\xi,t)\mapsto v(\xi,t)$ defined on $\rr\times[0,+\infty)$ by
\[
v(\xi,t) = u(x,t)
\quad\text{for}\quad
x=ct+\xi
\,.
\]
This function is a solution of the differential system
\[
v_t - c v_\xi = -\nabla V (v) + v_{\xi\xi}
\,.
\]
\subsubsection{Choice of the parameters and conditions on the speed \texorpdfstring{$c$}{c}}
%
A localized energy and two firewall functions associated with this solution will now be introduced. Let us call upon the quantity $\coeffEnZero$ involved in \cref{def_weight_en_V_ddag}, let us denote by $\kappa_0(m_-)$ and by $\kappa_0(m_+)$ the quantities defined in \vref{def_kappaZero} for the two minimum points $m_-$ and $m_+$, and let 
\[
\kappa_0 = \min\bigl(\kappa_0(m_-),\kappa_0(m_+)\bigr)
\quad\text{and}\quad
\eigVmin = \min\bigl(\eigVmin(m_-),\eigVmin(m_+)\bigr)
\,.
\]
Let $\cCutZero$ (speed of the cut-off point in the travelling frame) be a positive quantity, small enough so that the following inequalities hold:
\begin{equation}
\label{conditions_on_coeffEnZero_and_kappaZero_and_cutZero_relax}
\coeffEnZero\kappa_0\Bigl(\frac{\cCutZero}{2} + \frac{\kappa_0}{4}\Bigr) \le \frac{1}{2}
\quad\text{and}\quad
\coeffEnZero \, \cCutZero\, \kappa_0 \le \frac{1}{4} 
\quad\text{and}\quad
\frac{\kappa_0(\cCutZero+\kappa_0)}{2} \le \frac{\eigVmin}{16} 
\end{equation}
(compare with conditions \vref{conditions_kappaZero_lab} on $\kappa_0$ and \vref{conditions_kappa_cCut_coeffEn} on $\kappa$).
For instance, the quantity $\cCutZero$ may be chosen as
\[
\cCutZero = \min\biggl(\frac{1}{4\sqrt{\coeffEnZero}},\frac{\sqrt{\eigVmin}}{4}, \frac{\cHomPlus}{2},\frac{\abs{\cHomMinus}}{2}\biggr)
\,.
\]
Let us make on the parameter $c$ the following hypotheses (see comments below):
\begin{equation}
\label{hyp_param_relax}
\abs{c}\le\frac{\kappa_0}{10}
\quad\text{and}\quad
\abs{c}\le \frac{\sqrt{\eigVmin}}{4}
\quad\text{and}\quad
\abs{c}\le\frac{\cCutZero}{6}
\,.
\end{equation}
 According to hypotheses \textup{(\hyperlink{hypHom}{\hypHomRef})} and \textup{(\hyperlink{hypNoInv}{\hypNoInvRef})} and to the value of the quantity $\cCutZero$ chosen above, there exists a nonnegative time $T$ such that, for every time $t$ greater than or equal to $T$,
\begin{equation}
\label{hyp_pos_relax}
\begin{aligned}
\xHomMinus(t)&\le -\frac{11}{6}\cCutZero t
&\quad&\text{and}\quad&
-\frac{1}{6}\cCutZero t &\le \xEscMinus(t) \\
\text{and}\quad
\xEscPlus(t)&\le \frac{1}{6}\cCutZero t 
&\quad&\text{and}\quad&
\frac{11}{6}\cCutZero t &\le \xHomPlus(t)
\,.
\end{aligned}
\end{equation}
Let us briefly comment hypotheses \cref{hyp_param_relax,hyp_pos_relax}. 
\begin{itemize}
\item The relaxation scheme below will be applied either for $c$ equals $0$, or for $c$ very close to $0$. 
\item The choice of the value $1/10$ as upper bound of the quotient $\abs{c}/\kappa_0$ in \cref{hyp_param_relax} will turn out to be convenient to prove \vref{lem:der_fire_dichot_next} (that is, more precisely, to control the pollution terms involved in the time derivative of the firewall functions that will be defined below). The same is true concerning the choice of the factors $\pm 1/6$ and $\pm 11/6$ in \cref{hyp_pos_relax}. 
\end{itemize}
\subsubsection{Notation ``\texorpdfstring{$\pm$}{plus or minus}''}
Let us adopt, for the remaining of this \cref{sec:no_inv_implies_relax}, the following convention: the symbol ``$\pm$'' denotes one the the signs ``$+$'' and ``$-$'', this sign remaining the same along a whole expression, or an equality/inequality between two expressions, or a definition. 
\subsubsection{Normalized potential}
Let us introduce the ``normalized'' potential $V^\ddag: \rr^d\to\rr$, $v\mapsto V^\ddag(v)$ defined as
\begin{equation}
\label{def_V_ddag}
V^\ddag (v) = V(v) - \max\bigl(V(m_-),V(m_+)\bigr)
\,.
\end{equation}
As a consequence $\max\bigl(V^\ddag(m_-),V^\ddag(m_+)\bigr) = 0$, and $\nabla V$ and $\nabla V^\ddag$ are equal. With the convention above, it follows from inequality \cref{def_weight_en} satisfied by $\coeffEnZero$ that, for all $v$ in $\rr^d$, 
\begin{equation}
\label{def_weight_en_V_ddag}
\coeffEnZero \bigl(V^\ddag(v) - V^\ddag(m_\pm)\bigr)  +  \frac{1}{4}(v-m_\pm)^2\ge 0
\,,
\end{equation}
and it follows from inequalities \cref{v_nablaV_controls_square_around_loc_min,v_nablaV_controls_pot_around_loc_min} that, for all $v$ in $\rr^d$ satisfying $\abs{v-m_\pm}\le\dEsc(m_\pm)$, 
\begin{align}
\label{v_nablaV_controls_square_around_loc_min_ddag}
(v-m_\pm)\cdot \nabla V^\ddag(v) &\ge \frac{\eigVmin}{2} (v-m_\pm)^2 \,, \\
\text{and}\qquad
\label{v_nablaV_controls_pot_around_loc_min_ddag}
(v-m_\pm)\cdot \nabla V^\ddag(v) &\ge V^\ddag(v) - V^\ddag(m_\pm)
\,.
\end{align}
\subsubsection{Localized energy}
\label{subsubsec:en_dichot}
For every time $t$, let us introduce the three intervals
\[
\begin{aligned}
\iLeft(t) &= ( - \infty , -\cCutZero t ] \,, \\
\text{and}\qquad
\iMain(t) &= [-\cCutZero t, \cCutZero t] \,, \\
\text{and}\qquad
\iRight(t) &= [ \cCutZero t , +\infty) \,,
\end{aligned}
\]
and let us introduce the functions $\chi_0(\xi,t)$ and $\chi(\xi,t)$ (weight function for the localized energy) defined on $\rr\times[0,+\infty)$ by
\[
\chi_0(\xi,t)=
\left\{
\begin{aligned}
&1
&& \quad\text{if}\quad \xi\in\iMain(t)  \,, \\
& \exp\bigl(-\kappa_0 (\abs{\xi}-\cCutZero t)\bigr) 
&& \quad\text{if}\quad \xi\not\in\iMain(t)  \,,
\end{aligned}
\right.
\quad\text{and}\quad
\chi(\xi,t) = e^{c\xi} \chi_0(\xi,t)
\,,
\]
see \cref{fig:graph_weight_zero_dichot,fig:graph_weight_dichot}.
\begin{figure}[!htbp]
\centering
\includegraphics[width=.8\textwidth]{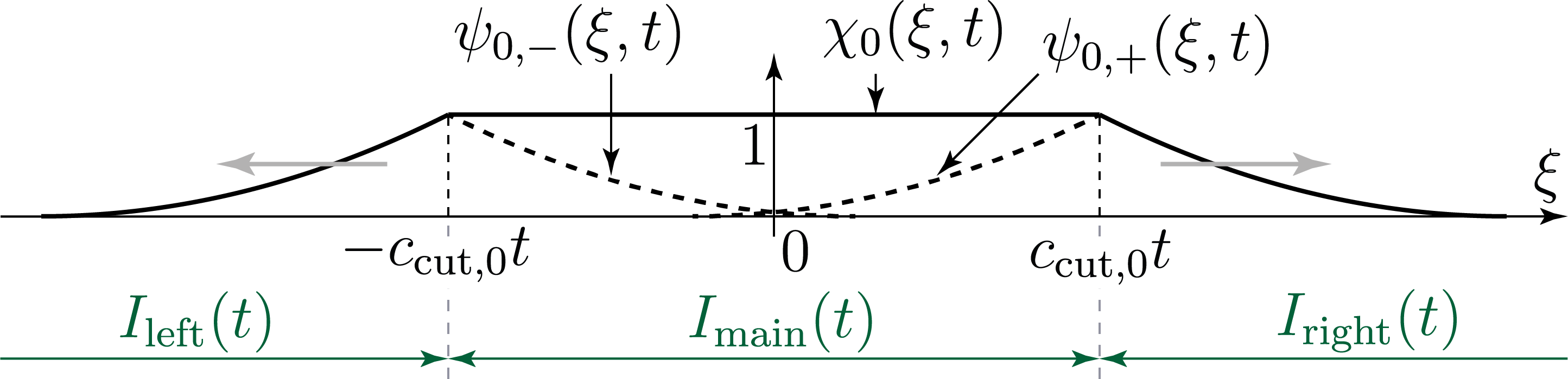}
\caption{Graphs of functions $\xi\mapsto\chi_0(\xi,t)$ and $\xi\mapsto\psi_{0,+}(\xi,t)$ and $\xi\mapsto\psi_{0,-}(\xi,t)$.}
\label{fig:graph_weight_zero_dichot}
\end{figure}
\begin{figure}[!htbp]
\centering
\includegraphics[width=\textwidth]{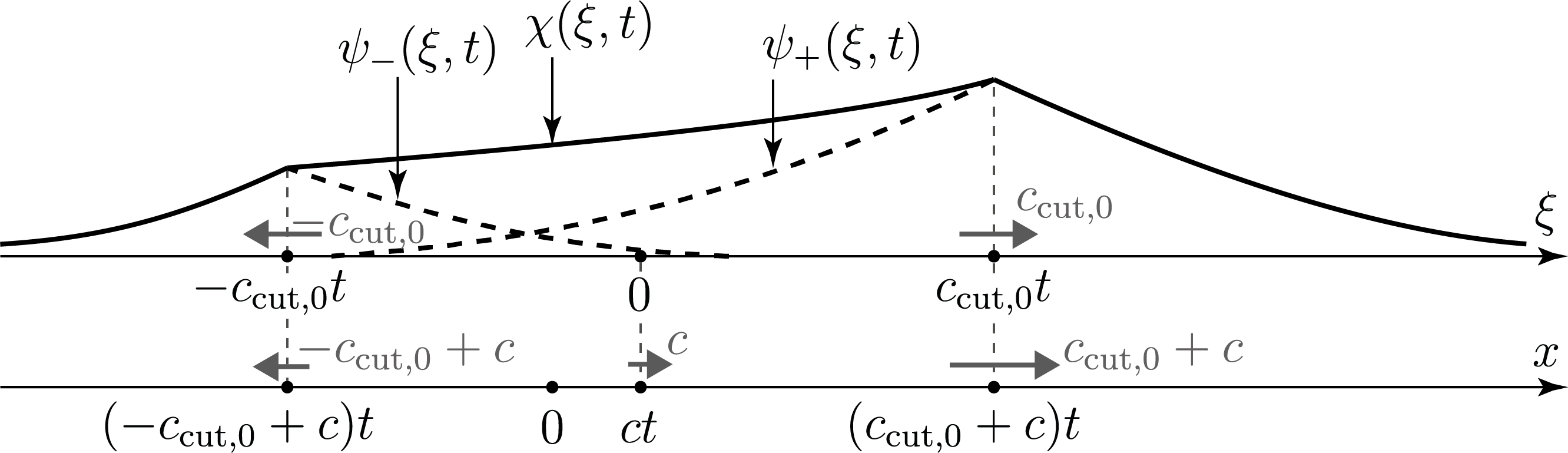}
\caption{Graphs of the weight functions $\xi\mapsto\chi(\xi,t)$ and $\xi\mapsto\psi_+(\xi,t)$ and $\xi\mapsto\psi_-(\xi,t)$ (case of a positive speed $c$).}
\label{fig:graph_weight_dichot}
\end{figure}
For all $t$ in $[0,+\infty)$, let us define the ``energy'' $\eee(t)$ by
\[
\eee(t) = \int_{\rr} \chi(\xi,t)\left(\frac{1}{2}v_\xi(\xi,t)^2 + V^\ddag\bigl( v(\xi,t)\bigr)\right) \, d\xi 
\,.
\]
The notation $\chi$ and $\eee$ is the same as in \cref{subsubsec:def_loc_en} but the definitions above are slightly different from those introduced in \cref{subsubsec:def_loc_en}. 
\begin{remark}
There is more than one possible choice (with equivalent outcomes) for the expressions of $\chi(\xi,t)$ and of the other weight functions $\psi_\pm(\xi,t)$ that will be defined in \cref{def_psi_plus_minus} below. The chosen expressions are closer to those of \vref{subsubsec:def_firewall_lab_frame} (laboratory frame) than to those of \vref{subsec:relax_sch_tr_fr} (travelling frame). In the context of this paper this is really a matter of taste. However, it seems that this choice extends more naturally to higher space dimension, \cite{Risler_noInvasionCaseHigherSpace_2020}.
\end{remark}
\subsubsection{Time derivative of localized energy}
\label{subsubsec:time_der_loc_en_dichot}
For all $t$ in $[0,+\infty)$, let us define the ``dissipation'' function by
\begin{equation}
\label{def_dissip_dichot}
\ddd(t) = \int_{\rr} \chi(\xi,t)\, v_t(\xi,t)^2 \, d\xi
\,.
\end{equation}
\begin{lemma}[time derivative of the localized energy]
\label{lem:ds_en_trav_f_prel_relax}
For all $t$ in $[0,+\infty)$, 
\begin{equation}
\label{ds_en_trav_f_prel_relax}
\eee'(t) \le -\frac{1}{2} \ddd(t) + \kappa_0 \int_{\rr\setminus\iMain(t)} \chi \Bigl( \frac{\cCutZero+\kappa_0}{2}v_\xi^2 + \cCutZero V^\ddag(v) \Bigr) \, d\xi
\,.
\end{equation}
\end{lemma}
\begin{proof}
For all $t$ in $[0,+\infty)$, it follows from expression \vref{ddt_loc_en} (time derivative of a localized energy) that
\begin{equation}
\label{ds_en_trav_f_prel_relax_proof_1}
\eee'(t) = - \ddd(t)  + \int_{\rr} \biggl( \chi_t\Bigl( \frac{1}{2}v_\xi^2 + V^\ddag(v)\Bigr) + (c\chi - \chi_\xi) v_\xi \cdot v_t \biggr) \, d\xi
\,.
\end{equation}
It follows from the definition of $\chi$ that:
\[
\chi_t(\xi,t)=e^{c\xi}\partial_t \chi_0(\xi,t) = 
\left\{
\begin{aligned}
& 0 
&& \quad\text{if}\quad \xi\in\iMain(t) \,, \\
&\kappa_0\, \cCutZero\,  \chi(\xi,t) 
&& \quad\text{if}\quad \xi\not\in\iMain(t) \,, 
\end{aligned}
\right.
\]
and 
\[
(c\chi - \chi_\xi) (\xi,t)= - e^{c\xi} \partial_\xi \chi_0 (\xi,t) = 
\left\{
\begin{aligned}
&0 
&& \quad\text{if}\quad \xi\in\iMain(t) \,, \\
&\sgn(\xi)\, \kappa_0\, \chi(\xi,t) 
&& \quad\text{if}\quad \xi\not\in\iMain(t) \,,
\end{aligned}
\right.
\]
where $\sgn(\xi)=\xi/\abs{\xi}$ denotes the sign of $\xi$. Thus, it follows from \cref{ds_en_trav_f_prel_relax_proof_1} that
\begin{equation}
\label{ds_en_trav_f_prel_relax_proof_2}
\eee'(t) =- \ddd(t)  + \kappa_0\int_{\rr\setminus\iMain(t)} \chi \biggl( \cCutZero\Bigl( \frac{1}{2}v_\xi^2 + V^\ddag(v)\Bigr) + \sgn(\xi)\, v_\xi \cdot v_t \biggr) \, d\xi 
\,.
\end{equation}
Using the inequality
\[
\sgn(\xi)\, \kappa_0 \, v_\xi \cdot v_t \le \frac{1}{2}v_t^2 + \frac{\kappa_0^2}{2} v_\xi^2
\,,
\]
inequality \cref{ds_en_trav_f_prel_relax} follows from equality \cref{ds_en_trav_f_prel_relax_proof_2}. \Cref{lem:ds_en_trav_f_prel_relax} is proved. 
\end{proof}
\subsubsection{Firewall functions}
\label{subsubsec:def_fire_relax}
The next task is to define two firewall functions to control the two last terms of the right-hand side of inequality \cref{ds_en_trav_f_prel_relax} above. Let us introduce the functions $\psi_{0,-}(\xi,t)$ and $\psi_{0,+}(\xi,t)$ and $\psi_-(\xi,t)$ and $\psi_+(\xi,t)$ (weight functions for those firewall functions) defined as (see \cref{fig:graph_weight_zero_dichot,fig:graph_weight_dichot} and the notation of \vref{def_firewall_lab_frame}):
\[
\begin{aligned}
\psi_{0,-}(\xi,t) &= T_{-\cCutZero t}\psi_0(x,t) = \exp\bigl(-\kappa_0\abs{\xi+\cCutZero t}\bigr)
\,,
\\
\text{and}\quad
\psi_{0,+}(\xi,t) &= T_{\cCutZero t}\psi_0(x,t) = \exp\bigl(-\kappa_0\abs{\xi-\cCutZero t}\bigr)
\,,
\end{aligned}
\]
and 
\begin{equation}
\label{def_psi_plus_minus}
\psi_{-}(\xi,t) = e^{c\xi}\psi_{0,-}(\xi,t)
\quad\text{and}\quad
\psi_{+}(\xi,t) = e^{c\xi}\psi_{0,+}(\xi,t)
\,;
\end{equation}
observe that 
\[
\chi(\xi,t) = \psi_{-}(\xi,t)
\quad\text{for}\quad
\xi\in\iLeft(t)
\quad\text{and}\quad
\chi(\xi,t) = \psi_{+}(\xi,t)
\quad\text{for}\quad
\xi\in\iRight(t)
\,.
\]
For every $t$ in $[0,+\infty)$, let us define the ``firewalls'' $\fff_-(t)$ and $\fff_+(t)$ as
\begin{equation}
\label{def_firewalls_relax}
\fff_\pm(t) = \int_{\rr} \psi_\pm(\xi,t)\biggl(\coeffEnZero \Bigl( \frac{1}{2}v_\xi(\xi,t)^2 + V^\ddag\bigl(v(\xi,t) - V^\ddag(m_\pm)\bigr) \Bigr) + \frac{1}{2}\bigl(v(\xi,t)-m_\pm\bigr)^2 \biggr) \, d\xi
\,.
\end{equation}
\subsubsection{Energy decrease up to firewalls}
\label{subsubsec:decrease_up_to_firewall_relax}
\begin{lemma}[energy decrease up to firewalls]
\label{lem:ds_en_trav_f_inprogress_relax}
There exists a quantity $\KEFZero$, depending only on $V$ and $m_-$ and $m_+$, such that, for every nonnegative time $t$, 
\begin{equation}
\label{ds_en_trav_f_inprogress_relax}
\eee'(t)\le -\frac{1}{2} \ddd(t) + \KEFZero\bigl( \fff_-(t) + \fff_+(t) \bigr)
\,.
\end{equation}
\end{lemma}
\begin{proof}
It follows from inequality \cref{ds_en_trav_f_prel_relax} that (observe the substitution of $\chi$ by $\psi_-$ on $\iLeft(t)$ and by $\psi_+$ on $\iRight(t)$ and the added nonnegative terms $-V^\ddag(m_-)$ and $-V^\ddag(m_+)$)
\[
\begin{aligned}
\eee'(t)  \le -\frac{1}{2} \ddd(t) &+ \kappa_0\int_{\iLeft(t)} \psi_-\Bigl( \frac{\cCutZero+\kappa_0}{2}v_\xi^2 + \cCutZero \bigl(V^\ddag(v) -V^\ddag(m_-)\bigr)  \Bigr) \, d\xi \\
 &+ \kappa_0\int_{\iRight(t)} \psi_+\Bigl( \frac{\cCutZero+\kappa_0}{2}v_\xi^2 + \cCutZero \bigl(V^\ddag(v) -V^\ddag(m_+)\bigr)  \Bigr) \, d\xi
\,.
\end{aligned}
\]
Thus,
\[
\begin{aligned}
&\eee'(t)  \le -\frac{1}{2} \ddd(t) \\
&+ \kappa_0\int_{\iLeft(t)} \psi_-\biggl[ \frac{\cCutZero+\kappa_0}{2} v_\xi^2+  \cCutZero \Bigl(V^\ddag(v) -V^\ddag(m_-) + \frac{1}{2\coeffEnZero}(v-m_-)^2 \Bigr)  \biggr] \, d\xi \\
&+ \kappa_0\int_{\iRight(t)} \psi_+\biggl[ \frac{\cCutZero+\kappa_0}{2} v_\xi^2 + \cCutZero\Bigl(V^\ddag(v) -V^\ddag(m_+) + \frac{1}{2\coeffEnZero}(v-m_+)^2 \Bigr)  \biggr] \, d\xi
\,.
\end{aligned}
\]
According to inequalities \vref{def_weight_en_V_ddag}, the quantities
\[
V^\ddag(v) -V^\ddag(m_\pm) + \frac{1}{2\coeffEnZero}(v-m_\pm)^2 = \frac{1}{\coeffEnZero}\Bigl(\coeffEnZero\bigl(V^\ddag(v) -V^\ddag(m_\pm)\bigr) + \frac{1}{2}(v-m_\pm)^2\Bigr)
\]
are nonnegative. As a consequence the previous inequality still holds if the factor $\cCutZero$ in front of these quantities is replaced by the greater factor $\cCutZero+\kappa_0$ and if the integration domains of the two integrals are extended to the whole real line. After these changes the inequality reads
\[
\eee'(t)  \le -\frac{1}{2} \ddd(t) + \frac{\kappa_0(\cCutZero+\kappa_0)}{\coeffEnZero}\bigl( \fff_-(t) + \fff_+(t) \bigr)
\,.
\]
Thus if $\KEFZero$ denotes the quantity $\kappa_0(\cCutZero+\kappa_0)/\coeffEnZero$, then inequality \cref{ds_en_trav_f_inprogress_relax} and \cref{lem:ds_en_trav_f_inprogress_relax} are proved. 
\end{proof}
\subsubsection{Firewalls linear decrease up to pollution}
\label{subsubsec:der_fire_dichot}
For all $t$ in $[0,+\infty)$, let
\[
\begin{aligned}
\SigmaEscMinus(t) &= \{\xi\in\rr : \abs{v(\xi,t)-m_-} > \dEsc(m_-) \} \\
\text{and}\quad
\SigmaEscPlus(t) &= \{\xi\in\rr : \abs{v(\xi,t)-m_+} > \dEsc(m_+) \} 
\,,
\end{aligned}
\]
and let us denote by $\nuFZero(m_-)$ and $\KFZero(m_-)$ (by $\nuFZero(m_+)$ and $\KFZero(m_+)$) the quantities denoted by $\nuFZero$ and $\KFZero$ in the proof of \vref{lem:dt_fire}, when the minimum point $m$ of \cref{lem:dt_fire} is replaced with $m_-$ (with $m_+$).
\begin{lemma}[firewalls linear decrease up to pollution]
\label{lem:approx_decr_fire_dichot}
For all $t$ in $[0,+\infty)$, 
\begin{equation}
\label{der_fire_dichot_prel}
\fff_\pm'(t) \le -\nuFZero(m_\pm) \fff_\pm(t) + \KFZero(m_\pm) \int_{\SigmaEscPlusMinus(t)} \psi_\pm (\xi,t) \, d\xi
\,.
\end{equation}
\end{lemma}
\begin{proof}
The proof is very similar to that of \vref{lem:decrease_firewall} or to that of \cref{lem:dt_fire} (see \cite[\GlobalRelaxationLemFirewallDecreaseUpToPollution]{Risler_globalRelaxation_2016}); however, since the definitions of the various parameters and functions are slightly different, the details of the calculations will be provided.

According to expressions \vref{ddt_loc_en,ddt_loc_L2} for the time derivatives of a localized energy and a localized $L^2$ functional, for all nonnegative time $t$,
\[
\begin{aligned}
\fff'_\pm(t) = & \int_{\rr} \Biggl[ \psi_\pm \Bigl( - \coeffEnZero v_t^2 - (v-m_\pm)\cdot \nabla V^\ddag(v) - v_\xi^2 \Bigr) \\
&+ \coeffEnZero \partial_t\psi_\pm   \Bigl( \frac{1}{2}v_\xi^2 + V^\ddag(v)-V^\ddag(m_\pm) \Bigr)  + \coeffEnZero (c\psi_\pm - \partial_\xi\psi_\pm)  v_\xi \cdot v_t \\
&+ \frac{\partial_t\psi_\pm + \partial^2_\xi\psi_\pm - c\partial_\xi\psi_\pm}{2}(v-m_\pm)^2 \Biggr] \, d\xi
\,.
\end{aligned}
\]
According to the definitions of $\psi_+$ and $\psi_-$, for all $(\xi,t)$ in $\rr\times[0,+\infty)$ (omitting the arguments $(\xi,t)$ of $\psi_\pm$ and of their derivatives),
\[
\begin{aligned}
\partial_t \psi_\pm = e^{c\xi}\partial_t \psi_{0,\pm}
\quad&\text{thus}\quad
\abs{\partial_t \psi_\pm } = \kappa_0 \cCutZero \psi_\pm \,, \\
c\psi_\pm -\partial_\xi \psi_\pm = -e^{c\xi} \partial_\xi \psi_{0,\pm}
\quad&\text{thus}\quad
\abs{c\psi_\pm -\partial_\xi \psi_\pm} = \kappa_0 \psi_\pm \,, \\
\partial_\xi^2 \psi_\pm - c \partial_\xi \psi_\pm 
= \partial_\xi(e^{c\xi} \partial_\xi \psi_{0,\pm})\qquad\qquad & \\
= e^{c\xi}(c\partial_\xi \psi_{0,\pm} + \partial_\xi^2\psi_{0,\pm})
\quad&\text{thus}\quad
\partial_\xi^2 \psi_\pm - c \partial_\xi \psi_\pm \le \kappa_0(\kappa_0+\abs{c})\psi_\pm
\end{aligned}
\]
(compare with the bounds \vref{bounds_psi_psi_s_psi_xi_etc}).
It follows that, for all nonnegative time $t$,
\[
\begin{aligned}
\fff'_\pm(t) \le \int_{\rr} \psi_\pm \Biggl[& -\coeffEnZero v_t^2 - (v-m_\pm)\cdot \nabla V^\ddag(v) - v_\xi^2 \\
&+ \coeffEnZero\, \kappa_0\, \cCutZero \Bigl( \frac{1}{2}v_\xi^2 + \abs{V^\ddag(v)-V^\ddag(m_\pm)} \Bigr)  + \coeffEnZero \, \kappa_0 \abs{v_\xi \cdot v_t} \\
&+  \frac{\kappa_0(\cCutZero + \kappa_0+ \abs{c})}{2}(v-m_\pm)^2
\Biggr] \, d\xi
\,,
\end{aligned}
\]
thus, using the inequality
\[
\coeffEnZero\, \kappa_0 \abs{v_\xi \cdot v_t} \le \coeffEnZero\, v_t^2 + \frac{\coeffEnZero\, \kappa_0^2}{4} v_\xi^2
\,,
\]
it follows that
\[
\begin{aligned}
\fff'_\pm(t) \le  \int_{\rr} \psi_\pm &\Biggl[ \biggl( \coeffEnZero\, \kappa_0 \Bigl( \frac{\cCutZero}{2} + \frac{\kappa_0}{4} \Bigr) - 1 \biggr) v_\xi^2 - (v-m_\pm)\cdot \nabla V^\ddag(v) \\
& + \coeffEnZero\, \kappa_0\, \cCutZero \abs{V^\ddag(v)-V^\ddag(m_\pm)} + \frac{\kappa_0(\cCutZero + \kappa_0+ \abs{c})}{2}(v-m_\pm)^2 \Biggr] \, d\xi
\,.
\end{aligned}
\]
According to the conditions \vref{conditions_on_coeffEnZero_and_kappaZero_and_cutZero_relax,hyp_param_relax} satisfied by $\coeffEnZero$ and $\kappa_0$ and $\cCutZero$ and $c$, it follows that
\[
\fff'_\pm(t) \le \int_{\rr} \psi_\pm \Bigl[ -\frac{1}{2}v_\xi^2 - (v-m_\pm)\cdot \nabla V^\ddag(v) + \frac{1}{4} \abs{V^\ddag(v)-V^\ddag(m_\pm)} + \frac{\eigVmin}{8} (v-m_\pm)^2 \Bigr] \, d\xi
\,,
\]
and as a consequence, 
\[
\begin{aligned}
\fff'_\pm(t) &+ \nuFZero(m_\pm) \fff_\pm(t)\le \int_{\rr} \psi_\pm \biggl[ -\frac{1}{2}(1-\nuFZero(m_\pm)\,\coeffEnZero)v_\xi^2 - (v-m_\pm)\cdot \nabla V^\ddag(v) \\
&+ \Bigl(\frac{1}{4} + \nuFZero(m_\pm)\,\coeffEnZero\Bigr)\abs{V^\ddag(v)-V^\ddag(m_\pm)} + \Bigl(\frac{\eigVmin}{8} + \frac{\nuFZero(m_\pm)}{2} \Bigr) (v-m_\pm)^2 \biggr] \, d\xi
\,.
\end{aligned}
\]
Since the quantity $\nuFZero(m_\pm)$ defined in \vref{def_nu_fire_zero} is such that
\[
\nuFZero(m_\pm)\, \coeffEnZero \le 1 
\quad\text{and}\quad
\nuFZero(m_\pm)\, \coeffEnZero \le \frac{1}{4} 
\quad\text{and}\quad
\frac{\nuFZero(m_\pm)}{2} \le \frac{\eigVmin(m_\pm)}{8} 
\]
(compare with conditions \vref{conditions_on_nuFire} on $\nuF$), it follows from the previous inequality that
\begin{equation}
\label{ds_fire_prel_relax}
\begin{aligned}
\fff'_\pm(t) &+ \nuFZero(m_\pm) \fff_\pm(t)\le \\
&\int_{\rr} \psi_\pm \Bigl[ - (v-m_\pm)\cdot \nabla V^\ddag(v) + \frac{1}{2}\abs{V^\ddag(v)-V^\ddag(m_\pm)} + \frac{\eigVmin}{4} (v-m_\pm)^2 \Bigr]\, d\xi
\,.
\end{aligned}
\end{equation}
According to \cref{v_nablaV_controls_square_around_loc_min_ddag,v_nablaV_controls_pot_around_loc_min_ddag}, the integrand of the integral at the right-hand side of this inequality is nonpositive as long as $x$ is \emph{not} in $\SigmaEscPlusMinus(t)$. Therefore this inequality still holds if the domain of integration of this integral is changed from $\rr$ to $\SigmaEscPlusMinus(t)$. Finally, according to the definition \vref{def_K_fire_zero} of the quantity $\KFZero$ and to the uniform bound \vref{att_ball_dichot} on the solution, inequality \cref{der_fire_dichot_prel} follows from \cref{ds_fire_prel_relax}. \Cref{lem:approx_decr_fire_dichot} is proved. 
\end{proof}
\subsubsection{Control over the pollution in the time derivative of the firewalls}
\label{subsubsec:flux_der_fire_dichot}
The following lemma calls upon the notation $T$ introduced for inequalities \cref{hyp_pos_relax}. 
\begin{lemma}[firewalls linear decrease up to pollution, continuation]
\label{lem:der_fire_dichot_next}
There exists a positive quantity $\KFprime$, depending only on $V$ and $m_-$ and $m_+$, such that, for every time $t$ greater than or equal to $T$,
\begin{equation}
\label{der_fire_dichot_next}
\fff_\pm'(t) \le - \nuFZero(m_\pm) \fff_\pm(t) + \KFprime \exp\Bigl(-\frac{\kappa_0\cCutZero}{2}t\Bigr)
\,.
\end{equation}
\end{lemma}
\begin{proof}
For all $t$ in $[0,+\infty)$, let 
\begin{equation}
\label{def_gggg_pm}
\mathcal{G}_-(t) = \int_{\SigmaEscMinus(t)} \psi_- (\xi,t) \, d\xi 
\quad\text{and}\quad
\mathcal{G}_+(t) = \int_{\SigmaEscPlus(t)} \psi_+ (\xi,t) \, d\xi 
\,,
\end{equation}
and
\[
\begin{aligned}
& \xiHomMinus(t) = \xHomMinus(t) - ct
\quad\text{and}\quad
\xiEscMinus(t) = \xEscMinus(t) - ct
 \\
\text{and}\quad
& \xiEscPlus(t) = \xEscPlus(t) - ct
\quad\text{and}\quad
\xiHomPlus(t) = \xHomPlus(t) - ct
\,.
\end{aligned}
\]
Assume that the time $t$ is \emph{greater than or equal to $T$}. Then it follows from hypotheses \cref{hyp_pos_relax} and from the last hypothesis of \cref{hyp_param_relax} that 
\begin{equation}
\label{hyp_relax_in_tf}
\begin{aligned}
\xiHomMinus(t)&\le -\frac{5}{3}\cCutZero t
&\quad&\text{and}\quad&
-\frac{1}{3}\cCutZero t &\le \xiEscMinus(t) \,, \\
\text{and}\quad
\xiEscPlus(t)&\le \frac{1}{3}\cCutZero t 
&\quad&\text{and}\quad&
\frac{5}{3}\cCutZero t &\le \xiHomPlus(t)
\,,
\end{aligned}
\end{equation}
see \cref{fig:setting_relax}. 
\begin{figure}[!htbp]
\centering
\includegraphics[width=\textwidth]{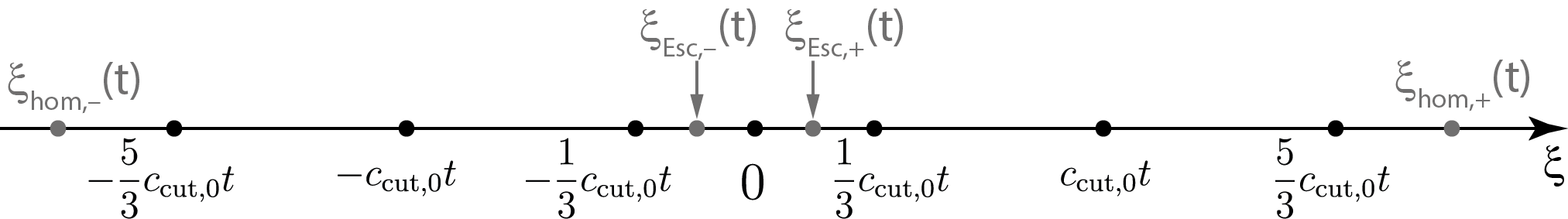}
\caption{Illustration of the notation and assumptions for the proof of \cref{prop:relax}.}
\label{fig:setting_relax}
\end{figure}
Besides, according to the definition of $\xEscPlus(t)$ and $\xEscMinus(t)$, 
\[
\begin{aligned}
\SigmaEscMinus(t) &\subset (-\infty,\xiHomMinus(t)] \cup [\xiEscMinus(t),+\infty) \\
\text{and}\quad
\SigmaEscPlus(t) &\subset (-\infty,\xiEscPlus(t)] \cup [\xiHomPlus(t),+\infty)
\,.
\end{aligned}
\]
Let us introduce the quantities
\[
\begin{aligned}
\GfrontMinus(t) &= \int_{-\infty}^{\xiHomMinus(t)} \psi_- (\xi,t) \, d\xi 
&\ &\text{and}\  &
\GbackMinus(t) &= \int_{\xiEscMinus(t)}^{+\infty} \psi_- (\xi,t) \, d\xi
\,,
\\
\text{and}\quad
\GbackPlus(t) &= \int_{-\infty}^{\xiEscPlus(t)} \psi_+ (\xi,t) \, d\xi 
&\ &\text{and}\  &
\GfrontPlus(t) &= \int_{\xiHomPlus(t)}^{+\infty} \psi_+ (\xi,t) \, d\xi 
\,.
\end{aligned}
\]
Then, it follows from the definition \cref{def_gggg_pm} of $\mathcal{G}_\pm(t)$ that
\[
\mathcal{G}_+(t) \le \GfrontPlus(t) + \GbackPlus(t)
\quad\text{and}\quad
\mathcal{G}_-(t) \le \GfrontMinus(t) + \GbackMinus(t)
\,.
\]
According to the definition of $\psi_+$ and $\psi_-$ and according to hypotheses \cref{hyp_param_relax} and inequalities \vref{hyp_relax_in_tf} it follows from explicit calculations that:
\[
\begin{aligned}
\GfrontMinus(t) &\le \frac{1}{\kappa_0+c} \exp\Bigl( \cCutZero \kappa_0 t + (\kappa_0+c) \xiHomMinus(t) \Bigr) 
&&\le \frac{1}{\kappa_0+c} \exp\Bigl(-\frac{\kappa_0\cCutZero}{2}t\Bigr) \,, \\
\GbackMinus(t) &\le \frac{1}{\kappa_0-c} \exp\Bigl(  -\cCutZero \kappa_0 t - (\kappa_0-c) \xiEscMinus(t)  \Bigr) 
&&\le \frac{1}{\kappa_0-c} \exp\Bigl(-\frac{\kappa_0\cCutZero}{2}t\Bigr) \,, \\
\GbackPlus(t) &\le \frac{1}{\kappa_0+c} \exp\Bigl( -\cCutZero \kappa_0 t + (\kappa_0+c) \xiEscPlus(t)  \Bigr) 
&&\le \frac{1}{\kappa_0-c} \exp\Bigl(-\frac{\kappa_0\cCutZero}{2}t\Bigr) \,, \\
\GfrontPlus(t) &\le \frac{1}{\kappa_0-c} \exp\Bigl( \cCutZero \kappa_0 t - (\kappa_0-c) \xiHomPlus(t)  \Bigr) 
&&\le \frac{1}{\kappa_0+c} \exp\Bigl(-\frac{\kappa_0\cCutZero}{2}t\Bigr) 
\,.
\end{aligned}
\]
It follows that
\[
\mathcal{G}_\pm(t)\le \frac{2\kappa_0}{\kappa_0^2-c^2} \exp\Bigl(-\frac{\kappa_0\cCutZero}{2}t\Bigr)
\,,
\]
and since according to the conditions \cref{hyp_param_relax} the quantity $\abs{c}$ is not larger than $\kappa_0/10$, it follows that
\[
\mathcal{G}_\pm(t)\le \frac{5}{2\kappa_0} \exp\Bigl(-\frac{\kappa_0\cCutZero}{2}t\Bigr)
\,,
\]
As a consequence, introducing the positive quantity $\KFprime$ defined as
\[
\KFprime = \frac{5}{2\kappa_0}\max\bigl(\KFZero(m_-),\KFZero(m_+)\bigr)
\,,
\]
inequality \cref{der_fire_dichot_next} follows from inequality \cref{der_fire_dichot_prel}. \Cref{lem:der_fire_dichot_next} is proved.
\end{proof}
\subsubsection{Nonnegativity of firewalls}
%
\begin{lemma}[nonnegativity of firewalls]
\label{lem:nonnegativity_firewalls}
For every nonnegative time $t$, 
\begin{equation}
\label{nonnegativity_firewalls}
\fff_\pm(t)\ge 0
\,.
\end{equation}
\end{lemma}
\begin{proof}
Inequality \cref{nonnegativity_firewalls} follows from inequality \vref{def_weight_en} and the definition \vref{def_firewalls_relax} of $\fff_\pm(t)$.
\end{proof}
\subsubsection{Energy decrease up to pollution}
\label{subsubsec:dt_en_final_dichot}
\begin{lemma}[energy decrease up to pollution]
\label{lem:dt_en_final_dichot}
There exist positive quantities $\KE$ and $\nuE$, depending only on $V$ and $m_-$ and $m_+$, such that, for every time $t$ greater than or equal to $T$,
\begin{equation}
\label{dt_en_final_dichot}
\eee'(t) \le -\frac{1}{2} \ddd(t) + \KE \exp\bigl(-\nuE (t-T)\bigr)
\,.
\end{equation}
\end{lemma}
\begin{proof}
Let 
\[
\nuE = \min\Bigl( \nuFZero(m_-), \nuFZero(m_+),\frac{\kappa_0\cCutZero}{4}\Bigr)
\,.
\]
According to Grönwall's inequality, it follows from inequalities \cref{der_fire_dichot_next} of \cref{lem:der_fire_dichot_next} that, for every time $t$ greater than or equal to $T$, 
\begin{align}
\fff_\pm(t) &\le \exp\bigl(- \nuFZero(m_\pm)(t-T)\bigr)\fff_\pm(T) \nonumber\\
&\qquad + \KFprime\int_T^t \exp\bigl(- \nuFZero(m_\pm)(t-s)\bigr)\exp\Bigl(-\frac{\kappa_0\cCutZero}{2}s\Bigr) \, ds \nonumber\\
&\le \exp\bigl(- \nuE(t-T)\Biggl(\fff_\pm(T) + \KFprime\exp\Bigl(-\frac{\kappa_0\cCutZero}{2}T\Bigr)\times \nonumber\\
& \qquad\int_T^t \exp\Bigl(- \bigl(\nuFZero(m_\pm)-\nuE\bigr)(t-s)\Bigr)\exp\biggl(-\Bigl(\frac{\kappa_0\cCutZero}{2}-\nuE\Bigr)(s-T)\biggr)\, ds \Biggr) \nonumber\\
&\le \exp\bigl(- \nuE(t-T)\bigr)\left(\fff_\pm(T) + \KFprime\int_T^t \exp\Bigl(-\frac{\kappa_0\cCutZero}{4}(s-T)\Bigr)\, ds\right)\nonumber\\
&\le \left(\fff_\pm(T) + \frac{4 \KFprime}{\kappa_0 \cCutZero}\right)\exp\bigl(- \nuE(t-T)\bigr)
\,.
\label{der_fire_dichot}
\end{align}
According to the $H^1_\text{ul}$-bound \vref{att_ball_X_dichot} for the solution, there exists a positive quantity $\fffMax$, depending only on $V$ and $m_+$ and $m_-$, such that the following estimates hold:
\[
\fff_+(T)\le \fffMax
\quad\text{and}\quad
\fff_-(T)\le \fffMax
\,.
\]
Thus, introducing the positive quantity
\[
\KE = 2 \KEFZero \Bigl(  \fffMax + \frac{4 \KFprime}{\kappa_0 \cCutZero}\Bigr)
\,,
\]
inequality \cref{dt_en_final_dichot} follows from inequalities \vref{ds_en_trav_f_inprogress_relax,der_fire_dichot}. \Cref{lem:dt_en_final_dichot} is proved.
\end{proof}
Inequality \cref{dt_en_final_dichot} is the key ingredient that will be applied in the next \cref{subsec:low_bd_en}. 
\subsection{Nonnegative asymptotic energy}
\label{subsec:low_bd_en}
\subsubsection{Notation}
Let us keep the notation and hypotheses of the previous \namecref{subsec:relax_sc_stand}. For every quantity $c$ close enough to $0$ so that conditions \vref{hyp_param_relax} be satisfied, let us denote by 
\[
\begin{aligned}
&v^{(c)}(\cdot,\cdot)
\quad\text{and}\quad
\chi^{(c)}(\cdot,\cdot)
\quad\text{and}\quad
\eee^{(c)}(\cdot) 
\quad\text{and}\quad
\SigmaEscPlus^{(c)}(\cdot) \\
\text{and}\quad
&\xiEscPlus^{(c)}(\cdot)
\quad\text{and}\quad
\xiHomPlus^{(c)}(\cdot)
\quad\text{and}\quad
\ddd^{(c)}(\cdot)
\end{aligned}
\]
the objects that were defined in \cref{subsec:relax_sc_stand} (with the same notation except the ``$(c)$'' superscript that is here to remind that these objects depend on the quantity $c$). For every such $c$, let us introduce the quantity $\eee^{(c)}(+\infty)$ in $\{-\infty\}\cup\rr$ defined as
\[
\eee^{(c)}(+\infty) = \liminf_{t\to+\infty} \eee^{(c)}(t)
\,,
\]
and let us call ``asymptotic energy at the speed $c$'' this quantity.
According to estimate \cref{dt_en_final_dichot} above, for every such $c$, 
\begin{equation}
\label{convergence_energy_almost_standing_frame}
\eee^{(c)}(t)\to \eee^{(c)}(+\infty)
\quad\text{as}\quad
t\to + \infty
\,.
\end{equation}
\subsubsection{Statement}
The aim of this \namecref{subsec:low_bd_en} is to prove the following proposition.
\begin{proposition}[nonnegative asymptotic energy at zero speed]
\label{prop:nonneg_asympt_en}
The quantity $\eee^{(0)}(+\infty)$ (the asymptotic energy at speed zero) is nonnegative.
\end{proposition}
The proof proceeds through the following lemmas and corollaries, that are rather direct consequences of the relaxation scheme set up in the previous \cref{subsec:relax_sc_stand}, and in particular of the inequality \cref{dt_en_final_dichot} for the time derivative of the energy.

Since according to the definition \vref{def_V_ddag} of $V^\ddag$ the maximum of $V^\ddag(m_+)$ and $V^\ddag(m_-)$ is assumed to be equal to $0$, it may be assumed (without loss of generality) that
\begin{equation}
\label{assumpt_V_ddag_of_m_plus_is_zero}
V^\ddag(m_+)=0
\,.
\end{equation}
\subsubsection{Nonnegative asymptotic energy at small nonzero speeds}
\begin{lemma}[nonnegative asymptotic energy at small nonzero speeds]
\label{lem:nonneg_en_prel}
For every quantity $c$ close enough to $0$ so that conditions \vref{hyp_param_relax} be satisfied, if in addition $c$ is positive, then
\[
\eee^{(c)}(+\infty) \ge 0
\,.
\]
\end{lemma}
\begin{proof}
Let $c$ be a positive quantity, close enough to $0$ so that conditions \cref{hyp_param_relax} be satisfied. With the notation of \cref{subsec:relax_sc_stand} (for the relaxation scheme in a frame travelling at the speed $c$), for all $t$ in $[0,+\infty)$, 
\[
\begin{aligned}
\eee^{(c)}(t) &= \int_{\rr} \chi^{(c)}(\xi,t) \Bigl( \frac{1}{2}v^{(c)}_\xi(\xi,t)^2 + V^\ddag\bigl(v^{(c)}(\xi,t)\bigr)\Bigr) \, d\xi \\
&\ge \int_{\rr} \chi^{(c)}(\xi,t) V^\ddag\bigl(v^{(c)}(\xi,t)\bigr) \, d\xi \\
&\ge \int_{\SigmaEscPlus^{(c)}(t)} \chi^{(c)}(\xi,t)  V^\ddag\bigl(v^{(c)}(\xi,t)\bigr)  \, d\xi
\,.
\end{aligned}
\]
Thus, considering the global minimum value of $V^\ddag$:
\[
V^\ddag_{\min} = \min_{u\in\rr^d} V^\ddag(u) \le 0 
\,,
\]
it follows that, for all $t$ in $[0,+\infty)$, 
\[
\begin{aligned}
\eee^{(c)}(t)&\ge V^\ddag_{\min} \int_{\SigmaEscPlus^{(c)}(t)} \chi^{(c)}(\xi,t) \, d\xi \\
&\ge V^\ddag_{\min}\biggl[ \int_{-\infty}^{\xiEscPlus^{(c)}(t)} \chi^{(c)}(\xi,t) \, d\xi + \int_{\xiHomPlus^{(c)}(t)}^{+\infty} \chi^{(c)}(\xi,t) \, d\xi \biggr] 
\,.
\end{aligned}
\]
According to the definition of $\chi^{(c)}(\xi,t)$ (see \vref{subsubsec:en_dichot}), this yields:
\[
\begin{aligned}
\eee^{(c)}(t)&\ge V^\ddag_{\min}\Bigl[ \int_{-\infty}^{\xiEscPlus^{(c)}(t)} \exp(c\xi) \, d\xi + \int_{\xiHomPlus^{(c)}(t)}^{+\infty} \exp\bigl(c\xi - \kappa_0(\xi - \cCutZero t)\bigr) \, d\xi \Bigr] \\
&= V^\ddag_{\min}\biggl[\frac{1}{c}\exp\left(c \xiEscPlus^{(c)}(t)\right) + \frac{1}{\kappa_0 -c} \exp\bigl(\kappa_0\cCutZero t-(\kappa_0-c)\xiHomPlus^{(c)}(t) \bigr) \biggr]
\,.
\end{aligned}
\]
Let us consider the arguments of the two exponential functions in this last expression.  
\begin{itemize}
\item Since $c$ is assumed to be positive and according to hypothesis \textup{(\hyperlink{hypNoInv}{\hypNoInvRef})}, the quantity $\xiEscPlus^{(c)}(t) = \xEscPlus(t) -ct$ goes to $-\infty$ as time goes to $+\infty$.
\item If $t$ is greater than or equal to $T$, then, according to inequality \vref{hyp_relax_in_tf} on $\xiHomPlus(t)$,
\[
\kappa_0\cCutZero t-(\kappa_0-c)\xiHomPlus^{(c)}(t) \le \cCutZero t \Bigl(\kappa_0 - \frac{5}{3}(\kappa_0-c)\Bigr) 
\,,
\]
and according to inequality \cref{hyp_param_relax} bounding from above the ratio $c/\kappa_0$, it follows that the right-hand side of this inequality goes to $-\infty$ as time goes to $+\infty$.
\end{itemize}
Thus $\eee^{(c)}(t)$ is bounded from below by a quantity which goes to $0$ as time goes to $+\infty$. \Cref{lem:nonneg_en_prel} is proved. 
\end{proof}
\subsubsection{Almost nonnegative energy at small nonzero speeds}
\begin{corollary}[almost nonnegative energy at small nonzero speeds]
\label{cor:low_bd_en_c}
For every quantity $c$ close enough to $0$ so that conditions \vref{hyp_param_relax} be satisfied, if in addition $c$ is positive, then, for every time $t$ greater than or equal to $T$,
\[
\eee^{(c)}(t) \ge -\frac{\KE}{\nuE} \exp\bigl(-\nuE (t-T)\bigr)
\,.
\]
\end{corollary}
\begin{proof}
The proof follows from previous \cref{lem:nonneg_en_prel} and inequality \cref{dt_en_final_dichot} of \cref{lem:dt_en_final_dichot}.
\end{proof}
\subsubsection{Continuity of energy with respect to the speed at \texorpdfstring{$c=0$}{c=0}}
\begin{lemma}[continuity of energy with respect to the speed at $c=0$]
\label{lem:cont_en_c}
For every $t$ in $(0,+\infty)$, 
\[
\eee^{(c)}(t)\to \eee^{(0)}(t)
\quad\text{as}\quad
c \to 0
\,.
\]
\end{lemma}
\begin{proof}
For all $t$ in $(0,+\infty)$, 
\[
\eee^{(0)}(t) = \int_{\rr} \chi^{(0)}(x,t)\Bigl( \frac{1}{2}u_x(x,t)^2 + V^\ddag\bigl( u(x,t)\bigr) \Bigr) \, dx 
\,,
\]
and, for every quantity $c$ close enough to $0$ so that hypotheses \vref{hyp_param_relax} be satisfied (substituting the notation $\xi$ used in \cref{subsec:relax_sc_stand} with $x$), 
\[
\eee^{(c)}(t) = \int_{\rr} \chi^{(c)}(ct+x,t)\Bigl( \frac{1}{2}u_x(ct+x,t)^2 + V^\ddag\bigl( u(ct+x,t)\bigr) \Bigr) \, dx
\,,
\]
and the result follows from the continuity of $\chi^{(c)}(\cdot,\cdot)$ with respect to $c$, the exponential decrease to $0$ of $\chi^{(c)}(x,t)$ when $x\to\pm\infty$, and the $\HoneulAlone$-bounds \vref{att_ball_X_dichot} for the solution.  
\end{proof}
\subsubsection{Almost nonnegative energy in a standing frame}
\begin{corollary}[almost nonnegative energy in a standing frame]
\label{cor:low_bd_en_0}
For every time $t$ greater than or equal to $T$,
\[
\eee^{(0)}(t) \ge -\frac{\KE}{\nuE} \exp\bigl(-\nuE (t-T)\bigr)
\,.
\]
\end{corollary}
\begin{proof}
This lower bound follows from \cref{cor:low_bd_en_c} and \cref{lem:cont_en_c}.
\end{proof}
\Cref{prop:nonneg_asympt_en} (``nonnegative asymptotic energy'') follows from \cref{cor:low_bd_en_0}.
\subsection{End of the proof of \texorpdfstring{\cref{prop:relax}}{Proposition \ref{prop:relax}}}
\label{subsec:end_pf_prop_relax}
\begin{lemma}[integrability of dissipation in the standing frame]
\label{lem:dissip_is_integrable}
The map
\[
t\mapsto \ddd^{(0)}(t) 
\]
is integrable on $[0,+\infty)$.
\end{lemma}
\begin{proof}
The statement follows from \cref{prop:nonneg_asympt_en} (``nonnegative asymptotic energy'') and from the upper bound \cref{dt_en_final_dichot} on the time derivative of energy.
\end{proof}
\begin{lemma}[relaxation]
\label{lem:relaxation}
The following limits hold:
\begin{equation}
\label{ut_app_zero_between_xHomMinus_and_xHomPlus}
\sup_{x\in[\xHomMinus(t),\xHomPlus(t)]} \abs{ u_t(x,t) } \to 0
\end{equation}
and, for every quantity $\varepsilon$ which is positive and smaller than $\abs{\cHomMinus}$ and than $\cHomPlus$, 
\begin{equation}
\label{relax_outside_center_area}
\sup_{x\in[\xHomMinus(t),-\varepsilon t]}\abs{u(x,t)-m_-} \to 0
\quad\text{and}\quad
\sup_{x\in[\varepsilon t,\xHomPlus(t)]}\abs{u(x,t)-m_+} \to 0
\end{equation}
as time goes to $+\infty$. 
\end{lemma}
\begin{proof}
For every quantity $\varepsilon$ which is positive and smaller than $\abs{\cHomMinus}$ and than $\cHomPlus$, according to hypothesis \textup{(\hyperlink{hypNoInv}{\hypNoInvRef})} and proceeding as in the proof of \vref{lem:cv_right_of_front}, the limits \cref{relax_outside_center_area} follow. In addition, according to the bounds \cref{bound_u_ut_ck} on the solution, it follows from these limits that 
\begin{equation}
\label{ut_goes_to_zero_beyond_eps_t}
\sup_{x\in[\xHomMinus(t),-\varepsilon t]\cup[\varepsilon t,\xHomPlus(t)]} \abs{ u_t(x,t) }\to 0
\quad\text{as}\quad
t\to+\infty
\,.
\end{equation}
It remains to prove that 
\[
\sup_{x\in[-\cCutZero t, \cCutZero t]} \abs{ u_t(x,t) } \to 0
\quad\text{as}\quad
t\to +\infty
\,.
\]
Let us proceed by contradiction and assume that this limit doesn't hold. Then, there exist a positive quantity $\varepsilon$ and a sequence $(x_n,t_n)_{n\in\nn}$ satisfying $t_n\to+\infty$ as $n\to+\infty$, such that, for every nonnegative integer $n$, 
\begin{equation}
\label{contradiction_hyp_u_t_not_small}
\abs{u_t(x_n,t_n)}\ge\varepsilon
\,.
\end{equation}
According \cref{ut_goes_to_zero_beyond_eps_t}, it may be assumed (up to dropping the first terms of the sequence $(x_n,t_n)_{n\in\nn}$) that, for every $n$ in $\nn$, $x_n$ belongs to the interval $[-\cCutZero t_n,\cCutZero t_n]$. 
By compactness (\cref{lem:compactness}), there exists an entire solution $\widebar{u}$ of system \cref{init_syst} such that, up to replacing $(x_n,t_n)_{n\in\nn}$ by a subsequence, with the notation of \cref{compactness},
\begin{equation}
\label{convergence_up_to_subsequence}
D^{2,1}u\bigl( x_n + \cdot, t_n + \cdot\bigr)\to D^{2,1}\tilde{u}
\quad\text{as}\quad
n\to+\infty
\,,
\end{equation}
uniformly on every compact subset of $\rr^2$. It follows from \cref{contradiction_hyp_u_t_not_small,convergence_up_to_subsequence} that the quantity $\abs{\widebar{u}_t(0,0)}$ is positive, so that the quantity
\[
\int_0^1 \left(\int_{\rr}e^{-\kappa_0\abs{\xi}}\abs{\widebar{u}_t(\xi,s)}^2 \, d\xi\right)\, ds
\]
is also positive. This quantity is greater than or equal to the quantity
\[
\liminf_{n\to+\infty}\int_0^1 \ddd^{(0)}(t_n+s)\, ds
\]
which is therefore also positive, a contradiction with the integrability of $t\mapsto\ddd^{(0)}(t)$ (\cref{lem:dissip_is_integrable}). \Cref{lem:relaxation} is proved.
\end{proof}
\begin{lemma}[$V(m_-)$ equals $V(m_+)$]
\label{lem:V_of_m_minus_equals_V_of_m_plus}
The following equalities hold:
\[
V^\ddag(m_-)=V^\ddag(m_+)=0\,,
\quad\text{or in other words:}\quad
V(m_-)=V(m_+)
\,.
\]
\end{lemma}
\begin{proof}
It follows from the definition \cref{def_V_ddag} of $V^\ddag$ and from the assumption \cref{assumpt_V_ddag_of_m_plus_is_zero} that $V^\ddag(m_+)$ equals $0$ and that $V^\ddag(m_-)$ is nonpositive. If $V^\ddag(m_-)$ was negative, then, according to the assertions \cref{ut_app_zero_between_xHomMinus_and_xHomPlus,relax_outside_center_area} above (and according to the bounds \vref{bound_u_ut_ck} for the solution), the following estimate would hold:
\[
\eee^{(0)}(t)\sim V^\ddag(m_-)\, \cCutZero \, t
\quad\text{as}\quad
t\to+\infty
\,,
\]
a contradiction with \cref{prop:nonneg_asympt_en}. \Cref{lem:V_of_m_minus_equals_V_of_m_plus} is proved.
\end{proof}
\begin{lemma}[convergence towards asymptotic energy]
\label{lem:asympt_energy_for_various_bounds}
For all quantities $c_-$ in\\ 
$(\cHomMinus,0)$ and $c_+$ in $(0,\cHomPlus)$, 
\begin{equation}
\label{equ_asympt_energy_for_various_bounds}
\int_{c_- t}^{c_+ t} \Bigl(\frac{1}{2}u_x(x,t)^2+V\bigl(u(x,t)\bigr)-V(m_\pm)\Bigr) \to \eee^{(0)}(+\infty)
\end{equation}
as time goes to $+\infty$. 
\end{lemma}
\begin{proof}
It follows from the limit \cref{convergence_energy_almost_standing_frame} that the quantity 
\[
\eee^{(0)}(t) = \int_{\rr}\chi^{(0)}(x,t)\Bigl(\frac{1}{2}u_x(x,t)^2 + V^\ddag\bigl(u(x,t)\bigr)\Bigr) \, dx
\]
goes to $\eee^{(0)}(+\infty)$ as time goes to $+\infty$, and according to \cref{lem:V_of_m_minus_equals_V_of_m_plus} $V\ddag(\cdot)$ equals $V(\cdot)-V(m_\pm)$. The fact that the same asymptotic behaviour holds for the integral in \cref{equ_asympt_energy_for_various_bounds} (whatever the values of $c_-$ and $c_+$) can thus (once again) be derived from inequality \cref{dt_fire} of \vref{lem:dt_fire} (as in the proof of \vref{lem:cv_right_of_front}). \Cref{lem:asympt_energy_for_various_bounds} is proved. 
\end{proof}
\begin{proof}[Proof of \cref{prop:relax}]
All statements of \cref{prop:relax} have been proved: 
\begin{enumerate}
\item equality between $V(m_-)$ and $V(m_+)$ is stated in \cref{lem:V_of_m_minus_equals_V_of_m_plus};
\item limits \cref{dissip_goes_to_zero_prop_relax,sol_goes_to_m_plus_minus_prop_relax} are stated in \cref{lem:relaxation};
\item according to \cref{prop:nonneg_asympt_en} the quantity $\eee^{(0)}(+\infty)$ is nonnegative, and, denoting by $\eeeResAsympt[u]$ this quantity, the limit \cref{asympt_energy_prop_relax} is stated in \cref{lem:asympt_energy_for_various_bounds}. 
\end{enumerate}
\Cref{prop:relax} is proved. 
\end{proof}
\section{Proof of \texorpdfstring{\cref{thm:main,prop:residual_asympt_en}}{Theorem \ref{thm:main} and Proposition \ref{prop:residual_asympt_en}}}
\label{sec:proof_thm1}
The aim of this \namecref{sec:proof_thm1} is to combine \cref{prop:inv_cv,prop:relax} (``invasion implies convergence'' and ``no invasion implies relaxation'') to complete the proof of \cref{thm:main,prop:residual_asympt_en}.
Not much remains to be said. 

As everywhere else, let us consider a function $V$ in $\ccc^2(\rr^d,\rr)$ satisfying the coercivity hypothesis \cref{hyp_coerc}. Let us assume in addition that the generic hypotheses \cref{hyp_gen} hold for the potential $V$, and let us consider two points $m_-$ and $m_+$ in $\mmm$ and a bistable solution $(x,t)\mapsto u(x,t)$ of system \cref{init_syst} connecting $m_-$ to $m_+$. 

Let us introduce the following notation:
\[
\begin{aligned}
\cHomMinus &= -(\cnoesc + 1)
\quad&\text{and}\quad
\xHomMinus(t) &= \cHomMinus t \\
\text{and}\quad
\cHomPlus &= \cnoesc + 1
\quad&\text{and}\quad
\xHomPlus(t) &= \cHomPlus  t
\,.
\end{aligned}
\]
According to the results of \cref{subsec:def_fire_zero,subsec:up_bd_inv_vel} (namely \vref{lem:inv} and inequality \vref{dt_fire}), for every positive quantity $L$, 
\[
\begin{aligned}
\sup_{\xi\in[-L,L]} \abs{u \bigl( \xHomMinus(t) + \xi\bigr) - m_-} &\to 0
\quad\text{as}\quad
t\to +\infty \\
\text{and}\quad
\sup_{\xi\in[-L,L]} \abs{u \bigl( \xHomPlus(t) + \xi\bigr) - m_+} &\to 0
\quad\text{as}\quad
t\to +\infty
\,,
\end{aligned}
\]
in other words hypotheses \textup{(\hyperlink{hypHom}{\hypHomRef})} of \cref{prop:relax} and \textup{(\hyperlink{hypHomRight}{\hypHomRightRef})} of \cref{prop:inv_cv} (and the symmetric hypothesis on the left) hold for these definitions. Thus, the two points $\xEscPlus(t)$ and $\xEscMinus(t)$ (for all $t$ in $[0,+\infty)$) and the corresponding asymptotic mean-sup speeds $\cEscPlus$ and $\cEscMinus$ may be defined exactly as in \cref{subsec:def_hyp_dichot}.

At this stage, two cases must be distinguished:
\begin{enumerate}
\item ``invasion'': $\max \bigl(\abs{\cEscMinus} , \cEscPlus \bigr)$ is positive;
\item ``no invasion'': $\cEscMinus$ is nonnegative and $\cEscPlus$ is nonpositive. 
\end{enumerate}

If the first case ``invasion'' occurs, then \cref{prop:inv_cv} can be applied (either to the right, or to the left, or on both sides). According to the statements of this proposition, behind the front(s) (to the right, or to the left, or on both sides) approached by the solution, ``new'' points $\xHomMinusNext(t)$ and $\xHomPlusNext(t)$ and new speeds $\cHomMinusNext$ and $\cHomPlusNext$ can be defined, for which hypothesis 
\textup{(\hyperlink{hypHom}{\hypHomRef})} is again satisfied. Thus the same process (definition of new ``Escape points'' and definition of their asymptotic mean-sup speeds and discussion at above about the signs of these speeds) can be repeated. And it can be repeated again, as long as case 1 (invasion) occurs. 

At each ``invasion'', the new bistable front approached by the solution replaces a stable homogeneous equilibrium by another one where (according to \vref{V_of_u_plus_minus_V_of_u_minus}) the potential takes a lower value. Since according to hypothesis \textup{(\hyperlink{hypCriticalValues}{\hypCriticalValuesRef})} the set of critical values of $V$ is finite, case 2 (no invasion) must eventually occur. 

When case 2 (no invasion) occurs, \cref{prop:relax} applies: the time derivative $u_t(x,t)$ of the solution goes to $0$ uniformly in the ``centre'' area between the two propagating terraces of travelling fronts, as time goes to $+\infty$, and there exists a nonnegative quantity $\eeeResAsympt[u]$ (``residual asymptotic energy'') such that, for all quantities $c_-$ in $(\cHomMinus,0)$ and $c_+$ in $(0,\cHomPlus)$, the limit \cref{asympt_energy_prop_relax} holds. In this case, all the arguments of \cite[\GlobalRelaxationSecApproachSetBistableStatSolAndConvergenceStandTerr]{Risler_globalRelaxation_2016} apply, up to the following minor changes: in \cite[\GlobalRelaxationSecApproachSetBistableStatSol]{Risler_globalRelaxation_2016} the quantity $c$ should be chosen smaller than $\min\left(\abs{\cHomMinus},\cHomPlus\right)$, and in \cite[\GlobalRelaxationSecApproachSetBistableStatSolAndConvergenceStandTerr]{Risler_globalRelaxation_2016} the statements should be restricted to the behaviour of the solution, at time $t$, in the space interval $[-ct,ct]$. Those arguments show that the solution approaches, in this ``centre area'', a standing terrace $\tttCentre$ of bistable stationary solutions (this completes the proof of \cref{thm:main}), and that the residual asymptotic energy $\eeeResAsympt[u]$ of the solution is equal to the energy $\eee[\tttCentre]$ of this standing terrace (this proves \cref{prop:residual_asympt_en}). The proof of \cref{thm:main,prop:residual_asympt_en} is complete.
\section{Some properties of the profiles of stationary solutions and travelling waves}
\label{sec:prop_stand_trav}
As everywhere else, let us consider a function $V$ in $\ccc^2(\rr^d,\rr)$ satisfying the coercivity hypothesis \cref{hyp_coerc}. 
\subsection{Asymptotics at the two ends of space}
\label{subsec:asympt_behav_tw}
Let $c$ denote a \emph{nonnegative} quantity, and let us consider the differential system governing the profiles of waves travelling at the speed $c$ (or ``standing'' if $c$ equals $0$): 
\begin{equation}
\label{syst_trav_front_bis}
\phi''=-c\phi'+\nabla V(\phi) \,.
\end{equation}
If $\xi\mapsto\phi(\xi)$ is a global solution of system \cref{syst_trav_front_bis}, recall that the \emph{$\alpha$-limit set of $\phi$} is the set $\displaystyle{\lim_\alpha \phi}$ defined as
\[
\lim_\alpha \phi = \bigcap_{\zeta\in\rr}\widebar{\{\phi(\xi):\xi<\zeta\}}
\,,
\]
and that this $\alpha$-limit set is a closed and connected subset of $\rr^d$. The following lemma is called upon in this paper only for a \emph{positive} quantity $c$, but the proof given below only requires that $c$ be \emph{nonnegative} (this allows to call upon this lemma in the case where $c$ equals zero in \cite{Risler_globalRelaxation_2016}).
\begin{lemma}[spatial asymptotics of the profiles of travelling waves]
\label{lem:asympt_behav_tw_2}
Let $m$ be a point of $\mmm$, and let $\xi\mapsto \phi(\xi)$ be a global solution of the differential system \cref{syst_trav_front_bis} satisfying 
\begin{equation}
\label{hyp_lem_asympt_behav_tw_2}
\abs{\phi(\xi)-m}\le \dEsc(m)
\quad\text{for every }
\xi \text{ in }[0,+\infty)
\quad\text{and}\quad
\phi(\cdot) \not\equiv m
\,.
\end{equation}
Then the following assertions hold. 
\begin{enumerate}
\item Both quantities $\abs{\phi(\xi)-m}$ and $\abs{\phi'(\xi)}$ go to $0$ as $\xi$ goes to $+\infty$.
\label{item:cv_spatial_asymptotics_tw}
\item For all $\xi$ in $[0,+\infty)$, the scalar product $\bigl(\phi(\xi)-m\bigr) \cdot \phi'(\xi)$ is negative. 
\label{item:transv_spatial_asymptotics_tw}
\item For all $\xi$ in $(0,+\infty)$, the quantity $\abs{\phi(\xi)-m}$ is smaller than $\dEsc(m)$. 
\label{item:closer_spatial_asymptotics_tw}
\item The supremum $\sup_{\xi\in\rr}\abs{\phi(\xi)-m}$ is larger than $\dEsc(m)$. 
\label{item:escape_spatial_asymptotics_tw}
\item In addition to assertion \cref{item:cv_spatial_asymptotics_tw} above, the quantities
\[
e^{c\xi}\abs{\phi(\xi)-m}
\quad\text{and}\quad
e^{c\xi}\abs{\phi'(\xi)}
\]
go to $0$ at an exponential rate when $\xi$ goes to $+\infty$. 
\label{item:exp_cv_spatial_asymptotics_tw}
\item If in addition $\phi(\cdot)$ is bounded on $\rr$ and $c$ is positive, then there exists a real quantity $V_{-\infty}$ such that
\[
\phi'(\xi)\xrightarrow[\xi\to -\infty]{}0
\quad\text{and}\quad
\lim_\alpha \phi \subset \bigl\{u\in\rr^d:V(u)=V_{-\infty}\text{ and } \nabla V(u)=0\bigr\}
\,,
\]
and such that
\begin{equation}
\label{V_of_m_plus_minus_V_at_minus_infinity}
V(m) - V_{-\infty} = c \int_\rr \phi'(\xi)^2 \, d\xi
\,.
\end{equation}
\label{item:minus_infty_asymptotics_tw}
\end{enumerate}
\end{lemma}
\begin{proof}
Let $m$ be a point of $\mmm$, and let $\xi\mapsto \phi(\xi)$ be a global solution of the differential system \cref{syst_trav_front_bis} satisfying \cref{hyp_lem_asympt_behav_tw_2}. For all $\xi$ in $\rr$, let us introduce the quantities
\begin{equation}
\label{def_Qloc_and_H}
\Qloc(\xi)=\frac{1}{2} \bigl(\phi(\xi)-m\bigr)^2
\quad\text{and}\quad
H(\xi) = \frac{1}{2}\phi'(\xi)^2 - V\bigl(\phi(\xi)\bigr)
\,.
\end{equation}
Then, it follows from system \cref{syst_trav_front_bis} that, for all $\xi$ in $\rr$, 
\[
\Qloc'(\xi) = \bigl(\phi(\xi)-m\bigr)\cdot \phi'(\xi)
\quad\text{and}\quad
\Qloc''(\xi) + c \Qloc'(\xi) = \phi'(\xi)^2 + \bigl(\phi(\xi)-m\bigr)\cdot\nabla V\bigl(\phi(\xi)\bigr)
\,,
\]
and thus it follows from inequality \vref{v_nablaV_controls_square_around_loc_min} that, for all $\xi$ in $[0,+\infty)$, 
\begin{equation}
\label{lower_bound_Q_second_plus_c_Q_prime}
\Qloc''(\xi)+c\Qloc'(\xi)\ge \phi'(\xi)^2+\eigVmin(m) \Qloc(\xi)
\,.
\end{equation} 
On the other hand, it again follows from system \cref{syst_trav_front_bis} that, for all $\xi$ in $\rr$, 
\begin{equation}
\label{expression_H_prime}
H'(\xi) = - c \phi'(\xi)^2 
\,,
\end{equation}
thus the function $H(\cdot)$ is non-increasing and thus bounded from above on $[0,+\infty)$. Besides, it follows from the first assumption of \cref{hyp_lem_asympt_behav_tw_2} that $V\bigl(\phi(\cdot)\bigr)$ is also bounded from above on $[0,+\infty)$. Thus, it follows from the expression \cref{def_Qloc_and_H} of $H(\xi)$ that $\abs{\phi'(\cdot)}$ is bounded on $[0,+\infty)$, and as a consequence the same is true for $\Qloc'(\cdot)$ and thus for $\Qloc'(\cdot)+c\Qloc(\cdot)$. Since the quantity $\Qloc'(\xi)+c\Qloc(\xi)$ is, according to inequality \cref{lower_bound_Q_second_plus_c_Q_prime}, non decreasing with respect to $\xi$ on $[0,+\infty)$, it follows that this quantity must converge towards a finite limit as $\xi$ goes to $+\infty$; and thus it follows from inequality \cref{lower_bound_Q_second_plus_c_Q_prime} that both functions $\abs{\phi'(\cdot)}^2$ and $\Qloc(\cdot)$ are integrable on $[0,+\infty)$. Since the derivative of the function $\Qloc(\cdot)$ is bounded on this interval, this function $\Qloc(\cdot)$ must converge towards $0$ as $\xi$ goes to $+\infty$. Thus $\phi(\xi)$ goes to $m$ as $\xi$ goes to $+\infty$, and as a consequence $V\bigl(\phi(\xi)\big)$ goes to $V(m)$ as $\xi$ goes to $+\infty$. Thus, since the function $H(\cdot)$ must converge to a finite limit at $+\infty$, it follows that $\phi'(\xi)^2$ must also go to a finite limit when $\xi$ goes to $+\infty$. Since $\phi'(\cdot)^2$ is integrable on $[0,+\infty)$, its limit at $+\infty$ must be $0$. Assertion \cref{item:cv_spatial_asymptotics_tw} is proved. 

It follows from assertion \cref{item:cv_spatial_asymptotics_tw} that the function $\Qloc'(\cdot)+c\Qloc(\cdot)$ converges towards $0$ at $+\infty$, and according to inequality \cref{lower_bound_Q_second_plus_c_Q_prime} (and since according to assumption \cref{hyp_lem_asympt_behav_tw_2} $\phi(\cdot)$ is not identically equal to $m$) the derivative of this function is positive on $[0,+\infty)$. It follows that this function is negative on $[0,+\infty)$, thus so is $\Qloc'(\cdot)$, which proves assertion \cref{item:transv_spatial_asymptotics_tw}. As a consequence, $\Qloc(\cdot)$ is strictly decreasing on $[0,+\infty)$, which proves assertion \cref{item:closer_spatial_asymptotics_tw}.

To prove assertion \cref{item:escape_spatial_asymptotics_tw}, let us proceed by contradiction and assume that $\abs{\phi(\xi)-m}$ remains not larger that $\dEsc(m)$ for all $\xi$ in $\rr$. It follows that the function $\Qloc(\cdot)$ is bounded on $\rr$. It also follows that inequality \cref{lower_bound_Q_second_plus_c_Q_prime} holds for all $\xi$ in $\rr$, thus the function $\Qloc'(\cdot)+c\Qloc(\cdot)$ is non-decreasing (and even, since $\phi(\cdot)$ is not identically equal to $m$, strictly increasing) on $\rr$. This function $\Qloc'(\cdot)+c\Qloc(\cdot)$ cannot go to $-\infty$ as $\xi$ goes to $-\infty$, or else (since $c$ is nonnegative) the same would be true for $\Qloc'(\cdot)$, a contradiction with the fact that $\Qloc(\cdot)$ is bounded. Thus $\Qloc'(\xi)+c\Qloc(\xi)$ converges to a finite limit as $\xi$ goes to $-\infty$. Again, according to inequality \cref{lower_bound_Q_second_plus_c_Q_prime}, this shows that the functions $\Qloc(\cdot)$ and $\abs{\phi'(\cdot)}^2$ are integrable, this time on $\rr$. It follows from the expression \cref{expression_H_prime} of $H'(\xi)$ that the function $H(\cdot)$ must converge towards a finite limit at $-\infty$. Thus $H(\cdot)$ is bounded on $\rr$, so that the same is true for $\phi'(\cdot)$. It follows that $\Qloc'(\cdot)$ is bounded on $\rr$, so that $\Qloc(\xi)$ must go to $0$ as $\xi$ goes to $-\infty$. Thus $V\bigl(\phi(\xi)\big)$ goes to $V(m)$ as $\xi$ goes to $-\infty$. As a consequence, since the function $H(\cdot)$ must converge towards a finite limit at $-\infty$, it follows that $\phi'(\xi)^2$ must also go to a finite limit when $\xi$ goes to $-\infty$. Since $\phi'(\cdot)^2$ is integrable on $\rr$, it must converge towards $0$ at $-\infty$. Finally the function $\Qloc'(\cdot)+c\Qloc(\cdot)$ converges towards the same limit (that is zero) at $-\infty$ and at $+\infty$, a contradiction with the fact that this function is strictly increasing on $\rr$. 

To prove assertion \cref{item:exp_cv_spatial_asymptotics_tw}, let us denote by $D$ the Hessian matrix $D^2V(m)$. Then the linearization at $(m,0)$ of system~\cref{syst_trav_front_bis} reads
\begin{equation}
\label{syst_fronts_lin}
\phi''=-c\phi'+D\phi\iff
\begin{pmatrix}
\phi \\ \psi
\end{pmatrix}'
=
\begin{pmatrix}
0&1 \\ D&-c
\end{pmatrix}
\begin{pmatrix}
\phi \\ \psi
\end{pmatrix}
\,,
\end{equation}
and the set of eigenvalues of this linear system is
\[
\left\{ \frac{-c\pm\sqrt{c^2+4\mu}}{2} : \mu \text{ is an eigenvalue of } D\right\}
\,.
\]
Since every eigenvalue of $D$ is positive, it follows that every nonpositive eigenvalue of system \cref{syst_fronts_lin} is less than $-c$, and assertion \cref{item:exp_cv_spatial_asymptotics_tw} follows. 

It remains to prove assertion \cref{item:minus_infty_asymptotics_tw}. Let us assume that $\phi(\cdot)$ is bounded on $\rr$ and that $c$ is positive. Since $H(\cdot)$ is non-increasing and since $H(\xi)$ goes to $-V(m)$ as $\xi$ goes to $+\infty$, there is a quantity $H_{-\infty}$ in $[-V(m),+\infty)\cup\{+\infty\}$ such that $H(\xi)$ goes to $H_{-\infty}$ as $\xi$ goes to $-\infty$. For all $\xi$ in $\rr$, let us introduce the quantity
\[
\Qglob(\xi) = \frac{1}{2}\phi(\xi)^2
\,.
\]
It follows from system \cref{syst_trav_front_bis} that, for all $\xi$ in $\rr$, 
\[
\Qglob'(\xi) = \phi(\xi)\cdot \phi'(\xi)
\quad\text{and}\quad
\Qglob''(\xi) + c \Qglob'(\xi) = \phi'(\xi)^2 + \phi(\xi)\cdot\nabla V\bigl(\phi(\xi)\bigr)
\,.
\] 
It follows from hypothesis \cref{hyp_coerc} that there exist positive quantities $\epsCoerc$ and $\Kcoerc$ such that, for every $u$ in $\rr^d$, 
\[
u\cdot \nabla V(u) \ge \epsCoerc u^2 - \Kcoerc
\,, 
\]
so that, for all $\xi$ in $\rr$, 
\[
\Qglob''(\xi) + c \Qglob'(\xi) \ge \phi'(\xi)^2 + \epsCoerc \phi(\xi)^2 - \Kcoerc
\,.
\]
For every negative quantity $\xi$, integrating this inequality between $\xi$ and $0$ yields:
\begin{equation}
\label{integrated_inequality_Qglob}
\Qglob'(0)-\Qglob'(\xi) + c \bigl(\Qglob(0)-\Qglob(\xi)\bigr) \ge \int_\xi^0 \bigl(\phi'(\zeta)^2 + \epsCoerc \phi(\zeta)^2 - \Kcoerc\bigr)\, d\zeta
\,.
\end{equation}
If the quantity $H_{-\infty}$ was equal to $+\infty$, then it would follow from the expression \cref{def_Qloc_and_H} of $H(\cdot)$ that $\abs{\phi'(\xi)}$ goes to $+\infty$ as $\xi$ goes to $-\infty$, and it would follow from inequality \cref{integrated_inequality_Qglob} that $\Qglob'(\xi)$ goes to $-\infty$ as $\xi$ goes to $-\infty$, a contradiction with the boundedness of $\phi(\cdot)$ and thus of $\Qglob(\cdot)$ on $\rr$. Thus $H_{-\infty}$ is finite. Then it follows from the expression \cref{expression_H_prime} of $H'(\cdot)$ and from the assumption that $c$ is positive that $\phi'(\cdot)$ is square integrable, and it follows from the expression \cref{def_Qloc_and_H} of $H(\cdot)$ that $\phi'(\cdot)$ is bounded, and it follows from system \cref{syst_trav_front_bis} that $\phi''(\cdot)$ is bounded (those two assertions hold in a neighbourhood of $-\infty$, thus also on $\rr$). Thus $\phi'(\xi)$ goes to $0$ as $\xi$ goes to $-\infty$, and if the quantity $-H_{-\infty}$ is denoted by $V_{-\infty}$, it follows that $V\bigl(\phi(\xi)\bigr)$ goes to $V_{-\infty}$ as $\xi$ goes to $-\infty$. It follows that the set $\displaystyle{\lim_\alpha \phi}$ must belong to the set of critical points of $V$ in the level set $V^{-1}\bigl(\{V_{-\infty}\}\bigr)$, and equality \cref{V_of_m_plus_minus_V_at_minus_infinity} follows from multiplying system \cref{syst_trav_front_bis} by $\phi'(\xi)$ and integrating over $\rr$. 
\Cref{lem:asympt_behav_tw_2} is proved. 
\end{proof}
\subsection{Vanishing energy of travelling waves invading a stable homogeneous equilibrium}
\label{subsec:zero_en_tw}
The following observation was, to the knowledge of the author, first made by Muratov (\cite[Proposition 3.3]{Muratov_globVarStructPropagation_2004}). For sake of completeness, a proof is given below.
\begin{lemma}[the energy of a bounded travelling wave invading a stable equilibrium vanishes]
\label{lem:zero_en_tw}
For every point $m$ of $\mmm$, every positive quantity $c$, and every function $\phi$ in the set $\Phi_c(m)$ of bounded profiles of waves travelling at the speed $c$ and ``invading'' the stable equilibrium $m$, the weighted energy
\[
\int_{\rr} e^{c\xi}\Bigl(\frac{1}{2}\phi'(\xi)^2 + V\bigl(\phi(\xi)\bigr) - V(m)\Bigr)\, d\xi
\,,
\]
is a convergent integral and its value is $0$. 
\end{lemma}
\begin{proof}
Let us take a point $m$ in $\mmm$, a positive quantity $c$, and a function $\phi$ in the set $\Phi_c(m)$ of bounded profiles of waves travelling at the speed $c$ and ``invading'' the stable equilibrium $m$, and let us consider the function
\[
\bar{\xi}\mapsto \eee_{c,\phi}(\bar{\xi}) = \int_{\rr} e^{c(\xi-\bar{\xi})}\Bigl(\frac{1}{2}\phi'(\xi)^2 + V\bigl(\phi(\xi)\bigr)- V(m) \Bigr) \, d\xi = e^{-c\bar{\xi}} \eee_{c,\phi}(0)
\,.
\]
According to assertions \cref{item:exp_cv_spatial_asymptotics_tw,item:minus_infty_asymptotics_tw} of \cref{lem:asympt_behav_tw_2} above this integral converges thus this function is defined for all $\bar{\xi}$ in $\rr$. Let us consider the derivative $\eee_{c,\phi}'(\bar{\xi})$ of this function at some quantity $\bar{\xi}$. On the one hand,
\begin{equation}
\label{derivative_of_eee_c_phi_first_expression}
\eee_{c,\phi}'(\bar{\xi}) = -c \eee_{c,\phi}(\bar\xi)
\,,
\end{equation}
and on the other hand, since 
\[
\eee_{c,\phi}(\bar{\xi}) = \int_{\rr} e^{c\xi}\Bigl(\frac{1}{2}\phi'(\bar{\xi}+\xi)^2 + V\bigl(\phi(\bar{\xi}+\xi) \bigr) - V(m) \Bigr)\, d\xi
\,,
\]
it follows, differentiating under the integral, that
\[
\eee_{c,\phi}'(\bar{\xi}) = \int_{\rr} e^{c\xi}  \phi'(\bar{\xi}+\xi) \cdot \Bigl( \phi''(\bar{\xi}+\xi) + \nabla V\bigl( \phi(\bar{\xi}+\xi) \bigr) \Bigr) \, d\xi
\]
thus according to the differential system \cref{syst_trav_front_bis},
\[
\begin{aligned}
\eee_{c,\phi}'(\bar{\xi}) &= \int_{\rr} e^{c\xi}  \phi'(\bar{\xi}+\xi) \cdot \Bigl( 2 \phi''(\bar{\xi}+\xi) + c \phi'(\bar{\xi}+\xi) \Bigr) \, d\xi \\
&= e^{-c\bar{\xi}} \int_{\rr} e^{c\xi} \phi'(\xi) \cdot \Bigl( 2 \phi''(\xi) + c \phi'(\xi) \Bigr) \, d\xi \\
&= e^{-c\bar{\xi}} \int_{\rr} \frac{d}{d\xi} \bigl( e^{c\xi}\phi'(\xi)^2 \bigr) \, d\xi 
\,.
\end{aligned}
\]
Since $c$ is positive and according to assertion \cref{item:minus_infty_asymptotics_tw} of \cref{lem:asympt_behav_tw_2}, the quantity $e^{c\xi}\phi'(\xi)^2$ goes to $0$ as $\xi$ goes to $-\infty$, and according to assertion \cref{item:exp_cv_spatial_asymptotics_tw} of \cref{lem:asympt_behav_tw_2} the same quantity also goes to $0$ as $\xi$ goes to $+\infty$. It follows that $\eee_{c,\phi}'(\bar{\xi})$ vanishes. Since $c$ is nonzero, the conclusion follows from \cref{derivative_of_eee_c_phi_first_expression}. \Cref{lem:zero_en_tw} is proved. 

\end{proof}
\subsubsection*{Acknowledgements} 
I am indebted to Thierry Gallay and Romain Joly for their help and interest through numerous fruitful discussions. 
\emergencystretch=1em
\printbibliography 
\bigskip
\mySignature
\end{document}